\documentclass[1 [leqno, 11pt]{amsart}
\usepackage{amssymb,  amsmath, amsmath, latexsym, amssymb, amsfonts, amsbsy, amsthm,mathtools, graphicx, CJKutf8, CJKnumb, CJKulem, extarrows, color, dsfont,verbatim}
\usepackage{amssymb, amsmath,amsmath,latexsym,amssymb,amsfonts,amsbsy, amsthm,mathtools,graphicx,CJKutf8,CJKnumb,CJKulem,color}
\usepackage{graphicx,epsfig}
\usepackage{graphicx}
\usepackage{amssymb}
\usepackage{graphicx}
\usepackage{graphicx,epsfig}
\usepackage{amsmath}
\usepackage{mathrsfs}
\usepackage{amssymb}

\usepackage{amssymb,mathrsfs,amsfonts,amsmath,amsbsy,tikz}\usepackage{fontenc}\usepackage{textcomp}

\usepackage{amssymb,  amsmath, amsmath, latexsym, amssymb, amsfonts, amsbsy, amsthm,mathtools, graphicx, CJKutf8, CJKnumb, CJKulem, extarrows, color, dsfont,verbatim}

\numberwithin{equation}{section}

\allowdisplaybreaks

\let\ep=\epsilon

\let\pa=\partial

\def\sech{{\rm{sech}}}

\def\curl{\mathop{\rm curl}\nolimits}

\def\gep{\gamma_\epsilon}

\newcommand{\beq}{\begin{equation}}
\newcommand{\eeq}{\end{equation}}
\newcommand{\ben}{\begin{eqnarray}}
\newcommand{\een}{\end{eqnarray}}
\newcommand{\beno}{\begin{eqnarray*}}
\newcommand{\eeno}{\end{eqnarray*}}

\newtheorem{Theorem}{Theorem}[section]

\newtheorem{definition}[Theorem]{Definition}
\newtheorem{lemma}[Theorem]{Lemma}
\newtheorem{proposition}[Theorem]{Proposition}

\newtheorem{remark}[Theorem]{Remark}

\newtheorem{Lemma}[Theorem]{Lemma}

\newtheorem{Corollary}[Theorem]{Corollary}

\begin{document}
\begin{CJK*}{GBK}{song}
\title[Stability and instability of Kelvin--Stuart cat's-eye flows]{\textbf {On the  stability and instability of Kelvin--Stuart cat's-eye flows}}

\author{Shasha Liao}
\address{Department of Mathematics, Georgia Institute of Technology, Atlanta, USA}
\email{ssliao@outlook.com}

\author{Zhiwu Lin}
\address{School of Mathematical Sciences, Fudan University,  200433, Shanghai, P. R. China}
\email{zwlin@fudan.edu.cn}

\author{Hao Zhu}
\address{School of Mathematics, Nanjing University,  210093, Nanjing, Jiangsu, P. R. China \&
 Faculty of Mathematics, University of Vienna, Oskar-Morgenstern-Platz 1, 1090 Vienna, Austria}
\email{haozhu@nju.edu.cn \& hao.zhu@univie.ac.at}

\date{\today}

\maketitle

\begin{abstract}
Kelvin--Stuart vortices are classical mixing layer flows with many applications in fluid mechanics, plasma physics and astrophysics. We prove that the whole family of Kelvin--Stuart vortices is nonlinearly orbitally stable
for co-periodic perturbations, and linearly unstable for multi-periodic and modulational perturbations.
This verifies a long-standing conjecture  since the discovery  of the Kelvin--Stuart  cat's-eye flows in the 1960s.
Kelvin--Stuart cat's eyes also appear as magnetic islands which are magnetostatic equilibria for the planar ideal MHD equations in plasmas.
We prove nonlinear orbital stability of Kelvin--Stuart magnetic islands for co-periodic perturbations, and give the first rigorous proof of coalescence instability for the whole family, which is important for magnetic reconnection.
\end{abstract}
\tableofcontents
\section{Introduction}
\subsection{Motivation and background}
In the 1960s, Kelvin--Stuart cat's-eye flows were discovered as one of the classical explicit families of non-parallel  steady solutions of the two-dimensional incompressible Euler equation. At about the same time, this structure also appeared in plasma physics as a family of static magnetic island equilibria of the planar ideal MHD equations. Their stability properties are tied to physically important phenomena, including vortex pairing and merger in hydrodynamics, as well as coalescence instability and magnetic reconnection in plasmas. In contrast with the extensively studied shear-flow setting, the linearized operators around these non-shear equilibria depend genuinely on both spatial variables. This inherent non-separability precludes a standard modal decomposition, rendering the stability analysis mathematically intractable by conventional means. This is one of the main reasons why the stability problem for the whole Kelvin--Stuart family has remained unresolved for more than half a century.

The purpose of this paper is to give a complete stability/instability theory for the whole Kelvin--Stuart family in the original unbounded strip. For the 2D Euler equation, we prove that every Kelvin--Stuart vortex is spectrally stable for co-periodic perturbations, nonlinearly orbitally stable in the co-periodic class, and linearly unstable for all multi-periodic and all modulational perturbations. For the planar ideal MHD equations, we prove co-periodic nonlinear orbital stability of the whole family of Kelvin--Stuart magnetic islands and give the first rigorous proof of coalescence instability for the whole family. A central new ingredient is a nonlinear change of variables revealing a hidden isospectral structure of the Kelvin--Stuart family, which is the structural reason why a complete analysis of the whole family becomes possible.

\subsubsection{Kelvin--Stuart  cat's-eye flows}
Consider the 2D Euler equation  for an incompressible inviscid fluid
 \begin{equation}\label{euler}
 \partial_t \vec{u} + (\vec{u}\cdot \nabla) \vec{u}  = -\nabla p, \quad \nabla \cdot \vec{u} = 0,
 \end{equation}
where $\vec{u} = (u_1, u_2)$ is the velocity field  and $p$ is the pressure. We study the fluid in the unbounded domain
$\Omega=\mathbb{T}_{2\pi}\times \mathbb{R}$, where $\mathbb{T}_{2\pi}$ means that the period is $2\pi$ in the $x$ direction.
The stream function $\psi$  satisfies   $ \vec{u} = \nabla ^\bot \psi = (\psi_y, -\psi_x)$.
 Taking the curl of \eqref{euler} gives the following evolution equation for the scalar-valued vorticity $\omega =- \Delta \psi$:
\begin{equation} \label{vor}
\partial_t \omega +  \{\omega , \psi\} = 0,
\end{equation}
where  $\{\omega , \psi\}  := \partial_y\psi \partial_x \omega - \partial_x\psi\partial_y \omega$ is the canonical Poisson bracket.

In 1967, Stuart \cite{stuart1967finite} found a family of exact solutions  to the 2D steady  Euler equation \eqref{vor}, now known as Kelvin--Stuart cat's-eye flows. Their stream functions are given explicitly by
\begin{equation}\label{catseye}
\psi_\epsilon(x,y) =  \ln \left(\frac{\cosh (y) + \epsilon \cos (x)}{\sqrt{1-\epsilon^2}} \right),\quad x\in\mathbb{T}_{2\pi},\quad y\in\mathbb{R}
\end{equation}
 with parameter $\epsilon \in [0, 1)$.  The streamlines for $\ep=0.5$ are shown in Figure \ref{fig:firstFig}. Such streamline patterns were already described by Kelvin \cite{kelvin1880disturbing}, and they model the rolling-up of a mixing layer into a chain of co-rotating vortices \cite{Tabeling-Perrin-Fauve1987}. Such cat's-eye flows have many applications. For example, their streamline patterns arise naturally in models of wave-current interaction in the ocean \cite{Martin2018}. They have also been proposed as potentially effective mixing mechanisms in industrial applications \cite{Rossi-Doorly-Kustrin2013}, and have been used to describe  tropical storms \cite{Dunkerton-Montgomery-Wang2009}. The vorticity and  velocity of the  Kelvin--Stuart cat's-eye flows are given by
 \begin{align} \label{steadyw}
\omega_\epsilon =& -\Delta \psi_\epsilon = \frac{-(1- \epsilon^2 )}{(\cosh y + \epsilon\cos x)^2},\\\label{steadyv}
\vec{u}_\epsilon =& (u_{\ep,1},u_{\ep,2})= (\partial_y\psi_{\epsilon }, -\partial_x\psi_{\epsilon })= \left(\frac{\sinh(y)}{\cosh y + \epsilon\cos x}, \frac{\epsilon\sin(x)}{\cosh y + \epsilon\cos x}\right).
\end{align}
  This family connects two important limiting regimes:
\begin{itemize}
\item \textbf{Shear case} ($\epsilon = 0$): hyperbolic tangent flow $$\psi_0 = \ln(\cosh (y)), \quad \omega_0 =\frac{-1}{\cosh ^2 (y)}, \quad \vec{u}_0 = (\tanh y, 0).$$
\item \textbf{Singular case} ($\epsilon =1$): a point-vortex row with vorticity concentrating at $$\{\cdots,(-3\pi,0),(-\pi,0),(\pi, 0), (3\pi,0),\cdots\}.$$
\end{itemize}

\begin{figure}[ht]
    \centering
\includegraphics[width=0.48\textwidth]{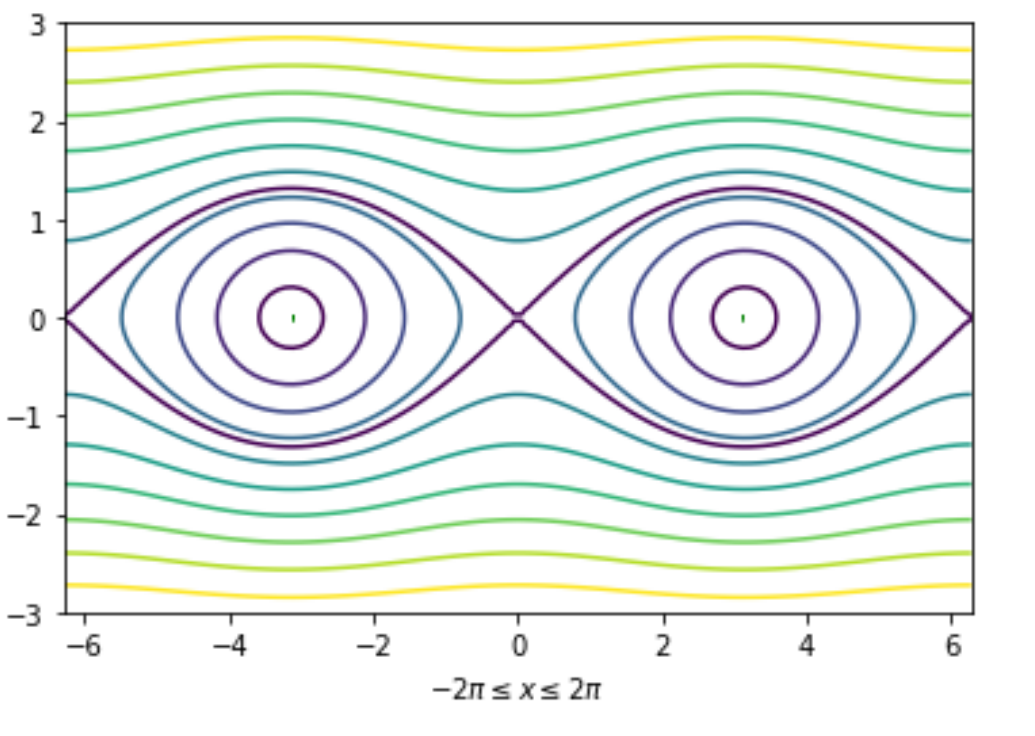}
\caption{Streamlines  for $\ep=0.5$}
	\label{fig:firstFig}
\end{figure}

The stream functions  satisfy the  Liouville's equation
 \begin{equation}\label{elip}
  - \Delta \psi_\epsilon = g(\psi_\epsilon)\quad \text{with}\quad g(\psi_\epsilon) =- e^{-2\psi_\epsilon},
 \end{equation}
 where $\ep\in[0,1)$.
Liouville-type equations $\Delta \phi = c_1 e^{c_2\phi}$ have important applications in fluid dynamics, space plasma physics, high-energy physics and differential geometry, where $c_1$ and $c_2$ are real numbers. We refer to the references \cite{Liouville1853,Richardson1921,Schindler2006,Bogatov-Kichenassamy22} for background on this broader context. Several exact solutions of Liouville's equation, including the Kelvin--Stuart cat's eyes, are known in the literature; see \cite{Crowdy97} and the references therein. In particular, Taylor \cite{Taylor2018} obtained a two-parameter family of cat's-eye solutions to \eqref{elip} with stream functions of the form
\begin{align}\label{Taylor 2 parameter solutions}
\psi_{\gamma,\sigma}(x,y)=\ln\left({\gamma\over2}e^y+{\sigma^2+1\over2\gamma}e^{-y}+\sigma\cos(x)\right),
\end{align}
where $\gamma$ and $\sigma$ are two independent positive numbers. Let $\sigma^2={\ep^2\over 1-\ep^2}$ and $\gamma={\kappa \over \sqrt{1-\ep^2}}$ for $\ep\in(0,1)$ and $\kappa>0$. Then
\begin{align*}
\psi_{\gamma,\sigma}(x,y)=\phi_{\kappa,\ep}(x,y)\triangleq\ln\left({{\kappa\over 2}e^y+{1\over 2\kappa}e^{-y}+\ep\cos(x)\over \sqrt{1-\ep^2}}\right)=\ln\left({\cosh(y+\ln(\kappa))+\ep\cos(x)\over \sqrt{1-\ep^2}}\right),
\end{align*}
which is exactly the translation of Stuart's solution $\psi_\epsilon(x,y)$ (see \eqref{catseye}) by $\ln(\kappa)$ in the $y$ direction.

Stability and instability of Stuart's exact solutions have attracted much attention since their discovery. A classical question, already raised in Stuart's original paper \cite{stuart1967finite}, is whether wavelength-doubling instability holds for the whole Kelvin--Stuart family \eqref{steadyw}. In modern terms, this asks whether every Kelvin--Stuart vortex is linearly unstable under double-periodic, or more generally multi-periodic, perturbations, a mechanism closely related to vortex pairing and merger. In the two extreme cases, Lamb \cite{lamb1932hydrodynamics} proved double-periodic instability for the singular row of point vortices corresponding to $\epsilon=1$, while for $\epsilon=0$ it follows from \cite{lin2003instability} that the hyperbolic tangent shear flow is unstable under all multi-periodic perturbations. For $0<\epsilon\ll 1$, Kelly \cite{kelly1967stability} numerically observed double-periodic instability. Motivated by these observations,
 Stuart himself conjectured in \cite{stuart1967finite} that ``{\it from a stability analysis, the wavelength doubling phenomenon might be
typical for all or many members of the class.}"
This conjectural picture is supported by further numerical studies: Pierrehumbert and Windnall \cite{pierrehumbert1982two} found double-periodic instability for $0\le \epsilon\le 0.3$ with a real most unstable eigenvalue, and Klaassen and Peltier \cite{Klaassen-Peltier1987} observed a slowly growing double-periodic mode at $\epsilon=0.1$. Triple-periodic instability is also physically relevant for collective vortex amalgamation, since it may drive the merger of three vortices into one or two \cite{Klaassen-Peltier1989}.

Modulational instability is a stronger notion than multi-periodic instability. Here, the perturbation takes the form $\omega(x,y)e^{i\alpha x}$, where $\omega$ is $2\pi$-periodic in $x$ and $\alpha\in\mathbb{R}\setminus\mathbb{Z}$. This type of instability is classical in the water-wave theory, beginning with the Benjamin--Feir instability of small-amplitude Stokes waves \cite{Benjamin-Feir1967}; for later developments, see, for instance, \cite{Bridges-Mielke1995,Berti-Maspero-Ventura2022,Nguyen-Strauss2023,Chen-Su2020}. More broadly, modulational instability has been studied in many dispersive models; see the survey \cite{Bronski-Hur-Johnson2016}. For a class of dispersive systems, it is proved that linear modulational instability  implies nonlinear instability \cite{Jin-Liao-Lin2019}.

For co-periodic perturbations, Holm, Marsden and Ratiu \cite{holm1986nonlinear} proved a nonlinear stability result in a truncated domain bounded by a pair of steady streamlines, and only for a restricted subfamily of Kelvin--Stuart vortices. For example,
in the domain bounded exactly by the separatrices (i.e. the trapped region), nonlinear stability holds true only for $\ep\in[0,\ep_0]$ according to their theory, where $\ep_0\approx0.525$. In the original unbounded strip $\Omega$, however, even the linear co-periodic stability of the whole family was previously unknown, let alone nonlinear orbital stability.

\subsubsection{Kelvin--Stuart  magnetic islands}
Independently, in 1965, Schmid-Burgk \cite{Schmid-Burgk1965} found the same family when studying self-gravitating isothermal gas layers, where \eqref{catseye} acts as the scaled gravitational potential. At about the same time, Fadeev {\it et al.} \cite{Fadeev et al-1965} found that the Kelvin--Stuart cat's-eye profiles also give static magnetic island equilibria of the planar ideal MHD equations, where \eqref{catseye} serves as the magnetic potential, see \eqref{Kelvin--Stuart cat's eyes-mhd-m-p}. For a plasma model including both gravitational and magnetic fields, Fleischer \cite{Fleischer1998} constructed a magnetohydrostatic equilibrium whose gravitational potential reduces to Schmid-Burgk's solution in the pure gravitational limit and whose magnetic flux function reduces to the solution of Fadeev {\it et al.} in the MHD limit.

We also study the stability and instability of the Kelvin--Stuart magnetic islands introduced in \cite{Fadeev et al-1965}. In the planar incompressible ideal MHD equations, writing the velocity field and magnetic field as $\vec{v}=\nabla^{\bot}\psi$ and $\vec{B}=\nabla^{\bot}\phi$, the scalar vorticity and current density are $\omega=-\Delta\psi$ and $J=-\Delta\phi$, where $\psi$ and $\phi$ are the scalar stream function and magnetic potential, respectively. The planar ideal MHD equations take the form (see \cite{holm1985nonlinear})
 \begin{align}\label{mhd}
\left\{ \begin{array}{lll} \partial_t \phi=\{\psi,\phi\},\\
 \partial_t \omega=\{\psi,\omega\}+\{J,\phi\}.
 \end{array} \right.
\end{align}
The Kelvin--Stuart magnetic island equilibria are $(\omega=0,\phi_\ep)$, where the steady magnetic potential
 \begin{align}\label{Kelvin--Stuart cat's eyes-mhd-m-p}\phi_\ep(x,y)=\ln \left(\frac{\cosh (y) + \epsilon \cos (x)}{\sqrt{1-\epsilon^2}} \right),\quad x\in\mathbb{T}_{2\pi},\quad y\in\mathbb{R}
 \end{align}
satisfies
 \begin{align*}
J^\epsilon =& -\Delta \phi_\epsilon = \frac{-(1- \epsilon^2 )}{(\cosh y + \epsilon\cos x)^2} =g(\phi_\epsilon),
\\
\vec{B}^\epsilon =& (B_{1,{\ep}},B_{2.{\ep}})= (\partial_y\phi_{\epsilon }, -\partial_x\phi_{\epsilon })= \left(\frac{\sinh(y)}{\cosh y + \epsilon\cos x}, \frac{\epsilon\sin(x)}{\cosh y + \epsilon\cos x}\right).
\end{align*}

For a chain of magnetic islands in a current slab, neighboring islands tend to merge in the nonlinear evolution. This coalescence instability is important in magnetic reconnection; see \cite{Pontin-Priest2022,Priest1985,Priest-Forbes2000}. At the linear level, it corresponds to double-periodic instability of $(\omega=0,\phi_{\ep})$. Finn and Kaw \cite{Finn-Kaw1977} numerically found that these magnetic island solutions are coalescence unstable for $\ep$ not close to $0$, predicting an instability threshold at some $\ep_0\in(0,1)$, with instability for $\ep\in(\ep_0,1)$ and stability for $\ep\in[0,\ep_0]$.
Pritchett and Wu \cite{Pritchett-Wu1979} numerically  obtained the instability growth rates as $\ep\to0$, thereby refuting the Finn-Kaw threshold hypothesis. Later, Bondeson  \cite{Bondeson1983} confirmed the coalescence instability for small $\ep$. However, no rigorous proof was previously known for the whole family.

For co-periodic perturbations, similar to the Euler case \cite{holm1986nonlinear}, Holm {\it et al.} \cite{holm1985nonlinear} proved nonlinear orbital stability of Kelvin--Stuart magnetic islands in a truncated domain for a restricted range of $\ep$. In the same truncated setting, Tassi \cite{Tassi2022} obtained a related stability result in a hot-ion model for a smaller parameter range. Nonlinear orbital stability of the whole family of Kelvin--Stuart magnetic islands in the original unbounded strip has remained open.

\subsection{Main results}

\subsubsection{Main results for the 2D Euler equation}

We now state our main results for the 2D Euler equation. Theorems \ref{main result2-multi-periodic perturbations}-\ref{main result4-nonlinear orbital stability} show that the Kelvin--Stuart family is unstable under all multi-periodic and all modulational perturbations, but stable in the co-periodic class. In particular, Theorem \ref{main result2-multi-periodic perturbations} gives a complete answer to Stuart's wavelength-doubling conjecture, while Theorem \ref{main result4-nonlinear orbital stability} proves co-periodic nonlinear orbital stability in the original unbounded strip.

 First, we provide a complete answer to Stuart's conjecture.

\begin{Theorem}\label{main result2-multi-periodic perturbations}
Let
$0 \leq \ep < 1$. Then the steady state $\omega_\ep$ in \eqref{steadyw} is linearly unstable for $2m\pi$-periodic perturbations, where $m\geq2$ is an integer.
\end{Theorem}

Linear instability for multi-periodic perturbations  implies modulational instability for some but not all rational modulational parameters, and thus far from all modulational  parameters. Our next  result is to cover all modulational parameters, which is stronger  than Theorem \ref{main result2-multi-periodic perturbations}.

\begin{Theorem}\label{main result3-modulation-unstable}
Let $0 \leq \ep < 1$. Then the steady state $\omega_\ep$ in \eqref{steadyw} is linearly modulationally unstable  for all $\alpha\in\mathbb{R}\setminus\mathbb{Z}$.
\end{Theorem}

Based on Theorems \ref{main result2-multi-periodic perturbations}-\ref{main result3-modulation-unstable}, it is expected to prove nonlinear instability for  multi-periodic or localized perturbations. To prove nonlinear instability for  localized perturbations in $\mathbb{R}^2$,
one may construct the unstable initial data in the form
$\omega_\ep(x,y)+2Re(\int_I\omega_u(\alpha;x,y) e^{i\alpha x}d\alpha)$, where
$I$ is a small interval near the most unstable frequency $\alpha_0$,
 $\omega_u(\alpha;x,y)$ is an eigenfunction of the eigenvalue $\lambda(\alpha)$ for the linearized operator  $J_{\ep,\alpha} L_{\ep,\alpha}$,
 $\{\lambda(\alpha):\alpha\in I\}$ is a curve of  unstable eigenvalues   bifurcating from the most unstable eigenvalue $\lambda(\alpha_0)$,   and $J_{\ep,\alpha}, L_{\ep,\alpha}$ are defined in \eqref{def-J-ep-al}-\eqref{def-L-ep-al}.

We next prove spectral stability of the whole family of Kelvin--Stuart vortices for co-periodic perturbations. We first state
our linear result.

\begin{Theorem}\label{main result1-co-periodic perturbations}
Let $0 \leq \ep < 1$. Then
the steady state $\omega_\ep$ in \eqref{steadyw} is spectrally stable for co-periodic perturbations.
\end{Theorem}

Based on spectral stability in Theorem \ref{main result1-co-periodic perturbations}, our main result for co-periodic perturbations is  that the whole family of Kelvin--Stuart vortices is nonlinear orbitally stable.

\begin{Theorem}\label{main result4-nonlinear orbital stability}
Let $\ep_0 \in (0, 1)$. For any $\kappa>0$, there exists $\delta=\delta(\ep_0,\kappa)>0$ such that if
\begin{align}\label{initial-condition-main result4-nonlinear orbital stability-2}
\inf_{(x_0,y_0)\in\Omega} d(\tilde \omega_0,\omega_{\ep_0}(x+x_0,y+y_0))+\inf_{(x_0,y_0)\in\Omega}\|\tilde \omega_0-\omega_{\ep_0}(x+x_0,y+y_0)\|_{L^2(\Omega)}<\delta,\end{align}
 then for any $t\geq0$, we have
\begin{align}\label{onlinear orbital stability-goal}
\inf_{(x_0,y_0)\in\Omega}d(\tilde \omega(t),\omega_{\ep_0}(x+x_0,y+y_0))<\kappa,
\end{align}
where $\tilde \omega(t)=\curl(\vec{v}(t))$, $\vec{v}(t)$ is a weak solution to the nonlinear 2D Euler equation \eqref{euler} with the initial vorticity
 \begin{align}\label{def-X-non-ep}
\tilde \omega(0)=\tilde \omega_0\in Y_{non}= \bigg\{ \tilde \omega|\tilde\omega\in L^1(\Omega)\cap L^2(\Omega),
  y\tilde \omega\in L^1(\Omega),\tilde \omega<0,\iint_{\Omega}\tilde \omega dxdy=-4\pi\bigg\}.
\end{align} The distance functional $d$ is defined by
\begin{align}\nonumber
d(\tilde\omega,\omega_\ep)&=\iint_{\Omega}(h(\tilde \omega ) - h(\omega_\ep)  - \psi_\ep(\tilde  \omega-\omega_\ep)+(G*(\tilde \omega-\omega_\ep))(\tilde \omega-\omega_\ep))dxdy, \quad\tilde \omega\in Y_{non},
\end{align}
where $h(s)={1\over2}(s-s\ln(-s))$ for $s<0$ and $G(x,y)=-{1\over 4\pi}\ln(\cosh(y)-\cos(x))$.
\end{Theorem}

\medskip
\noindent

The initial smallness assumption in \eqref{initial-condition-main result4-nonlinear orbital stability-2} is not optimal; see Remark \ref{rem-weaken-condition-main-result4} for a refinement in which the smallness of the $L^2$ initial vorticity perturbation is replaced by an $L^2$-bound on the initial vorticity. Theorem \ref{main result4-nonlinear orbital stability} also yields quantitative control of the vorticity in $L^a(\Omega)$ for every
$a\in[1,2)$, and in $L^2(\Omega)$ under an additional $L^3$-bound on the initial vorticity.
See Remark \ref{rem-La-control-main-result4}.

\bigskip

\noindent\textbf{Remark on the admissible perturbation class.} The vorticity \(\omega_\epsilon\) of the Kelvin--Stuart cat's-eye flow decays exponentially as
\(y\to\pm\infty\). The admissible perturbed vorticity \(\tilde\omega\in Y_{non}\), however, is required only to satisfy the mild decay and moment conditions
\[
\tilde\omega\in L^1(\Omega)\cap L^2(\Omega),\qquad
y\tilde\omega\in L^1(\Omega),
\]
which ensure that the pseudoenergy is well-defined. Since the background velocity converges to \((\pm1,0)\) as \(y\to\pm\infty\), we impose the circulation constraint
\[
\iint_\Omega \tilde\omega\,dxdy=-4\pi,
\]
so that the perturbed flow has the same asymptotic velocity jump.

The sign condition \(\tilde\omega<0\) is not a pointwise decay requirement on
\(\tilde\omega-\omega_\ep\). It only requires the perturbed vorticity itself to remain in the same negative-vorticity regime as the Kelvin--Stuart profile
\[
\omega_\ep=-e^{-2\psi_\ep}<0,
\]
so that the Casimir functional is well-defined. Physically, this corresponds to redistributing the same-signed vorticity of the rolled-up shear layer, rather than introducing oppositely signed, counter-rotating vorticity.

\subsubsection{Main results for the MHD equations}
We next state the corresponding results for the Kelvin--Stuart magnetic islands. The first theorem gives a rigorous proof of coalescence instability for the whole family at the linear level.

\begin{Theorem}\label{main result1-mhd-all1} Let $0 \leq \ep < 1$.  Then

$(1)$
 the  magnetic island solution $(\omega=0,\phi_{\ep})$ is linearly unstable for double-periodic perturbations,

$(2)$
 the  magnetic island solution $(\omega=0,\phi_{\ep})$  is spectrally stable for co-periodic perturbations.
\end{Theorem}

Then we prove nonlinear orbital stability of the whole family of Kelvin--Stuart magnetic islands for co-periodic perturbations.

\begin{Theorem}\label{main result1-mhd-all}
 Assume that

 $({\rm i})$ for the initial data $\tilde \omega(0)=\tilde \omega_0\in\tilde Y$ and $\tilde \phi(0)=\tilde \phi_0\in\tilde Z_{non,\ep}$,  there exists a global weak solution $(\tilde \omega(t),\tilde \phi(t))$ in the distributional sense to the nonlinear MHD equations \eqref{mhd} such that $\tilde \omega(t)\in\tilde Y$ and $\tilde \phi(t)\in \tilde  Z_{non,\ep}$ for $t\geq0$,

 $({\rm ii})$ the distance functional $\hat d((\tilde \omega(t),\tilde \phi(t)),(0,\phi_{\ep}))$
is  continuous on $t$,

$({\rm iii})$  the energy-Casimir functional $\hat H$ satisfies that  $\hat H(\tilde \omega(t),\tilde \phi(t))\leq \hat H(\tilde \omega(0),\tilde \phi(0))$  and $\iint_\Omega e^{-j\tilde\phi(t)}dxdy$ is conserved for $t\geq0$ and $j=2,3$. \\
Let $\ep_0\in(0,1)$. For any $\kappa>0$, there exists $\delta=\delta(\ep_0,\kappa)>0$ such that if
\begin{align}\label{initial data-mhd}
\inf_{(x_0,y_0)\in\Omega} \hat d((\tilde\omega_0,\tilde \phi_0),(0,\phi_{\ep_0}(x+x_0,y+y_0)))+\left|\iint_{\Omega}(e^{-2\tilde \phi_0}-e^{-2\phi_{\ep_0}})dxdy\right|<\delta,
 \end{align}
  then
  for any $t\geq0$, we have
\begin{align}\label{onlinear orbital stability-goal-mhd}
\inf_{(x_0,y_0)\in\Omega} \hat d((\tilde\omega(t),\tilde \phi(t)),(0,\phi_{\ep_0}(x+x_0,y+y_0)))<\kappa,
\end{align}
where the distance  $\hat d$ is defined in \eqref{distance-mhd}, the functional $\hat H$ is defined in \eqref{EC-functional-mhd}, and the spaces $\tilde Y, \tilde Z_{non,\ep}$ are defined in \eqref{tilde Y-space}, \eqref{def-Z-non-ep}, respectively.
\end{Theorem}


\subsection{Main ideas in the proof}
This paper has one core structural discovery -- the hidden isospectral/integrable structure revealed by the
nonlinear change of variables $(x,y)\mapsto(\theta_\ep,\gamma_\ep)$ defined by \eqref{transformation-isospectrum-inroduction1}-\eqref{transformation-isospectrum-inroduction2} -- and each part of the proof
develops a different consequence of this underlying structure. The co-periodic spectral analysis identifies the exact
linear structure of the whole family; the multi-periodic and modulational instability arguments
exploit more involved transformed spectral problems and the separable Hamiltonian structures; the nonlinear Euler stability theory combines
this spectral information with a dual pseudoenergy-Casimir framework adapted to the unbounded
strip; and the MHD analysis is a further application of the same geometry and Hamiltonian
ideas.

Theorem~1.4 should also be compared with the work of Holm--Marsden--Ratiu \cite{holm1986nonlinear}. Their theory applies in a truncated domain bounded by a pair of steady streamlines and only for a restricted subfamily of Kelvin--Stuart vortices, whereas the present paper treats the whole family in the original unbounded strip. This is not a routine extension: in the unbounded setting one needs both a full co-periodic linear analysis and a different nonlinear framework. A more detailed comparison is given in Subsection \ref{Proof-nonlinear-orbital-stability-Kelvin--Stuart-vortices-co-periodic-perturbations}.

\subsubsection{Proof of spectral  stability of Kelvin--Stuart vortices for co-periodic perturbations} \label{subsection-Proof-of-spectral-stability-of-Kelvin--Stuart-vortices-for-co-periodic-perturbations} We begin with the co-periodic spectral problem. In general, the linear stability analysis of
non-parallel flows is difficult. Our starting point is that the linearized vorticity equation
around $\omega_\ep$ has the Hamiltonian form
 \begin{equation}\label{hami}
 \partial_t \omega = J_\epsilon L_\epsilon\omega, \quad \omega \in X_\ep,
 \end{equation}
 where
 \begin{align}\label{J-ep-J-ep-def}J_\epsilon = -g'(\psi_\epsilon)\vec{u}_\epsilon\cdot\nabla: X_\ep^* \supset D(J_\epsilon) \rightarrow X_\ep, \quad
 L_\epsilon = \frac {1} {g'(\psi_\epsilon)} - (-\Delta)^{-1}: X_\ep \rightarrow X_\ep^*,\end{align}
\begin{align}\label{spaceXep}
X_\ep = \left\{\omega\bigg| \iint_\Omega \frac{|\omega|^2}{g'_\epsilon(\psi_\epsilon)} dxdy < \infty, \iint_\Omega \omega dxdy = 0 \right\},\quad \epsilon\in[0,1),
\end{align}
and $(-\Delta)^{-1}\omega$ is clarified in Lemmas \ref{1-1correspond} and \ref{1-1correspond-ep}.
The constraint $\iint_\Omega\omega dxdy=0$ in $X_\ep$ is   again due to the asymptotic behavior of the velocity.
In contrast with the truncated-domain setting in \cite{holm1986nonlinear}, the
original unbounded strip $\Omega$ requires several new ingredients to handle the loss of
compactness. In particular, we introduce two weighted Poincar\'e-type inequalities, see
\eqref{Poincare inequality I-ep22} and \eqref{Poincare inequality II-ep22}, in a new Hilbert
space $\tilde X_\ep$ of stream functions, defined in \eqref{tilde-X-e}.

The Hamiltonian structure \eqref{hami} allows us to use the index formula
 \begin{align}\label{index-formula-stuart}
k_{r,\ep} + 2k_{c,\ep}+2k_{i,\ep}^{\leq0}+k_{0,\ep}^{\leq0} = n^-(L_\ep),
\end{align}
where $k_{r,\ep}$ is the sum of algebraic multiplicities of positive eigenvalues of $J_\ep L_\ep$,
$k_{c,\ep}$ is the sum of algebraic multiplicities of eigenvalues of $J_\ep L_\ep$ in the first
quadrant, $k_{i,\ep}^{\leq 0}$ is the total number of non-positive directions of
$\langle L_\ep\cdot,\cdot\rangle$ restricted to the generalized eigenspaces of pure imaginary
eigenvalues of $J_\ep L_\ep$ with positive imaginary parts, and $k_{0,\ep}^{\leq 0}$ is the number
of non-positive directions of $\langle L_\ep\cdot,\cdot\rangle$ restricted to the generalized
kernel of $J_\ep L_\ep$ modulo $\ker L_\ep$. The formula \eqref{index-formula-stuart} was
developed for general Hamiltonian systems in \cite{lin2022instability}.

By  \eqref{index-formula-stuart}, a sufficient condition for the spectral stability of the Kelvin--Stuart vortex is that
the energy quadratic form is non-negative, that is,
\begin{align*}
\langle L_\epsilon\omega,\omega\rangle\geq0,\quad\omega\in  X_\ep.
\end{align*}
This is equivalent to the non-negativity of the dual quadratic form:
 \begin{align}\label{dual energy quadratic form non-negative}
 \langle\tilde A_\epsilon\psi,\psi\rangle\geq0,\quad \psi\in \tilde{X}_\ep,
 \end{align}
where
    \begin{align*}
 \tilde{A}_\ep=-\Delta-g'(\psi_\ep)(I - P_\ep): \tilde{X}_\ep \rightarrow \tilde{X}_\ep^*,
\end{align*}
and the $1$-dimensional projection $P_\ep \psi={1\over8\pi}\iint_\Omega g'(\psi_\ep)\psi \,dxdy$
arises from the constraint $\iint_\Omega\omega$ $dxdy=0$.
To confirm that $\tilde A_\epsilon\geq0$, it is equivalent to show that the principal eigenvalue of the associated  PDE eigenvalue problem
\begin{align}\label{eigenvalue problem-introduction}
-\Delta \psi = \lambda g'(\psi_\ep)(\psi -  P_\ep\psi), \quad \psi \in \tilde{X}_\ep
\end{align}
is equal to $1$. We also prove that
 \begin{align}\label{dual energy quadratic form ker}
 \dim(\ker(\tilde A_\epsilon))=3,
 \end{align}
where the kernel directions come from translation in $x$, translation in $y$, and variation of
the parameter $\ep$. This non-degeneracy property is crucial in the nonlinear orbital stability theory.

We first consider the shear case ($\ep=0$). Because the variables separate, the problem \eqref{eigenvalue problem-introduction} reduces
to a family of Sturm-Liouville ODE eigenvalue problems \eqref{mode0}-\eqref{modek} for the
Fourier modes. Guided by the numerical computation in Subsection \ref{eigenfunction-motivation}
and by the first few exact eigenpairs in \eqref{eigen value-function}, we introduce the change
of variable $\gamma = \tanh(y)$. This unexpectedly transforms \eqref{mode0}-\eqref{modek}
into the classical Legendre-type equations \eqref{eigenvalue problem for 0 mode} and
\eqref{eigenvalue problem2 non-zero modes varepsilon=0}, which can then be solved explicitly
in terms of the Legendre and associated Legendre polynomials. In particular, the principal
eigenvalue of \eqref{eigenvalue problem-introduction} is $1$, which yields spectral stability for
$\ep=0$.

For the Kelvin--Stuart vortices ($0<\ep<1$), the PDE eigenvalue problem
\eqref{eigenvalue problem-introduction} cannot be solved by separating the original variables
$(x,y)$. This is the main  difficulty in the linear analysis. We overcome it by introducing a nonlinear change
of variables $(x,y)\mapsto(\theta_\ep,\gamma_\ep)$ under which the associated PDE eigenvalue
problems \eqref{eigenvalue problem-introduction} decouple. The change of variables is
  \begin{align}\label{transformation-isospectrum-inroduction1}
 \theta_\ep(x,y) & = \left\{ \begin{array}{llll} \arccos \left( \frac{\xi_\ep}{\sqrt{1-\gamma_\ep^2}} \right) & \mbox{ for } & (x,y) \in [0, \pi]\times\mathbb{R}, \\
 2\pi - \arccos \left( \frac{\xi_\ep}{\sqrt{1-\gamma_\ep^2}} \right) & \mbox{ for } & (x,y) \in (\pi, 2\pi]\times\mathbb{R}, \end{array}\right.\\\label{transformation-isospectrum-inroduction2}
 \gamma_\ep(x,y) & = \frac{\sqrt{1 - \epsilon^2}\sinh(y)}{\cosh(y)+\epsilon \cos(x)}\quad\text{for}\quad  (x,y) \in [0, 2\pi]\times\mathbb{R},
 \end{align}
where
$\xi_\ep(x, y)   = (1 - \epsilon^2) \frac{ \partial \psi_\ep}{\partial \ep} = \frac{\epsilon \cosh(y) + \cos(x)}{\cosh(y)+\epsilon \cos(x)}$.
These variables are compatible with the shear case, and the parameter $\ep$ for the whole
family is fully encoded in them. Under this transformation, we prove that $\tilde A_\ep$ is
isospectral to $\tilde A_0$ (i.e. they have the same eigenvalues). In particular,
\eqref{dual energy quadratic form non-negative} and \eqref{dual energy quadratic form ker}
follow for the whole family, which is exactly the information needed later in the proof of
nonlinear orbital stability. For the motivation behind the variables
$(\theta_\ep,\gamma_\ep)$, we refer to \eqref{eigen-p-cat-eyes}-\eqref{def-xi-ep}.

\subsubsection{Proof of linear  instability of Kelvin--Stuart vortices for multi-periodic perturbations}
As in the co-periodic case, the linearized equation around $\omega_\ep$ can be written as the
Hamiltonian system
$
 \partial_t \omega = J_{\epsilon,m} L_{\epsilon,m}\omega,  \omega \in X_{\ep,m},
 $
where the subscript $m$ denotes $2m\pi$-periodic perturbations with $m\geq2$.
The difference from the co-periodic problem is that $n^-(L_{\epsilon,m})>0$, where
$n^-(L_{\epsilon,m})$ is the negative dimension of the energy quadratic form
$\langle L_{\epsilon,m}\cdot,\cdot\rangle$. If one tries to use the same type of index formula
$
 k_{r,\ep,m} + 2k_{c,\ep,m}+2k_{i,\ep,m}^{\leq0}+k_{0,\ep,m}^{\leq0} = n^-(L_{\ep,m})
$
as in the co-periodic case, then one must compute $k_{i,\ep,m}^{\leq0}$ and $k_{0,\ep,m}^{\leq0}$,
which depend on the spectral information of $J_{\epsilon,m}L_{\epsilon,m}$ on the imaginary axis
and are difficult to analyze. Here the indices $k_{r,\ep,m}, k_{c,\ep,m}, k_{i,\ep,m}^{\leq0},
k_{0,\ep,m}^{\leq0}$ are defined similarly as in \eqref{index-formula-stuart}.

A key observation is that the linearized vorticity equation can be reformulated as the separable
Hamiltonian system
\begin{align}\label{s Hamiltonian system-introduction}
\partial_t \left( \begin{array}{c} \omega_1 \\ \omega_2 \end{array} \right) = \left( \begin{array}{cc} 0 & B_\ep \\ -B'_\ep & 0 \end{array} \right)\left( \begin{array}{cc} L_{\ep,e} & 0 \\ 0 & L_{\ep,o} \end{array} \right) \left( \begin{array}{c} \omega_1 \\ \omega_2 \end{array} \right),
\end{align}
which reflects the symmetry of the steady state in the $y$-direction together with the fact
that $L_{\ep,o}\geq0$. Here,
\begin{align*}B_\ep &= -g'(\psi_\ep) \vec{u}_\ep \cdot \nabla : X_{\ep, o}^* \supset D(B_\ep) \rightarrow X_{\ep, e}, \\
 L_{\ep,o} &= \frac{1}{g'(\psi_\ep)} - (-\Delta)^{-1}: X_{\ep, o} \rightarrow X_{\ep, o}^*, \quad\quad
 L_{\ep,e} = \frac{1}{g'(\psi_\ep)} - (-\Delta)^{-1}: X_{\ep, e} \rightarrow X_{\ep, e}^*,
 \end{align*}
$X_{\ep, e} = \left\{ \omega \in X_{\ep,m} \mid \omega \text{ is even in }y \right\}$ and
$X_{\ep, o} = \left\{ \omega \in X_{\ep,m} \mid \omega \text{ is odd in }y \right\}$.
This formulation leads to the exact unstable-mode counting formula
$n^-\left(L_{\ep,e}|_{\overline{{R}(B_\ep L_{\ep,o})}}\right)$. Moreover,
$\overline{{R}(B_\ep L_{\ep,o})}=\overline{{R}(B_\ep)}$ by Lemma \ref{range BLo-B}. Hence,
$\omega_\ep$ is linearly unstable if and only if
$$n^-\left(L_{\ep,e}|_{\overline{{R}(B_\ep )}}\right)>0.$$
This is equivalent to
\begin{align}\label{multi-periodic instabilibity criterion intruduction}
 n^-\left(\hat{A}_{\ep,e}\right)>0,
\end{align}
where the operator $\hat{A}_{\ep,e}$ is
$$\hat{A}_{\ep,e} = - \Delta - g'(\psi_\ep)(I - \hat{P}_{\ep,e}): \tilde{X}_{\ep, e} \rightarrow \tilde{X}^*_{\ep, e}.$$
Here, the operator $\hat{P}_{\ep,e}$, defined in \eqref{def-hat-P-ep-e}, is an infinite-dimensional
projection onto $\ker (B_\ep')$ and can be traced back to the constraint space $\overline{{R}(B_\ep)}$ for
$L_{\ep,e}$.

Because of the nonlocal projection $\hat{P}_{\ep,e}$, the spectrum of $\hat{A}_{\ep,e}$ is difficult to find  explicitly. To prove instability it is therefore enough to construct a suitable  test function
$\psi$ such that $\langle\hat{A}_{\ep,e}\psi,\psi\rangle<0$. For the $4k\pi$-periodic case, the test
function \eqref{test-even} is built from an explicit eigenfunction of the associated PDE
eigenvalue problem
\begin{align}\label{eigenvalue-problem-elliptic-m}
-\Delta \psi = \lambda g'(\psi_\ep)(\psi -  P_{\ep,m}\psi),  \psi \in \tilde{X}_{\ep,m},
\end{align}
for which the nonlocal projection term vanishes,
where $P_{\ep,m}$ is a one-dimensional projection defined
analogously to $P_\ep$. For the $(4k+2)\pi$-periodic case, however,
one cannot choose a periodic test function that annihilates the nonlocal term, and the
construction is much more subtle. Our test functions are delicate combinations of explicit
eigenfunctions in different regions: see \eqref{test-odd} for $\ep\in[0,{4\over5}]$ and
\eqref{test-odd-2} for $\ep\in\left({4\over5},1\right)$. The split into these two parameter
ranges is made so as to keep the contribution of the projection term sufficiently small. To
control this term, we reduce the estimates to a nested property of the trapped regions in the
variables $(\theta_\ep,\gamma_\ep)$; see Lemma \ref{D-theta-gamma-ep-nest}. In particular, we find that the
level curves of $\omega_\ep$ in the alternative variables $(\xi_\ep,\eta_\ep)$ are parts of some ellipses in
the closed unit disk $D_1$, where $(\xi_\ep,\eta_\ep)$ are given in \eqref{three-kers3} and
\eqref{three-kers1}. We obtain the desired property by proving that the inner boundary elliptic curves are nested.

\subsubsection{Proof of  modulational  instability of Kelvin--Stuart vortices}
The proof is mostly analytical; the only computer-assisted step is the evaluation of the
integral in \eqref{func-b2-alpha}-\eqref{func-b1-alpha-2}.
In this setting, the linearized vorticity equation is first written as the complex Hamiltonian
system \eqref{complex Ham-modu}. To apply the index formula \eqref{index-formula-neg}, we
rewrite \eqref{complex Ham-modu} as the real separable Hamiltonian system
\eqref{sep-hamiltonian-alpha}. This leads to the instability criterion in Lemma
\ref{L e-hat A-alpha}, formulated in terms of a dual quadratic form associated with a different
nonlocal projection from the multi-periodic case. We then construct the test function
\eqref{test-function-modulational-instability} by the first eigenfunction of the associated PDE
eigenvalue problem \eqref{elip02-alpha}, and verify that the corresponding dual quadratic form
is negative for every $\alpha\in(0,{1\over2}]$.

In both the multi-periodic and modulational arguments, the test functions are built from
eigenfunctions corresponding to the first few eigenvalues of
\eqref{eigenvalue-problem-elliptic-m}
or of \eqref{elip02-alpha}. These eigenvalue problems are more involved than the co-periodic
problem \eqref{eigenvalue problem-introduction}, both in the original variables and in the
transformed variables. To solve them, we introduce two further transformations,
\eqref{transformation-modu} and \eqref{modulational-transformation}, which convert the ODEs
for the non-zero modes into Gegenbauer equations. This makes it possible to solve the relevant
eigenvalue problems explicitly in terms of Gegenbauer (ultraspherical) polynomials.

\subsubsection{Proof of nonlinear orbital stability of Kelvin--Stuart vortices for co-periodic perturbations}\label{Proof-nonlinear-orbital-stability-Kelvin--Stuart-vortices-co-periodic-perturbations}
We first recall the truncated-domain strategy in \cite{holm1986nonlinear}. There, Holm,
Marsden and Ratiu used Arnol$'$d's original energy-Casimir method
\cite{Arnold65,Arnold69} in a truncated domain $\Omega_{trun}$ bounded by a pair of streamlines.
To highlight the idea, we ignore the boundary effect here.
Writing the energy-Casimir functional as
$\tilde H(\tilde\omega)= \iint_{\Omega_{trun}}\left(h(\tilde\omega)-{1\over2}|\nabla\tilde\psi|^2\right)
 dxdy$,
 one has $\tilde H'(\omega_\ep)=0$ and
\begin{align*}
\tilde H(\tilde\omega)-\tilde H(\omega_\ep)=\iint_{\Omega_{trun}}\left((h(\tilde\omega)-h(\omega_\ep)-h'(\omega_\ep)\omega)-{1\over2}|\nabla\psi|^2\right)dxdy,
\end{align*}
where  $\tilde \omega$ and $\tilde \psi$ denote the perturbed vorticity and stream function, and $h(s)=\int_0^sg^{-1}(\tilde s)d\tilde s=-\int_0^s{1\over 2}\ln(-\tilde s)d\tilde s
={1\over2}(s-s\ln(-s))$ for $s<0$. Crucially, $h''(\omega_\ep)$ admits both a positive lower bound
$c_\ep$ and a positive upper bound $C_{trun}$ on $\Omega_{trun}$. By extending
$h|_{Ran(\omega_\ep)}$ to the whole real line with the same bounds for the second derivative,
one obtains
$${1\over2} C_{trun}\|\omega\|_{L^2(\Omega_{trun})}^2\geq\iint_{\Omega_{trun}}\left(h(\tilde\omega)-h(\omega_\ep)-h'(\omega_\ep)\omega\right)dxdy\geq {1\over2} c_\ep\|\omega\|_{L^2(\Omega_{trun})}^2,$$
where $C_{trun}\to\infty$ as the size of the truncated domain tends to infinity, while $c_\ep$
depends only on $\ep$. For the second term, one has the Poincar\'e inequality
\begin{align}\label{Poincare type inequality truncated domain}
\iint_{\Omega_{trun}}|\nabla\psi|^2dxdy\leq k_{\min}^{-2}\|\omega\|_{L^2(\Omega_{trun})}^2
\end{align}
with $k_{\min}^2$ the principal eigenvalue of $-\Delta$ on a rectangle containing $\Omega_{trun}$. If $k_{\min}^{-2}<c_\ep,$
which can be enforced by shrinking the domain and restricting the range of \(\ep\), then combining the above estimates yields
\begin{align}\label{EC-bound}
\frac12 C_{trun}\|\omega^0\|_{L^2(\Omega_{{trun}})}^2
\ge
\tilde H(\tilde\omega)-\tilde H(\omega_\ep)
\ge
\frac12 (c_\ep-k_{\min}^{-2})\|\omega\|_{L^2(\Omega_{{trun}})}^2,
\end{align}
where \(\omega^0\) is the initial perturbation.
 This establishes nonlinear stability. When the truncated
domain is large or \(\ep\) ranges over the whole family, however, the condition $k_{\min}^{-2}<c_\ep$ fails, and this argument breaks
down. In the full strip the situation is harder still: on the one hand,
\eqref{Poincare type inequality truncated domain} is unavailable, and on the other hand,
$h''(\omega_\ep)$ is unbounded from above.

We now explain our strategy in the original unbounded domain $\Omega$. Since the perturbed
velocity tends to $(\pm1,0)$ as $y\to\pm\infty$, the classical kinetic energy
$\iint_{\Omega}|\vec{u}|^2dxdy$ is not finite. We therefore replace it by the pseudoenergy
$\iint_{\Omega}(G\ast\tilde \omega)\tilde \omega \,dxdy$ and consider the pseudoenergy-Casimir
functional
$ H(\tilde\omega)= \iint_{\Omega}\left(h(\tilde\omega)-{1\over2}(G\ast\tilde \omega)\tilde \omega\right)dxdy.$ Then
\begin{align}\label{pec-introduction}
 H(\tilde\omega)- H(\omega_\ep)=\iint_{\Omega}\left((h(\tilde\omega)-h(\omega_\ep)-h'(\omega_\ep)\omega)-{1\over2}(G*\omega)\omega\right)dxdy.
\end{align}
Because $h''(\omega_\ep)$ is unbounded from above, the enstrophy norm used in the truncated
domain is no longer appropriate in $\Omega$, and one cannot extend
$h|_{Ran(\omega_\ep)}$ to a convex function on the whole real axis. Instead, we define our
distance functional as the sum of the first term in \eqref{pec-introduction} and the
pseudoenergy. This gives the required upper bound of $H(\tilde\omega)- H(\omega_\ep)$ from the initial data. The lower bound,
however, requires a new argument, since the bounded-domain proof based on
\eqref{Poincare type inequality truncated domain} is unavailable. Our strategy can be
summarized as follows.

\begin{enumerate}
\item We first try to analyze the Taylor expansion of $H$ directly at $\omega_\ep$. The first
variation satisfies $H'(\omega_\ep)=0$, and the second variation is exactly the linear energy
quadratic form:
$\langle H''(\omega_\ep)\omega,\omega\rangle=\langle L_\ep\omega,\omega\rangle$.
The difficulty is that $H$ is not $C^2$ near $\omega_\ep$, so the remainder terms cannot be
controlled directly. We therefore introduce, via the Legendre transformation, the dual functional of the stream functions
\begin{align*}
\mathscr{B}_\ep( \psi)=
  & \iint_{\Omega} \left(\frac 1 2 |\nabla \psi|^2 -\frac 1 4 g'(\psi_\ep)(e^{-2\psi} + 2\psi - 1)\right) dxdy,\quad \psi\in \tilde X_\ep,
\end{align*}
and prove that it is $C^2$ on $\tilde X_\ep$. This is enough to control the remainder terms.
Moreover, $\mathscr{B}_\ep'(0)=0$, and the second variation corresponds to the dual linear
quadratic form:
$$\langle \mathscr{B}_\ep''(0)\psi,\psi\rangle=\langle A_\ep\psi,\psi\rangle,$$
where $A_\ep=\tilde A_\ep-g'(\psi_\ep)P_\ep$.

\item
 Our precise linear spectral analysis shows that  $A_\ep\geq0$ and $\dim(\ker(A_\ep))=3$ with kernel directions generated by translation in $x$,
translation in $y$, and variation of the parameter $\ep$.  This allows us to prove nonlinear orbital stability with respect to the full three-dimensional orbit, generated by translations in $(x,y)$ together with variation along the Kelvin--Stuart family.

\item To obtain nonlinear orbital stability of a fixed Kelvin--Stuart vortex, modulo only the
translations in $x$ and $y$, we use the additional Casimir constraint
$\iint_{\Omega} (-\omega)^{3\over2} \,dxdy$ to keep the $\epsilon$-parameter variation small for all time.
This allows us to pass from $3$D orbital stability of the family to $2$D orbital stability of a
fixed member of the family.

\item Finally, if one works directly with weak solutions, the distance functional need not be
continuous in time, so the solution may jump between neighborhoods of
different steady states. To overcome this difficulty, we first smooth the initial data, construct
approximate strong solutions, and prove nonlinear orbital stability for these approximants. We
then pass to the weak limit -- using the convexity of the Casimir functional and a careful study of
the convergence of the approximating initial data -- to obtain the nonlinear orbital stability for the weak solutions.
\end{enumerate}

\medskip

\noindent{\bf{Comparison with the previous work of Holm--Marsden--Ratiu.}}\medskip

The methodological differences between our work and that of Holm--Marsden--Ratiu \cite{holm1986nonlinear} can be summarized as follows.

\begin{itemize}
\item \emph{Linear and spectral structure.} Because of the domain truncation and the corresponding restriction on the parameter $\ep$, no linear stability analysis is needed in \cite{holm1986nonlinear}. In the unbounded domain, however, the nonlinear theory for the whole family must begin with a complete co-periodic linear analysis. The principal novelty of the present paper lies here: the nonlinear change of variables \eqref{transformation-isospectrum-inroduction1}-\eqref{transformation-isospectrum-inroduction2} reveals a hidden symmetry of the spectral problem and yields an isospectral reduction of the whole Kelvin--Stuart family to the hyperbolic tangent shear case. Consequently, the seemingly non-separable PDE eigenvalue problem \eqref{eigenvalue problem-introduction} becomes exactly reducible to classical ODEs of Legendre or Gegenbauer type. This hidden structure underlies not only the co-periodic spectral stability and nonlinear orbital stability analysis, but also the multi-periodic and modulational instability arguments.

\item \emph{Nonlinear framework in the unbounded strip.} The argument in \cite{holm1986nonlinear} relies on the uniform convexity of the Casimir functional in vorticity, together with Poincar\'e-type inequalities, to obtain the upper and lower bounds \eqref{EC-bound} for the energy-Casimir functional. Our approach is fundamentally different.
    Instead of directly analyzing the non-$C^2$  pseudoenergy-Casimir functional in vorticity, we pass via the Legendre transformation to a dual functional in terms of stream function with the required $C^2$-regularity in $\tilde X_\ep$.
    This, combined with the delicate linear spectral analysis, allows us to
    establish $3$D orbital stability, incorporating translations and parameter variation. We then reduce it to $2$D orbital stability for a fixed Kelvin--Stuart vortex by means of an additional Casimir constraint. For weak solutions, we construct approximate strong solutions by mollification, establish stability for these smooth approximants, and then pass to the limit.
\end{itemize}

The behavior under double-periodic perturbations also differs sharply between the truncated domain and the original unbounded strip. In \cite{holm1986nonlinear}, a related argument gives nonlinear stability for double-periodic perturbations when the truncated domain is sufficiently small and the allowed range of $\ep$ is sufficiently restricted.
 This suggests that imposing sufficient constraints on the truncation and the $\ep$-parameter suppresses the onset of pairing instability.
 In contrast, in the original unbounded domain the whole Kelvin--Stuart family is always unstable under double-periodic perturbations, exactly as conjectured by Stuart in \cite{stuart1967finite} and proved here in Theorem~\ref{main result2-multi-periodic perturbations}.

\subsubsection{Proof of   stability and instability  of Kelvin--Stuart magnetic islands}
Compared with the separable Hamiltonian form \eqref{s Hamiltonian system-introduction} in
the Euler case, the linearized planar ideal MHD equations around the magnetic island
$(0,\phi_{\ep})$ have a different separable Hamiltonian structure
\begin{align*}
\partial_t \left( \begin{array}{c} \phi \\ \omega \end{array} \right) = \left( \begin{array}{cc} 0 & D_\ep \\ -D_{\ep}' & 0 \end{array} \right)\left( \begin{array}{cc}-\Delta-g'(\phi_{\ep}) & 0 \\ 0 & (-\Delta)^{-1} \end{array} \right) \left( \begin{array}{c} \phi \\ \omega \end{array} \right)
\end{align*}
for  co-periodic perturbations, where
$\phi\in\tilde W_{\ep}=\{ \phi \in\dot{H}^1(\Omega) | \iint_\Omega g'(\phi_\ep)\phi dxdy=0\}$ is the perturbation of magnetic potential,
$\omega\in \tilde X_\ep^*$ is the perturbation of vorticity,
  and $D_\ep=-\{\phi_{\ep},\cdot\}:\tilde X_\ep\supset D(D_\ep)\to\tilde W_{\ep}$.
Based on this structure, the criterion for co-periodic spectral stability is
$$n^-\left(\tilde A_{\ep}|_{\overline{R(D_{\ep})}}\right)=0.$$
Spectral stability of $(0,\phi_{\ep})$ is then recovered from the Euler linear analysis, since
$\tilde A_\ep|_{\tilde X_\ep}\geq0$.
Similarly, the criterion for multi-periodic linear instability is
\begin{align}\label{multi-periodic linear instability-introduction}
n^-\left(\tilde A_{\ep,m}|_{\overline{R(D_{\ep,m})}}\right)\geq1,
\end{align}
where the subscript $m$ denotes $2m\pi$-periodic perturbations with $m\geq2$.
The condition \eqref{multi-periodic linear instability-introduction} is more restrictive than
\eqref{multi-periodic instabilibity criterion intruduction} in the Euler case. Nevertheless,
thanks to the symmetry of the test function $\tilde{\psi}_\ep$ in \eqref{test-even}, this function
belongs to $\overline{R(D_{\ep,2})}$, and we obtain linear instability of $(\omega=0,\phi_{\ep})$
for double-periodic perturbations. This gives the coalescence instability for the whole family
of Kelvin--Stuart magnetic islands and rigorously confirms the physical observations in
\cite{Finn-Kaw1977,Pritchett-Wu1979,Bondeson1983}.

Nonlinear orbital stability of Kelvin--Stuart magnetic islands for co-periodic perturbations is
proved by the energy-Casimir method. Besides the difficulties already present in the Euler
case, there is an additional issue in the MHD nonlinear analysis. In the Euler case, the
perturbation of the stream function is allowed to differ by a constant because
$\iint_\Omega\omega\,dxdy=0$. In the MHD case, however, the perturbation of the magnetic
potential cannot be shifted by a constant, and after translations it need not lie in the space
$\tilde X_\ep$. Consequently, the $C^2$ regularity of the energy-Casimir functional cannot be
proved directly on $\tilde X_\ep$. Our remedy is to add the projection term
$P_\ep \phi={1\over8\pi}\iint_\Omega g'(\phi_\ep)\phi \,dxdy$ to the energy-Casimir functional,
which allows a constant discrepancy in the perturbation. This makes it possible to prove the
$C^2$ regularity of the main term of the functional on $\tilde X_\ep$ and thereby exploit the
linear analysis. The remainder term created by the projection is then shown to be higher order
with respect to the distance functional.

 \subsubsection{Further context}
 Kelvin--Stuart vortices also arise in other physical and geometric settings. They have been used in models of planetary rings, including spatial structures in Saturn's ring system \cite{Shukla-Sen1996}, and they also appear as solutions of the Liouville's equation in certain dusty plasma models. More recently, Stuart vortices have been generalized from the plane to non-rotating and rotating spheres \cite{Crowdy04,Constantin-Crowdy-Krishnamurthy-Wheeler2021}, as well as to a torus and a hyperbolic sphere \cite{Sakajo2019,Yoon20}. See also \cite{Klaassen-Peltier1991,Dauxois-Fauve-Tuckerman1996,BD01,Majda-Bertozzi02,Constantin-Krishnamurthy2019,Krishnamurthy2019,Krishnamurthy2021} for further discussions of Kelvin--Stuart vortices and related equilibria. It would be interesting to study the stability of these generalized Stuart vortices by the methods developed here.

The rest of this paper is organized as follows. In Section 2, we prove that the steady state
$\omega_\ep$ with $\ep\in[0,1)$ is spectrally stable for co-periodic perturbations. In Section 3,
we prove linear instability for multi-periodic perturbations, and in Section 4 we prove linear
modulational instability. Section 5 establishes nonlinear orbital stability of Kelvin--Stuart
vortices for co-periodic perturbations. Section 6 contains numerical illustrations. In Section 7,
we study stability and instability of the magnetic island solutions $(\omega=0,\phi_{\ep})$ of the
planar ideal MHD equations \eqref{mhd} for co-periodic and double-periodic perturbations.
In the Appendix, we prove the existence of weak solutions to the 2D Euler equation in the
unbounded domain $\Omega$ with non-vanishing velocity at infinity.

\section{Spectral stability for  co-periodic perturbations}\label{co-periodic-linear}

In this section, we study the linear stability of the steady states \(\omega_\ep\) for co-periodic perturbations. We prove that the whole Kelvin--Stuart family is spectrally stable for all \(\ep\in[0,1)\).

We first formulate the linearized vorticity equation as a Hamiltonian PDE, and then reduce the self-adjoint part of the linearized vorticity operator to an elliptic operator on the space of stream functions.
\subsection{Hamiltonian formulation of the linearized Euler equation}
Linearizing the vorticity equation \eqref{vor} around the steady state $\omega_\ep$, we have
 \begin{equation*}
 \partial_t \omega + \partial_y \psi_{\epsilon} \partial_x \omega - \partial_x \psi_{\epsilon}\partial_y \omega + \partial_y \psi \partial_x \omega_\epsilon - \partial_x\psi\partial_y \omega_\epsilon = 0,
 \end{equation*}
 which can be rewritten as
 \begin{equation}\label{linearized vorticity equation}
 \partial_t\omega = - \vec{u}_\epsilon\cdot \nabla\omega + g'(\psi_\epsilon)\vec{u}_\epsilon\cdot\nabla\psi,
 \end{equation}
 where we used $\omega_\epsilon=g(\psi_\epsilon)$ by \eqref{elip}. Note that
  \begin{align}\label{def-g-psi-ep-derivative}
    g'(\psi_\epsilon) = 2 e^{-2\psi_\epsilon} = \frac{2(1-\ep^2)}{(\cosh (y) + \ep \cos(x))^2}>0,\quad(x,y)\in\Omega,\; \epsilon\in[0,1).
  \end{align}
   The linearized equation \eqref{linearized vorticity equation} has
 the following  Hamiltonian structure
 \begin{equation*}
 \partial_t \omega = J_\epsilon L_\epsilon\omega, \quad \omega \in X_\ep,
 \end{equation*}
 where
 \begin{align*}J_\epsilon = -g'(\psi_\epsilon)\vec{u}_\epsilon\cdot\nabla: X_\ep^* \supset D(J_\epsilon) \rightarrow X_\ep, \quad
 L_\epsilon = \frac {1} {g'(\psi_\epsilon)} - (-\Delta)^{-1}: X_\ep \rightarrow X_\ep^*,\end{align*}
\begin{align*}
X_\ep = \left\{\omega\bigg| \iint_\Omega \frac{|\omega|^2}{g'_\epsilon(\psi_\epsilon)} dxdy < \infty, \iint_\Omega \omega dxdy = 0 \right\},\quad \epsilon\in[0,1),
\end{align*}
$X_\ep^*$ is the dual space of $X_\ep$ and $(-\Delta)^{-1}\omega$ is defined as the unique weak solution to the Poisson equation
\begin{align}\label{Poisson equation}
-\Delta \psi = \omega
\end{align}
 in $\tilde X_\epsilon$
 (see Lemmas \ref{1-1correspond} and \ref{1-1correspond-ep}).  Here, $\tilde X_\epsilon$ is defined in \eqref{tilde-X0} and \eqref{tilde-X-e} for $\epsilon=0$ and  $\epsilon\in(0,1)$, respectively.

The vorticity space $X_\ep$ equipped with the inner product
$$(\omega_1, \omega_2) = \iint_\Omega \frac{\omega_1\omega_2}{g'_\epsilon(\psi_\epsilon)} dxdy$$ is a Hilbert space since it is a closed subspace of the Hilbert space $L^2_{\frac{1}{g'(\psi_\ep)}}(\Omega).$
We  denote  the dual bracket between $X_\ep$ and $X_\ep^*$ by $\langle \cdot, \cdot \rangle$. Thanks to the Poincar\'e inequality in Lemmas  \ref{poincare1} and \ref{poincare1ep}, we will prove that  $\langle L_\epsilon\cdot,\cdot\rangle$ is  a bounded  symmetric bilinear form on  $X_\ep$, see Lemmas  \ref{Lbounded} and \ref{Lbounded-ep}.

We impose the condition
\[
\iint_\Omega \omega dxdy=0
\]
in the definition of \(X_\ep\) because the perturbation must preserve the asymptotic velocity jump. Indeed, by \eqref{steadyv},
\[
\lim_{y\to\pm\infty}\vec u_\ep(x,y)=(\pm1,0),
\qquad x\in\mathbb T_{2\pi},\ \ep\in[0,1),
\]
and the perturbed velocity \(\vec v=(v_1,v_2)\) is required to have the same asymptotic behavior:
\[
\lim_{y\to\pm\infty}\vec v(x,y)=(\pm1,0).
\]
Hence the perturbed vorticity \(\tilde\omega\) satisfies
\begin{align}\label{perturbed vorticity condition}
\iint_\Omega \tilde\omega(x,y)dxdy
=
-\int_0^{2\pi} v_1(x,y)\big|_{y=-\infty}^{\infty}\,dx
=
-4\pi
=
\iint_\Omega \omega_\ep(x,y)dxdy.
\end{align}
Therefore, for the vorticity perturbation \(\omega=\tilde\omega-\omega_\ep\), we must have
$
\iint_\Omega \omega dxdy=0.
$

To understand  linear stability of the steady state $\omega_\ep$, it suffices to study the spectrum of the operator $J_\ep L_\ep$ on  $X_\ep$. Based on Hamiltonian structure of the linearized equation \eqref{hami}, we will study the spectral distribution  of $J_\ep L_\ep$  by the index formula \eqref{index-formula-stuart} developed in \cite{lin2022instability}. To verify the assumptions  in the Index Theorem (see {\bf(H1)-(H3)} in Lemma \ref{theorem-index}) and compute the indices  $n^{0}(L_\ep)$ and $n^{-}(L_\ep)$ (i.e. the number of kernel and negative directions of the self-adjoint operator $L_\ep$),   we will define  a dual  elliptic operator $\tilde{A}_\ep$ on a Hilbert space $\tilde{X}_\ep$ of stream functions, and reduce the computation of  the two indices  to the kernel and negative dimensions of $\tilde{A}_\ep$.

We divide the discussions into the  case $\ep = 0$ (hyperbolic tangent shear flow) and the case $0<\ep<1 $ (Kelvin--Stuart cat's-eye flows) separately.
\subsection{Dual quadratic form and variational problem for the shear case}

The advantage of the shear case \(\ep=0\) is that $g'(\psi_0) = {2}{\sech^2(y)}$ depends only on \(y\). This allows us to separate the variables \((x,y)\) and reduce the analysis to one-dimensional problems.
\subsubsection{Space of stream functions, Poisson equation and energy quadratic form}
First, we define explicitly   the space of stream functions such that the Poisson equation \eqref{Poisson equation} is well-posed in this space.
\begin{lemma}\label{Hilbert}
The function space
\begin{align}\label{tilde-X0}
\tilde{X}_0 = \left\{ \psi \bigg| \|\nabla \psi\|_{L^2(\Omega)} < \infty\quad {\rm{ and }}\quad\widehat\psi_0(0)={1\over2\pi} \int_{0}^{2\pi} \psi(x, 0)d x= 0 \right\}
\end{align} equipped with the inner product $$(\psi_1, \psi_2) = \iint_\Omega \nabla \psi_1 \cdot \nabla \psi_2 dxdy, \quad \forall\; \psi_1, \psi_2 \in \tilde{X}_0$$ is a Hilbert space.
\end{lemma}
Note that two functions differing by a constant represent the same element of the space $\dot{H}^1(\Omega)$. We add the condition $\widehat \psi_0(0)={1\over 2\pi}\int_{0}^{2\pi} \psi(x, 0)d x= 0$ in \eqref{tilde-X0} to fix the additive constant and make $\tilde{X}_0$ a Hilbert space.
\begin{proof}
First, we prove that $\|\psi\|_{\tilde{X}_0}= \|\nabla \psi\|_{L^2(\Omega)} = 0$ implies $\psi = 0$ in $\tilde{X}_0$. Since $\psi(x,y)=\sum_{k\in\mathbb{Z}}\widehat{\psi}_{k}(y)e^{ik x}$,  we have
\begin{align}\label{tilde-X0-norm}
\|\nabla \psi\|_{L^2(\Omega)}^2
=  2\pi \left( \int_{-\infty}^{+\infty} \sum_{k\neq0} k^2 \left|\widehat{\psi}_k(y)\right|^2 dy + \int_{-\infty}^{+\infty}\left( \left| \widehat{\psi}_0'(y)\right|^2  +  \sum_{k\neq 0}\left| \widehat{\psi}_k'(y)\right|^2\right) dy \right).
\end{align}
Then we infer from $\|\nabla \psi\|_{L^2(\Omega)} = 0$ that
$ \widehat{\psi}_k= 0$ for $k\neq 0$   and $\widehat{\psi}_0' =0$.
By the condition $\widehat{\psi}_0(0)= 0$, we have
$$\widehat{\psi}_0(y) = \widehat{\psi}_0(0) + \int_0^{y}\widehat{\psi}_0'(s) ds = 0$$
for $y \in \mathbb{R}$.
So $\widehat{\psi}_k = 0$ for  $k\in\mathbb{Z}$, and thus, $\psi= 0$.
Now we prove the completeness of the space $\tilde{X}_0$.
Let $\{\psi_m \}_{m=1}^{+\infty}$ be a Cauchy sequence in $\tilde{X}_0$, i.e.
$\|\psi_m - \psi_n\|_{\tilde{X}_0} \to 0$ as $m,n \rightarrow \infty$,
where
\begin{align}\label{psi-m-dec}
\psi_m(x,y) = \widehat{\psi}_{m,0}(y) + \sum_{k\neq0}\widehat{\psi}_{m,k}(y)e^{ikx} =: \widehat{\psi}_{m,0}(y) + {\psi}_{m,\neq 0}(x,y)
\end{align}
for $m\geq1$.
By \eqref{tilde-X0-norm}, we have
\begin{align*}
\|\psi_m\|^2_{\tilde{X}_0}
& = \| \widehat{\psi}_{m,0}'\|^2_{L^2(\Omega)} + \|\nabla\psi_{m,\neq 0}\|^2_{L^2(\Omega)} < \infty.
\end{align*}
Since
\begin{align*}
\|\psi_{m,\neq 0}\|^2_{L^2(\Omega)} =&2\pi \int_{-\infty}^{+\infty} \sum_{k\neq0}  \left|\widehat{\psi}_{m,k}(y)\right|^2 dy \\
\leq&
2\pi \int_{-\infty}^{+\infty} \sum_{k\neq0}\left(  k^2 \left|\widehat{\psi}_{m,k}(y)\right|^2 +\left| \widehat{\psi}_{m,k}'(y)\right|^2\right) dy = \|\nabla\psi_{m,\neq 0}\|^2_{L^2(\Omega)}, \end{align*}
we have $\psi_{m,\neq 0}\in H^1(\Omega)$. Similarly, we have $\|\psi_{m,\neq 0} -  \psi_{n,\neq 0}\|_{H^1(\Omega)}^2\leq 2\|\nabla(\psi_{m,\neq 0} -  \psi_{n,\neq 0})\|_{L^2(\Omega)}^2$ $\leq2\|\psi_m -  \psi_n\|_{\tilde{X}_0}^2$ for $m,n\geq1$. Since $\|\psi_m - \psi_n\|_{\tilde{X}_0} \to 0$ as $m,n \rightarrow \infty$, we obtain that $\{\psi_{m,\neq 0} \}_{m=1}^{+\infty}$ is a Cauchy sequence in the Hilbert space $H^1(\Omega)$. Then
there exists $ \psi_{\neq0} \in H^1(\Omega)$   such that
$ \psi_{m,\neq 0} \to\psi_{\neq0}$ in $H^1(\Omega)$. By the Trace Theorem, $\{\psi_{m,\neq 0}(\cdot,0) \}_{m=1}^{+\infty}$ is a Cauchy sequence in $L^{2}(\mathbb{T}_{2\pi})$ (and thus in $L^{1}(\mathbb{T}_{2\pi})$). Then
$$\widehat{\psi}_{\neq0,0}(0)={1\over2\pi}\int_{0}^{2\pi}\psi_{\neq0}(x,0)dx=\lim_{m\to\infty}{1\over2\pi}\int_{0}^{2\pi}\psi_{\neq0}(x,0)dx=0.$$
Thus, $\widehat{\psi}_{\neq0,0}\in\tilde{X}_0$.
Since $ \|\widehat{\psi}_{m,0}'-\widehat{\psi}_{n,0}'\|_{L^2(\Omega)}\leq \|{\psi}_{m}-{\psi}_{n}\|_{\tilde{X}_0}$,
 $\{\widehat{\psi}_{m,0}' \}_{m=1}^{+\infty}$ is a Cauchy sequence in the Hilbert space $L^2(\Omega)$.
Thus, there exists $ \psi^0_* \in L^2(\Omega)$ such that
$\widehat{\psi}_{m,0}'\to  \psi^0_*$ in $L^2(\Omega)$.
Now we define $$\psi^0(y) = \int_{0}^{y} \psi^0_*(s) ds\quad\text{for}\quad y\in\mathbb{R}.$$
Then $\psi^0(0) = 0$ and
$ \widehat{\psi}_{m,0} \to \psi^{0}$ in $\tilde{X}_0$.
Let $\psi^*(x,y) = \psi^0(y) + \psi_{\neq0}(x,y)$ for $(x,y)\in\Omega$. Then $\psi^*\in \tilde{X}_0$ and
$$\|\psi_m - \psi^*\|_{\tilde{X}_0}\leq  \|\widehat\psi_{m,0} - \psi^0\|_{\tilde{X}_0} +  \| \psi_{m,\neq 0} -  \psi_{\neq0}\|_{\tilde{X}_0} \to0$$
as $m\to\infty$. Thus, $\tilde{X}_0$ is a Hilbert space.
\end{proof}
\subsubsection{Poincar\'e inequalities}
First, we give a Poincar\'e-type inequality for functions with exponential decay  weight.
\begin{lemma}[Poincar\'e inequality I-$0$]\label{poincare1}
For any $\psi \in \tilde{X}_0$, we have
\begin{align}\label{Poincare inequality I022}
\iint_\Omega g'(\psi_0)|\psi|^2 dxdy  \leq C \|\nabla \psi\|_{L^2(\Omega)}^2.
\end{align}
\end{lemma}
\begin{proof}
For $\psi \in \tilde{X}_0$, we have
\begin{align*}
\iint_\Omega g'(\psi_0)|\psi|^2 dxdy
& = 2\pi\left( \int_{-\infty}^{+\infty} g'(\psi_0) \left|\widehat{\psi}_0\right|^2 dy  + \int_{-\infty}^{+\infty} g'(\psi_0) \sum_{k\neq 0}\left|\widehat{\psi}_k\right|^2 dy \right)\\
& = 2\pi (I + II).
\end{align*}
Since $0<g'(\psi_0(y)) = 2 \sech^2(y) \leq 2$ for $y\in\mathbb{R}$, we get by \eqref{tilde-X0-norm} that for the part of non-zero modes,
\begin{align*}
II \leq 2\int_{-\infty}^{+\infty} \sum_{k\neq 0}\left|\widehat{\psi}_k\right|^2 dy
\leq C \|\nabla \psi\|_{L^2(\Omega)}^2.
\end{align*}
For the  part of zero mode, by the fact that
$\widehat{\psi}_0(0)  = 0$, we have
\begin{align*}
I& = \int_{-\infty}^{+\infty} g'(\psi_0) \left|\int_0^y \widehat{\psi}_0'(s) ds\right|^2 dy
 \leq\|\widehat{\psi}_0'\|_{L^2(\mathbb{R})}^2 \int_{-\infty}^{+\infty} g'(\psi_0)|y| dy
\leq C\|\nabla \psi\|_{L^2(\Omega)}^2
\end{align*}
since $ g'(\psi_0)$ decays exponentially near $\pm\infty$.
\end{proof}

We define  a $1$-dimensional projection operator $P_0$ on $ \tilde{X}_0$ by
\begin{align}\label{P0-psi-def}
P_0\psi = \frac{\iint_\Omega g'(\psi_0)\psi dxdy}{\iint_\Omega g'(\psi_0) dxdy}=\frac{\iint_\Omega g'(\psi_0)\psi dxdy}{8\pi},\quad \psi \in \tilde{X}_0,
\end{align}
where we used
 $$\iint_\Omega g'(\psi_0) dxdy = \int_{-\infty}^{\infty} \int_0^{2\pi} 2 \sech^2(y) dxdy = 8\pi.$$
The projection $P_0$ will be used later to introduce  a suitable dual elliptic operator acting at the stream functions.
\begin{Corollary}\label{projection}
 The projection operator $P_0 $ is well-defined  on $ \tilde{X}_0$.
\end{Corollary}
\begin{proof}
By  Lemma \ref{poincare1}, we have
\begin{align}\nonumber
|P_0\psi|
& \leq \frac{1}{8\pi} \iint_\Omega g'(\psi_0)|\psi| dxdy
 \leq \frac{1}{8\pi} \left(\iint_\Omega g'(\psi_0)|\psi|^2 dxdy\right)^{1/2} \left( \iint_\Omega g'(\psi_0)dxdy \right)^{1/2} \\\label{p0-psi-estimates-2}
& \leq C \|\nabla \psi\|_{L^2(\Omega)}.
\end{align}
\end{proof}
Next, we give another Poincar\'e-type  inequality, which involves the projection defined above.
\begin{lemma}[Poincar\'e inequality II-$0$]\label{poincare2}
For any $\psi \in \tilde{X}_0$,
we have
\begin{align}\label{Poincare inequality II022}
\iint_\Omega g'(\psi_0)|\psi - P_0\psi|^2 dxdy  \leq C \|\nabla \psi\|_{L^2(\Omega)}^2.
\end{align}
\end{lemma}
\begin{proof}
By Corollary \ref{projection}, we have
\begin{align}\label{projection-estimate-nabla-psi}
\iint_\Omega g'(\psi_0)|P_0\psi|^2 dxdy = 8\pi |P_0\psi|^2 \leq C \|\nabla \psi\|_{L^2(\Omega)}^2.
\end{align}
Then
\begin{align*}
\iint_\Omega g'(\psi_0)|\psi - P_0\psi|^2 dxdy  \leq 2 \iint_\Omega g'(\psi_0)\left(|\psi|^2  + |P_0\psi|^2 \right)dxdy  \leq C  \|\nabla \psi\|_{L^2(\Omega)}^2
\end{align*}
by  Lemma \ref{poincare1} and \eqref{projection-estimate-nabla-psi}.
\end{proof}
Now we consider the existence and uniqueness of the weak solution to the Poisson equation \eqref{Poisson equation} in $\tilde{X}_0$.
\begin{lemma}\label{1-1correspond}
For  $\omega \in X_0$, the Poisson equation
\eqref{Poisson equation}
has a unique weak solution in $\tilde{X}_0$.
\end{lemma}
\begin{proof}
By  Lemma \ref{poincare1},  we have
\begin{align*}
\iint_\Omega \omega \tilde {\psi} dxdy
& \leq \left( \iint_\Omega \frac{|\omega|^2}{g'(\psi_0)} dxdy  \right)^{1/2} \left( \iint_\Omega g'(\psi_0) |\tilde\psi|^2 dxdy  \right)^{1/2}  \leq C \|\omega\|_{X_0} \|\tilde{\psi}\|_{\tilde{X}_0}
\end{align*}
for any $\tilde{\psi}\in\tilde{X}_0$. Note that $\tilde{X}_0$ is a Hilbert space by Lemma \ref{Hilbert}. Thus,
by the Riesz Representation Theorem, there exists a unique $\psi \in \tilde{X}_0$ such that
$$\iint_\Omega \omega \tilde{\psi} dxdy = \langle \omega, \tilde{\psi} \rangle = (\psi, \tilde{\psi}) = \iint_\Omega \nabla \psi \cdot \nabla \tilde{\psi}dxdy.$$
Then $\psi $ is the unique weak solution in $\tilde X_0$ to the Poisson equation \eqref{Poisson equation}.
\end{proof}

For $\omega \in X_0$, we denote $(-\Delta)^{-1}\omega\in\tilde{X}_0$ to be the weak solution of the Poisson equation
\eqref{Poisson equation}.
Then we prove that the bilinear form
\begin{align}\label{L0-quadratic form}
 \langle L_0\omega_1,\omega_2\rangle=\iint_\Omega\left(\frac {\omega_1\omega_2} {g'(\psi_0)} - (-\Delta)^{-1}\omega_1\omega_2 \right) dxdy,\quad\omega_1,\omega_2\in X_0
\end{align}
 is bounded and symmetric  on $ X_0$.
\begin{lemma}\label{Lbounded}
For  $\omega_1,\omega_2 \in X_0$, we have
$\langle L_0 \omega_1, \omega_2 \rangle=\langle \omega_1, L_0 \omega_2 \rangle \leq C\|\omega_1\|_{X_0}\|\omega_2\|_{X_0}.$
\end{lemma}
\begin{proof}
For $\omega \in X_0$, let $\psi=(-\Delta)^{-1}\omega\in\tilde{X}_0$, we infer from  Lemma \ref{poincare1} that
\begin{align*}
\|\psi\|_{\tilde X_0}^2=\iint_\Omega \omega \psi dxdy
&   \leq C\|\omega\|_{X_0} \|\psi\|_{\tilde{X}_0},
\end{align*}
which gives $\|\psi\|_{\tilde X_0}\leq C \|\omega\|_{X_0}$. Let $\psi_i=(-\Delta)^{-1}\omega_i\in\tilde{X}_0$ for $i=1,2$. Then
\begin{align*}
\langle L_0 \omega_1, \omega_2 \rangle
 = &\iint_\Omega \left(\frac{\omega_1\omega_2}{g'(\psi_0)} dxdy -  \nabla\psi_1\cdot\nabla\psi_2\right) dxdy=
 \langle  \omega_1, L_0\omega_2 \rangle
\end{align*}
and
\begin{align*}
\langle L_0 \omega_1, \omega_2 \rangle
 \leq& \|\omega_1\|_{X_0}\|\omega_2\|_{X_0} + \|\psi_1\|_{\tilde{X}_0}\|\psi_2\|_{\tilde{X}_0}
 \leq C\|\omega_1\|_{X_0}\|\omega_2\|_{X_0}.
\end{align*}
\end{proof}

\subsubsection{Compact embedding lemma   and the variational problems}
Define
 \begin{align}\label{A0}
 \tilde{A}_0=-\Delta-g'(\psi_0)(I - P_0): \tilde{X}_0 \rightarrow \tilde{X}_0^*,
\end{align}
where the negative Laplacian operator should be understood in the weak sense.
Then
\begin{align}\label{A0-quadratic form}
 \langle\tilde{A}_0\psi,\psi\rangle=\iint_\Omega|\nabla\psi|^2-g'(\psi_0)(\psi - P_0\psi)^2dxdy,\quad\psi\in \tilde{X}_0
\end{align}
defines a  bounded symmetric quadratic form on $\tilde{X}_0$ by the Poincar\'e inequality II-0 \eqref{Poincare inequality II022}.
Define another elliptic operator without the projection
\begin{equation}\label{A0 without projection}
A_0 =-\Delta -g'(\psi_0):\tilde{X}_0 \rightarrow \tilde{X}_0^*.
\end{equation}
The corresponding quadratic form
\begin{align*}
\langle A_0 \psi,\psi\rangle=\iint_{\Omega}\left(|\nabla \psi|^2-g'(\psi_0)|\psi|^2\right)dxdy,\quad\psi\in \tilde{X}_0
\end{align*}
is bounded and symmetric  on $\tilde{X}_0$ by the Poincar\'e inequality I-$0$ \eqref{Poincare inequality I022}.
Then
\begin{align}\label{tilde A0-A0}
\langle\tilde  A_0 \psi,\psi\rangle=\langle A_0 \psi,\psi\rangle+{\left(\iint_\Omega g'(\psi_0)\psi dxdy\right)^2\over \iint_\Omega g'(\psi_0)dxdy}=\langle A_0 \psi,\psi\rangle+8\pi(P_0 \psi)^2,\quad \psi\in\tilde X_0,
\end{align}
where we used $ \iint_\Omega g'(\psi_0)dxdy=8\pi$.
In particular,
\begin{equation*}
n^{\leq0}(\tilde A_0)\leq n^{\leq0}(A_0),\quad n^{-}(\tilde A_0)\leq n^{-}(A_0),
\end{equation*}
where $n^{\leq0}(\tilde A_0)$ and $n^-(\tilde A_0)$ are the number of non-positive and negative eigenvalues of $\tilde A_0$, respectively.
The operator $A_0$ and its quadratic form are useful in our study on nonlinear stability of the steady states.

Then we show that the study  on the dimensions of kernel and negative subspaces of the quadratic form $\langle L_0\cdot,\cdot\rangle$
defined in \eqref{L0-quadratic form}
  could be reduced to the corresponding dimensions for $ \langle\tilde{A}_0\cdot,\cdot\rangle$.
\begin{lemma}\label{equal-indices0}
\begin{align*}
\dim\ker(L_0)=\dim\ker(\tilde{A}_0) \quad {\rm{and}} \quad n^-(L_0)=n^-(\tilde{A}_0).
\end{align*}
\end{lemma}
\begin{proof}
First, we prove that $\dim\ker(L_0) =\dim\ker(\tilde{A}_0)$.

For  $\omega \in \ker L_0$, let $\psi = (-\Delta)^{-1} \omega\in \tilde X_0$, we have
\begin{align}\label{L0-dual-omega}\langle L_0 \omega,\tilde \omega\rangle = \iint_\Omega\left(\frac{\omega\tilde \omega}{g'(\psi_0)} - \psi\tilde \omega\right) dxdy = 0,\quad\forall\; \tilde \omega\in X_0.
\end{align}
For any $\tilde \psi\in \tilde X_0$, we define $\omega_{\tilde \psi}=g'(\psi_0)(\tilde \psi - P_0\tilde\psi)$.
Then $\iint_\Omega\omega_{\tilde \psi}dxdy=0$, and thus, $\omega_{\tilde \psi}\in X_0$ by  Lemma \ref{poincare2}. By \eqref{L0-dual-omega}, we have
\begin{align*}\langle L_0 \omega, \omega_{\tilde \psi}\rangle= \iint_\Omega\left(\omega\tilde \psi - g'(\psi_0)\psi(\tilde\psi - P_0\tilde\psi)\right)dxdy =  \iint_\Omega\left(\omega\tilde \psi - g'(\psi_0)(\psi-P_0\psi)\tilde\psi\right) dxdy = 0,\end{align*}
where we used $\iint_\Omega\omega dxdy=0$ and $\iint_\Omega g'(\psi_0)(\tilde\psi - P_0\tilde\psi)dxdy=\iint_\Omega g'(\psi_0)(\psi-P_0\psi)dxdy=0$. This implies that $\psi\in \ker(\tilde{A}_0)$ since
\begin{align*}\langle\tilde{A}_0 \psi,\tilde \psi\rangle= \iint_\Omega\left(\omega\tilde \psi - g'(\psi_0)(\psi-P_0\psi)\tilde\psi \right)dxdy = 0,\quad\forall \; \tilde \psi\in \tilde X_0.\end{align*}
 Thus,  $\dim\ker(L_0) \leq \dim\ker(\tilde{A}_0)$.

For  $\psi \in \ker \tilde{A}_0$, let $\omega = g'(\psi_0)(\psi - P_0\psi)$, we have $\omega\in X_0$ and
\begin{align}\label{tilde-A0-psi}\langle\tilde{A}_0 \psi,\tilde\psi\rangle =\iint_\Omega \left(-\Delta \psi\tilde\psi - g'(\psi_0)(\psi - P_0\psi)\tilde\psi \right)dxdy =0,\quad\forall \; \tilde \psi\in \tilde X_0.\end{align}
For any $\tilde \omega\in X_0$, let $\psi_{\tilde \omega}=(-\Delta)^{-1}\tilde \omega\in \tilde X_0$, we have
\begin{align*}\langle L_0 \omega, \tilde{\omega}\rangle =& \iint_\Omega\left( {\omega\tilde \omega\over g'(\psi_0)}-(-\Delta)^{-1}\omega\tilde\omega \right)dxdy= \iint_\Omega \left((\psi-P_0\psi)\tilde \omega-\omega(-\Delta)^{-1}\tilde\omega \right) dxdy\\
=& \iint_\Omega \left(\psi(-\Delta)\psi_{\tilde \omega}-g'(\psi_0)(\psi - P_0\psi)\psi_{\tilde \omega} \right)dxdy\\
=&\iint_\Omega  \left( -\Delta\psi\psi_{\tilde \omega}-g'(\psi_0)(\psi - P_0\psi)\psi_{\tilde \omega}\right) dxdy=0
\end{align*}
by \eqref{tilde-A0-psi}, which gives  $L_0 \omega = 0$.  This proves  $\dim\ker(L_0) \geq \dim\ker(\tilde{A}_0)$, and thus, $\dim\ker(L_0) = \dim\ker(\tilde{A}_0)$.

For any $\omega \in X_0$, let $\psi = (-\Delta)^{-1} \omega\in \tilde X_0$ and we have
\begin{align}\nonumber
\langle L_0 \omega, \omega\rangle
& = \iint_\Omega \left(\frac{|\omega|^2}{g'(\psi_0)} -   \psi\omega\right) dxdy  = \iint_\Omega |\nabla \psi|^2 dxdy + \iint_\Omega \left(\frac{|\omega|^2}{g'(\psi_0)} -  2 \psi\omega\right) dxdy \\\nonumber
& =\|\nabla \psi\|_{L^2(\Omega)}^2 +\iint_\Omega\left( \frac{|\omega|^2}{g'(\psi_0)} -  2 (\psi-P_0\psi)\omega \right)dxdy\\\nonumber
& \geq\|\nabla \psi\|_{L^2(\Omega)}^2 - \iint_\Omega g'(\psi_0)(\psi - P_0\psi)^2 dxdy \\\label{L0omega-omega}
& = \|\nabla \psi\|_{L^2(\Omega)}^2 - \iint_\Omega g'(\psi_0)(\psi - P_0\psi)\psi dxdy  = \langle\tilde{A}_0 \psi, \psi\rangle.
\end{align}
Thus, $n^{\leq 0} (L_0) \leq n^{\leq 0} (\tilde{A}_0)$.

For any $\psi \in \tilde{X}_0$, let $\tilde{\omega} = g'(\psi_0)(\psi - P_0\psi)$, we have
$\tilde{\omega} \in X_0$, $\psi_{\tilde{\omega}} = (-\Delta)^{-1}\tilde{\omega}\in \tilde X_0$, and
\begin{align*}
\langle\tilde{A}_0 \psi, \psi\rangle
& = \iint_\Omega\left( |\nabla \psi|^2  - g'(\psi_0)(\psi - P_0\psi)^2 \right)dxdy
 = \iint_\Omega \left(|\nabla \psi|^2  - \frac{\tilde{\omega}^2}{g'(\psi_0)}\right)dxdy \\
& = \iint_\Omega \left(\frac{\tilde{\omega}^2}{g'(\psi_0)} + |\nabla \psi|^2  - 2  \tilde{\omega}(\psi - P_0\psi)\right)dxdy\\
&= \iint_\Omega \left(\frac{\tilde{\omega}^2}{g'(\psi_0)} + |\nabla \psi|^2  - 2  \tilde{\omega}\psi\right)dxdy
 = \iint_\Omega \left(\frac{\tilde{\omega}^2}{g'(\psi_0)} +  |\nabla \psi|^2 - 2  \nabla\psi_{\tilde\omega}\cdot\nabla\psi \right)dxdy \\
  &\geq \iint_\Omega \left(\frac{\tilde{\omega}^2}{g'(\psi_0)} -  |\nabla \psi_{\tilde \omega}|^2\right) dxdy= \langle L_0 \tilde{\omega}, \tilde{\omega}\rangle.
\end{align*}
This proves $n^{\leq 0} (L_0) \geq n^{\leq 0} (\tilde{A}_0)$. Then $n^{\leq 0} (L_0) =n^{\leq 0} (\tilde{A}_0)$, which, along with
$\dim\ker(L_0) = \dim\ker(\tilde{A}_0)$, gives $n^{-} (L_0) =n^{-} (\tilde{A}_0)$.
\end{proof}

To compute $n^{-}(\tilde{A}_0)$, we study the variational problem
\begin{align}\label{variational problem1}
\lambda_1= \inf_{\psi\in\tilde X_0}{\iint_\Omega|\nabla\psi|^2dxdy\over\iint_\Omega g'(\psi_0)(\psi - P_0\psi)^2dxdy}.
\end{align}
$\lambda_1$ is finite due to the Poincar\'e inequality II-$0$ \eqref{Poincare inequality II022}.
We
need the following compact embedding result.
\begin{lemma}\label{compact2P}
$(1)$ $\tilde X_0$ is compactly embedded in $L_{g'(\psi_0)}^2(\Omega)$.

 $ (2)$ $\tilde X_0$ is compactly embedded in
 \begin{equation*}
 Z_{0}:=\left\{\psi\bigg|\iint_{\Omega}g'(\psi_0)|\psi-P_0\psi|^2dxdy<\infty\right\}.
  \end{equation*}
\end{lemma}

\begin{proof} First, we prove (1).
By the Poincar\'e inequality I-$0$ \eqref{Poincare inequality I022}, $\tilde X_0$ is embedded in $L_{g'(\psi_0)}^2(\Omega)$. To prove that the embedding is compact, let $\{\psi_n\}_{n\geq1}$ be a bounded sequence in $\tilde{X}_0$. We decompose $\psi_n=\widehat{\psi}_{n,0}+\psi_{n,{\neq0}}$ as in \eqref{psi-m-dec}.  By  \eqref{tilde-X0-norm} we have
 \begin{align}\label{psi-n-0-mode-unidform bound}
 \|\widehat{\psi}_{n,0}'\|_{L^2(\mathbb{R})}<C \quad\text{and} \quad\|\psi_{n,{\neq0}}\|_{H^1(\Omega)}<C,\quad n\geq1.
 \end{align}
For any $\kappa > 0$, there exists $K>0$ such that $g'(\psi_0(y))=2\sech^2(y)<\kappa$ for $y\in (-\infty,-K]\cup[K,\infty)$, and
 \begin{align*}
 \int_{(-\infty, -K)\cup(K,\infty)}g'(\psi_0)|y|dy =2\int_{(-\infty, -K)\cup(K,\infty)} \sech^2(y)|y|dy <\kappa.
 \end{align*}
  Then by \eqref{psi-n-0-mode-unidform bound} and $\widehat{\psi}_{n,0}(0)=0$ for $n\geq1$, we have
\begin{align*}
&\int_{(-\infty, -K)\cup(K,\infty)}g'(\psi_0)(\widehat{\psi}_{n,0}-\widehat{\psi}_{m,0})^2d y\\
 \leq& \|\widehat{\psi}_{n,0}'-\widehat{\psi}_{m,0}'\|_{L^2(\mathbb{R})}^2\int_{(-\infty, -K)\cup(K,\infty)}g'(\psi_0)|y|dy \leq C\kappa
\end{align*}
and
\begin{align*}
\int_0^{2\pi}\int_{(-\infty, -K)\cup(K,\infty)} g'(\psi_0) (\psi_{n,{\neq 0}}-\psi_{m,{\neq 0}})^2 dydx \leq  \kappa \| \psi_{n,{\neq 0}} - \psi_{m,{\neq 0}}\|_{H^1(\Omega)}^2 \leq C \kappa
\end{align*}
for $m,n\geq1$.
Thus,
\begin{align*}
&\int_0 ^ {2\pi} \int_{(-\infty, -K)\cup(K,\infty)} g'(\psi_0) (\psi_n -\psi_m)^2 dydx \\
\leq&2\int_0 ^ {2\pi} \int_{(-\infty, -K)\cup(K,\infty)} g'(\psi_0) \left((\widehat{\psi}_{n,0}-\widehat{\psi}_{m,0})^2+ (\psi_{n,{\neq 0}}-\psi_{m,{\neq 0}})^2 \right)dydx\leq C \kappa.
\end{align*}
Since $\|\widehat{\psi}_{n,0}\|_{L^2(-K,K)}^2\leq 2K^2\|\widehat{\psi}_{n,0}'\|_{L^2(-K,K)}^2\leq C_K$, we infer from \eqref{psi-n-0-mode-unidform bound} that
$\{\sqrt{g'(\psi_0)}\psi_n\}_{n\geq1}$ is a  bounded sequence in $H^1(\mathbb{T}_{2\pi} \times [-K,K])$. Since the embedding $H^1\hookrightarrow L^2(\mathbb{T}_{2\pi} \times [-K,K])$ is compact, then up to a subsequence, there exists $N>0$ such that $\|\psi_n -\psi_m\|_{L^2_{g'(\psi_0)}(\mathbb{T}_{2\pi} \times [-K,K])}=\|\sqrt{g'(\psi_0)}(\psi_n-\psi_m)\|_{L^2(\mathbb{T}_{2\pi} \times [-K,K])}<\kappa$ for $m,n>N$. Thus, up to a subsequence,
\begin{align*}
& \|\psi_n -\psi_m\|_{L^2_{g'(\psi_0)}(\Omega)}^2 = \|\sqrt{g'(\psi_0)}(\psi_n-\psi_m)\|_{L^2(\mathbb{T}_{2\pi} \times [-K,K])}^2  \\ &+\|\sqrt{g'(\psi_0)}(\psi_n-\psi_m)\|_{L^2(\mathbb{T}_{2\pi} \times ((-\infty, -K)\cup(K,\infty)))}^2   \leq \kappa^2+ C \kappa
\end{align*}
for $m,n>N$, which implies that there exists $\psi_* \in L^2_{g'(\psi_0)}(\Omega)$ such that $\psi_n\rightarrow\psi_*$ in $ L^2_{g'(\psi_0)}(\Omega)$.

Then we prove (2).
By the Poincar\'e inequality II-$0$ \eqref{Poincare inequality II022}, $\tilde X_0$ is embedded in $Z_0$. Let $\{\psi_n\}_{n\geq1}$ be a bounded sequence in $\tilde{X}_0$. By (1), we know that there exists $\psi_* \in L^2_{g'(\psi_0)}(\Omega)$ such that, up to a subsequence, $\psi_n \rightarrow \psi_*$ in $L^2_{g'(\psi_0)}(\Omega)$, and  it follows from \eqref{p0-psi-estimates-2} that
\begin{align*}
|P_0( \psi_n - \psi_*)| \leq C \|\psi_n - \psi_*\|_{L^2_{g'(\psi_0)}(\Omega)} \rightarrow 0\quad \text{as}\quad n \rightarrow  \infty.
\end{align*}
Thus, up to a subsequence, we have
\begin{align*}
&\iint_{\Omega} g'(\psi_0) \left((\psi_n - \psi_*) - P_0 (\psi_n -  \psi_*)\right)^2 dxdy \\
\leq & 2 \iint_{\Omega} g'(\psi_0) \left((\psi_n - \psi_*)^2 + \left(P_0 (\psi_n -  \psi_*)\right)^2\right) dxdy \\
\leq & 2\|\psi_n - \psi_*\|_{L^2_{g'(\psi_0)}(\Omega)}^2 + C|P_0( \psi_n -  \psi_*)|^2 \\
\leq & C \|\psi_n - \psi_*\|_{L^2_{g'(\psi_0)}(\Omega)}^2 \rightarrow 0 \quad\text{as}\quad n \rightarrow  \infty.
\end{align*}
\end{proof}

Since  the embedding  $\tilde X_0\hookrightarrow Z_{0}$ is  compact, a standard argument in variational method implies that
 the  infimum in  \eqref{variational problem1} can be attained in $\tilde X_0$, and we can
inductively define $\lambda_n$ as follows for $n\geq1$,
\begin{align}\nonumber
\lambda_n=& \inf_{\psi \in \tilde X_0, (\psi, \psi_{i})_{Z_0} = 0, i = 1, 2, \cdots, n-1}{\iint_\Omega|\nabla\psi|^2dxdy\over\iint_\Omega g'(\psi_0)(\psi - P_0\psi)^2dxdy}\\\label{variational problem2}
=&\min_{\psi \in \tilde X_0, (\psi, \psi_{i})_{Z_0} = 0, i = 1, 2, \cdots, n-1}{\iint_\Omega|\nabla\psi|^2dxdy\over\iint_\Omega g'(\psi_0)(\psi - P_0\psi)^2dxdy},
\end{align}
where the infimum for $\lambda_i$ is attained at $\psi_{i} \in \tilde X_0$ and $\iint_\Omega g'(\psi_0)(\psi_{i} - P_0\psi_{i})^2 dxdy = 1$, $1\leq i \leq n-1$.
To solve the variational problem \eqref{variational problem2}, we compute the first variation of $G(\psi)={\iint_\Omega|\nabla\psi|^2dxdy\over\iint_\Omega g'(\psi_0)(\psi - P_0\psi)^2dxdy}$ at $\psi_{n}$:
\begin{align*}
\frac{d}{d \tau} G(\psi_{n} + \tau \psi)|_{\tau = 0} = \iint_{\Omega} 2\left(-\Delta\psi_n - \lambda_ng'(\psi_0)(\psi_n - P_0\psi_n)\right)\psi dxdy,\quad\forall\; \psi\in \tilde X_0.
\end{align*}
Due to the fact that  $\widehat{\psi}_0(0)=0$ for $\psi\in \tilde X_0$, we
derive the corresponding Euler-Lagrangian equation
\begin{align}\label{elip0}
-\Delta \psi = \lambda g'(\psi_0)(\psi - P_0\psi)+a\delta(y), \quad \psi \in \tilde{X}_0,
\end{align}
where $\delta$ is the Dirac delta function and $a\in\mathbb{R}$ is to be determined. Thanks to the projection $P_0$, integrating \eqref{elip0} on $\Omega$, we have
$$
2\pi a=\iint_\Omega-\Delta \psi -\lambda g'(\psi_0)(\psi - P_0\psi) dxdy=0\;\;\Longrightarrow \;\;a=0,
$$
and thus,  we arrive at the associated eigenvalue problem
\begin{align}\label{elip02}
-\Delta \psi = \lambda g'(\psi_0)(\psi - P_0\psi), \quad \psi \in \tilde{X}_0.
\end{align}
Since $g'(\psi_0)$ depends only on $y$, we can use the Fourier expansion of
$\psi$ to separate the variables.
Since $\psi(x,y)=\sum_{k\in\mathbb{Z}}\widehat{\psi}_{k}(y)e^{ik x} \in \tilde{X}_0$, we infer from \eqref{tilde-X0-norm} that
\begin{align}\label{def-space-Y0Yk}
\widehat{\psi}_0 \in Y_0 = \{\phi | \phi \in \dot{H}^1(\mathbb{R}), \phi(0) = 0\}\quad\text{and}\quad\widehat{\psi}_{k} \in Y_1 =  H^1(\mathbb{R})\quad \text{for}\quad k\neq0.
\end{align}
Plugging the Fourier expansion $\psi(x,y)=\sum_{k\in\mathbb{Z}}\widehat{\psi}_{k}(y)e^{ik x}$ into \eqref{elip02}, we get the eigenvalue problem for the  $0$-mode
\begin{align}\label{mode0}
- \phi'' = 2\lambda \sech^2(y) (I - P_0) \phi, \quad \phi\in Y_0,
\end{align}
with $$P_0 \phi = \frac{1}{2} \int_{\mathbb{R}} \sech^2(y) \phi(y) dy,$$
and the eigenvalue problem for the  $k$-mode
\begin{align}\label{modek}
-  \phi'' + k^2 \phi = 2\lambda \sech^2(y)\phi, \quad \phi\in Y_1, \quad k \neq 0,
\end{align}
since $$P_0 (\phi e^{ikx}) =\frac{1}{4\pi} \iint_\Omega  \sech^2(y) \phi(y)e^{ikx} dxdy = 0.$$

\subsection{Exact solutions to the associated  eigenvalue problems for the shear case}
\subsubsection{A change of variable}\label{Change of variable}

Our motivation for introducing a change of variable comes from the eigenvalue problem \eqref{mode0} for the \(0\)-mode. Differentiating the steady-state equation
\[
-\Delta \psi_0=g(\psi_0)
\]
with respect to \(y\), we find that \(\lambda=1\) is an eigenvalue of \eqref{mode0}, with corresponding eigenfunction \(\tanh(y)\); see also (16.3) in \cite{Laugesen2011}. Guided by the numerical computation in Subsection \ref{eigenfunction-motivation}, we further identify another eigenvalue \(\lambda=3\) with eigenfunction \(\tanh^2(y)\). This suggests that all eigenfunctions of \eqref{mode0} might be polynomials in \(\tanh(y)\). Substituting such polynomials into \eqref{mode0}, we obtain the following five eigenvalues and corresponding eigenfunctions:

\begin{equation}\label{eigen value-function}\begin{aligned}\begin{array}{llll}
&\lambda_1=1=1, &\phi_1(y)=\tanh(y),\\
&\lambda_2=1+2=3, &\phi_2(y)=\tanh^2(y),\\
&\lambda_3=1+2+3=6, &\phi_3(y)=5\tanh^3(y)-3\tanh(y),\\
&\lambda_4=1+2+3+4=10, &\phi_4(y)=7\tanh^4(y)-6\tanh^2(y),\\
&\lambda_5=1+2+3+4+5=15, &\phi_5(y)=9\tanh^5(y)-10\tanh^3(y)+{15\over7}\tanh(y).
\end{array}
\end{aligned}
\end{equation}
This suggests that the eigenvalues of \eqref{mode0} might be given by
\[
\lambda_n=\frac{n(n+1)}{2},
\]
with eigenfunctions that are polynomials in \(\tanh(y)\). Guided by \eqref{eigen value-function}, we therefore introduce the change of variable
\begin{align}\label{change of variable for 0 mode}
\gamma=\tanh(y)\in(-1,1).
\end{align}
The key point is that this transformation converts the eigenvalue problem \eqref{mode0} for the \(0\)-mode and the eigenvalue problem \eqref{modek} for the non-zero modes into classical Legendre and general Legendre equations, with the projection terms and function spaces built into the formulation. This will be explained in the next subsection. For the Kelvin--Stuart vortices \(\omega_\ep\) with \(0<\ep<1\), we later introduce a more delicate change of variables in Subsection \ref{change of variables for cat's eyes  flows}, which again reduces the corresponding eigenvalue problems to Legendre-type boundary value problems. As a result, the stability analysis of Kelvin--Stuart vortices turns out to be closely related to spherical harmonics.

In the new variables $(x,\gamma)$, we rewrite the spaces of stream functions $\tilde X_0$ and  $Z_{0}$,  Poincar\'e inequality I-II (see \eqref{Poincare inequality I022}, \eqref{Poincare inequality II022}) and the compact embedding   $\tilde X_0\hookrightarrow Z_{0}$, respectively. These statements in the new variables are also useful in establishing the correspondence of stream functions between the hyperbolic tangent shear case ($\ep=0$) and the cat's-eye case ($0<\ep<1$).

First, the space $\tilde X_0$ in \eqref{tilde-X0} is rewritten as the following space in the new variables $(x,\gamma)$.
\begin{lemma}\label{Hilbert-new variables-0}
The function space
\begin{align}\label{tilde-Y0-def}
\tilde{Y}_0 = \left\{ \Psi\bigg|\iint_{\tilde \Omega}\left({1\over1-\gamma^2}|\Psi_{x}|^2+(1-\gamma^2)|\Psi_{\gamma}|^2\right)d x d\gamma< \infty \text{ and } \widehat{\Psi}_0(0)=0 \right\}
\end{align}
equipped with the inner product $$(\Psi_1, \Psi_2) = \iint_{\tilde{\Omega}}  \left({1\over1-\gamma^2}(\Psi_1)_{x}(\Psi_2)_{x} +(1-\gamma^2)(\Psi_1)_{\gamma}(\Psi_2)_{\gamma}\right)d x d\gamma, \quad \forall\; \Psi_1, \Psi_2 \in \tilde{Y}_0$$ is a Hilbert space, where $\tilde \Omega=\mathbb{T}_{2\pi}\times [-1,1]$.
\end{lemma}
\begin{proof}
For $\psi_i(x,y)=\Psi_i(x,\gamma)$, $i=1,2$, we have
\begin{align}\label{cor-0-new}
\iint_\Omega\nabla \psi_1\cdot\nabla \psi_2dxdy
=\iint_{\tilde \Omega}\left({1\over1-\gamma^2}(\Psi_1)_{x}(\Psi_2)_{x}+(1-\gamma^2)(\Psi_1)_{\gamma}(\Psi_2)_{\gamma}\right)d x d\gamma.
\end{align}
Moreover, $y=0\Longleftrightarrow\gamma=0$, and thus,
\begin{align}\label{0 mode x gamma}
\widehat{\psi}_0(0)=\widehat{\Psi}_0(0)
\end{align}
for $\psi(x,y)=\Psi(x,\gamma)$.
The conclusion follows from \eqref{cor-0-new}-\eqref{0 mode x gamma} and  the fact that $\tilde X_0$ is a Hilbert space by Lemma \ref{Hilbert}.
\end{proof}
Let  $\psi \in \tilde{X}_0$ and $\Psi \in \tilde{Y}_0$ such that $\psi(x,y) = \Psi(x, \gamma)$. It follows from \eqref{cor-0-new} that
\begin{align}\label{cor-0-new-norm}
\|\psi\|_{\tilde X_0}^2= \|\nabla \psi\|_{L^2(\Omega)}^2=\iint_{\tilde \Omega}\left({1\over1-\gamma^2}|\Psi_{x}|^2+(1-\gamma^2)|\Psi_{\gamma}|^2\right)d x d\gamma=\|\Psi\|_{\tilde Y_0}^2.
\end{align}
Corresponding to $P_0$ in \eqref{P0-psi-def}, we define a $1$-dimensional projection operator $\tilde P_0$ on $ \tilde{Y}_0$ by
 \begin{align}\label{def-tilde-P0-Psi}
 \tilde P_0\Psi = \frac{\iint_{\tilde\Omega} \Psi dxd\gamma}{\iint_{\tilde\Omega}  dxd\gamma}=\frac{\iint_{\tilde \Omega} \Psi dxd\gamma}{4\pi},\quad \Psi \in \tilde{Y}_0.
 \end{align}
Then we prove that $\tilde P_0$ is well-defined on $ \tilde{Y}_0$, and give the Poincar\'e-type inequalities in the new variables $(x,\gamma)$.
\begin{lemma}\label{Poincare ineqalities-new-variable0}
$(1)$ Poincar\'e inequality $\textup{I-}0'$:
\begin{align*}
\|\Psi\|_{L^2(\tilde\Omega)}^2  \leq C \iint_{\tilde \Omega}\left({1\over1-\gamma^2}|\Psi_{x}|^2+(1-\gamma^2)|\Psi_{\gamma}|^2\right)d x d\gamma=C\|\Psi\|_{\tilde Y_0}^2, \quad \Psi \in \tilde{Y}_0.
\end{align*}

$(2)$ The projection operator $\tilde P_0 $ is well-defined  on $ \tilde{Y}_0$, $|\tilde P_0\Psi|\leq C \|\Psi\|_{\tilde Y_0}$, and $P_0\psi=\tilde P_0\Psi$ for  $\psi \in \tilde{X}_0$ and $\Psi \in \tilde{Y}_0$ such that $\psi(x,y) = \Psi(x, \gamma)$.

$(3)$ Poincar\'e inequality $\textup{II-}0'$:
\begin{align*}
\iint_{\tilde \Omega}|\Psi - \tilde P_0\Psi|^2 dxd\gamma  \leq C \iint_{\tilde \Omega}\left({1\over1-\gamma^2}|\Psi_{x}|^2+(1-\gamma^2)|\Psi_{\gamma}|^2\right)d x d\gamma=C\|\Psi\|_{\tilde Y_0}^2, \quad \Psi \in \tilde{Y}_0.
\end{align*}
\end{lemma}
\begin{proof}  Let $\psi(x,y)=\Psi(x,\gamma)$. Then $\psi\in \tilde {X}_0$. First, we prove (1).
By Lemma  \ref{poincare1} and \eqref{cor-0-new-norm}, we have
\begin{align*}
&2\iint_{\tilde \Omega}|\Psi|^2 dxd\gamma=\iint_\Omega g'(\psi_0)|\psi|^2 dxdy  \\
\leq& C \|\nabla \psi\|_{L^2(\Omega)}^2=C \iint_{\tilde \Omega}\left({1\over1-\gamma^2}|\Psi_{x}|^2+(1-\gamma^2)|\Psi_{\gamma}|^2\right)d x d\gamma.
\end{align*}

Next, we prove (2). By \eqref{P0-psi-def} and \eqref{def-tilde-P0-Psi}, we have $P_0\psi=\tilde P_0\Psi$.  Thus, we infer from
\eqref{p0-psi-estimates-2} that
\begin{align*}
|\tilde P_0\Psi|=|P_0\psi|\leq C \| \psi\|_{\tilde X_0}=C \|\Psi\|_{\tilde Y_0}.
\end{align*}

Finally, we prove (3).
By Lemma \ref{poincare2}, $P_0\psi=\tilde P_0\Psi$ and \eqref{cor-0-new-norm} we have
\begin{align*}
&2\iint_{\tilde \Omega} |\Psi -\tilde P_0\Psi|^2 dxd\gamma =\iint_\Omega g'(\psi_0)|\psi - P_0\psi|^2 dxdy  \\
\leq & C \|\nabla \psi\|_{L^2(\Omega)}^2=C \iint_{\tilde \Omega}\left({1\over1-\gamma^2}|\Psi_{x}|^2+(1-\gamma^2)|\Psi_{\gamma}|^2\right)d x d\gamma.
\end{align*}
\end{proof}

Now we give the  compact embedding lemma in the new variables.
\begin{lemma}\label{compact2P-new-variable0}
$(1)$ $\tilde Y_0$ is compactly embedded in $L^2(\tilde \Omega)$.

$(2)$  $\tilde Y_0$ is compactly embedded in
 \begin{equation*}
 \tilde Z_{0}:=\left\{\Psi\bigg|\iint_{\tilde \Omega}|\Psi-\tilde P_0\Psi|^2dxd\gamma<\infty\right\}.
  \end{equation*}
\end{lemma}
\begin{proof} We only prove (2), and the proof of (1) is similar.
By Lemma \ref{Poincare ineqalities-new-variable0} (3), $\tilde Y_0$ is embedded in $\tilde Z_{0}$. Let $\{\Psi_n\}_{n\geq1}$ be a bounded sequence in $\tilde{Y}_0$ and $\psi_n(x,y)=\Psi_n(x,\gamma)$. Then it follows from \eqref{cor-0-new-norm} that  $\{\psi_n\}_{n\geq1}$ is a bounded sequence in $\tilde{X}_0$. By Lemma \ref{compact2P} (2), there exists $\psi_*\in Z_0$ such that up to a subsequence, $\|\psi_n-\psi_*\|_{Z_0}\to0$. Let $\Psi_*(x,\gamma)=\psi_*(x,y)$. Then $\Psi_*\in \tilde Z_0$ and up to a subsequence,  $\|\Psi_n-\Psi_*\|_{\tilde Z_0}=\|\psi_n-\psi_*\|_{Z_0}\to0$.
\end{proof}
\subsubsection{Solutions to the eigenvalue problems}\label{Solutions to the eigenvalue problem ep=0}
 We  study the eigenvalue problems \eqref{mode0} for the $0$-mode and \eqref{modek} for the non-zero modes, separately.\medskip

\noindent{\bf{Eigenvalue problem for the $0$-mode.}}\medskip

In this part, we solve the eigenvalue problem \eqref{mode0} for the $0$-mode. We use the change of variable $\gamma=\tanh(y)$ and
denote $\phi(y)=\phi(\tanh^{-1}(\gamma))=\varphi(\gamma)$.
 Then $d\gamma=(1-\gamma^2)dy={1\over2}g'(\psi_0)dy$ and
\begin{align*}
\phi'(y)&=(1-\gamma^2)\varphi'(\gamma),\;\;\phi''(y)=(1-\gamma^2)(-2\gamma\varphi'(\gamma)+(1-\gamma^2)\varphi''(\gamma)),\\
P_0\phi&={1\over4}\int_{\mathbb{R}}g'(\psi_0)\phi(y)dy={1\over2}\int_{-1}^{1}\varphi(\gamma)d\gamma=:\hat P_0\varphi.
\end{align*}

Since
\begin{align}\label{phi-der-varohi-der}
\int_{\mathbb{R}}|\phi'(y)|^2dy=\int_{-1}^1(1-\gamma^2)|\varphi'(\gamma)|^2d\gamma,
\end{align}
 the space $Y_0$ (see \eqref{def-space-Y0Yk}) for $\phi$ in the variable $y$ is transformed to
\begin{equation*}
\hat Y_0=\left\{\varphi\bigg|\int_{-1}^1(1-\gamma^2)|\varphi'(\gamma)|^2d\gamma<\infty\text{ and }\varphi(0)=0 \right\}
\end{equation*}
for $\varphi$ in the new variable $\gamma$.
Thus, the eigenvalue problem \eqref{mode0} is transformed to
\begin{align}\label{eigenvalue problem for 0 mode}
-\left((1-\gamma^2)  \varphi'\right)' = 2 \lambda(\varphi-\hat{P}_0\varphi) \quad \text{on}\quad (-1,1),\quad\varphi\in \hat Y_0.
\end{align}
If we neglect the term $-2\lambda\hat P_0\varphi$ and change the space $\hat Y_0$ to $L^2(-1,1)$ for a while, \eqref{eigenvalue problem for 0 mode} surprisingly becomes the  Legendre equation
\begin{equation}\label{eigenvalue problem3}
-\left((1-\gamma^2)  \varphi'\right)' = 2 \lambda\varphi \quad \text{on}\quad (-1,1),\quad \varphi\in L^2(-1,1).
\end{equation}
If we require that  the solution is regular at $\gamma=\pm1$, then
it is well-known that the eigenvalues to the boundary value problems \eqref{eigenvalue problem3} are  $\lambda_n={n(n+1)\over2}$ for $n\geq0$, and the corresponding eigenfunctions are the Legendre polynomials
\begin{align*}
L_n(\gamma)={1\over2^nn!}{d^n\over d\gamma^n}(\gamma^2-1)^n.
\end{align*}
Moreover, $\{L_n\}_{n=0}^\infty$ is a complete and  orthogonal basis in $L^2(-1,1)$ \cite{Weidmann80}.

By \eqref{phi-der-varohi-der} and the fact that $d\gamma=(1-\gamma^2)dy={1\over2}g'(\psi_0)dy$, we get the Poincar\'e inequalities
 in the new variable $\gamma$, which are direct consequence of Lemma \ref{Poincare ineqalities-new-variable0} (1) and (3).
\begin{lemma}\label{Poincare inequalities compact embedding result  new variables}
For any $\varphi \in \hat{Y}_0$, we have
\begin{align*}
 \|\varphi\|^2_{L^2(-1,1)}   \leq C\int_{-1}^1 (1-\gamma)^2| \varphi'|^2d\gamma,\quad \|\varphi-\hat P_0\varphi\|^2_{L^2(-1,1)}   \leq C\int_{-1}^1 (1-\gamma)^2| \varphi'|^2d\gamma.
\end{align*}
\end{lemma}
Thus, in the new variable $\gamma$, $\hat Y_0$ is embedded in $L^2(-1,1)$. Let us compare the eigenfunctions $\phi_{n}$, $1\leq n\leq 5$, in \eqref{eigen value-function} with the  Legendre polynomials
\begin{equation*}\begin{aligned}\begin{array}{llll}
&L_1(\gamma)=\gamma,
\quad L_2(\gamma)={1\over2}(3\gamma^2-1),\quad
L_3(\gamma)={1\over2}(5\gamma^3-3\gamma),\\
&L_4(\gamma)={1\over8}(35\gamma^4-30\gamma^2+3),\quad
L_5(\gamma)={1\over8}(63\gamma^5-70\gamma^3+15\gamma).
\end{array}
\end{aligned}
\end{equation*}
Then we find that up to a constant factor,
\begin{equation*}
\phi_{n}(y)=L_{n}(\tanh(y))-L_{n}(0)=L_{n}(\gamma)-L_{n}(0), \;\;1\leq n\leq 5.
\end{equation*}
This provides a hint that
 the eigenvalues  for  \eqref{eigenvalue problem for 0 mode} might be $\lambda_n={n(n+1)\over2}$, $n\geq1$, with corresponding eigenfunctions $L_{n}(\gamma)-L_{n}(0)$, which is confirmed in the next lemma.

\begin{lemma}\label{sol to eigenvalue problem}
All the eigenvalues  of the eigenvalue problem \eqref{eigenvalue problem for 0 mode} are $\lambda_n={n(n+1)\over2}$, $n\geq1$. For $n\geq1$, the eigenspace associated to $\lambda_n={n(n+1)\over2}$ is $\text{span}\{L_{n}(\gamma)-L_{n}(0)\}.$  Consequently, all the eigenvalues  of the eigenvalue problem \eqref{mode0} are $\lambda_n={n(n+1)\over2}$, $n\geq1$. For $n\geq1$, the eigenspace associated to $\lambda_n={n(n+1)\over2}$ is $\text{span}\{L_{n}(\tanh(y))-L_{n}(0)\}$.
\end{lemma}
\begin{proof}
Due to the presence of the projection term, we need to  check that $\varphi(\gamma)=\varphi_n(\gamma)=L_{n}(\gamma)-L_{n}(0)\in \hat Y_0$ and $\lambda=\lambda_n={n(n+1)\over2}$  solve \eqref{eigenvalue problem for 0 mode}.
Thanks to the property of Legendre polynomials that
\begin{equation*}\label{int-0}\int_{-1}^{1} L_n(\gamma) d\gamma=0\end{equation*}
for $n\geq1$ \cite{Byerly59},
we have
 $\hat P_0\varphi_n=\hat P_0 (L_n(\gamma)-L_n(0))=-L_n(0)$,  and thus,
\begin{equation*}\begin{aligned}\begin{array}{llll}
&((1-\gamma^2)\varphi_n')'+2\lambda(\varphi_n-\hat P_0\varphi_n)=(1-\gamma^2)\varphi_n''-2\gamma\varphi_n'+2\lambda(\varphi_n-\hat P_0\varphi_n)\\
=&(1-\gamma^2)(L_n(\gamma)-L_n(0))''-2\gamma(L_n(\gamma)-L_n(0))'+2\lambda((L_n(\gamma)-L_n(0))+L_n(0))\\
=&(1-\gamma^2)L_n''(\gamma)-2\gamma L_n'(\gamma)+2\lambda L_n(\gamma)=0.
\end{array}
\end{aligned}
\end{equation*}
Since  $\varphi_n(0)=0$ and $\int_{-1}^1(1-\gamma^2)|\varphi_n'(\gamma)|^2d\gamma<\infty$, we have $\varphi_n\in \hat Y_0$.
So $\varphi_n$ solves  \eqref{eigenvalue problem for 0 mode}.

Next, we prove that the eigenspace associated to $\lambda_n={n(n+1)\over2}$ is $\text{span}\{\varphi_n\}$, and there are no more eigenvalues for \eqref{eigenvalue problem for 0 mode}. From the variational problem, we know that it suffices to prove that $\{\varphi_n\}_{n=1}^\infty$ is a complete and orthogonal basis of  $\hat Y_0$ under the inner product
$$(\varphi_1, \varphi_2)_{ \hat Z_0} = \int_{-1}^1 (\varphi_1 - \hat{P}_0\varphi_1)(\varphi_2  - \hat{P}_0 \varphi_2)  d \gamma, \quad\forall \varphi_1, \varphi_2 \in \hat Z_0,$$
where $\hat Z_0 := \{ \varphi | \int_{-1}^1 |\varphi - \hat{P}_0 \varphi|^2 d \gamma < \infty \}$ corresponds to the space $\{ \phi | \int_{\mathbb{R}} g'(\psi_0) |\phi - {P}_0 \phi|^2 d y < \infty \}$ in the original variable $y$.

 To see this, we note that
\begin{align}\nonumber
(\varphi_n,\varphi_m)_{\hat Z_0}=&\int_{-1}^1(\varphi_n-\hat P_0 \varphi_n)(\varphi_m-\hat P_0 \varphi_m)d\gamma=\int_{-1}^1(\varphi_n+L_n(0))(\varphi_m+L_m(0))d\gamma\\\nonumber
=&\int_{-1}^1L_nL_md\gamma=\left\{\begin{array}{llll}0,\;\;\;\;\;\;\;\text{if}\;\;m\neq n,\\
{2\over2n+1},\;\;\text{if}\;\;m= n.\end{array}\right.
\end{align}
This proves the orthogonality of  $\{\varphi_n\}_{n=1}^\infty$.
For any $\varphi\in \hat Y_0$, by Lemma \ref{Poincare inequalities compact embedding result  new variables}  we have  $\varphi\in L^2(-1,1)$  and thus, $\varphi(\gamma)=\sum_{n=0}^{\infty}a_nL_n(\gamma)$, where $a_n={{2n+1\over2}}\int_{-1}^1\varphi L_n d\gamma$.
Note that $\varphi\in \hat Y_0$ implies  $\varphi(0)=\sum_{n=0}^{\infty}a_nL_n(0)=0$. Thus, we have \begin{align*}
\varphi(\gamma)=\sum_{n=0}^{\infty}a_n(L_n(\gamma)-L_n(0))=\sum_{n=1}^{\infty}a_n\varphi_n(\gamma)
\end{align*}
 for $\gamma\in(-1,1)$, with
\begin{align*}
a_n={{2n+1\over2}}\int_{-1}^1(\varphi-\hat P_0 \varphi)(\varphi_n-\hat P_0 \varphi_n)d\gamma=(\varphi,\varphi_n)_{\hat Z_0}.
\end{align*}
 For any $\varepsilon>0$, there exists $N_\varepsilon>0$ such that
 \begin{align*}
 \left\|\varphi-\sum_{n=0}^{N_\varepsilon}a_nL_n\right\|_{L^2(-1,1)}<{\varepsilon\over4}\quad
 \text{and}\quad\left|\sum_{n=0}^{N_\varepsilon}a_nL_n(0)\right|<{\sqrt{2}\varepsilon\over8}. \end{align*}
 Then \begin{align*}
 \left\|\hat P_0\left(\varphi-\sum_{n=1}^{N_\varepsilon}a_n\varphi_n\right)\right\|_{L^2(-1,1)}=\sqrt{2}\left|\hat P_0\left(\varphi-\sum_{n=1}^{N_\varepsilon}a_n\varphi_n\right)\right|\leq  \left\| \varphi-\sum_{n=1}^{N_\varepsilon}a_n\varphi_n\right\|_{L^2(-1,1)},
 \end{align*}
 and
 \begin{align}\nonumber
 &\left\|\varphi-\sum_{n=1}^{N_\varepsilon}a_n\varphi_n\right\|_{\hat Z_0}\leq
 \left\|\varphi-\sum_{n=1}^{N_\varepsilon}a_n\varphi_n\right\|_{L^2(-1,1)}+\left\|\hat P_0\left(\varphi-\sum_{n=1}^{N_\varepsilon}a_n\varphi_n\right)\right\|_{L^2(-1,1)}\\\nonumber
 \leq &
2\left\|\varphi-\sum_{n=1}^{N_\varepsilon}a_n\varphi_n\right\|_{L^2(-1,1)}= 2\left\|\varphi-\sum_{n=0}^{N_\varepsilon}a_n(L_n-L_n(0))\right\|_{L^2(-1,1)}
 \\\nonumber
 \leq&
2\left\|\varphi-\sum_{n=0}^{N_\varepsilon}a_nL_n\right\|_{L^2(-1,1)}+2\left\|\sum_{n=0}^{N_\varepsilon}a_nL_n(0)\right\|_{L^2(-1,1)}<{\varepsilon\over2}+{\varepsilon
 \over2}=\varepsilon.
 \end{align}
This proves the completeness of  $\{\varphi_n\}_{n=1}^\infty$.
\end{proof}

\noindent{\bf{Eigenvalue problem for the non-zero mode.}}\medskip

For the $k$-mode with $k\neq0$, we solve the eigenvalue problem
\eqref{modek}.
It suffices to consider $k\geq1$. We use the change of variable  \eqref{change of variable for 0 mode}
 and denote $\phi(y)=\varphi(\gamma)$.
Since $$\|\phi\|_{H^1(\mathbb{R})}^2=\int_{-1}^1\left({1\over1-\gamma^2}|\varphi(\gamma)|^2+(1-\gamma^2)|\varphi'(\gamma)|^2\right)d\gamma,$$
the space $Y_1=H^1(\mathbb{R})$ for $\phi$ in the variable $y$ is transformed to
\begin{align}\label{def-hat-Y1}
\hat Y_1=\left\{\varphi\bigg|\int_{-1}^1\left({1\over1-\gamma^2}|\varphi(\gamma)|^2+(1-\gamma^2)|\varphi'(\gamma)|^2\right)d\gamma<\infty \right\}
\end{align}
for $\varphi$ in the new variable $\gamma$. Then the eigenvalue problem \eqref{modek}  is equivalent to the general Legendre equation
\begin{equation}\label{eigenvalue problem2 non-zero modes varepsilon=0}
-((1-\gamma^2)\varphi')'+{k^2\over1-\gamma^2}\varphi =2\lambda\varphi \quad \text{on}\quad (-1,1),\quad\varphi\in \hat Y_1.
\end{equation}

 The Poincar\'e inequality in Lemma \ref{Poincare ineqalities-new-variable0} (3) reads as follows.
\begin{lemma}\label{Poincare inequalities compact embedding result new variables k mode}
For any $\varphi \in \hat{Y}_1$, we have
\begin{align*}
 \|\varphi\|^2_{L^2(-1,1)}   \leq C\int_{-1}^1\left({1\over1-\gamma^2}|\varphi(\gamma)|^2+(1-\gamma^2)|\varphi'(\gamma)|^2\right)d\gamma.
\end{align*}
\end{lemma}
Then we give all the eigenvalues of \eqref{eigenvalue problem2 non-zero modes varepsilon=0} with corresponding eigenfunctions.
\begin{lemma}\label{sol to eigenvalue problem non-zero modes varepsilon=0 original} Fix $k\geq1.$ Then
all the eigenvalues  of the eigenvalue problem \eqref{eigenvalue problem2 non-zero modes varepsilon=0} are $\lambda_n={n(n+1)\over2}$, $n\geq k$. For $n\geq k$, the eigenspace associated to $\lambda_n={n(n+1)\over2}$ is $\text{span}\{L_{n,k}(\gamma)\}$.
Consequently, all the eigenvalues  of the eigenvalue problem \eqref{modek} are $\lambda_n={n(n+1)\over2}$, $n\geq k$. For $n\geq k$, the eigenspace associated to $\lambda_n={n(n+1)\over2}$ is $\text{span}\{L_{n,k}(\tanh(y))\}$.
\end{lemma}
\begin{proof}
It is well-known in \cite{Courant-Hilbert53} that for $n\geq k$ and $\lambda_n={n(n+1)\over2}$, the associated Legendre polynomials of $k$-th order  \begin{equation}\nonumber
L_{n,k}(\gamma)=(1-\gamma^2)^{k\over2}{d^k\over d\gamma^k}L_n(\gamma)
\end{equation}
are   solutions of the equation in \eqref{eigenvalue problem2 non-zero modes varepsilon=0}.
Note that $k\geq1$ implies
\begin{equation*}\begin{aligned}
&\int_{-1}^1{1\over 1-\gamma^2}|L_{n,k}(\gamma)|^2d\gamma=\int_{-1}^1(1-\gamma^2)^{k-1}\left|{d^k\over d\gamma^k}L_n(\gamma)\right|^2d\gamma<\infty,\\
&\int_{-1}^1{ (1-\gamma^2)}|L'_{n,k}(\gamma)|^2d\gamma=\int_{-1}^1(1-\gamma^2)^{k-1}\left|-k\gamma{d^k\over d\gamma^k}L_n(\gamma)+(1-\gamma^2){d^{k+1}\over d\gamma^{k+1}}L_n(\gamma)\right|^2d\gamma<\infty,
\end{aligned}\end{equation*}
and thus, $L_{n,k}\in \hat Y_1$. Thus, $\lambda_n={n(n+1)\over2}$ is an eigenvalue of \eqref{eigenvalue problem2 non-zero modes varepsilon=0} with corresponding eigenfunction  $L_{n,k}(\gamma)$, where $n\geq k$. It suffices to show that $\{L_{n,k}\}_{n=k}^\infty$ is a complete and orthogonal basis of $\hat Y_1$ under the inner product of  $L^2(-1,1)$.
In fact,  $\{L_{n,k}\}_{n=k}^\infty$ is a complete and orthogonal basis of $L^2(-1,1)$  \cite{Courant-Hilbert53,Dominguez-Heuer-Sayas11}. The conclusion follows from the embedding $\hat Y_1\hookrightarrow L^2(-1,1)$ by Lemma \ref{Poincare inequalities compact embedding result new variables k mode}.
\end{proof}

In summary, under the new coordinate $(x,\gamma=\tanh(y))\in\mathbb{T}_{2\pi}\times (-1,1)$,  the associated eigenvalue problem \eqref{elip02} is transformed to
\begin{align}\label{elip02-x-gamma}
-{1\over1-\gamma^2}\pa_x^2\Psi-\pa_\gamma\left((1-\gamma^2)\pa_\gamma\Psi\right)=2\lambda(\Psi-\tilde P_0\Psi),\quad \Psi\in\tilde Y_0,
\end{align}
where $\Psi(x, \gamma) = \psi(x, y)$, $\tilde P_0$ is defined in \eqref{def-tilde-P0-Psi} and $\tilde{Y}_0$ is given in \eqref{tilde-Y0-def}.

Combining the conclusions for  the $0$-mode in Lemma \ref{sol to eigenvalue problem} and for the non-zero modes in Lemma \ref{sol to eigenvalue problem non-zero modes varepsilon=0 original}, we solve the eigenvalue problems \eqref{elip02-x-gamma} and \eqref{elip02}.
\begin{Theorem}\label{associate_ep0}
All the eigenvalues of the eigenvalue problem \eqref{elip02-x-gamma} are $ \lambda_n  = \frac{n(n+1)}{2}$, $n\geq1$. For $n\geq1$, the eigenspace associated to $\lambda_n$ is  spanned by
\begin{align*}
 L_{n}(\gamma) - L_n(0), \quad  L_{n,k}(\gamma)\cos(kx), \quad L_{n,k}(\gamma)\sin(kx), \quad  1 \leq k\leq n.
 \end{align*}
Consequently, all the eigenvalues of the associated eigenvalue problem \eqref{elip02} are $\lambda_n  = \frac{n(n+1)}{2}$, $n\geq1$. For $n\geq1$, the eigenspace associated to $\lambda_n$ is  spanned by
\begin{align}\label{sol-elip02}
 L_{n}(\tanh(y)) - L_n(0), \quad  L_{n,k}(\tanh(y))\cos(kx), \quad L_{n,k}(\tanh(y))\sin(kx), \quad  1 \leq k\leq n.
 \end{align}
\end{Theorem}
In particular, we obtain the kernel of  the operator $\tilde A_0$ and a decomposition of $\tilde X_{0}$ as follows.

\begin{Corollary}\label{kernel of  the operator tilde A0 and a decomposition of tilde X0}
$(1)$ $\ker (\tilde A_0)={\rm{span}}\left\{\tanh(y),{\cos(x)\over \cosh(y)},{\sin(x)\over \cosh(y)}\right\}$.

$(2)$ Let $\tilde X_{0+}=\tilde X_0 \ominus\ker (\tilde A_0)$. Then
\begin{align*}
\langle \tilde A_0 \psi,\psi\rangle \geq {2\over3} \| \psi\|_{\tilde X_0}^2, \quad \quad \psi\in \tilde X_{0+}.
\end{align*}
\end{Corollary}

\begin{proof}
By Theorem \ref{associate_ep0}, we infer that $\lambda_1=1$ is the principal eigenvalue of \eqref{elip02} with multiplicity $3$, and the corresponding eigenfunctions are $\tanh(y),{\cos(x)\over \cosh(y)},{\sin(x)\over \cosh(y)}$. This proves (1).

For $\psi\in\tilde X_0$ and $\phi\in\ker (\tilde A_0)$, we note that $(\psi,\phi)_{Z_0}=\iint_{\Omega}g'(\psi_0)(\psi-P_0\psi)\phi dxdy=\iint_{\Omega}g'(\psi_0)\psi\phi dxdy=\iint_{\Omega}\psi(-\Delta)\phi dxdy=(\psi,\phi)_{\tilde X_0}$, where we used $P_0\phi=0$.
Since $\lambda_2=3$ is the second eigenvalue of \eqref{elip02}, we get by the variational problem \eqref{variational problem2} that
\begin{align*}
{1\over3}\iint_{\Omega}|\nabla\psi|^2dxdy\geq\iint_{\Omega}g'(\psi_0)(\psi-P_0\psi)^2dxdy,\quad \psi\in\tilde X_{0+},
\end{align*}
and thus, by \eqref{A0-quadratic form} we have
\begin{align*}
 \langle\tilde{A}_0\psi,\psi\rangle=\iint_\Omega|\nabla\psi|^2-g'(\psi_0)(\psi - P_0\psi)^2dxdy\geq {2\over3} \| \psi\|_{\tilde X_0}^2.
\end{align*}
This proves (2).
\end{proof}
We also get  the kernel of  the operator $A_0$ defined in \eqref{A0 without projection} and a decomposition of $\tilde X_{0}$ associated to $A_0$, which plays important roles in the study on nonlinear stability.
\begin{Corollary}\label{kernel of  the operator A0 and a decomposition of tilde X0}
$(1)$ $\ker ( A_0)=\ker (\tilde A_0)={\rm{span}}\left\{\tanh(y),{\cos(x)\over \cosh(y)},{\sin(x)\over \cosh(y)}\right\}$.

$(2)$ Let $\tilde X_{0+}$ be defined as above. Then
\begin{align*}
\langle  A_0 \psi,\psi\rangle \geq C_0 \| \psi\|_{\tilde X_0}^2, \quad \quad \psi\in \tilde X_{0+}
\end{align*}
for some $C_0>0$.
\end{Corollary}
\begin{proof} (1)
Since $P_0|_{\ker ( A_0)}=0$, we have by \eqref{tilde A0-A0} that $\ker (\tilde A_0)\subset \ker ( A_0)$.
For $\psi=\widehat\psi_0+\psi_{\neq0}\in \ker ( A_0)\backslash\ker (\tilde A_0)$, we have $\psi=\widehat\psi_0$ since $\tilde A_0\psi_{\neq0}=A_0\psi_{\neq0}=0$.
Then $\langle A_0 \widehat\psi_0,\phi\rangle=2\pi\int_\mathbb{R}\left(\widehat\psi_0'\phi'-g'(\psi_0)\widehat\psi_0\phi \right)dy=0$ for $\phi\in Y_0=\{\phi|\phi\in\dot{H}^1(\mathbb{R}),\phi(0)=0\}$. Thus,
  $-\widehat\psi_0''-g'(\psi_0)\widehat\psi_0=a_0\delta(y)$ for some  $a_0\in \mathbb{R}$. Thus,  $-\widehat\psi_0''-g'(\psi_0)\widehat\psi_0=0$ for $y\neq0$. Then $\widehat\psi_0(y)=c_1\tanh(y)+c_2(y\tanh(y)-1)$ for $y\neq0$. Since $y\tanh(y)-1\notin \dot{H}^1(\mathbb{R})$, we have $\widehat\psi_0(y)=c_1\tanh(y)$. Thus, $\ker (\tilde A_0)=\ker ( A_0)$.

(2) First, we claim that $\langle A_0\phi,\phi\rangle\geq0$ for $\phi\in Y_0$.
In fact, since $(\sech^2(y))'=-2\sech^2(y)$ $\tanh(y)$, we have
\begin{align*}
\langle A_0\phi,\phi\rangle=&2\pi\int_{-\infty}^\infty\left(|\phi'(y)|^2+{(\sech^2(y))'\over\tanh(y)}\phi(y)^2\right)dy\\
=&2\pi\int_{-\infty}^\infty|\phi'(y)|^2dy+2\pi{\sech^2(y)\phi(y)^2\over\tanh(y)}\bigg|_{-\infty}^\infty\\
&-
2\pi\int_{-\infty}^\infty\left({2\phi(y)\phi'(y)\sech^2(y)\over\tanh(y)}
-{\phi(y)^2\sech^4(y)\over\tanh^2(y)}\right)dy\\
=&2\pi\int_{-\infty}^\infty\left(\phi'(y)
-{\phi(y)\sech^2(y)\over\tanh(y)}\right)^2dy\geq0,
\end{align*}
where we used $\phi(y)^2\leq\|\phi'\|_{L^2(\mathbb{R})}^2|y|$, $\phi(y)=\tanh(y)\sum_{k\geq0}P_k(\tanh(y))$, and $P_k(\tanh(y))$ is a polynomial of degree $k$ in $\tanh(y)$.

Let $\psi=\widehat\psi_0+\psi_{\neq0}\in \tilde X_0$. Then
$\langle A_0 \psi_{\neq0},\psi_{\neq0}\rangle =
\langle \tilde A_0 \psi_{\neq0},\psi_{\neq0}\rangle \geq 0$ by Theorem \ref{associate_ep0}.
Thus, $
\langle A_0\psi,\psi\rangle= \langle A_0\widehat\psi_0,\widehat\psi_0\rangle+\langle A_0\psi_{\neq0},\psi_{\neq0}\rangle\geq0$.
Since  $\tilde X_0$ is compactly embedded in $L^2_{g'(\psi_0)}(\Omega)$ by Lemma \ref{compact2P}, we have
\begin{align*}\inf_{\psi \in \tilde X_0, (\psi, \phi)_{L^2_{g'(\psi_0)}(\Omega)} = 0, \phi\in\ker(A_0)}{\iint_\Omega|\nabla\psi|^2dxdy\over\iint_\Omega g'(\psi_0)\psi^2dxdy}=\mu_0>1,
\end{align*}
which implies that
\begin{align*}
 \langle A_0\psi,\psi\rangle=\iint_\Omega|\nabla\psi|^2-g'(\psi_0)\psi^2dxdy\geq \left(1-{1\over\mu_0}\right) \| \psi\|_{\tilde X_0}^2,\quad \psi\in\tilde X_{0+},
\end{align*}
where we used  $(\psi,\phi)_{L^2_{g'(\psi_0)}(\Omega)}=\iint_{\Omega}g'(\psi_0)\psi\phi dxdy=\iint_{\Omega}\nabla\psi\cdot\nabla\phi dxdy=(\psi,\phi)_{\tilde X_0}$ for $\phi\in\ker (\tilde A_0)$.
\end{proof}

\begin{remark} If we neglect the projection term $-\lambda g'(\psi_0)P_0\psi$ in \eqref{elip02}, the equation takes the form
 \begin{align}\label{elip02-without-projection}
-\Delta \psi = \lambda g'(\psi_0)\psi.
\end{align}
By changing the variable $y$ to $\gamma=\tanh(y)$ and denoting $\psi(x,y)=\Psi(x,\gamma)$, we have
\begin{align*}
-{1\over1-\gamma^2}\pa_x^2\Psi-\pa_\gamma\left((1-\gamma^2)\pa_\gamma\Psi\right)=2\lambda\Psi.
\end{align*}
Furthermore, by changing the variable $\gamma$ to $\beta=\cos^{-1}(\gamma)$, $\beta\in(0,\pi)$, and denoting $\Psi(x,\gamma)=\hat \Psi(x,\beta)$, we have
\begin{align}\label{spherical Laplacian}
-\Delta^*\hat\Psi=-{1\over\sin^2(\beta)}\pa_x^2\hat \Psi-{1\over \sin(\beta)}\pa_\beta\left(\sin(\beta)\pa_\beta\hat\Psi\right)=2\lambda\hat \Psi,
\end{align}
where $\Delta^*$ is the spherical Laplacian. It is well-known \cite{Courant-Hilbert53} that
 if $\hat\Psi\in L^2(S^2)$, and the boundary terms $\hat\Psi(\cdot,0)$ and $\hat\Psi(\cdot,\pi)$ are regular, then all the eigenvalues  of \eqref{spherical Laplacian} are $\lambda=\frac{n(n+1)}{2}$ with $n\geq0$.
For $n\geq0$, the eigenspace associated to $\lambda_n$ is  spanned by
$$ L_{n}(\cos(\beta)) , \quad  L_{n,k}(\cos(\beta))\cos(kx), \quad L_{n,k}(\cos(\beta))\sin(kx), \quad  0 \leq k\leq n,$$
which are exactly the spherical harmonic functions of degree $n$ and order $k$. Moreover, the  spherical harmonic functions
 form a complete and orthonormal basis of $L^2(S^2)$.
 Correspondingly, we find a series of solutions to \eqref{elip02-without-projection}
\begin{align*} L_{n}(\tanh(y)) , \quad  L_{n,k}(\tanh(y))\cos(kx), \quad L_{n,k}(\tanh(y))\sin(kx), \quad  0 \leq k\leq n,
\end{align*}
with $\lambda=\lambda_n =\frac{n(n+1)}{2}$, where $n\geq0$ is an integer.
 The difference between  \eqref{elip02-without-projection}  and our case  \eqref{elip02} is that we need to deal with the projection occurring in the equation \eqref{elip02}  as well as  the function spaces.
 The change of variables $\gamma=\tanh(y)$ and $\beta=\cos^{-1}(\gamma)$
 is interesting independently.
\end{remark}

\subsection{Change of variables for  Kelvin--Stuart vortices  and reduction to the shear case}\label{Kelvin--Stuart cats eyes flows}
Unlike the hyperbolic tangent shear flow ($\ep=0$), the Kelvin--Stuart vortex $\omega_\ep$ ($0<\ep<1$) depends genuinely on both $x$ and $y$, so the problem is no longer separable. In the original variables $(x,y)$, this prevents us from decomposing the associated eigenvalue problem arising from the variational problem
 into a family of 1-dimensional eigenvalue problems, as in the reduction from \eqref{elip02} to \eqref{mode0}-\eqref{modek} for the shear case.
We overcome this difficulty by introducing a suitable change of variables, which reduces the non-shear case $0<\ep<1$ to the shear case $\ep=0$.

\subsubsection{Change of variables}\label{change of variables for cat's eyes  flows}
The main difficulty for the Kelvin--Stuart vortex $\omega_\ep$ ($0<\ep<1$) is to understand
the associated eigenvalue problem
\begin{align}\label{eigen-p-cat-eyes}
-\Delta \psi = \lambda g'(\psi_\epsilon)(I - P_\epsilon)\psi
\end{align}
in a suitable function space $\tilde X_\ep$ (see \eqref{tilde-X-e}).
Here, $g'(\psi_\epsilon)$ is defined in \eqref{def-g-psi-ep-derivative} and $P_\epsilon$ (see  \eqref{P-ep}) is a similar projection as $P_0$.
The change of variable $\gamma=\tanh(y)$ used in the shear case does not work here, since $g'(\psi_\epsilon)$ depends essentially on $x$. In the shear case ($\ep=0$), the transformation $\gamma=\tanh(y)$ is motivated by the explicit eigenpairs in \eqref{eigen value-function} for the eigenvalue problem \eqref{mode0}. For the non-shear case ($0<\ep<1$), we therefore look for explicit solutions of \eqref{eigen-p-cat-eyes}, which in turn suggest a suitable change of variables.
By taking derivative of  $-\Delta \psi_\epsilon = g(\psi_\epsilon)$, we see that    $\lambda=1$ is an eigenvalue   of $
-\Delta \psi = \lambda g'(\psi_\epsilon)\psi, \psi\in\dot{H}^1(\Omega)
$  with   eigenfunctions  $\partial_x \psi_\ep, \partial_y \psi_\ep$ and $\partial_\epsilon \psi_\ep$  for all $0<\epsilon<1$. The eigenfunctions  could be viewed  as bifurcation from  the three eigenfunctions of the eigenvalue $\lambda=1$ for the corresponding equation
$-\Delta\psi=\lambda g'(\psi_0)\psi, \psi\in\dot{H}^1(\Omega)$  (i.e. $\ep=0$) as follows:
\begin{align} \label{bifurcation1}\begin{array}{llll}
\ep=0&&0<\ep<1\\
{\sin(x)\over \cosh(y)}&\longrightarrow&\frac{\sin(x)}{\cosh(y)+\epsilon \cos(x)}=-{1\over \ep}{\partial \psi_\ep\over\partial x},\\
 \tanh(y)&\longrightarrow& \frac{\sinh(y)}{\cosh(y)+\epsilon \cos(x)}=\frac{ \partial \psi_\ep}{\partial y}, \\
 {\cos(x)\over \cosh(y)}&\longrightarrow&  \frac{\epsilon \cosh(y) + \cos(x)}{\cosh(y)+\epsilon \cos(x)}=(1 - \epsilon^2) \frac{ \partial \psi_\ep}{\partial \ep}.
 \end{array}
\end{align}
This gives a hint that $\cosh(y)$ for $\ep=0$ branches to $\cosh(y)+\epsilon \cos(x)$ for $0<\ep<1$, and $\cos(x)$ branches to $\epsilon \cosh(y) + \cos(x)$.
 Motivated by this observation, we find that $\lambda=3$ is also an eigenvalue of $
-\Delta \psi = \lambda g'(\psi_\epsilon)\psi, \psi\in\dot{H}^1(\Omega)$ for all $0<\lambda<1$, since the eigenfunctions can be obtained by the similar bifurcation:
\begin{equation} \label{bifurcation2}
\begin{array}{llll}
\ep=0&&0<\ep<1\\
3\tanh^2-1&\longrightarrow&3\left({\sqrt{1-\ep^2}\sinh(y)\over \cosh(y)+\epsilon \cos(x)}\right)^2-1=3\left(\sqrt{1 - \epsilon^2}\frac{ \partial \psi_\ep}{\partial y}\right)^2-1,\\
 {\sin(x)\sinh(y)\over \cosh^2(y)}&\longrightarrow& \frac{\sin(x)\sinh(y)}{(\cosh(y)+\epsilon \cos(x))^2}=-{1\over \ep}{\partial \psi_\ep\over\partial x}\frac{ \partial \psi_\ep}{\partial y}, \\
 {\sinh(y)\cos(x)\over \cosh^2(y)}&\longrightarrow&  \frac{\sinh(y)(\ep\cosh(y)+\cos(x))}{(\cosh(y)+\epsilon \cos(x))^2}=\frac{ \partial \psi_\ep}{\partial y}\left((1 - \epsilon^2) \frac{ \partial \psi_\ep}{\partial \ep}\right),\\
 {\sin(2x)\over\cosh^2(y)}&\longrightarrow&  \frac{\sin(x)(\ep\cosh(y)+\cos(x))}{(\cosh(y)+\epsilon \cos(x))^2}=-{1\over \ep}{\partial \psi_\ep\over\partial x}\left((1 - \epsilon^2) \frac{ \partial \psi_\ep}{\partial \ep}\right),\\
 {\cos(2x)\over\cosh^2(y)}&\longrightarrow&  \frac{(\ep\cosh(y)+\cos(x))^2-(\sqrt{1-\ep^2}\sin(x))^2}{(\cosh(y)+\epsilon \cos(x))^2}=\left((1 - \epsilon^2) \frac{ \partial \psi_\ep}{\partial \ep}\right)^2-\left(-{\sqrt{1 - \epsilon^2}\over \ep}{\partial \psi_\ep\over\partial x}\right)^2.
 \end{array}
\end{equation}
This gives a hint that $\sin(x)$ for $\ep=0$ branches to $\sqrt{1-\ep^2}\sin(x)$ for $0<\ep<1$, and $\sinh(y)$ branches to $\sqrt{1-\ep^2}\sinh(y)$.
This also motivates us to rescale   $\partial_x \psi_\ep, \partial_y \psi_\ep$ and $\partial_\epsilon \psi_\ep$ to be
\begin{align}\label{three-kers1}
\eta_\ep(x, y) & :=  \frac{-\sqrt{1 - \epsilon^2}}{\ep} \frac{ \partial \psi_\ep}{\partial x}= \frac{\sqrt{1 - \epsilon^2}\sin(x)}{\cosh(y)+\epsilon \cos(x)},\\\label{three-kers2}
\gamma_\ep(x, y)  & :=  \sqrt{1 - \epsilon^2}\frac{ \partial \psi_\ep}{\partial y}= \frac{\sqrt{1 - \epsilon^2}\sinh(y)}{\cosh(y)+\epsilon \cos(x)},\\\label{three-kers3}
\xi_\ep(x, y)  & := (1 - \epsilon^2) \frac{ \partial \psi_\ep}{\partial \ep} = \frac{\epsilon \cosh(y) + \cos(x)}{\cosh(y)+\epsilon \cos(x)},
\end{align}
since  the above eigenfunctions of $\lambda=3$ can be written as polynomials of $\eta_\ep$, $\gamma_\ep$ and $\xi_\ep$, and
\begin{align}\label{eta-gamma-xi}\eta_\ep^2 + \gamma_\ep^2 + \xi_\ep^2 = 1.\end{align}
Now, we know how to bifurcate  $\cos(x),\sin(x),\sinh(y)$ and $\cosh(y)$ from $\ep=0$ to $0<\ep<1$.
However,   $\cos(kx)$ and $\sin(kx)$  appear in the eigenfunctions in \eqref{sol-elip02} for $\ep=0$. It is difficult to study how such functions branch to the case $0<\ep<1$. Our  observation is that using the De Moivre's formulae, we  can expand $\cos(kx)$ and $\sin(kx)$ by $\sin(x)$ and $\cos(x) $ as follows:
 \begin{align}\label{De Moivres formula1}
 &\cos(kx)=\sum_{j=0}^k\begin{pmatrix}k\\j\end{pmatrix}\cos^j(x)\sin^{k-j}(x)\cos\left({(k-j)\pi\over2}\right),\\\label{De Moivres formula2}
 &\sin(kx)=\sum_{j=0}^k\begin{pmatrix}k\\j\end{pmatrix}\cos^j(x)\sin^{k-j}(x)\sin\left({(k-j)\pi\over2}\right).
 \end{align}
 In this way, the bifurcation of $\cos(kx)$ and $\sin(kx)$ reduce to that of $\cos(x)$ and $\sin(x)$. Now, every component in the eigenfunctions of \eqref{sol-elip02}
 is a combination of $\cos(x),\sin(x),\sinh(y)$ and $\cosh(y)$.
Using the above branches and after direct computations, the branches of  the eigenfunctions
are  polynomials of the three functions  $
\eta_\epsilon,
\gamma_\epsilon,$ and $
\xi_\epsilon$:
\begin{align}\label{ep larger than 0-sol1}
&L_{n}(\gamma_\epsilon)-L_{n}(0)\\\label{ep larger than 0-sol2}
 &{d^k\over d\gamma_\ep^k}L_n(\gamma_\ep)
 \sum_{j=0}^k\begin{pmatrix}k\\j\end{pmatrix}\xi_\ep^j\eta_\ep^{k-j}\cos\left({(k-j)\pi\over2}\right),\\\label{ep larger than 0-sol3}
 &{d^k\over d\gamma_\ep^k}L_n(\gamma_\ep)
 \sum_{j=0}^k\begin{pmatrix}k\\j\end{pmatrix}\xi_\ep^j\eta_\ep^{k-j}\sin\left({(k-j)\pi\over2}\right).
\end{align}
Another approach to obtain \eqref{ep larger than 0-sol2}-\eqref{ep larger than 0-sol3} is first applying the De Moivre's formulae to the eigenfunctions  $L_{n,k}(\tanh(y))\cos(kx)$  and $L_{n,k}(\tanh(y))\sin(kx)$   in \eqref{sol-elip02} for $\ep=0$ to get
\begin{align}\label{ep=0-cos}
 &L_{n,k}(\tanh(y))\cos(kx)
 ={d^k\over d\gamma_0^k}L_n(\gamma_0)
 \sum_{j=0}^k\begin{pmatrix}k\\j\end{pmatrix}\xi_0^j\eta_0^{k-j}\cos\left({(k-j)\pi\over2}\right),\\\label{ep=0-sin}
 &L_{n,k}(\tanh(y))\sin(kx)=
 {d^k\over d\gamma_0^k}L_n(\gamma_0)
 \sum_{j=0}^k\begin{pmatrix}k\\j\end{pmatrix}\xi_0^j\eta_0^{k-j}\sin\left({(k-j)\pi\over2}\right),
\end{align}
and then carrying out the branches from $\xi_0$, $\gamma_0$, $\eta_0$ to $\xi_\ep$, $\gamma_\ep$, $\eta_\ep$,
where $\gamma_0=\gamma=\tanh(y)$,  $\xi_0=\cos(x)\text{sech}(y)=\cos(x)\sqrt{1-\gamma_0^2}$, and $\eta_0=\sin(x)\text{sech}(y)=\sin(x)\sqrt{1-\gamma_0^2}$.
By induction one can prove that   the functions in \eqref{ep larger than 0-sol1}-\eqref{ep larger than 0-sol3} are exactly  eigenfunctions of $
-\Delta \psi = \lambda g'(\psi_\epsilon)\psi
$ with $\lambda=n(n+1)/2$ for all $0<\epsilon<1$.
A natural question is whether there are  other linearly independent eigenfunctions.
With this problem  and our approach for $\ep=0$ in mind, we proceed to look for  change of variables for $0<\ep<1$.
Since $\gamma_\ep$ is branched from $\tanh(y)$ and recall that the change of variable is $y\mapsto \tanh(y)$ for $\ep=0$,  it is reasonable to define a new variable $\gamma_\ep$  for $0<\ep<1$.
The discovery of the other new variable, which is denoted by $\theta_\epsilon$ and should be branched  from the original variable $x$, is more subtle. Note that the eigenfunctions  \eqref{ep larger than 0-sol2}-\eqref{ep larger than 0-sol3} for  $0<\ep<1$  have the same forms with the eigenfunctions  \eqref{ep=0-cos}-\eqref{ep=0-sin} for $\ep=0$. The left hand sides of  \eqref{ep=0-cos}-\eqref{ep=0-sin}  for $\ep=0$ inspire us that  in the new variables $(\theta_\ep,\gamma_\ep)$,   the eigenfunctions for $0<\ep<1$ might  have the same forms
  $L_{n,k}(\gamma_\epsilon)\cos(k\theta_\epsilon)$  and $L_{n,k}(\gamma_\epsilon)\sin(k\theta_\epsilon)$. Applying the De Moivre's formula to $\cos(k\theta_\epsilon)$  and $\sin(k\theta_\epsilon)$, we have
 \begin{align}\nonumber
&L_{n,k}(\gamma_\epsilon)\cos(k\theta_\epsilon)\\\label{eigen-ep-gamma-theta1}
=&
 {d^k\over d\gamma_\ep^k}L_n(\gamma_\ep) \sum_{j=0}^k\begin{pmatrix}k\\j\end{pmatrix}\left(\sqrt{1-\gamma_\ep^2}\cos(\theta_\ep)\right)^j\left(\sqrt{1-\gamma_\ep^2}\sin(\theta_\ep)\right)^{k-j}\cos\left({(k-j)\pi\over2}\right),\\\nonumber
 &L_{n,k}(\gamma_\epsilon)\sin(k\theta_\epsilon)\\\label{eigen-ep-gamma-theta2}
=&
 {d^k\over d\gamma_\ep^k}L_n(\gamma_\ep) \sum_{j=0}^k\begin{pmatrix}k\\j\end{pmatrix}\left(\sqrt{1-\gamma_\ep^2}\cos(\theta_\ep)\right)^j\left(\sqrt{1-\gamma_\ep^2}\sin(\theta_\ep)\right)^{k-j}\sin\left({(k-j)\pi\over2}\right).
  \end{align}
Comparing the factors in  \eqref{ep larger than 0-sol2}-\eqref{ep larger than 0-sol3} and \eqref{eigen-ep-gamma-theta1}-\eqref{eigen-ep-gamma-theta2}, and in view of  \eqref{eta-gamma-xi}, we can
define the other new variable as  an angle $\theta_\ep \in [0, 2\pi]$ such that
\begin{align}\label{def-eta-ep}
\eta_\ep& = \sqrt{1-\gamma_\ep^2} \sin(\theta_\ep), \\\label{def-xi-ep}
\xi_\ep& = \sqrt{1-\gamma_\ep^2} \cos(\theta_\ep),
\end{align}
where $\ep\in[0,1)$. In summary, we change the original variables $(x,y)$ to the new ones  $(\theta_\ep,\gamma_\ep)$ as follows
 \begin{align} \label{transf1}
 \theta_\ep(x,y) & = \left\{ \begin{array}{rcl} \arccos \left( \frac{\xi_\ep}{\sqrt{1-\gamma_\ep^2}} \right) & \mbox{ for } & (x,y) \in [0, \pi]\times\mathbb{R}, \\
 2\pi - \arccos \left( \frac{\xi_\ep}{\sqrt{1-\gamma_\ep^2}} \right) & \mbox{ for } & (x,y) \in (\pi, 2\pi]\times\mathbb{R}, \end{array}\right.\\
  \label{transf2}
 \gamma_\ep(x,y) & = \frac{\sqrt{1 - \epsilon^2}\sinh(y)}{\cosh(y)+\epsilon \cos(x)}\quad\text{for}\quad  (x,y) \in [0, 2\pi]\times\mathbb{R}.
 \end{align}
Here, $(\theta_\ep, \gamma_\ep) \in \tilde \Omega = \mathbb{T}_{2\pi} \times [-1, 1]$  and $\ep\in[0,1)$.
 The change of variables in $\eqref{transf1}$ and $\eqref{transf2}$ is well-defined and plays an important role in solving the associated eigenvalue problem \eqref{eigen-p-cat-eyes}. First,  $\eqref{transf1}$-$\eqref{transf2}$ reduce to  the change of variable in the shear case $\ep = 0$ as $\gamma_0 = \tanh(y) = \gamma$ and $\theta_0 = x$. Second,
for the new variables $\theta_\ep$ and $\gamma_\ep$,
  the Jacobian of this transformation is
\begin{align}\label{Jacobian of the transformation-ep}
\frac{\partial (\theta_\ep, \gamma_\ep)}{\partial (x, y)} = \frac{\partial \theta_\ep}{\partial x}\frac{\partial \gamma_\ep}{\partial y} - \frac{\partial \theta_\ep}{\partial y}\frac{\partial \gamma_\ep}{\partial x} = \frac 1 2 g'(\psi_\epsilon) > 0,
\end{align}
where $\ep\in[0,1)$. More importantly, the parameter $\ep$ is fully encoded into the new variables. This enables us to reduce the eigenvalue problem in the cat's-eye case ($0<\ep<1$) to the hyperbolic tangent shear case ($\ep=0$),  which has been studied in Subsection \ref{Solutions to the eigenvalue problem ep=0}. More precisely, the associated eigenvalue problem \eqref{eigen-p-cat-eyes} is transformed to \eqref{eigenvalue problem-ep-new}, which is the same as \eqref{elip02-x-gamma}.  In particular, the eigenfunctions \eqref{ep larger than 0-sol1}-\eqref{ep larger than 0-sol3} form a complete and orthogonal basis after taking the projection terms and specific spaces in consideration.

By direct computation, we  obtain many  properties of $\eta_\ep, \gamma_\ep, \xi_\ep$ and $\theta_\ep$. We present some of them below in Propositions \ref{prop0}, \ref{prop1} and \ref{prop3}.
\begin{proposition}\label{prop0}
$(1)$ In terms of $\eta_\ep, \gamma_\ep, \xi_\ep$ and $\ep$, the steady state $ \omega_\ep$ is represented by
\begin{align}\label{omega-xi-eta-gamma-ep}\omega_\ep = -\left(\frac{(\xi_\ep - \ep)^2}{1-\ep^2} + \eta_\ep^2\right).\end{align}
$(2)$ The partial derivatives of $\eta_\ep(x, y), \gamma_\ep(x, y), \xi_\ep(x, y)$ and $\theta_\ep(x, y)$ are represented by
\begin{align*}\frac{\partial \xi_\ep}{\partial x} =& -\frac{\eta_\ep (1-\xi_\ep \ep)}{\sqrt{1-\ep^2}}, \quad \frac{\partial \xi_\ep}{\partial y} = -\frac{\gamma_\ep (\xi_\ep - \ep)}{\sqrt{1-\ep^2}},\quad
\frac{\partial \eta_\ep}{\partial x} = \frac{\xi_\ep - \ep + \eta_\ep^2 \ep}{\sqrt{1-\ep^2}}, \quad \frac{\partial \eta_\ep}{\partial y} = \frac{-\gamma_\ep \eta_\ep}{\sqrt{1-\ep^2}},\\
\frac{\partial \gamma_\ep}{\partial x} =& \frac{\ep \gamma_\ep \eta_\ep}{\sqrt{1-\ep^2}}, \quad \frac{\partial \gamma_\ep}{\partial y} = \frac{1-\xi_\ep \ep - \gamma_\ep^2}{\sqrt{1-\ep^2}},\quad
\frac{\partial \theta_\ep}{\partial x} = \frac{\gamma_{\ep y}}{1-\gamma_\ep^2}, \quad \frac{\partial \theta_\ep}{\partial y} = - \frac{\gamma_{\ep x}}{1-\gamma_\ep^2}.\end{align*}
\end{proposition}

As a consequence, the representation of $\psi_\ep = - \frac 1 2 \ln(-\omega_\ep)$ and $ g'(\psi_\ep) = -2\omega_\ep$ in terms of $\eta_\ep, \gamma_\ep, \xi_\ep$ and $\ep$ can be directly obtained by  \eqref{omega-xi-eta-gamma-ep}.
\begin{proof}
By \eqref{three-kers3}, we have
\begin{align}\label{cosh-cos}
\frac{\cosh(y)}{\cos(x)} = \frac{1-\xi_\ep \ep}{\xi_\ep - \ep}.
\end{align}
Together with \eqref{three-kers1}-\eqref{three-kers2},
we get
\begin{align}\label{tan-tanh}
\tan(x) = \frac{\sqrt{1-\ep^2}\eta_\ep}{\xi_\ep - \ep}, \quad  \tanh(y) = \frac{\sqrt{1-\ep^2}\gamma_\ep}{1- \xi_\ep\ep}.
\end{align}
Then
$$\omega_\ep =  - \frac{(1-\ep^2) \sec^2(x)}{\left( \frac{\cosh(y)}{\cos(x)} + \ep \right)^2} = -\left(\frac{(\xi_\ep - \ep)^2}{1-\ep^2} + \eta_\ep^2\right).$$
Moreover,
\begin{align}\label{tan-theta-ep}
\tan(\theta_\ep) = \frac{\eta_\ep}{\xi_\ep}.
\end{align}
The conclusions in (2) then follow from taking partial derivatives on \eqref{cosh-cos}, \eqref{tan-tanh} and \eqref{tan-theta-ep}.
\end{proof}

\begin{proposition}\label{prop1}
With $(\theta_\ep, \gamma_\ep)$ defined in \eqref{transf1}-\eqref{transf2}, we have
\begin{itemize}
\item $(\theta_\ep)_x^2 + (\theta_\ep)_y^2 =\frac 1 2 \frac{ g'(\psi_\epsilon)}{1-\gamma_\ep^2}.$

\item $-\Delta \theta_\ep = -(\theta_\ep)_{xx} - (\theta_\ep)_{yy} = 0.$

\item $$ -\Delta \eta_\ep = g'(\psi_\epsilon) \eta_\ep, \quad
-\Delta \gamma_\ep = g'(\psi_\epsilon) \gamma_\ep, \quad
-\Delta \xi_\ep = g'(\psi_\epsilon) \xi_\ep.$$

\item \begin{align*}  \nabla \eta_\ep \cdot \nabla \gamma_\ep &= -\frac 1 2 g'(\psi_\epsilon) \eta_\ep\gamma_\ep, \quad  \nabla \eta_\ep \cdot \nabla \eta_\ep = \frac 1 2 g'(\psi_\epsilon)(1 - \eta_\ep^2),  \\
 \nabla \gamma_\ep \cdot \nabla \xi_\ep &= -\frac 1 2 g'(\psi_\epsilon) \gamma_\ep\xi_\ep, \quad \nabla \gamma_\ep \cdot \nabla \gamma_\ep = \frac 1 2 g'(\psi_\epsilon) (1 - \gamma_\ep^2 ), \\
 \nabla \xi_\ep \cdot \nabla \eta_\ep &= -\frac 1 2 g'(\psi_\epsilon) \xi_\ep\eta_\ep, \quad \nabla \xi_\ep \cdot \nabla \xi_\ep = \frac 1 2 g'(\psi_\epsilon)( 1 - \xi_\ep^2).
\end{align*}

\item
\begin{align*}
-\Delta (\eta_\ep\gamma_\ep) &=  3g'(\psi_\epsilon) \eta_\ep\gamma_\ep, \quad -\Delta (3\eta_\ep^2-1) = 3g'(\psi_\epsilon) (3\eta_\ep^2 - 1), \\
-\Delta (\gamma_\ep \xi_\ep) &= 3g'(\psi_\epsilon) \gamma_\ep \xi_\ep, \quad -\Delta (3\gamma_\ep^2-1) = 3g'(\psi_\epsilon) (3\gamma_\ep^2 - 1), \\
-\Delta (\xi_\ep \eta_\ep) & =   3g'(\psi_\epsilon) \xi_\ep\eta_\ep,  \quad -\Delta (3\xi_\ep^2-1) = 3g'(\psi_\epsilon) (3\xi_\ep^2 - 1).
\end{align*}

\end{itemize}
\end{proposition}

\begin{proposition}\label{prop3} Let $\Psi(\theta_\ep, \gamma_\ep) = \psi(x(\theta_\ep, \gamma_\ep),y(\theta_\ep, \gamma_\ep))$. Then
\begin{align}\label{laplacian}-\Delta \psi = \frac 1 2 g'(\psi_\ep) \left( -\frac{\Psi_{\theta_\ep \theta_\ep}}{1-\gamma_\ep^2} - \left( (1-\gamma_\ep^2)\Psi_{\gamma_\ep} \right)_{\gamma_\ep} \right)\end{align}
and
\begin{align}\label{gradient-psi}\| \nabla \psi\|_{L^2(\Omega)}^2=\iint_{\tilde \Omega}\left({1\over1-\gamma_\ep^2}|\Psi_{\theta_\ep}|^2+(1-\gamma_\ep^2)|\Psi_{\gamma_\ep}|^2\right)d \theta_\ep d\gamma_\ep.\end{align}
\end{proposition}
\begin{proof} First, we prove \eqref{laplacian}.
By Proposition \ref{prop1}, we have
$-\Delta \theta_\ep = 0$, $(\theta_\ep)_x(\gamma_\ep)_x + (\theta_\ep)_y(\gamma_\ep)_y=0$, $(\theta_\ep)_x^2 + (\theta_\ep)_y^2 =\frac 1 2 \frac{ g'(\psi_\epsilon)}{1-\gamma_\ep^2}$,
$-\Delta \gamma_\ep = g'(\psi_\epsilon)\gamma_\ep$, and $(\gamma_\ep)_x^2 + (\gamma_\ep)_y^2  = \frac 1 2 g'(\psi_\epsilon)(1-\gamma_\ep^2).$ Thus,
\begin{align*}
\begin{split}
- \Delta \psi & = - \psi_{xx} - \psi_{yy} \\
&= -\Psi_{\theta_\ep \theta_\ep}((\theta_\ep)_x^2 + (\theta_\ep)_y^2) + \Psi_{\theta_\ep}(-\Delta \theta_\ep) - \Psi_{\gamma_\ep \gamma_\ep}\left((\gamma_\ep)_x^2 + (\gamma_\ep)_y^2 \right) + \Psi_{\gamma_\ep}(-\Delta \gamma_\ep) \\
&= -\frac 1 2 g'(\psi_\epsilon) \frac{\Psi_{\theta_\ep \theta_\ep}}{1-\gamma_\ep^2} - \frac 1 2 g'(\psi_\epsilon)(1-\gamma_\ep^2)\Psi_{\gamma_\ep \gamma_\ep} +  g'(\psi_\epsilon)\Psi_{\gamma_\ep}\gamma_\ep\\
&= \frac 1 2 g'(\psi_\ep) \left( -\frac{\Psi_{\theta_\ep \theta_\ep}}{1-\gamma_\ep^2} - \left( (1-\gamma_\ep^2)\Psi_{\gamma_\ep} \right)_{\gamma_\ep} \right)
\end{split}
\end{align*}
and
\begin{align*}
\|\nabla \psi\|^2_{L^2(\Omega)}
& = \iint_\Omega\left( |\psi_x|^2 + |\psi_y|^2\right)dxdy \\
& = \iint_\Omega\left( |\Psi_{\theta_\ep}|^2 \left((\partial_x \theta_\ep)^2 + (\partial_y \theta_\ep)^2\right) + |\Psi_{\gamma_\ep}|^2 \left((\partial_x \gamma_\ep)^2 + (\partial_y \gamma_\ep)^2\right)\right) dxdy \\
& = \iint_\Omega \frac 1 2 g'(\psi_\ep) \left( {1\over1-\gamma_\ep^2}|\Psi_{\theta_\ep}|^2+(1-\gamma_\ep^2)|\Psi_{\gamma_\ep}|^2 \right) dx dy \\
&=\int_{-1}^1\int_{0}^{2\pi} \left({1\over1-\gamma_\ep^2}|\Psi_{\theta_\ep}|^2+(1-\gamma_\ep^2)|\Psi_{\gamma_\ep}|^2\right)d \theta_\ep d\gamma_\ep.
\end{align*}
\end{proof}
Similar to \eqref{gradient-psi}, we have
\begin{align}\label{psi-X-ep-inner product}
( \psi_1,\psi_2)_{\tilde X_\ep}=\iint_{\tilde \Omega}\left({1\over1-\gamma_\ep^2}(\Psi_{1})_{\theta_\ep}(\Psi_{2})_{\theta_\ep}+(1-\gamma_\ep^2)(\Psi_{1})_{\gamma_\ep}(\Psi_{2})_{\gamma_\ep}\right)d \theta_\ep d\gamma_\ep\end{align}
for $\Psi_i(\theta_\ep, \gamma_\ep) = \psi_i(x(\theta_\ep, \gamma_\ep),y(\theta_\ep, \gamma_\ep))$, $i=1,2$.
Then we will prove that under the new coordinate $(\theta_\ep, \gep)$, the associated eigenvalue problem \eqref{eigen-p-cat-eyes} can be reduced to the corresponding one \eqref{elip02-x-gamma} in the  case  $\ep = 0$, which is solved in Theorem \ref{associate_ep0}. To this end, we preliminarily clarify the space of stream functions, solvability of the Poisson equation and boundedness of the energy quadratic form in the next subsection.
\subsubsection{Space of stream functions, Poisson equation and energy quadratic form}
Let $0<\epsilon<1$ and $\Psi(\theta_\ep, \gamma_\ep) = \psi(x(\theta_\ep, \gamma_\ep),y(\theta_\ep, \gamma_\ep))$.
Recall that the space $\tilde X_0$ of stream functions $\psi$  for  $\ep = 0$ is $\dot{H}^1(\Omega)$ with an additional condition that $\widehat{\psi}_0(0)=0$. If we use the same space  $\tilde X_0$ for $0<\epsilon<1$, then
$n^-(A_\epsilon)\geq1$ for the elliptic operator $A_\epsilon$ without projection (see Remark \ref{X-ep-X0-A-ep-neg}), which is inapplicable in the proof of  nonlinear stability. Furthermore,
it is inappropriate to establish an isomorphism for the spaces  of stream functions   between $\epsilon=0$ and $0<\epsilon<1$, since the variable $\theta_\epsilon$ involves $x$ and $y$ in a very coupled way so  that in  the new variables, $\widehat{\psi}_0$ is no longer  the $0$-mode of $\Psi$ after writing it in the Fourier series with respect to $\theta_\epsilon$.  Instead, our choice is to  replace the condition $\widehat{\psi}_0(0)=0$ by $\widehat{\Psi}_0(0)=0$ in the definition of the space  of stream functions, where $\widehat{\Psi}_0(0)={1\over 2\pi}\int_{0}^{2\pi}\Psi(\theta_\epsilon,0)d\theta_\epsilon$. In this way, we can ensure not only that $\dim \ker (A_\epsilon)=3$ and $n^-(A_\epsilon)=0$ (see Corollary \ref{kernel of  the operator A-ep and a decomposition of tilde Xep}), but also that the spaces of stream functions for $\epsilon=0$ and $0<\epsilon<1$ are isomorphic.  Noting that $y=0$ if and only if $\gamma_\epsilon=0$, by Proposition \ref{prop0} (2)  we have
\begin{align}\nonumber
\widehat{\Psi}_0(0)=&{1\over2\pi}\int_{0}^{2\pi}\Psi(\theta_\epsilon,0)d\theta_\epsilon={1\over2\pi}\int_{0}^{2\pi}\psi(x(\theta_\epsilon,0),0){\partial{\theta_\epsilon}\over \partial x}|_{y=0}dx\\\nonumber
=&{1\over2\pi}\int_{0}^{2\pi}\psi(x,0){\gamma_{\epsilon y}}|_{y=0}dx
={1\over2\pi\sqrt{1-\ep^2}}\int_{0}^{2\pi}\psi(x,0)(1-\xi_\ep \ep)|_{y=0}dx\\\label{widehat-Psi-to-psi-1+ep-cos}
=&{\sqrt{1-\ep^2}\over2\pi}\int_{0}^{2\pi}\psi(x,0){1\over1+\epsilon\cos(x)} dx.
\end{align}
Thus,  we define the space  of stream functions specifically in  the original variables as follows
\begin{align}\label{tilde-X-e}\tilde{X}_\ep = \left\{ \psi\bigg| \iint_\Omega |\nabla \psi|^2 dxdy <  \infty \text{ and } \int_{0}^{2\pi}\psi(x,0){1\over1+\epsilon\cos(x)} dx = 0  \right\}.\end{align}
In  the new variables, by \eqref{gradient-psi}-\eqref{widehat-Psi-to-psi-1+ep-cos} $\tilde{X}_\ep$ is equivalent  to the following space
\begin{align*}
\tilde{Y}_\ep = \left\{ \Psi \bigg| \iint_{\tilde \Omega}\left({1\over1-\gamma_\ep^2}|\Psi_{\theta_\ep}|^2+(1-\gamma_\ep^2)|\Psi_{\gamma_\ep}|^2\right)d \theta_\ep d\gamma_\ep< \infty \text{ and } \widehat{\Psi}_0(0)=0 \right\},
\end{align*}
where $\tilde \Omega = \mathbb{T}_{2\pi} \times [-1, 1]$.
Noting that $\tilde{Y}_\ep$ is the same space as $\tilde Y_0$ as defined in \eqref{tilde-Y0-def}, we thus get the following result.

\begin{lemma}\label{hilbert-ep}
Let $0<\epsilon<1$. Then

$(1)$ the function space $\tilde{Y}_\ep$ equipped with the inner product
 $$(\Psi_1, \Psi_2) = \iint_{\tilde{\Omega}}  \left({1\over1-\gamma_\ep^2}(\Psi_1)_{\theta_\ep}(\Psi_2)_{\theta_\ep} +(1-\gamma_\ep^2)(\Psi_1)_{\gamma_\ep}(\Psi_2)_{\gamma_\ep}\right)d \theta_\ep d\gamma_\ep, \quad \forall\; \Psi_1, \Psi_2 \in \tilde{Y}_\ep$$
 is  a Hilbert space;

$(2)$
the function space $\tilde{X}_\ep$ equipped with the inner product
$$(\psi_1, \psi_2) = \iint_{\Omega} \nabla \psi_1 \cdot \nabla \psi_2 dxdy, \quad \forall\; \psi_1, \psi_2 \in \tilde{X}_\ep$$
 is a Hilbert space. Moreover,
\begin{align}\label{Psi-psi-norm-eq} \| \psi\|_{\tilde X_\epsilon}^2= \| \nabla \psi\|_{L^2(\Omega)}^2=\iint_{\tilde \Omega}\left({1\over1-\gamma_\ep^2}|\Psi_{\theta_\ep}|^2+(1-\gamma_\ep^2)|\Psi_{\gamma_\ep}|^2\right)d \theta_\ep d\gamma_\ep=\|\Psi\|_{\tilde Y_\epsilon}^2
\end{align}
for   $\psi \in \tilde{X}_\ep$ and $\Psi \in \tilde{Y}_\ep$ such that $\psi(x,y) = \Psi(\theta_\ep, \gamma_\ep)$.
\end{lemma}
\begin{proof}
(1) follows from Lemma \ref{Hilbert-new variables-0}, and (2) is obtained by \eqref{gradient-psi}-\eqref{widehat-Psi-to-psi-1+ep-cos} and (1).
\end{proof}

Then we give the Poincar\'e inequality I for $0<\epsilon<1$.

\begin{lemma}[Poincar\'e inequality I-$\ep$]\label{poincare1ep}
$(1)$ For any $\Psi \in \tilde{Y_\ep}$, we have
\begin{align*}
 \|\Psi\|_{L^2(\tilde \Omega)}^2  \leq C \iint_{\tilde \Omega}\left({1\over1-\gamma_\ep^2}|\Psi_{\theta_\ep}|^2+(1-\gamma_\ep^2)|\Psi_{\gamma_\ep}|^2\right)d \theta_\ep d\gamma_\ep.
\end{align*}

$(2)$
For any $\psi \in \tilde{X_\ep}$, we have
\begin{align}\label{Poincare inequality I-ep22}
\iint_\Omega g'(\psi_\epsilon)|\psi|^2 dxdy  \leq C \|\nabla \psi\|_{L^2(\Omega)}^2.
\end{align}
\end{lemma}
\begin{proof}
(1) is the same as Lemma \ref{Poincare ineqalities-new-variable0} (1). To prove (2), let $  \Psi(\theta_\ep, \gamma_\ep)=\psi(x,y)$ for $\psi\in\tilde{X_\ep}$. By \eqref{Jacobian of the transformation-ep} we have
\begin{align}\label{Psi-psi-L2-ep}
2\iint_{\tilde \Omega}|\Psi|^2d \theta_\ep d\gamma_\ep= \iint_\Omega g'(\psi_\epsilon)|\psi|^2 dxdy.
\end{align}
 By \eqref{gradient-psi} and \eqref{Psi-psi-L2-ep}, we know that  (2) is a restatement of  (1) in the original variables $(x,y)$.
\end{proof}

For $0<\epsilon<1$, we define the projection
\begin{align}\label{P-ep}P_\ep \psi := \frac{\iint_\Omega g'(\psi_\ep)\psi dxdy}{\iint_\Omega g'(\psi_\ep) dxdy}={\iint_\Omega g'(\psi_\ep)\psi dxdy\over8\pi},\quad\psi \in \tilde{X}_\ep,\end{align}
 and
\begin{align}\label{P-ep-new-variables}
\tilde P_\ep\Psi :=\frac{\iint_{\tilde \Omega}  \Psi d \theta_\ep d \gamma_\ep}{\iint_{\tilde \Omega}   d \theta_\ep d \gamma_\ep}= \frac{\iint_{\tilde \Omega}  \Psi d \theta_\ep d \gamma_\ep}{4\pi},\quad \Psi \in \tilde{Y}_\ep.\end{align}

\begin{Corollary}\label{P-welldefined}
The projections $P_\ep $
 and
$\tilde P_\ep$ are well-defined. Moreover,
$P_\ep \psi = \tilde P_\ep\Psi$
for  $\psi \in \tilde{X}_\ep$ and $\Psi \in \tilde{Y}_\ep$ such that $\psi(x,y) = \Psi(\theta_\ep, \gamma_\ep)$.
\end{Corollary}
\begin{proof}
The projection $\tilde P_\ep$ is the same as  $\tilde P_0$ in \eqref{def-tilde-P0-Psi}. Let $\psi \in \tilde{X}_\ep$ and $\Psi \in \tilde{Y}_\ep$ such that $\psi(x,y) = \Psi(\theta_\ep, \gamma_\ep)$. Then $\tilde P_\ep$ is well-defined and $|\tilde P_\epsilon\Psi|\leq C \|\Psi\|_{\tilde Y_\epsilon}$ by Lemma
\ref{Poincare ineqalities-new-variable0} (2).
By \eqref{Jacobian of the transformation-ep},  $P_\ep \psi = \tilde P_\ep\Psi$ follows directly from  the definitions of $P_\ep$ and $\tilde P_\ep$. Then we have by \eqref{Psi-psi-norm-eq} that
\begin{align}\label{projection-controlled by-X-ep}
|P_\epsilon\psi|=|\tilde P_\epsilon\Psi|\leq C \|\Psi\|_{\tilde Y_\epsilon}=C \| \psi\|_{\tilde X_\epsilon}.
\end{align}
\end{proof}

Next,  we give the Poincar\'e inequality II for $0<\epsilon<1$.
\begin{lemma}[Poincar\'e inequality II-$\ep$]\label{poincare2ep}
$(1)$ For any $\Psi \in \tilde{Y_\ep}$,
we have
\begin{align*}
\iint_{\tilde \Omega} (\Psi -\tilde P_\epsilon\Psi)^2 d \theta_\ep d\gamma_\ep  \leq C \iint_{\tilde \Omega}\left({1\over1-\gamma_\ep^2}|\Psi_{\theta_\ep}|^2+(1-\gamma_\ep^2)|\Psi_{\gamma_\ep}|^2\right)d \theta_\ep d\gamma_\ep.
\end{align*}

$(2)$ For any $\psi \in \tilde{X}_\ep$,
we have
\begin{align}\label{Poincare inequality II-ep22}
\iint_\Omega g'(\psi_\epsilon)(\psi - P_\epsilon\psi)^2 dxdy  \leq C \|\nabla \psi\|_{L^2(\Omega)}^2.
\end{align}

\end{lemma}
\begin{proof}
(1)  follows from Lemma \ref{Poincare ineqalities-new-variable0} (3). By \eqref{Jacobian of the transformation-ep}, \eqref{Psi-psi-norm-eq} and Corollary \ref{P-welldefined},  we infer that   (2) is a restatement of  (1) in the original variables $(x,y)$.
\end{proof}
By Lemma \ref{hilbert-ep} (2) and  the Poincar\'e inequality I-$\ep$ \eqref{Poincare inequality I-ep22}, one can prove the existence and uniqueness of solutions in $\tilde X_\ep$ to the Poisson equation $-\Delta \psi = \omega\in X_\ep$ in the weak sense. The proof is similar to
Lemma  \ref{1-1correspond}, and we omit it.

\begin{lemma}\label{1-1correspond-ep}
For any $\omega \in X_\ep$, the Poisson equation
\begin{align*}
-\Delta \psi = \omega
\end{align*}
has a unique weak solution in $\tilde{X}_\ep$.
\end{lemma}
Recall that  $L_\ep$ and $X_\epsilon$ are defined in \eqref{J-ep-J-ep-def}-\eqref{spaceXep}, and the corresponding quadratic form for $L_\ep$ is
\begin{align*}
 \langle L_\epsilon\omega,\omega\rangle=\iint_\Omega\left(\frac {|\omega|^2} {g'(\psi_\epsilon)} - (-\Delta)^{-1}\omega\omega \right) dxdy,\quad\omega\in X_\epsilon.
\end{align*}
In view of Lemmas \ref{poincare1ep} (2) and \ref{1-1correspond-ep}, one can prove that $\langle L_\epsilon\cdot,\cdot\rangle$ is bounded on $ X_\epsilon$ by a similar way as Lemma \ref{Lbounded}.
\begin{lemma}\label{Lbounded-ep}
For any $\omega_1,\omega_1 \in X_\ep$, we have
$\langle L_\ep \omega_1, \omega_2 \rangle=\langle  \omega_1, L_\ep\omega_2 \rangle < C\|\omega_1\|_{X_\epsilon}\|\omega_2\|_{X_\epsilon}$.
\end{lemma}
\subsubsection{Reduction of the eigenvalue problems from Kelvin--Stuart vortex to hyperbolic tangent shear flow}

Define two elliptic operators
 \begin{align}\label{tilde-A-ep-A-ep}
 \tilde{A}_\ep=-\Delta-g'(\psi_\ep)(I - P_\ep): \tilde{X}_\ep \rightarrow \tilde{X}_\ep^*\quad\text{and}\quad
 A_\ep =-\Delta -g'(\psi_\ep):\tilde{X}_\ep \rightarrow \tilde{X}_\ep^*.
\end{align}
Then the corresponding quadratic forms
\begin{align*}
 \langle\tilde{A}_\ep\psi,\psi\rangle=&\iint_\Omega\left(|\nabla\psi|^2-g'(\psi_\ep)(\psi - P_\ep\psi)^2\right)dxdy
 \end{align*}
 and
 \begin{align*}
 \langle A_\ep \psi,\psi\rangle=&\iint_{\Omega}\left(|\nabla \psi|^2-g'(\psi_\ep)|\psi|^2\right)dxdy
\end{align*}
are bounded and symmetric on $\tilde{X}_\ep$ by the Poincar\'e inequalities I-$\ep$ \eqref{Poincare inequality I-ep22}, II-$\ep$ \eqref{Poincare inequality II-ep22}.
Then similar to \eqref{tilde A0-A0}, we have
\begin{align*}
\langle\tilde  A_\ep \psi,\psi\rangle=\langle A_\ep \psi,\psi\rangle+8\pi(P_\ep \psi)^2,\quad\psi\in \tilde X_\ep.
\end{align*}
Thus,
\begin{equation*}
n^{\leq0}(\tilde A_\ep)\leq n^{\leq0}(A_\ep),\quad n^{-}(\tilde A_\ep)\leq n^{-}(A_\ep).
\end{equation*}
By means of Lemmas \ref{poincare2ep} (2) and \ref{1-1correspond-ep}, we have the following result by a similar argument to Lemma \ref{equal-indices0}.

\begin{lemma}\label{equal-indices} Let $0<\epsilon<1$. Then
$$\dim \ker (\tilde{A}_\ep) = \dim \ker (L_\ep), \quad n^-(\tilde{A}_\ep) = n^-(L_\ep).$$
\end{lemma}

To compute $n^-(\tilde{A}_\ep)$, we also need the compact embedding results.

\begin{lemma}\label{compact2P-new-variable-ep} Let $0<\ep<1$.
$(1)$ $\tilde Y_\ep$ is compactly embedded in $L^2(\tilde \Omega)$ and
\begin{equation*}
 \tilde Z_{\ep}:=\left\{\Psi\bigg|\iint_{\tilde \Omega}|\Psi-\tilde P_\ep\Psi|^2d\theta_\ep d\gamma_\ep<\infty\right\},
  \end{equation*}
respectively.

$(2)$  $\tilde X_\ep$ is compactly embedded in $L_{g'(\psi_\ep)}^2(\Omega)$ and
 \begin{equation*}
 Z_{\ep}:=\left\{\psi\bigg|\iint_{\Omega}g'(\psi_\ep)|\psi-P_\ep\psi|^2dxdy<\infty\right\},
  \end{equation*}
respectively.
\end{lemma}
\begin{proof}
(1) is equivalent to Lemma \ref{compact2P-new-variable0}. (2) is a consequence of (1), \eqref{Psi-psi-norm-eq} and Corollary \ref{P-welldefined}.
\end{proof}

By the compact embedding  $\tilde X_\ep\hookrightarrow Z_{\ep}$,   we can
inductively define $\lambda_{n}(\ep)$ as follows
\begin{align}\label{variational problem2-ep}
\lambda_n(\ep)=& \inf_{\psi \in \tilde X_\ep, (\psi, \psi_{i})_{Z_\ep} = 0, i = 1, 2, \cdots, n-1}{\iint_\Omega|\nabla\psi|^2dxdy\over\iint_\Omega g'(\psi_\ep)(\psi - P_\ep\psi)^2dxdy},\quad n\geq1,
\end{align}
where the  infimum for $\lambda_i(\ep)$ is attained at $\psi_{i}\in \tilde X_\ep$ and $\iint_\Omega g'(\psi_\ep)(\psi_{i} - {P_\ep}\psi_{i})^2 dxdy = 1$, $1\leq i \leq n-1$.
By computing the first variation of the functional $G_\ep(\psi)={\iint_\Omega|\nabla\psi|^2dxdy\over\iint_\Omega g'(\psi_\ep)(\psi - P_\ep\psi)^2dxdy}$ at $\psi_{_n}$,  we have
\begin{align*}
&\frac{d}{d \tau} G_\ep(\psi_{n} + \tau \psi)|_{\tau = 0} = 2\iint_{\Omega} \left(-\Delta\psi_n - \lambda_n(\epsilon)g'(\psi_\ep)(\psi_n - P_\ep\psi_n)\right)\psi dxdy\\
=&2\iint_{\tilde\Omega}\left(-{1\over1-\gamma_\ep^2}\pa_{\theta_\ep}^2\Psi_n-\pa_{\gamma_\ep}\left((1-\gamma_\ep^2)\pa_{\gamma_\ep}\Psi_n\right)
-2\lambda_n(\epsilon)(\Psi_n-\tilde P_\ep\Psi_n)\right)\Psi d\theta_\ep d\gamma_\ep
\end{align*}
for $\psi\in \tilde X_\ep$ and $\Psi\in \tilde Y_\ep$ with $\psi(x,y)=\Psi(\theta_\ep,\gamma_\ep)$, where $\Psi_n(\theta_\ep,\gamma_\ep)=\psi_n(x,y)$.
Since   $\widehat{\Psi}_0(0)=0$ for $\Psi\in \tilde Y_\ep$, we
derive the Euler-Lagrangian equation in the new variables
\begin{align}\label{elip-ep}
-{1\over1-\gamma_\ep^2}\pa_{\theta_\ep}^2\Psi-\pa_{\gamma_\ep}\left((1-\gamma_\ep^2)\pa_{\gamma_\ep}\Psi\right)
=2\lambda(\Psi-\tilde P_\ep\Psi)+a\delta(\gamma_\ep), \quad \Psi \in \tilde{Y}_\ep,
\end{align}
where $a\in\mathbb{R}$ is to be determined. By the definition of  $\tilde P_\ep$ in \eqref{P-ep-new-variables}, integrating \eqref{elip-ep} on $\tilde \Omega$, we have
$$
2\pi a=\iint_{\tilde \Omega}\left(-{1\over1-\gamma_\ep^2}\pa_{\theta_\ep}^2\Psi-\pa_{\gamma_\ep}\left((1-\gamma_\ep^2)\pa_{\gamma_\ep}\Psi\right)
-2\lambda(\Psi-\tilde P_\ep\Psi)\right) d\theta_\ep d\gamma_\ep=0\;\;\Longrightarrow \;\;a=0,
$$
and thus,  we get the eigenvalue problem
\begin{align}\label{eigenvalue problem-ep-new}
-{1\over1-\gamma_\ep^2}\pa_{\theta_\ep}^2\Psi-\pa_{\gamma_\ep}\left((1-\gamma_\ep^2)\pa_{\gamma_\ep}\Psi\right)
=2\lambda(\Psi-\tilde P_\ep\Psi), \quad \Psi \in \tilde{Y}_\ep,
\end{align}
which, in the original variables, is exactly
\begin{align}\label{eigenvalue problem-ep-original}
-\Delta \psi = \lambda g'(\psi_\ep)(\psi -  P_\ep\psi), \quad \psi \in \tilde{X}_\ep.
\end{align}
Noting that  the eigenvalue problem \eqref{eigenvalue problem-ep-new} is the same  as
\eqref{elip02-x-gamma}, we have the following conclusions by Theorem \ref{associate_ep0}.

\begin{Theorem}\label{associate_ep-new-variable-original-variable}
All the eigenvalues of the eigenvalue problem \eqref{eigenvalue problem-ep-new} are $\lambda_n  = \frac{n(n+1)}{2}, n\geq1$. For $n\geq1$, the eigenspace associated to $\lambda_n$ is  spanned by
\begin{align*}
 L_{n}(\gamma_\ep) - L_n(0), \quad  L_{n,k}(\gamma_\ep)\cos(k\theta_\ep), \quad L_{n,k}(\gamma_\ep)\sin(k\theta_\ep), \quad  1 \leq k\leq n.
 \end{align*}
Consequently, all the eigenvalues of the associated eigenvalue problem \eqref{eigenvalue problem-ep-original} are $\lambda_n  = \frac{n(n+1)}{2}, n\geq1$. For $n\geq1$, the eigenspace associated to $\lambda_n$ is  spanned by
\begin{align*}
& L_{n}(\gamma_\ep(x,y)) - L_n(0), \quad  L_{n,k}(\gamma_\ep(x,y))\cos(k\theta_\ep(x,y)), \\
 & L_{n,k}(\gamma_\ep(x,y))\sin(k\theta_\ep(x,y)), \quad  1 \leq k\leq n,
 \end{align*}
where $\gamma_\ep(x,y)$ and $\theta_\ep(x,y)$ are defined in \eqref{transf1}-\eqref{transf2},  $L_{n,k}(\gamma_\ep)=(1-\gamma_\ep^2)^{k\over2}{d^k\over d\gamma_\ep^k}L_n(\gamma_\ep)$, and $L_n$ is the Legendre polynomial of degree $n$.
\end{Theorem}
Then we get the kernel of  the operators $\tilde A_\ep$ and $A_\ep$, as well as  decompositions of $\tilde X_{\ep}$ associated to the two operators.

\begin{Corollary}\label{kernel of  the operator tilde A-ep and a decomposition of tilde Xep}
$(1)$ $\ker (\tilde A_\ep)={\rm{span}}\left\{\eta_\ep(x,y), \gamma_\ep(x,y), \xi_\ep(x,y)\right\}$.

$(2)$ Let $\tilde X_{\ep+}=\tilde X_\ep \ominus\ker (\tilde A_\ep)$. Then
\begin{align*}
\langle \tilde A_\ep \psi,\psi\rangle \geq {2\over3} \| \psi\|_{\tilde X_\ep}^2, \quad \quad \psi\in \tilde X_{\ep+}.
\end{align*}
\end{Corollary}

\begin{proof}
By means of  Theorem \ref{associate_ep-new-variable-original-variable} and \eqref{variational problem2-ep}, the proof is similar to Corollary \ref{kernel of  the operator tilde A0 and a decomposition of tilde X0}. Here, we used $\tilde P_\ep\eta_\ep={1\over 4\pi}\iint_{\tilde \Omega}\sqrt{1-\gamma_\ep^2}\sin(\theta_\ep)d\theta_\ep d\gamma_\ep=0,$ $\tilde P_\ep\gamma_\ep={1\over 4\pi}\iint_{\tilde \Omega}\gamma_\ep d\theta_\ep d\gamma_\ep=0$, and $\tilde P_\ep\xi_\ep={1\over 4\pi}\iint_{\tilde \Omega}\sqrt{1-\gamma_\ep^2}\cos(\theta_\ep)d\theta_\ep d\gamma_\ep=0$ by \eqref{P-ep-new-variables}.
\end{proof}
The decomposition of $\tilde X_{\ep}$ associated to $A_\ep$ will be used  in the study on nonlinear stability.
\begin{Corollary}\label{kernel of  the operator A-ep and a decomposition of tilde Xep}
$(1)$ $\ker ( A_\ep)=\ker (\tilde A_\ep)={\rm{span}}\left\{\eta_\ep(x,y), \gamma_\ep(x,y), \xi_\ep(x,y)\right\}$.

$(2)$ Let $\tilde X_{\ep+}$ be defined as above. Then
\begin{align*}
\langle  A_\ep \psi,\psi\rangle \geq C_0 \| \psi\|_{\tilde X_\ep}^2, \quad \quad \psi\in \tilde X_{\ep+}
\end{align*}
for some $C_0>0$.
\end{Corollary}
\begin{proof}
Define the quadratic form
 \begin{align*}
 \langle\mathscr{A}_\ep\Psi,\Psi\rangle=\iint_{\tilde\Omega}\left({|\partial_{\theta_\ep}\Psi|^2\over 1-\gamma_\ep^2}+(1-\gamma_\ep^2)|\partial_{\gamma_\ep}\Psi|^2-2|\Psi|^2\right)d\theta_\ep d\gamma_\ep,\quad \Psi\in\tilde{Y}_\ep,
 \end{align*}
 where $\ep\in[0,1)$. Note that $\langle\mathscr{A}_\ep\Psi,\Psi\rangle=\langle{A}_\ep\psi,\psi\rangle$ for   $\psi \in \tilde{X}_\ep$ and $\Psi \in \tilde{Y}_\ep$ such that $\psi(x,y) = \Psi(\theta_\ep, \gamma_\ep)$, where $\ep\in[0,1)$.
 By Corollary \ref{kernel of  the operator A0 and a decomposition of tilde X0}, $\ker(\mathscr{A}_0)=\text{span}\{\gamma_0,\sqrt{1-\gamma_0^2}\cos(x),$ $\sqrt{1-\gamma_0^2}\sin(x)\}$, and $\langle\mathscr{A}_0\Psi,\Psi\rangle\geq C_0\|\Psi\|_{\tilde Y_0}$ for $\Psi\in \tilde Y_{0+}$, where $\tilde Y_{0+}=\tilde Y_0 \ominus\ker (\mathscr{A}_0)$. Thus, we  have $\ker(\mathscr{A}_\ep)=\text{span}\{\gamma_\ep,\sqrt{1-\gamma_\ep^2}\cos(\theta_\ep),$ $\sqrt{1-\gamma_\ep^2}\sin(\theta_\ep)\}$, and $\langle\mathscr{A}_\ep\Psi,\Psi\rangle\geq C_0\|\Psi\|_{\tilde Y_\ep}$ for $\Psi\in \tilde Y_{\ep+}$, where $\tilde Y_{\ep+}=\tilde Y_{\ep} \ominus\ker (\mathscr{A}_{\ep})$ and $\ep\in(0,1)$. This proves (1)-(2).
\end{proof}

\begin{remark}\label{X-ep-X0-A-ep-neg}
In the definition of $\tilde X_\epsilon$, if we replace the condition $\widehat\Psi_0(0)=0$ by $\widehat \psi_0(0)=0$  as in $\tilde X_0$ for $\ep\in(0,1)$, then $n^-(A_\ep)\geq1$. In fact,
$\partial_\ep\psi_\ep\not\in \tilde X_\ep$ since
\begin{align*}
(\widehat{\partial_\ep\psi_\ep})_{\,0}(0)={1\over 2\pi}\int_0^{2\pi}\partial_\ep\psi_\ep(x,0)dx={1\over 2\pi}\int_0^{2\pi}\left({\epsilon\over 1-\epsilon^2}+{\cos(x)\over 1+\ep\cos(x)}\right)dx={1\over \ep-\ep^3}\neq0
\end{align*}
for $\ep\in(0,1)$. This implies that $\partial_\ep\psi_\ep-c_\ep\in\tilde X_\ep$ for $c_\ep={1\over \ep-\ep^3}$. Then
\begin{align*}
&\langle A_\ep(\partial_\ep\psi_\ep-c_\ep),\partial_\ep\psi_\ep-c_\ep\rangle=\langle(-\Delta -g'(\psi_\ep))(\partial_\ep\psi_\ep-c_\ep),\partial_\ep\psi_\ep-c_\ep\rangle\\
=&\langle g'(\psi_\ep)c_\ep,\partial_\ep\psi_\ep-c_\ep\rangle=-c_\ep^2\iint_{\Omega}g'(\psi_\ep)dxdy<0,
\end{align*}
where we used $-\Delta\partial_\ep\psi_\ep=g'(\psi_\ep)\partial_\ep\psi_\ep$ and $\iint_{\Omega} g(\psi_\ep)dxdy=8\pi\Longrightarrow\iint_{\Omega} g'(\psi_\ep)\partial_\ep \psi_\ep dxdy=0$. Thus, $n^-(A_\ep)\geq1$.
\end{remark}
\subsection{The proof of linear stability of Kelvin--Stuart vortices}
Based on our solutions to  the eigenvalue problems \eqref{elip02} and \eqref{eigenvalue problem-ep-original}, we  prove linear stability of the hyperbolic tangent shear flow and the Kelvin--Stuart vortices for  co-periodic perturbations. The approach is to apply the following index formula for general linear Hamiltonian PDEs developed in \cite{lin2022instability}.

\begin{lemma}
\label{theorem-index}
Consider a linear Hamiltonian system
$$\partial_t \omega = JL\omega, \quad \omega \in X,$$
where $X$ is a real Hilbert space.
Assume that

{\rm \textbf{(H1)}} $J:X^{\ast} \supset D(J)  \rightarrow X$ is anti-self-dual.

{\rm \textbf{(H2)}} $L:X\rightarrow X^{\ast}$ is bounded and self-dual. Moreover, there exists a decomposition of $X$ into the direct sum of three closed subspaces
$$X=X_{-}\oplus\ker L\oplus X_{+}, \quad n^-(L)= \dim X_- < \infty$$ satisfying

$\quad$ {\rm \textbf{(H2.a)}}
 $\left\langle L\omega ,\omega \right\rangle <0$ for all $\omega \in X_- \backslash \{0\}$;

$\quad$ {\rm \textbf{(H2.b)}} there exists $\delta >0$ such that
\[
\left\langle L\omega ,\omega \right\rangle \geq\delta \left\Vert \omega \right\Vert _{X}
^{2},\quad\forall \; \omega \in X_{+}.
\]

{\rm \textbf{(H3)}} $\dim\ker
L<\infty$.\\
Then
\begin{align}\label{index-formula}
k_r + 2k_c+2k_i^{\leq0}+k_0^{\leq0} = n^-(L),
\end{align}
where $k_r$ is the sum of algebraic multiplicities of positive eigenvalues of $JL$, $k_c$ is the sum of algebraic multiplicities of eigenvalues of $JL$ in the first quadrant, $k_i^{\leq 0}$ is the total number of non-positive dimensions of $\langle L\cdot, \cdot \rangle$ restricted to the generalized eigenspaces of pure imaginary eigenvalues of $JL$ with positive imaginary parts, and $k_0^{\leq 0}$ is the number of non-positive directions of $\langle L\cdot, \cdot \rangle$ restricted to the generalized kernel of $JL$ modulo $\ker L$.
\end{lemma}

Now  we are in a position to prove  Theorem \ref{main result1-co-periodic perturbations}.
 \begin{proof}[Proof of Theorem \ref{main result1-co-periodic perturbations}]
We check \textbf{(H1-3)} in Lemma \ref{theorem-index} and then  apply the index formula \eqref{index-formula-stuart}  to prove  spectral stability of $\omega_\ep$, $0\leq \ep<1$. Recall that $J_\epsilon$, $L_\epsilon$ and $X_\ep$ are defined in \eqref{J-ep-J-ep-def}-\eqref{spaceXep}.
First, we define the space $\hat L^2(\Omega)=\{\omega\in L^2(\Omega)|\iint_\Omega \sqrt{g'(\psi_\ep)}\omega dxdy=0\}$ and the isometry
$$S: L^2(\Omega) \rightarrow X_\ep, \quad S\omega = \sqrt{g'(\psi_\ep)}\omega.$$
 Since $g'(\psi_\ep)\cdot$ and $\vec{u}_\ep \cdot \nabla$ are commutative, and $\nabla \cdot \vec{u}_\ep = 0$,
\begin{align}\label{tilde-J-ep}\tilde{J_\ep} := S^{-1} J_\ep (S')^{-1} = -\vec{u}_\ep \cdot \nabla:(\hat L^2(\Omega))^*\supset D(\tilde{J_\ep}) \rightarrow \hat L^2(\Omega)\end{align}
is anti-self-dual, where
$$D(\tilde{J_\ep}) =  \left\{ \omega \in  (\hat L^2(\Omega))^* | (\vec{u}_\ep \cdot \nabla) \omega \in \hat L^2(\Omega) \text{ in the distribution sense} \right\}.$$
Then $J_\ep' = -J_\ep$, and thus, \textbf{(H1)} is satisfied.
By Lemmas  \ref{Lbounded} and \ref{Lbounded-ep},
 the  operator $L_\ep:X_\ep\to X_\ep^*$ is self-dual and bounded for $0\leq \ep<1$.

It follows from Corollaries  \ref{kernel of  the operator tilde A0 and a decomposition of tilde X0} and  \ref{kernel of  the operator tilde A-ep and a decomposition of tilde Xep} that
$$n^-(\tilde{A}_\ep) = 0, \quad \dim \ker (\tilde{A}_\ep) = 3 \quad\text{ for all } \epsilon \in [0,1),$$
and $\tilde X_\ep$ can be decomposed as  $\tilde X_\ep=\ker (\tilde A_\ep) \oplus\tilde X_{\ep+}$ such that
\begin{align}\label{tilde-A-ep-psi-psi-Xep+}
\langle \tilde A_\ep \psi,\psi\rangle \geq {2\over3} \| \psi\|_{\tilde X_\ep}^2, \quad \quad \psi\in \tilde X_{\ep+}.
\end{align}
Then Lemmas \ref{equal-indices0} and \ref{equal-indices} tell us
$$n^-(L_\ep) = n^-(\tilde{A}_\ep) = 0, \quad \dim \ker (L_\ep) = \dim \ker (\tilde{A}_\ep) = 3 \quad\text{ for all } \epsilon \in [0,1).$$
Thus,  \textbf{(H2.a)} and \textbf{(H3)} are satisfied.
Since $\ker (\tilde{A}_\ep)={\rm{span}}\left\{\eta_\ep(x,y), \gamma_\ep(x,y), \xi_\ep(x,y)\right\}$  for all  $\epsilon \in [0,1)$, the kernel of $L_\ep$ is given explicitly by
\begin{align}\label{ker-L-ep}
\ker (L_\ep)={\rm{span}}\left\{g'(\psi_\ep)\eta_\ep(x,y), g'(\psi_\ep)\gamma_\ep(x,y), g'(\psi_\ep)\xi_\ep(x,y)\right\}.
\end{align}
Noting that $n^-(L_\ep)=0$, we  decompose $X_\ep$ into
$$X_\ep = \ker L_\ep \oplus X_{\ep+}.$$
  To verify \textbf{(H2.b)}, let us first note that for any $\omega \in X_{\ep+}$, we have $\psi=(-\Delta)^{-1}\omega\in \tilde X_{\ep+}$. In fact, it follows from \eqref{ker-L-ep} that $\tilde \omega:=g'(\psi_\ep)\tilde \psi\in\ker (L_\ep)$ for any $\tilde\psi\in \ker(\tilde A_\ep)$, and thus,   $(\psi,\tilde \psi)_{\tilde X_\ep}=\iint_{\Omega}-\Delta\psi\tilde \psi dxdy=\iint_{\Omega}{\omega\tilde \omega\over g'(\psi_\ep)} dxdy=(\omega,\tilde \omega)_{X_\ep}=0$.
By a similar argument to
\eqref{L0omega-omega}, we infer from \eqref{tilde-A-ep-psi-psi-Xep+} that
$$\langle L_\ep \omega, \omega \rangle \geq \langle \tilde{A}_\ep \psi, \psi \rangle\geq {2\over3} \|\nabla \psi\|_{L^2(\Omega)}^2, \quad \omega\in X_{\ep+}.$$
So, we have
\begin{align}\nonumber
\langle L_\ep \omega, \omega \rangle & = \kappa \iint_\Omega\left( \frac{\omega^2}{g'(\psi_\ep)} - |\nabla \psi|^2 \right)dxdy + (1-\kappa)\langle L_\ep \omega, \omega \rangle \\\nonumber
& \geq \kappa \iint_\Omega\left( \frac{\omega^2}{g'(\psi_\ep)} - |\nabla \psi|^2 \right)dxdy + {2\over3}(1-\kappa)  \|\nabla \psi\|_{L^2(\Omega)}^2\\\label{positive-decom}
& \geq \kappa\iint_\Omega \frac{\omega^2}{g'(\psi_\ep)} dxdy=\kappa\|\omega\|_{X_\ep}^2, \quad\forall \; \omega \in X_{\ep+}
\end{align}
by choosing $\kappa > 0$ such that ${2\over3}(1-\kappa) > \kappa$. This verifies \textbf{(H2.b)}. Now by the index formula \eqref{index-formula-stuart}, we have
  $$k_{r,\ep} + 2k_{c,\ep}+2k_{i,\ep}^{\leq0}+k_{0,\ep}^{\leq0}  = n^-(L_\epsilon) = 0.$$
In particular, $$k_{r,\ep} = 2k_{c,\ep}= 0,$$
which implies that there exist no exponential unstable solutions to the linearized vorticity equation \eqref{hami}. Therefore, the steady solution $\omega_\ep$ is spectrally stable.
\end{proof}

\section{Linear instability for  multi-periodic perturbations}\label{multi-periodic-linear}
In this section, we prove the linear instability of Kelvin--Stuart cat's-eye flows for $2m\pi$-periodic perturbations with $m\geq2$.
\subsection{Parity decomposition in the $y$ direction and separable Hamiltonian structure}

Let $\Omega_m = \mathbb{T}_{2m\pi} \times \mathbb{R}$ for $m\geq2$.
As in  \eqref{hami} for co-periodic perturbations, the linearized equation around the Kelvin--Stuart vortex $\omega_\ep$ can be written as
the  Hamiltonian system
 \begin{equation}\label{hami-m}
 \partial_t \omega = J_{\epsilon,m} L_{\epsilon,m}\omega, \quad \omega \in X_{\ep,m},
 \end{equation}
 where
 \begin{align*}
 J_{\ep,m} = -g'(\psi_\epsilon)\vec{u}_\epsilon\cdot\nabla: X_{\ep,m}^* \supset D(J_{\epsilon,m}) \rightarrow X_{\ep,m}, \quad
 L_{\epsilon,m} = \frac {1} {g'(\psi_\ep)} - (-\Delta)^{-1}: X_{\ep,m} \rightarrow X_{\ep,m}^*,\end{align*}
 and
\begin{align*}
X_{\ep,m} = \left\{\omega\bigg| \iint_{\Omega_m} \frac{|\omega|^2}{g'_\epsilon(\psi_\epsilon)} dxdy < \infty, \iint_{\Omega_m} \omega dxdy = 0 \right\},\quad \epsilon\in[0,1).
\end{align*}
To understand the linear stability/instability of the Kelvin--Stuart vortices for multi-periodic perturbations, we first try to compute the index $n^-(L_{\ep,m})$ as in the co-periodic case. Unlike the co-periodic case, $n^-(L_{\ep,m})>0$ in the multi-periodic case. Thus, if
 $$
k_{r,\ep,m} + 2k_{c,\ep,m}+2k_{i,\ep,m}^{\leq0}+k_{0,\ep,m}^{\leq0} = n^-(L_{\ep,m})
$$
  as \eqref{index-formula-stuart} in the co-periodic case,
we have to compute the two indices
$k_{i,\ep,m}^{\leq0}$ and $k_{0,\ep,m}^{\leq0}$ for $J_{\epsilon,m} L_{\epsilon,m}$, which involves a tough and tedious  study  on the pure imaginary eigenvalues of $J_{\epsilon,m} L_{\epsilon,m}$. Here, $k_{r,\ep,m}, k_{c,\ep,m}, k_{i,\ep,m}^{\leq0}, k_{0,\ep,m}^{\leq0}$ are the indices defined similarly as in \eqref{index-formula-stuart}.
To avoid such a difficult part, we  observe that $g'(\psi_\epsilon)\vec{u}_\epsilon\cdot\nabla$ is odd in $y$ and  $ g'(\psi_\epsilon)$ is even in $y$, which implies that $L_{\epsilon,m}$ maps  odd (even) functions in $y$ to  odd (even) functions in $y$, while  $J_{\epsilon,m}$ maps odd (even) functions in $y$ to even (odd) functions in $y$. Based on this observation, we find that the linearized equation \eqref{hami-m} has indeed a separable  Hamiltonian structure. To make it clear, we give some preliminaries. Define two space
\begin{align*}
X_{\ep, e} = \left\{ \omega \in X_{\ep,m} | \omega \text{ is even in }y \right\}\quad\text{and}\quad
X_{\ep, o} = \left\{ \omega \in X_{\ep,m} | \omega \text{ is odd in }y \right\}.\end{align*}
Then $X_{\ep,m}, X_{\ep, e}$ and  $X_{\ep, o}$ are  Hilbert spaces with the $\frac{1}{g'(\psi_\ep)}$-weighted $L^2$ inner product on $\Omega_m$, since they are closed subspaces of $L_{\frac{1}{g'(\psi_\ep)}}^2(\Omega_m)$. Without loss of generality, we denote the dual space of  $X_{\ep, o}$ (resp. $X_{\ep, e}$) restricted to the class of odd (resp. even) functions by $X_{\ep, o}^*$ (resp. $X_{\ep, e}^*$).
Based on above properties on $L_{\epsilon,m}$ and  $J_{\epsilon,m}$, we can define
\begin{align*}B_\ep &= -g'(\psi_\ep) \vec{u}_\ep \cdot \nabla : X_{\ep, o}^* \supset D(B_\ep) \rightarrow X_{\ep, e}, \\
 L_{\ep,o} &= \frac{1}{g'(\psi_\ep)} - (-\Delta)^{-1}: X_{\ep, o} \rightarrow X_{\ep, o}^* \quad\text{and}\quad
 L_{\ep,e} = \frac{1}{g'(\psi_\ep)} - (-\Delta)^{-1}: X_{\ep, e} \rightarrow X_{\ep, e}^*.
 \end{align*}
Here, $(-\Delta)^{-1}\omega$ is the unique  weak solution in $\tilde X_{\ep,o}$ or $\tilde X_{\ep,e}$ of $-\Delta\psi=\omega$ for $\omega\in X_{\ep, o} \text{ or }X_{\ep, e}$, see Lemma \ref{1-1correspond-Lbounded-m} (1).
Then the dual operator of $B_\ep$ is
$$B'_\ep = g'(\psi_\ep) \vec{u}_\ep \cdot \nabla : X_{\ep, e}^* \supset D(B'_\ep) \rightarrow X_{\ep, o}.$$
We decompose $\omega \in X_{\ep,m}$ as $\omega = \left( \begin{array}{c} \omega_1 \\ \omega_2 \end{array} \right) $ such that $\omega_1 \in X_{\ep, e}$ and $\omega_2 \in X_{\ep, o}$. Then the linearized equation \eqref{hami-m} can be written as  the following separable Hamiltonian system
\begin{align}\label{sep-hamiltonian}
\partial_t \left( \begin{array}{c} \omega_1 \\ \omega_2 \end{array} \right) = \left( \begin{array}{cc} 0 & B_\ep \\ -B'_\ep & 0 \end{array} \right)\left( \begin{array}{cc} L_{\ep,e} & 0 \\ 0 & L_{\ep,o} \end{array} \right) \left( \begin{array}{c} \omega_1 \\ \omega_2 \end{array} \right),
\end{align}
or
$$\partial_t \omega = \mathbf{J}_{\ep,m} \mathbf{L}_{\ep,m}  \omega,$$
where $\omega \in \mathbf{X}_{\ep,m} = X_{\ep, e} \times X_{\ep, o}$ and
\begin{align*}\mathbf{J}_{\ep,m} = \left( \begin{array}{cc} 0 & B_\ep \\ -B'_\ep & 0 \end{array} \right): \mathbf{X}_{\ep,m}^* \supset D(\mathbf{J}_{\ep,m}) \rightarrow \mathbf{X}_{\ep,m},\quad\mathbf{L}_{\ep,m} = \left( \begin{array}{cc} L_{\ep,e} & 0 \\ 0 & L_{\ep,o} \end{array} \right): \mathbf{X}_{\ep,m} \rightarrow \mathbf{X}_{\ep,m}^*.\end{align*}
One of the advantage of the separable Hamiltonian system is a  precise counting formula  of unstable modes, see the next lemma \cite{lin2020separable,lin2021linear}.

\begin{lemma}  \label{indice-theorem-sep}
Let $X$ and $Y$ be real Hilbert spaces. Consider a linear Hamiltonian system of the separable form
\begin{align}\label{sep-hamil}
\partial_t \left( \begin{array}{c} u \\ v \end{array} \right) = \left( \begin{array}{cc} 0 & B \\ -B' & 0 \end{array} \right)\left( \begin{array}{cc} L & 0 \\ 0 & A \end{array} \right) \left( \begin{array}{c} u \\ v \end{array} \right) = \mathbf{J}  \mathbf{L} \left( \begin{array}{c} u \\ v \end{array} \right),
\end{align}
where $u \in X$ and $v \in Y$.
Assume that
\begin{itemize}
\item[{\textbf{(G1)}}] The operator $B: Y^* \supset D(B) \rightarrow X$ and its dual operator $B': X^* \supset D(B') \rightarrow Y$ are densely defined and closed.
\item[{\textbf{(G2)}}] The operator $A: Y \rightarrow Y^*$ is bounded and self-dual. Moreover, there exist $\delta > 0$  and a closed subspace $Y_+ \subset Y$ such that
$$ Y = \ker A \oplus Y_+,  \quad \langle Au, u \rangle \geq \delta \|u\|_Y^2,\quad \forall\; u \in Y_+.$$
\item[{\textbf{(G3)}}] The operator $L: X \rightarrow X^*$ is bounded and self-dual, and there exists a decomposition of $X$ into the direct sum of three closed subspaces
$$X = X_- \oplus \ker L \oplus X_+, \quad \dim \ker L < \infty,  \quad n^-(L) = \dim X_- < \infty$$
satisfying
\begin{itemize}
\item[{\textbf{(G3.a)}}]
$\langle Lu, u \rangle < 0$ for all $u\in X_- \backslash \{0\}$;
\item[{\textbf{(G3.b)}}] there exists $\delta > 0$ such that
$$\langle Lu, u \rangle \geq \delta \|u\|_X^2, \quad\forall \; u \in X_+.$$
\end{itemize}
\item[{\textbf{(G4)}}] $\dim \ker L < \infty$ \text{and} $\dim \ker A < \infty$.
\end{itemize}
Then the operator $\mathbf{JL}$ generates a $C^0$ group $e^{t\mathbf{JL}}$ of bounded linear operators on $\mathbf{X} = X \times Y$ and there exists a decomposition
$$\mathbf{X} = E^u \oplus E^c \oplus E^s$$
of closed subspaces $E^{u,s,c}$ with the following properties:

{\rm(i)} $E^c, E^u$ and $ E^s$ are invariant under $e^{t\mathbf{JL}}$.

{\rm(ii)} $E^u (E^s)$ only consists of eigenvectors corresponding to positive (negative) eigenvalues of $\mathbf{JL}$ and
\begin{align}\label{index-formula-neg}
\dim E^u = \dim E^s = n^-\left(L|_{\overline{R(BA)}} \right),
\end{align}
where $n^-\left(L|_{\overline{R(BA)}} \right)$ denotes the number of negative modes of $\langle L\cdot, \cdot \rangle|_{{\overline{R(BA)}} }$. If $n^-\left(L|_{\overline{R(BA)}} \right) >0$, then there exists $M>0$ such that
\begin{align}\label{trichotomy}
|e^{t\mathbf{JL}}|_{E^s}| \leq Me^{-\lambda_ut},\quad t \geq 0; \quad |e^{t\mathbf{JL}}|_{E^u}| \leq Me^{\lambda_ut},\quad t \leq 0,
\end{align}
where $\lambda_u = \min \{ \lambda | \lambda \in \sigma(\mathbf{JL}_{E^u}) \} > 0.$

{\rm(iii)} The quadratic form $\langle \mathbf{L}\cdot, \cdot \rangle$ vanishes on $E^{u,s}$, i.e. $\langle \mathbf{L}\mathbf{u}, \mathbf{u} \rangle = 0$ for all $\mathbf{u} \in E^{u,s}$, but is non-degenerate on $E^u \oplus E^s$ and
$$E^c = \{ \mathbf{u} \in \mathbf{X} | \langle \mathbf{L}\mathbf{u}, \mathbf{v} \rangle = 0, \forall \;\mathbf{v} \in E^s \oplus E^u \}.$$
There exists $M > 0$ such that
\begin{align}\label{trichotomy2}|e^{t\mathbf{JL}}|_{E^c} | \leq M(1+|t|^3), \quad   t \in \mathbb{R}.\end{align}
\end{lemma}

Lemma \ref{indice-theorem-sep} reveals that under the assumptions $\textbf{(G1-4)}$, the solutions of \eqref{sep-hamil} is spectrally stable if and only if $ L|_{{\overline{R(BA)}} }\geq0$. Moreover,  the number of unstable modes is  $n^-\left(L|_{\overline{R(BA)}} \right)$. In addition, the exponential trichotomy estimates \eqref{trichotomy}-\eqref{trichotomy2} are useful in the  study of the nonlinear dynamics, including nonlinear instability and invariant manifolds, near an unstable steady state.

To prove linear instability of the Kelvin--Stuart vortices, we apply the index formula \eqref{index-formula-neg} to the Hamiltonian system \eqref{sep-hamiltonian} after verifying assumptions $\textbf{(G1-4)}$ in Lemma \ref{indice-theorem-sep}. It then suffices to show that  $n^-\left(L_{\ep,e}|_{\overline{R(B_\ep L_{\ep,o})}}\right)>0$. As we explain below, this reduces to the construction of suitable test functions for an associated elliptic operator.


First, we show that  the Hamiltonian system \eqref{sep-hamiltonian} satisfies {\textbf{(G1)}} in Lemma \ref{indice-theorem-sep}.
Since  $(C_0^\infty(\Omega_m)/\mathbb{R})\cap X_{\ep, o}^*\subset D( B_\ep)$ and $(C_0^\infty(\Omega_m)/\mathbb{R})\cap X_{\ep, e}^*\subset D( B_\ep')$, we know that both $B_\ep$ and $B_\ep'$ are densely defined. To prove that they are closed operators, we first prove that the operator $\hat{J}_{\ep,m} = -g'(\psi_\ep) \vec{u}_\ep \cdot \nabla : \hat{X}_{\ep,m}^* \supset D(\hat{J}_{\ep,m}) \rightarrow \hat{X}_{\ep,m}$ with $\hat{X}_{\ep,m} = L^2_{\frac{1}{g'(\psi_\ep)}}(\Omega_m)$ is closed. To show this, by a similar argument to \eqref{tilde-J-ep}, we know that
$\hat{J}_{\ep,m}$
is anti-self-dual, (i.e. $\hat{J}_{\ep,m}'= -\hat{J}_{\ep,m}$), and thus, $\hat{J}_{\ep,m}$ is closed. Since $B_\ep$ and $B'_\ep$ are restrictions of  $\hat{J}_{\ep,m}$  to two closed subspaces of $\hat{X}_{\ep,m}$, we infer that  both $B_\ep$ and $B'_\ep$ are also closed operators, which can be verified directly by  Proposition 1 in Chapter 5 of \cite{Weidmann80}.

To confirm that system \eqref{sep-hamiltonian} satisfies {\textbf{(G2-4)}} in Lemma \ref{indice-theorem-sep}, we transform the operators $L_{\ep,o} $
 and  $L_{\ep,e}$ of vorticity to elliptic operators of stream functions like we did for the co-periodic case.  To this end, we
use the new variables $(\theta_\ep,\gamma_\ep)$ for $(x,y)\in[0,2\pi]\times\mathbb{R}$, and add the definitions $\theta_\ep(x,y)$ and $\gamma_\ep(x,y)$ for $(x,y)\in(2\pi,2m\pi]\times\mathbb{R}$ by $2\pi$-periodic extensions in the $\theta_\ep$ direction.
First, we  give the spaces of stream functions.
Let
\begin{align}\label{space-tilde-Xep-m}
\tilde{X}_{\ep,m} = \left\{\psi \bigg| \iint_{\Omega_m} |\nabla \psi|^2 dxdy <  \infty \text{ and } \int_{0}^{2m\pi}\psi(x,0){1\over1+\epsilon\cos(x)} dx = 0 \right\},
\end{align}
where $\ep\in[0,1)$.
By \eqref{gradient-psi}-\eqref{widehat-Psi-to-psi-1+ep-cos}, in the new variables, $\tilde{X}_{\ep,m}$ is equivalent  to the following space
\begin{align}\label{def-Y-ep-m}
\tilde{Y}_{\ep,m} = \left\{ \Psi \bigg| \iint_{\tilde \Omega_{m}}\left({1\over1-\gamma_\ep^2}|\Psi_{\theta_\ep}|^2+(1-\gamma_\ep^2)|\Psi_{\gamma_\ep}|^2\right)d \theta_\ep d\gamma_\ep< \infty \text{ and } \widehat{\Psi}_0(0)=0 \right\},
\end{align}
 where
  $\tilde{\Omega}_{m} = \mathbb{T}_{2m\pi} \times [-1, 1]$.
Then we  define
$$\tilde{X}_{\ep, e} = \left\{ \psi \in \tilde{X}_{\ep,m} | \psi \text{ is even in }y \right\} \quad \text{and} \quad \tilde{X}_{\ep, o} = \left\{ \psi \in \tilde{X}_{\ep,m} | \psi \text{ is odd in }y \right\},$$
$$\tilde{Y}_{\ep, e} = \left\{ \Psi \in \tilde{Y}_{\ep,m} | \Psi \text{ is even in } \gamma_\ep  \right\} \quad \text{and} \quad \tilde{Y}_{\ep, o} = \left\{ \Psi \in \tilde{Y}_{\ep,m} | \Psi \text{ is odd in }\gamma_\ep \right\}.$$
Following the same steps in Lemmas \ref{Hilbert}, \ref{Hilbert-new variables-0} and \ref{hilbert-ep}, we can prove that $\tilde{X}_{\ep,m}$ is a Hilbert space under the inner product
$$(\psi_1, \psi_2)_{\tilde{X}_{\ep,m}} = \iint_{\Omega_m} \nabla \psi_1 \cdot \nabla \psi_2 dxdy, \quad \forall \;\psi_1, \psi_2 \in \tilde{X}_{\ep,m}.$$
Then $\tilde{X}_{\ep, e}$ and  $\tilde{X}_{\ep, o}$ are Hilbert spaces since they are closed subspaces of $\tilde{X}_{\ep,m}$.
 Correspondingly, $\tilde{Y}_{\ep,m}$ is also a Hilbert space under the inner product
 $$(\Psi_1, \Psi_2)_{\tilde{Y}_{\ep,m}} = \iint_{\tilde{\Omega}_{m}}  \left({1\over1-\gamma_\ep^2}(\Psi_1)_{\theta_\ep}(\Psi_2)_{\theta_\ep} +(1-\gamma_\ep^2)(\Psi_1)_{\gamma_\ep}(\Psi_2)_{\gamma_\ep}\right)d \theta_\ep d\gamma_\ep, \; \forall\; \Psi_1, \Psi_2 \in \tilde{Y}_{\ep,m},$$
and so are $\tilde{Y}_{\ep, e}$ and  $\tilde{Y}_{\ep, o}$.
Moreover,
\begin{align*}
(\psi_1, \psi_2)_{\tilde{X}_{\ep,m}}=(\Psi_1, \Psi_2)_{\tilde{Y}_{\ep,m}}
\end{align*}
for   $\psi_i \in \tilde{X}_{\ep,m}$ and $\Psi_i \in \tilde{Y}_{\ep,m}$ such that $\psi_i(x,y) = \Psi_i(\theta_\ep, \gamma_\ep)$, $i=1,2$.
Then we give the Poincar\'e inequality I for $\ep\in[0,1)$:
\begin{align}\label{Poincare inequality I-m}
\iint_{\Omega_m} g'(\psi_\epsilon)|\psi|^2 dxdy  \leq C \|\nabla \psi\|_{L^2(\Omega_m)}^2,\quad \psi \in \tilde{X}_{\ep,m},
\end{align}
and correspondingly, in the new variables,
\begin{align}\label{Poincare inequality I-m-new}
 \|\Psi\|_{L^2(\tilde \Omega_{m})}^2  \leq C \iint_{\tilde \Omega_m}\left({1\over1-\gamma_\ep^2}|\Psi_{\theta_\ep}|^2+(1-\gamma_\ep^2)|\Psi_{\gamma_\ep}|^2\right)d \theta_\ep d\gamma_\ep,\quad \Psi \in \tilde{Y}_{\ep,m}.
\end{align}
The proof of \eqref{Poincare inequality I-m}-\eqref{Poincare inequality I-m-new} is similar to Lemmas \ref{poincare1} and \ref{Poincare ineqalities-new-variable0} (1) for $\ep=0$, and similar to Lemma \ref{poincare1ep} for $\ep\in(0,1)$.
Let the projection be defined by
\begin{align}\label{def-projection-multi periodic}
P_{\ep,m} \psi = \frac{\iint_{\Omega_m}g'(\psi_\ep) \psi dxdy}{\iint_{\Omega_m}g'(\psi_\ep) dxdy} = \frac{1}{8m\pi}\iint_{\Omega_m}g'(\psi_\ep) \psi dxdy,\quad\psi\in \tilde X_{\ep,m},
\end{align}
and in the new variables, the corresponding projection is
$$\tilde{P}_{\ep,m} \Psi = \frac{\iint_{\tilde{\Omega}_{m}} \Psi d\theta_\ep d\gep}{\iint_{\tilde{\Omega}_{m}}1 d\theta_\ep d\gep} = \frac{1}{4m\pi}\iint_{\tilde{\Omega}_{m}}\Psi d\theta_\ep d\gep,\quad\Psi\in\tilde Y_{\ep,m}.$$
By \eqref{Poincare inequality I-m}-\eqref{Poincare inequality I-m-new}, $P_{\ep,m}$ and $\tilde{P}_{\ep,m}$ are well-defined on $\tilde X_{\ep,m}$ and $\tilde Y_{\ep,m}$, respectively. Then we give the
Poincar\'e inequality II for $\ep\in[0,1)$:
\begin{align}\label{Poincare inequality II-m}
\iint_{\Omega_m} g'(\psi_\epsilon)(\psi - P_{\epsilon,m}\psi)^2 dxdy  \leq C \|\nabla \psi\|_{L^2(\Omega_m)}^2, \quad \psi \in \tilde{X}_{\ep,m},
\end{align}
and correspondingly, in the new variables,
\begin{align}\nonumber
&\iint_{\tilde \Omega_{m}} (\Psi -\tilde P_{\epsilon,m}\Psi)^2 d \theta_\ep d\gamma_\ep  \\\label{Poincare inequality II-m-new}
\leq& C \iint_{\tilde \Omega_{m}}\left({1\over1-\gamma_\ep^2}|\Psi_{\theta_\ep}|^2+(1-\gamma_\ep^2)|\Psi_{\gamma_\ep}|^2\right)d \theta_\ep d\gamma_\ep,\;\Psi \in \tilde{Y}_{\ep,m}.
\end{align}
The proof of \eqref{Poincare inequality II-m}-\eqref{Poincare inequality II-m-new} is similar to Lemmas \ref{poincare2} and \ref{Poincare ineqalities-new-variable0} (3) for $\ep=0$, and similar to Lemma \ref{poincare2ep} for $\ep\in(0,1)$.
By the fact that $X_{\ep,o}$ (resp.  $X_{\ep,e}$) is a Hilbert space and the Poincar\'e inequality I \eqref{Poincare inequality I-m}, one can prove the following results by a similar argument to Lemmas \ref{1-1correspond} and \ref{Lbounded}.

\begin{lemma}\label{1-1correspond-Lbounded-m} Let $\ep\in[0,1)$.
$(1)$ For $\omega \in X_{\ep,o}$ (resp.  $X_{\ep,e}$), the Poisson equation
$-\Delta\psi=\omega$
has a unique weak solution in $\tilde{X}_{\ep,o}$ (resp.  $\tilde X_{\ep,e}$).

$(2)$
For  $\omega_1,\omega_2 \in X_{\ep,o}$, we have
$\langle L_{\ep,o} \omega_1, \omega_2 \rangle=\langle \omega_1, L_{\ep,o} \omega_2 \rangle \leq C\|\omega_1\|_{X_{\ep,o}}\|\omega_2\|_{X_{\ep,o}}.$

$(3)$
For  $\omega_1,\omega_2 \in X_{\ep,e}$, we have
$\langle L_{\ep,e} \omega_1, \omega_2 \rangle=\langle \omega_1, L_{\ep,e} \omega_2 \rangle \leq C\|\omega_1\|_{X_{\ep,e}}\|\omega_2\|_{X_{\ep,e}}.$
\end{lemma}
By Lemma \ref{1-1correspond-Lbounded-m} (2)-(3), both $L_{\ep,o}:X_{\ep, o}\rightarrow X_{\ep, o}^*$ and $L_{\ep,e}:X_{\ep, e}\rightarrow X_{\ep, e}^*$ are self-dual and bounded.

\subsection{Exact solutions to the associated eigenvalue problems  for the multi-periodic case}
Next, we consider the decomposition of $X_{\ep, o}$ and $X_{\ep, e}$ associated to $L_{\ep,o}$ and $L_{\ep,e}$, respectively.
Define the elliptic operators
$$\tilde{A}_{\ep,o} = -\Delta - g'(\psi_\ep)(I-P_{\ep,m}) =  -\Delta - g'(\psi_\ep): \tilde{X}_{\ep, o} \rightarrow \tilde{X}_{\ep, o}^*$$
and
$$\tilde{A}_{\ep,e} = -\Delta - g'(\psi_\ep)(I-P_{\ep,m}): \tilde{X}_{\ep, e} \rightarrow \tilde{X}_{\ep, e}^*,$$
where we used $P_{\ep,m}\psi=0$ for $\psi\in\tilde{X}_{\ep, o}$.
The dual space of  $\tilde X_{\ep, o}$ (resp. $\tilde X_{\ep, e}$)  restricted to the class of odd (resp. even) functions is denoted  by $\tilde X_{\ep, o}^*$ (resp. $\tilde X_{\ep, e}^*$).
Based on Lemma \ref{1-1correspond-Lbounded-m} and  \eqref{Poincare inequality II-m}, we  prove
\begin{align}\label{n-ker-o}
n^-(L_{\ep,o}) &= n^-(\tilde A_{\ep,o}), \quad \dim \ker(L_{\ep,o}) = \dim \ker(\tilde A_{\ep,o}),\\\label{n-ker-e}
n^-(L_{\ep,e}) & = n^-(\tilde A_{\ep,e}), \quad \dim \ker(L_{\ep,e}) = \dim \ker(\tilde A_{\ep,e})
\end{align}
by a similar way as  Lemma \ref{equal-indices0}. Similar to Lemmas \ref{compact2P}, \ref{compact2P-new-variable0} and \ref{compact2P-new-variable-ep},
 $\tilde Y_{\ep,m}$ is compactly embedded in $L^2(\tilde \Omega_{m})$ and
$$
 \tilde Z_{\ep,m}:=\left\{\Psi\bigg|\iint_{\tilde \Omega_{m}}|\Psi-\tilde P_{\ep,m}\Psi|^2d\theta_\ep d\gamma_\ep<\infty\right\},
 $$
respectively. Correspondingly,
 $\tilde X_{\ep,m}$ is compactly embedded in $L_{g'(\psi_\ep)}^2(\Omega_m)$ and
$$
 Z_{\ep,m}:=\left\{\psi\bigg|\iint_{\Omega_m}g'(\psi_\ep)|\psi-P_{\ep,m}\psi|^2dxdy<\infty\right\},
$$
respectively.
Thus,  we can
inductively define
\begin{align}\label{variational problem2-ep-m}
\lambda_n(\ep,m)=& \inf_{\psi \in \tilde X_{\ep,m}, (\psi, \psi_{i})_{Z_{\ep,m}} = 0, i = 1, 2, \cdots, n-1}{\|\psi\|_{\tilde{X}_{\ep,m}}^2\over\iint_{\Omega_m} g'(\psi_\ep)(\psi - P_{\ep,m}\psi)^2dxdy},\quad n\geq1,
\end{align}
where the  infimum for $\lambda_i(\ep,m)$ is attained at $\psi_{i} \in \tilde X_{\ep,m}$ and $\iint_{\Omega_m} g'(\psi_\ep)(\psi_{i} - {P_{\ep,m}}\psi_{i})^2 dxdy = 1$, $1\leq i \leq n-1$.
Then in the new variables,
\begin{align}\label{variational problem2-ep-m-new variables}
\lambda_n(\ep,m)
=&\inf_{\Psi \in \tilde Y_{\ep,m}, (\Psi, \Psi_{i})_{\tilde Z_{\ep,m}} = 0, i = 1, 2, \cdots, n-1}{\|\Psi\|_{\tilde{Y}_{\ep,m}}^2\over\iint_{\tilde \Omega_{m}}2|\Psi-\tilde P_{\ep,m}\Psi|^2d\theta_\ep d\gamma_\ep},\quad
n\geq1.
\end{align}
By a similar argument to \eqref{variational problem2-ep}-\eqref{eigenvalue problem-ep-original},  we arrive at the eigenvalue problem
\begin{align}\label{eigenvalue problem-ep-new-m}
-\pa_{\gamma_\ep}\left((1-\gamma_\ep^2)\pa_{\gamma_\ep}\Psi\right)-{1\over1-\gamma_\ep^2}\pa_{\theta_\ep}^2\Psi
=2\lambda(\Psi-\tilde P_{\ep,m}\Psi), \quad \Psi \in \tilde{Y}_{\ep,m},
\end{align}
which, in the original variables, is exactly
\begin{align}\label{eigenvalue problem-ep-original-m}
-\Delta \psi = \lambda g'(\psi_\ep)(\psi -  P_{\ep,m}\psi), \quad \psi \in \tilde{X}_{\ep,m}.
\end{align}
In the new variables $(\theta_\ep,\gamma_\ep)$, we use the Fourier expansion $\Psi(\theta_\ep,\gamma_\ep)=\sum_{k\in\mathbb{Z}}\widehat{\Psi}_{k}(\gamma_\ep)e^{i{k\over m} \theta_\ep}$ to separate the variables, and study the eigenvalue problem \eqref{eigenvalue problem-ep-new-m} for the $0$-mode and the non-zero modes, separately.  For the $0$-mode,
the eigenvalue problem  is
\begin{align}\label{eigenvalue problem for 0 mode-m}
-\left((1-\gamma_\ep^2)  \varphi'\right)' = 2 \lambda(\varphi-\hat{P}_{0}^\ep\varphi) \quad \text{on}\quad (-1,1),\quad\varphi\in \hat Y_{0}^\ep,
\end{align}
where $\hat{P}_{0}^\ep\varphi={1\over2}\int_{-1}^{1}\varphi(\gamma_\ep)d\gamma_\ep$  and
\begin{equation*}
\hat Y_0^\ep=\left\{\varphi\bigg|\int_{-1}^1(1-\gamma_\ep^2)|\varphi'(\gamma_\ep)|^2d\gamma_\ep<\infty\text{ and }\varphi(0)=0 \right\}.
\end{equation*}
Since the eigenvalue problem \eqref{eigenvalue problem for 0 mode-m} for the $0$-mode  is the same as \eqref{eigenvalue problem for 0 mode}, by applying Lemma
\ref{sol to eigenvalue problem},
all the eigenvalues  of the eigenvalue problem \eqref{eigenvalue problem for 0 mode-m} with corresponding eigenfunctions are as follows:
\begin{align}\label{0 mode-eigenvalue-eigenfunction-m}
\lambda_{n,0}={n(n+1)\over2},\quad\varphi_{n,0}(\gamma_\ep)= L_{n}(\gamma_\ep)-L_{n}(0),\quad n\geq1.
\end{align}

The difference comes from the non-zero modes. For the $k$-mode, the eigenvalue problem \eqref{eigenvalue problem-ep-new-m}
is
\begin{equation}\label{eigenvalue problem2 non-zero modes varepsilon=0m}
-((1-\gamma_\ep^2)\varphi')'+{{k^2\over m^2}\over1-\gamma_\ep^2}\varphi =2\lambda\varphi \quad \text{on}\quad (-1,1),\quad\varphi\in \hat Y_1^\ep,
\end{equation}
where $k\neq0$ and
\begin{align}\label{Y-k-ep-def}
\hat Y_1^\ep=\left\{\varphi\bigg|\int_{-1}^1\left({1\over1-\gamma_\ep^2}|\varphi(\gamma_\ep)|^2+(1-\gamma_\ep^2)|\varphi'(\gamma_\ep)|^2\right)d\gamma_\ep
<\infty \right\},
\end{align}
which is the same space
 $\hat Y_1$ defined in \eqref{def-hat-Y1} if we replace the variable $\gamma_\ep$ by $\gamma$ in \eqref{Y-k-ep-def}.
To the best of our knowledge, the existing approach to solving the eigenvalue problem \eqref{eigenvalue problem2 non-zero modes varepsilon=0m} is via the hypergeometric functions directly, but it seems a tedious task to compute all the eigenvalues and corresponding eigenfunctions in this way. Our method is motivated as follows. For $m=2$ and $k=1$, we observe that $\varphi(\gamma_\ep)=(1-\gamma_\ep^2)^{1\over 4}$ and $\lambda={3\over 8}$ solve \eqref{eigenvalue problem2 non-zero modes varepsilon=0m}. Taking
$
\varphi=(1-\gamma_\ep^2)^{1\over4}\phi,
$
then $\phi$ solves
\begin{equation}\label{tran-m=2-equa}
(1-\gamma_\ep^2)\phi''-3\gamma_\ep\phi'+\left(-{3\over4}+2\lambda\right)\phi =0 \quad \text{on}\quad (-1,1),\quad\phi\in W_{1\over 2},
\end{equation}
where $W_{1\over 2}=\{\phi|(1-\gamma_\ep^2)^{1\over4}\phi\in\hat Y_1^\ep\}$.
Then $\phi=1$ and $\lambda={3\over 8}$ solve \eqref{tran-m=2-equa}. Moreover, $\phi=\gamma_\ep$ and $\lambda={15\over 8}$ also solve \eqref{tran-m=2-equa}. As in the co-periodic case, our perspective is that all the eigenfunctions for \eqref{tran-m=2-equa}  might be polynomials of $\gamma_\ep$. They are indeed polynomials of $\gamma_\ep$ after we find that \eqref{tran-m=2-equa} is exactly the Gegenbauer differential equation
\begin{equation}\label{tran-m=2-equa-g-2}
(1-\gamma_\ep^2)\phi''-(2\beta+1)\gamma_\ep\phi'+n(n+2\beta)\phi =0 \quad \text{on}\quad (-1,1)
\end{equation}
for $\beta=1$ in \eqref{tran-m=2-equa-g-2} and $\lambda={1\over2} \left(n^2+2n+{3\over4}\right)$, $n\geq0$, in \eqref{tran-m=2-equa}. All the solutions of \eqref{tran-m=2-equa-g-2} are given by  Gegenbauer polynomials.
To
 solve the eigenvalue problem \eqref{eigenvalue problem2 non-zero modes varepsilon=0m} for general $k\geq1$ and $m\geq2$,
we introduce the transformation
\begin{align}\label{transformation-multi-periodic perturbations}
\varphi=(1-\gamma_\ep^2)^{k\over2m}\phi.
\end{align}
Then \eqref{eigenvalue problem2 non-zero modes varepsilon=0m} is transformed to
\begin{equation}\label{eigenvalue problem2 non-zero modes varepsilon=0-transform}
(1-\gamma_\ep^2)\phi''-2\left({k\over m}+1\right)\gamma_\ep\phi'+\left(-{k^2\over m^2}-{k\over m}+2\lambda\right)\phi =0 \quad \text{on}\quad (-1,1),\quad\varphi\in W_{k\over m},
\end{equation}
where $W_{k\over m}=\{\phi|(1-\gamma_\ep^2)^{k\over2m}\phi\in\hat Y_1^\ep\}$. It is  well-known \cite{Suetin2001} that the
Gegenbauer polynomials
\begin{align}\label{Gegenbauer polynomials}
C_n^\beta(\gamma_\ep)={(-1)^n\over 2^n n!}{\Gamma(\beta+{1\over2})\Gamma(n+2\beta)\over\Gamma(2\beta)\Gamma(\beta+n+{1\over2})}(1-\gamma_\ep^2)^{-\beta+{1\over2}}{d^n\over d\gamma_\ep^n}\left((1-\gamma_\ep^2)^{n+\beta-{1\over2}}\right)
\end{align}
are solutions of
 the   Gegenbauer
 differential equations
\begin{equation}\label{Gegenbauer differential equation}
(1-\gamma_\ep^2)\phi''-(2\beta+1)\gamma_\ep\phi'+n\left(n+2\beta\right)\phi =0 \quad \text{on}\quad (-1,1),\quad\phi\in L_{\hat g_\beta}^2(-1,1),
\end{equation}
where $n\geq0$ and $\hat g_\beta(\gamma_\ep)=(1-\gamma_\ep^2)^{\beta-{1\over2}}$. Moreover, $\{C_n^\beta\}_{n=0}^\infty$ is a complete and  orthogonal basis of $ L_{\hat g_\beta}^2(-1,1)$ for $\beta>-{1\over2}$. Set
\begin{align*}
\beta\triangleq{k\over m}+{1\over2},\quad \lambda\triangleq{1\over2}\left({k^2\over m^2}+{k\over m}+n^2+{2nk\over m}+n\right)={1\over2}\left(n+{k\over m}\right)\left(n+{k\over m} +1\right),
\end{align*}
and then the two equations
in \eqref{Gegenbauer differential equation} and \eqref{eigenvalue problem2 non-zero modes varepsilon=0-transform} surprisingly coincide.
Furthermore, $(1-\gamma_\ep^2)^{k\over2m}C_n^\beta\in\hat Y_1^\ep$ for $n\geq0$.
In fact,
\begin{align}\nonumber
&\int_{-1}^1\left({1\over1-\gamma_\ep^2}(1-\gamma_\ep^2)^{k\over m}|C_n^\beta(\gamma_\ep)|^2+(1-\gamma_\ep^2)\left|\left((1-\gamma_\ep^2)^{k\over 2m}C_n^\beta(\gamma_\ep)\right)'\right|^2\right)d\gamma_\ep\\\nonumber
=&\int_{-1}^1(1-\gamma_\ep^2)^{{k\over m}-1}|C_n^\beta(\gamma_\ep)|^2d\gamma_\ep\\
&+\int_{-1}^1\left|-{k\over m}\gamma_\ep(1-\gamma_\ep^2)^{{k\over 2m}-{1\over2}} C_n^\beta(\gamma_\ep)+(1-\gamma_\ep^2)^{{k\over 2m}+{1\over2}}(C_n^\beta(\gamma_\ep))'\right|^2d\gamma_\ep
<\infty.\label{1-gammacnykep}
\end{align}
This implies that
\begin{align*}
&\varphi_{n,{k\over m}}(\gamma_\ep)\triangleq(1-\gamma_\ep^2)^{k\over2m}C_n^{{k\over m}+{1\over2}}(\gamma_\ep)\in\hat Y_1^\ep,\quad\lambda=\lambda_{n,{k\over m}}\triangleq{1\over2}\left(n+{k\over m}\right)\left(n+{k\over m} +1\right)
\end{align*}
solves \eqref{eigenvalue problem2 non-zero modes varepsilon=0m} for $n\geq0$.
Since $\{C_n^\beta\}_{n=0}^\infty$ is a complete and  orthogonal basis of $ L_{\hat g_\beta}^2(-1,1)$, and
\begin{align*}
\int_{-1}^1\hat g_\beta(\gamma_\ep) C_{n_1}^\beta (\gamma_\ep)C_{n_2}^\beta (\gamma_\ep) d\gamma_\ep=&\int_{-1}^1(1-\gamma_\ep^2)^{k\over m}C_{n_1}^\beta (\gamma_\ep)C_{n_2}^\beta (\gamma_\ep) d\gamma_\ep \\
=&\int_{-1}^1\varphi_{n_1,{k\over m}}(\gamma_\ep)\varphi_{n_2,{k\over m}}(\gamma_\ep)  d\gamma_\ep
\end{align*}
for $n_1,n_2\geq0$, we know that  $\{\varphi_{n,{k\over m}}\}_{n=0}^\infty$ is a complete and  orthogonal basis of $ L^2(-1,1)$.
Since $\hat Y_1^\ep$ is embedded in $ L^2(-1,1)$ by Lemma
\ref{Poincare inequalities compact embedding result new variables k mode}, we infer that
$\{\varphi_{n,{k\over m}}\}_{n=0}^\infty$ is a complete and  orthogonal basis of $\hat Y_1^\ep$ under the inner product of $ L^2(-1,1)$.
In summary, the eigenvalue problem \eqref{eigenvalue problem2 non-zero modes varepsilon=0m} is solved as follows.

\begin{lemma}\label{sol to eigenvalue problem non-zero modes varepsilon=0} Fix $m\geq2$ and $k\geq1.$ Then
all the eigenvalues  of the eigenvalue problem \eqref{eigenvalue problem2 non-zero modes varepsilon=0m} are $\lambda_{n,{k\over m}}={1\over2}\left(n+{k\over m}\right)\left(n+{k\over m} +1\right)$, $n\geq 0$. For $n\geq0$, the eigenspace associated to $\lambda_{n,{k\over m}}$ is $\text{span}\{\varphi_{n,{k\over m}}(\gamma_\ep)\}=\text{span}\{(1-\gamma_\ep^2)^{k\over2m}C_n^{{k\over m}+{1\over2}}(\gamma_\ep)\}$.
\end{lemma}

Combining \eqref{0 mode-eigenvalue-eigenfunction-m} and Lemma \ref{sol to eigenvalue problem non-zero modes varepsilon=0}, we solve
the eigenvalue problem \eqref{eigenvalue problem-ep-new-m} (and hence, \eqref{eigenvalue problem-ep-original-m}).
\begin{Theorem}\label{sol to eigenvalue problem varepsilon=0-pde} Fix $m\geq2$.

$(1)$
All the eigenvalues  of the eigenvalue problem \eqref{eigenvalue problem-ep-new-m} are
\begin{align}\label{km-eigenvalues}
{1\over2}n\left(n+1\right), &\quad n\geq1,\\\label{nonkm-eigenvalues}
{1\over2}\left(n+{i\over m}\right)\left( n+{i\over m}+1\right), &\quad 1\leq i\leq m-1,\; n\geq0.
\end{align} The corresponding eigenspaces are given as follows.
 \begin{itemize}
 \item For $n\geq1$,
the eigenspace associated to the eigenvalue  ${1\over2}n\left( n+1\right)$ is spanned by
 \begin{align}\label{km-eigenfunctions}
L_{n}(\gamma_\ep)-L_{n}(0),\;\; L_{n,j}(\gamma_\ep)\cos(j\theta_\ep),\;\;
 L_{n,j}(\gamma_\ep)\sin(j\theta_\ep), \quad  1 \leq j\leq n.
 \end{align}

\item For $1\leq i\leq m-1$ and $n\geq0$,
the eigenspace associated to the eigenvalue  ${1\over2}\left(n+{i\over m}\right)$ $\left( n+{i\over m}+1\right)$ is spanned by
\begin{align}\nonumber
 &(1-\gamma_\ep^2)^{(n-j)m+i\over2m}C_j^{{(n-j)m+i\over m}+{1\over2}}(\gamma_\ep)\cos\left({(n-j)m+i\over m}\theta_\ep\right),\\\label{nonkm-eigenfunctions}
 &(1-\gamma_\ep^2)^{(n-j)m+i\over2m}C_j^{{(n-j)m+i\over m}+{1\over2}}(\gamma_\ep)\sin\left({(n-j)m+i\over m}\theta_\ep\right),\;\;0\leq j\leq n.
 \end{align}
 \end{itemize}

$(2)$
All the eigenvalues  of the associated eigenvalue problem \eqref{eigenvalue problem-ep-original-m} are
given in \eqref{km-eigenvalues}-\eqref{nonkm-eigenvalues}. The corresponding eigenspaces are given as follows.
 \begin{itemize}
 \item For $n\geq1$,
the eigenspace associated to the eigenvalue  ${1\over2}n\left( n+1\right)$ is spanned by
 \begin{align*}
L_{n}(\gamma_\ep(x,y))-L_{n}(0),\;\; L_{n,j}(\gamma_\ep(x,y))\cos(j\theta_\ep(x,y)),\;\;
 L_{n,j}(\gamma_\ep(x,y))\sin(j\theta_\ep(x,y)), \quad  1 \leq j\leq n.
 \end{align*}
\item For $1\leq i\leq m-1$ and $n\geq0$,
the eigenspace associated to the eigenvalue  ${1\over2}\left(n+{i\over m}\right)$ $\left( n+{i\over m}+1\right)$ is spanned by
\begin{align*}
 &(1-\gamma_\ep(x,y)^2)^{(n-j)m+i\over2m}C_j^{{(n-j)m+i\over m}+{1\over2}}(\gamma_\ep(x,y))\cos\left({(n-j)m+i\over m}\theta_\ep(x,y)\right),\\
 &(1-\gamma_\ep(x,y)^2)^{(n-j)m+i\over2m}C_j^{{(n-j)m+i\over m}+{1\over2}}(\gamma_\ep(x,y))\sin\left({(n-j)m+i\over m}\theta_\ep(x,y)\right),\;\;0\leq j\leq n.
 \end{align*}
 \end{itemize}
Here $\theta_\ep(x,y)$ and  $\gamma_\ep(x,y)$ are defined in \eqref{transf1} and \eqref{transf2}.

 In particular,
the multiplicity of $ {1\over2}n\left(n+1\right)$ is $2n+1$ for  $n\geq1$, and the multiplicity of
${1\over2}\left(n+{i\over m}\right)\left( n+{i\over m}+1\right)$ is $2n+2$ for $ 1\leq i\leq m-1$ and  $n\geq0$.
\end{Theorem}
\begin{proof}
By
\eqref{0 mode-eigenvalue-eigenfunction-m} and Lemma \ref{sol to eigenvalue problem non-zero modes varepsilon=0}  the set of all the eigenvalues of \eqref{eigenvalue problem-ep-new-m} is
\begin{align*}
&\left\{{1\over2}n\left(n+1\right)\right\}_{n=1}^\infty\cup\left(\bigcup_{k=1}^\infty\left\{{1\over2}\left(n+{k\over m}\right)\left(n+{k\over m} +1\right)\right\}_{n=0}^\infty\right)\\
=&\left\{{1\over2}n\left(n+1\right)\right\}_{n=1}^\infty\cup
\left(\bigcup_{i=1}^{m-1}\left\{{1\over2}\left(n+{i\over m}\right)\left(n+{i\over m} +1\right)\right\}_{n=0}^\infty\right).
\end{align*}

Let  $n\geq1$. Then ${1\over2}n\left( n+1\right)$  is the eigenvalue of the $0$-mode with
an eigenfunction   $L_{n}(\gamma_\ep)-L_{n}(0)$. It is also  the eigenvalue $\lambda_{n-j,{k\over m}}$ of the $k=jm$ mode with an eigenfunction
$(1-\gamma_\ep^2)^{j\over2}C_{n-j}^{j+{1\over2}}(\gamma_\ep)$ for $1\leq j\leq n$. Then up to a constant factor, the equality \begin{align*}(1-\gamma_\ep^2)^{j\over2}C_{n-j}^{j+{1\over2}}(\gamma_\ep)=L_{n,j}(\gamma_\ep)\end{align*} gives \eqref{km-eigenfunctions}.

Let $1\leq i\leq m-1$ and  $n\geq0$. Then  ${1\over2}\left(n+{i\over m}\right)\left( n+{i\over m}+1\right)$ is the eigenvalue $\lambda_{j,{k\over m}}$ of the $k=(n-j)m+i$ mode with an eigenfunction $(1-\gamma_\ep^2)^{(n-j)m+i\over2m}C_j^{{(n-j)m+i\over m}+{1\over2}}(\gamma_\ep)$ for  $0\leq j\leq  n$, which gives \eqref{nonkm-eigenfunctions}.
\end{proof}

As an application,
we prove that  $\tilde A_{\ep,o}$ and $L_{\ep,o}$ are  non-negative, present their explicit kernel,  and obtain  decompositions of
$\tilde X_{\ep,o}$ and $X_{\ep,o}$  associated to the two operators. This verifies  {\textbf{(G2)}} in Lemma \ref{indice-theorem-sep} for \eqref{sep-hamiltonian}.

\begin{Corollary}\label{A-L-dec-o}
 Let $\ep\in[0,1)$. Then

$(1)$ $\ker (\tilde A_{\ep,o})=\textup{span}\{\gamma_\ep(x,y)\}$ and $\ker (L_{\ep,o})=\textup{span}\{g'(\psi_\ep)\gamma_\ep(x,y)\}$. Thus,  $\dim \ker(L_{\ep,o})=\dim\ker (\tilde A_{\ep,o})=1$.

$(2)$ Let $\tilde X_{\ep,o+}=\tilde X_{\ep,o} \ominus\ker (\tilde A_{\ep,o})$ and $ X_{\ep,o+}= X_{\ep,o} \ominus\ker (L_{\ep,o})$. Then
\begin{align*}
\langle \tilde A_{\ep,o} \psi,\psi\rangle \geq \left(1-{2m^2\over (m+1)(2m+1)}\right) \| \psi\|_{\tilde X_{\ep,o}}^2, \quad\forall \psi\in \tilde X_{\ep,o+},
\end{align*}
and there exists $\delta>0$ such that
\begin{align*}
\langle L_{\ep,o} \omega,\omega\rangle \geq \delta \| \omega\|_{ X_{\ep,o}}^2, \quad \forall \omega\in  X_{\ep,o+}.
\end{align*}
\end{Corollary}
\begin{proof}
Note that $\psi(x,y)$ is odd in $y$ if and only if $\Psi(\theta_\ep, \gep)$ is odd in $\gep$ for $\psi \in \tilde{X}_{\ep,m}$ and $\Psi \in \tilde{Y}_{\ep,m}$ such that $\psi(x,y) = \Psi(\theta_\ep, \gamma_\ep)$. Thus, $\psi\in\tilde{X}_{\ep,o}$ if and only if $\Psi\in\tilde{Y}_{\ep,o}$. We consider the eigenvalue problem \eqref{eigenvalue problem-ep-new-m} with $\Psi \in \tilde{Y}_{\ep,o}$ by separating it into the Fourier modes.

For the $0$-mode, the eigenvalue problem \eqref{eigenvalue problem-ep-new-m} is reduced to \eqref{eigenvalue problem for 0 mode-m}. Noting that the eigenfunction $\varphi_{n,0}$ in \eqref{0 mode-eigenvalue-eigenfunction-m} is odd if and only if $n\geq1$ is odd, we obtain that all the eigenvalues and corresponding eigenfunctions  are given in \eqref{0 mode-eigenvalue-eigenfunction-m} for odd integers $n\geq1$. Thus, the principal eigenvalue  for the $0$-mode is $1$ with an eigenfunction $\gamma_\ep$. This implies that
 there is no contribution to the negative directions of $ \tilde A_{\ep,o} $ from the $0$-mode, and $\gamma_\ep(x,y)\in\ker( \tilde A_{\ep,o} )$.

For the $k$-mode with $k\neq0$, the eigenvalue problem \eqref{eigenvalue problem-ep-new-m} is reduced to \eqref{eigenvalue problem2 non-zero modes varepsilon=0m}. Noting that the eigenfunction $\varphi_{n,{k\over m}}(\gamma_\ep)$ in Lemma \ref{sol to eigenvalue problem non-zero modes varepsilon=0} is odd if and only if $n\geq0$ is odd, we know that
all the eigenvalues and corresponding  eigenfunctions  are given in Lemma  \ref{sol to eigenvalue problem non-zero modes varepsilon=0}  for odd integers $n\geq0$. Thus, the principal eigenvalue  for the $k$-mode is ${1\over2}\left(1+{k\over m}\right)\left(2+{k\over m} \right)>1$. Then  there is no contribution to the negative and kernel directions of $ \tilde A_{\ep,o} $ from the $k$-mode.
This confirms that $\ker (\tilde A_{\ep,o})=\textup{span}\{\gamma_\ep(x,y)\}$.

Since the second  eigenvalue  for the $0$-mode is $6$ and the  principal eigenvalue  for the $k$-mode is ${1\over2}\left(1+{k\over m}\right)\left(2+{k\over m} \right)>1$ with $k\neq0$, by the variational problem \eqref{variational problem2-ep-m}-\eqref{variational problem2-ep-m-new variables} we have
\begin{align*}
\iint_{\Omega_m}|\nabla\psi|^2dxdy\geq{1\over2}\left(1+{1\over m}\right)\left(2+{1\over m}\right)\iint_{\Omega_m}g'(\psi_\ep)(\psi-P_{\ep,m}\psi)^2dxdy,\;\psi\in\tilde X_{\ep,o+},
\end{align*}
where $\tilde X_{\ep,o+}=\tilde X_{\ep,o} \ominus\ker (\tilde A_{\ep,o})$.
Thus,
\begin{align*}
 \langle\tilde A_{\ep,o}\psi,\psi\rangle=&\iint_{\Omega_m} \left(|\nabla\psi|^2-g'(\psi_\epsilon)(\psi - P_{\epsilon,m}\psi)^2\right) dxdy \\
 \geq& \left(1-{2m^2\over (m+1)(2m+1)}\right) \| \psi\|_{\tilde X_{\ep,o}}^2
\end{align*}
for $\psi\in\tilde X_{\ep,o+}$.

By \eqref{n-ker-o},  $\ker (L_{\ep,o})=\textup{span}\{g'(\psi_\ep)\gamma_\ep(x,y)\}$.
The proof of $\langle L_{\ep,o} \omega,\omega\rangle \geq \delta \| \omega\|_{ X_{\ep,o}}^2$ for   $\omega\in  X_{\ep,o+} $ is similar to  \eqref{positive-decom}.
\end{proof}

Next, we give
 the explicit negative directions and  kernel of  the operators $\tilde A_{\ep,e}$ and $L_{\ep,e}$, as well as   decompositions of
$\tilde X_{\ep,e}$ and $X_{\ep,e}$  associated to $\tilde A_{\ep,e}$ and $L_{\ep,e}$, respectively. This verifies  {\textbf{(G3)}} in Lemma \ref{indice-theorem-sep} for \eqref{sep-hamiltonian}.

\begin{Corollary}\label{A-L-dec-e}
 Let $\ep\in[0,1)$. Then

$(1)$  the negative subspaces of  $\tilde X_{\ep, e}$ and  $ X_{\ep, e}$  associated to $\tilde A_{\ep,e}$ and $L_{\ep,e}$ are
 \begin{align*}
\tilde X_{\ep,e-}&=\textup{span}\left\{(1-\gamma_\ep^2)^{i\over 2m}\cos\left({i\theta_\ep\over m}\right),(1-\gamma_\ep^2)^{i\over 2m}\sin\left({i\theta_\ep\over m}\right), 1\leq i\leq m-1\right\},\\
 X_{\ep,e-}&=\textup{span}\left\{g'(\psi_\ep)(1-\gamma_\ep^2)^{i\over 2m}\cos\left({i\theta_\ep\over m}\right),g'(\psi_\ep)(1-\gamma_\ep^2)^{i\over 2m}\sin\left({i\theta_\ep\over m}\right), 1\leq i\leq m-1\right\},
\end{align*}
respectively, where $\gamma_\ep=\gamma_\ep(x,y)$ and $\theta_\ep=\theta_\ep(x,y)$.
Thus,  $\dim \tilde X_{\ep,e-}=\dim X_{\ep,e-}=2(m-1)$.

$(2)$ $\ker (\tilde A_{\ep,e})=\textup{span}\{(1-\gamma_\ep^2)^{1\over 2}\cos\left({\theta_\ep}\right),(1-\gamma_\ep^2)^{1\over 2}\sin\left({\theta_\ep}\right)\}$ and $\ker (L_{\ep,e})=\textup{span}\{g'(\psi_\ep)(1-\gamma_\ep^2)^{1\over 2}\cos\left({\theta_\ep}\right),g'(\psi_\ep)(1-\gamma_\ep^2)^{1\over 2}\sin\left({\theta_\ep}\right)\}$. Thus,  $\dim\ker (\tilde A_{\ep,e})=\dim \ker(L_{\ep,e})=2$.

$(3)$ Let  $ X_{\ep,e+}= X_{\ep,e} \ominus\left(\ker (L_{\ep,e})\oplus X_{\ep,e-}\right)$ and $\tilde X_{\ep,e+}=\tilde X_{\ep,e} \ominus\left(\ker (\tilde A_{\ep,e})\oplus \tilde  X_{\ep,e-}\right)$. Then
\begin{align*}
\langle \tilde A_{\ep,e} \psi,\psi\rangle \geq \left(1-{2m^2\over (m+1)(2m+1)}\right) \| \psi\|_{\tilde X_{\ep,e}}^2, \quad\forall \psi\in \tilde X_{\ep,e+},
\end{align*}
there exists $\delta>0$ such that
\begin{align*}
\langle L_{\ep,e} \omega,\omega\rangle \geq \delta \| \omega\|_{ X_{\ep,e}}^2, \quad \forall \omega\in  X_{\ep,e+}.
\end{align*}
\end{Corollary}

\begin{proof}
Note that $\psi\in\tilde{X}_{\ep,e}$ if and only if $\Psi\in\tilde{Y}_{\ep,e}$ for $\psi \in \tilde{X}_{\ep,m}$ and $\Psi \in \tilde{Y}_{\ep,m}$ such that $\psi(x,y) = \Psi(\theta_\ep, \gamma_\ep)$.  We also consider the eigenvalue problem \eqref{eigenvalue problem-ep-new-m} with $\Psi \in \tilde{Y}_{\ep,e}$ by separating it into the Fourier modes.

For the $0$-mode, the eigenvalue problem \eqref{eigenvalue problem-ep-new-m} is reduced to \eqref{eigenvalue problem for 0 mode-m}. Since $\varphi_{n,0}$ in \eqref{0 mode-eigenvalue-eigenfunction-m} is even if and only if $n\geq1$ is even, all the eigenvalues and corresponding eigenfunctions   are given in \eqref{0 mode-eigenvalue-eigenfunction-m} for even integers $n\geq1$. Thus, the principal eigenvalue  for the $0$-mode is $3$. This implies that
 there is no contribution to the negative directions and kernel of $ \tilde A_{\ep,e} $ from the $0$-mode.

For the $k$-mode with $k\neq0$, the eigenvalue problem \eqref{eigenvalue problem-ep-new-m} is reduced to \eqref{eigenvalue problem2 non-zero modes varepsilon=0m}. Since $\varphi_{n,{k\over m}}(\gamma_\ep)$ in Lemma \ref{sol to eigenvalue problem non-zero modes varepsilon=0} is even if and only if $n\geq0$ is even, we know that
all the eigenvalues and corresponding  eigenfunctions are given in Lemma  \ref{sol to eigenvalue problem non-zero modes varepsilon=0}  for even integers $n\geq0$. Thus, the principal eigenvalue  for the $k$-mode is ${1\over2}{k\over m}\left({k\over m}+1 \right)$ with an eigenfunction $(1-\gamma_\ep^2)^{k\over 2m}$.
 For the $k$-mode  with $1\leq k\leq m-1$, the principal eigenvalue satisfies ${1\over2}{k\over m}\left({k\over m}+1 \right)<1$, which gives $2m-2$
 negative directions of $ \tilde A_{\ep,e} $
\begin{align*} (1-\gamma_\ep^2)^{k\over 2m}\cos\left({k\theta_\ep\over m}\right),(1-\gamma_\ep^2)^{k\over 2m}\sin\left({k\theta_\ep\over m}\right), 1\leq k\leq m-1.
 \end{align*}
 For the $m$-mode, the principal eigenvalue is $1$, which implies that
 \begin{align*}(1-\gamma_\ep^2)^{1\over 2}\cos\left({\theta_\ep}\right),(1-\gamma_\ep^2)^{1\over 2}\sin\left({\theta_\ep}\right) \in \ker (\tilde A_{\ep,e}).\end{align*}
 For the $k$-mode  with $k\geq m+1$, the principal eigenvalue satisfies
 \begin{align}\label{k-m+1-pr}
 {1\over2}{k\over m}\left({k\over m}+1 \right)\geq{1\over2}\left({1\over m}+1\right)\left({1\over m}+2 \right)>1.
 \end{align}
  For the $k$-mode  with $k\geq 1$, the second eigenvalue satisfies
    \begin{align}\label{k-1-se}
 {1\over2}\left({k\over m}+2\right)\left({k\over m} +3\right)>3.
 \end{align}
Then $\tilde X_{\ep,e-}$ and $\ker (\tilde A_{\ep,e})$ have no more linearly independent functions, and thus, are given in (1)-(2).

Note that the principal  eigenvalue  for the $0$-mode is $3$. By \eqref{k-m+1-pr}-\eqref{k-1-se},  the minimal  eigenvalue, which is larger than $1$,  for the non-zero modes is ${1\over2}\left({1\over m}+1\right)\left({1\over m}+2 \right)$. By the variational problem \eqref{variational problem2-ep-m}-\eqref{variational problem2-ep-m-new variables} we also have
\begin{align*}
\iint_{\Omega_m}|\nabla\psi|^2dxdy\geq{1\over2}\left(1+{1\over m}\right)\left(2+{1\over m}\right)\iint_{\Omega_m}g'(\psi_\ep)(\psi-P_{\ep,m}\psi)^2dxdy,\;\psi\in\tilde X_{\ep,e+},
\end{align*}
where $\tilde X_{\ep,e+}=X_{\ep,e} \ominus\left(\ker (L_{\ep,e})\oplus X_{\ep,e-}\right)$.
Thus,
\begin{align*}
 \langle\tilde A_{\ep,e}\psi,\psi\rangle
 \geq& \left(1-{2m^2\over (m+1)(2m+1)}\right) \| \psi\|_{\tilde X_{\ep,e}}^2,\quad \psi\in\tilde X_{\ep,e+}.
\end{align*}

The rest of the proof follows from \eqref{n-ker-e} and a similar argument to \eqref{positive-decom}.
\end{proof}

By Corollaries \ref{A-L-dec-o}-\ref{A-L-dec-e}, the assumptions \textbf{(G2-4)} in Lemma \ref{indice-theorem-sep} are verified for the Hamiltonian system
\eqref{sep-hamiltonian}.

\subsection{A linear instability criterion}
Applying Lemma \ref{indice-theorem-sep} to the Hamiltonian system \eqref{sep-hamiltonian},  the criterion for  linear instability of the cat's-eye flows is  that   $n^-\left( L_{\ep,e} |_{\overline{R(B_\ep L_{\ep,o})}} \right) \geq 1$. First, we study the relation between $\overline{R(B_\ep L_{\ep,o})}$ and $\overline{R(B_\ep)}$.
\begin{lemma}\label{range BLo-B}
$\overline{R(B_\ep L_{\ep,o})}=\overline{R(B_\ep)}$.
\end{lemma}
\begin{proof}
Recall that $L_{\ep,o}:X_{\ep,o}\to X_{\ep,o}^*$ is a self-dual operator, and $B_\ep: X_{\ep,o}^*\supset D(B_\ep)\to X_{\ep,e}$.
For a Hilbert space $X$, we denote  $S_X: X^*\to X$ to be the isomorphism defined by the Riesz representation theorem.
Let $\tilde L_{\ep,o}\triangleq S_{X_{\ep, o}}L_{\ep,o}:X_{\ep, o} \rightarrow X_{\ep, o}$ and $\tilde B_\ep\triangleq B_\ep S_{X_{\ep, o}}^{-1}:X_{\ep, o} \supset D(\tilde B_\ep) \rightarrow X_{\ep, e}$. Then $\tilde L_{\ep,o}$ is a self-adjoint operator.  Noting that $\overline{R(B_\ep L_{\ep,o})}=\overline{R(\tilde B_\ep \tilde L_{\ep,o})}$ and $\overline{R(B_\ep)}=\overline{R(\tilde B_\ep)}$, we will prove that $\overline{R(\tilde B_\ep \tilde L_{\ep,o})}=\overline{R(\tilde B_\ep)}$.
It is equivalent to show that $\ker(\tilde L_{\ep,o}\tilde B_\ep^*)=\ker(\tilde B_{\ep}^*)$, where $\tilde B_{\ep}^*$ is the adjoint operator of $\tilde B_{\ep}$.

It is clear that $\ker(\tilde B_{\ep}^*)\subset\ker(\tilde L_{\ep,o}\tilde B_\ep^*)$. If $\omega\in \ker(\tilde L_{\ep,o}\tilde B_\ep^*)$, then
$\tilde L_{\ep,o}\tilde B_\ep^*\omega=0$. By Corollary \ref{A-L-dec-o}, we have $\ker(\tilde L_{\ep,o})=\ker( L_{\ep,o})=\text{span}\{g'(\psi_\ep)\gamma_\ep\}$.
Thus, $\tilde B_\ep^*\omega=Cg'(\psi_\ep)\gamma_\ep$ for some $C\in\mathbb{R}$. If $C=0$, then $\omega\in\ker(\tilde B_\ep^*)$. If $C\neq0$, we will get a contradiction. In fact, since $\overline{R(\tilde B_\ep^*)}=\ker(\tilde B_\ep^{**})^{\perp}$ and $\ker(\tilde B_\ep)\subset\ker(\tilde B_\ep^{**})$, we have
\begin{align}\label{B-ep-omega-varpi-inner product}
(\tilde B_\ep^*\omega,\varpi)_{X_{\ep, o}}=0
\end{align}
 for any $\varpi\in\ker(\tilde B_\ep)$, where ``$\perp$" is under the inner product of $X_{\ep, o}$.
We denote
\begin{align}\label{def-rho0}
\rho_0 = \psi_\ep(0,0) = \ln\left(\sqrt{\frac{1+\ep}{1-\ep}}\right).
\end{align}
Let $f\in C_c^\infty(\rho_0,\infty)$, $f\geq0$ and $f\not\equiv0$.
We construct
\begin{align*}
\varpi_\ep(x,y)=  &\left\{ \begin{array}{ll}
         f(\psi_\ep(x,y)) & \mbox{for $\psi_\ep(x,y)>\rho_0$ and $y>0$},\\
         0 & \mbox{for $-\rho_0\leq\psi_\ep(x,y)\leq\rho_0$},\\
         -f(\psi_\ep(x,y)) & \mbox{for $\psi_\ep(x,y)>\rho_0$ and $y<0$}.
        \end{array} \right.
\end{align*}
Then $\varpi_\ep$ is odd in $y$ and $\varpi_\ep\in\ker(\tilde B_\ep)$.
 By
\eqref{three-kers2}, we have
\begin{align*}
\gamma_\ep= \frac{\sqrt{1 - \epsilon^2}\sinh(y)}{\cosh(y)+\epsilon \cos(x)}
&\left\{ \begin{array}{ll}
       >0 & \mbox{for  $y>0$},\\
         <0 & \mbox{for  $y<0$}.
        \end{array} \right.
\end{align*}
Then
\begin{align*}
(\tilde B_\ep^*\omega,\varpi_\ep)_{X_{\ep, o}}=(Cg'(\psi_\ep)\gamma_\ep,\varpi_\ep)_{X_{\ep, o}}\neq0.
\end{align*}
 This  contradicts \eqref{B-ep-omega-varpi-inner product}. Thus, $\omega\in \ker(\tilde B_\ep^*)$ and $\ker(\tilde L_{\ep,o}\tilde B_\ep^*)=\ker(\tilde B_{\ep}^*)$.
\end{proof}
\begin{remark}
 In the above proof, the key point is to show that  $\tilde B_\ep^*\omega=g'(\psi_\ep)\gamma_\ep$ has no solutions in $X_{\ep,e}$.
 We now give an intuitive explanation.
Indeed,  by
\eqref{three-kers2}, we have $\tilde B_\ep^*\omega=g'(\psi_\ep)\gamma_\ep= g'(\psi_\ep) \sqrt{1 - \epsilon^2}{ \partial_y \psi_\ep}$.
Formally, we have $(\vec{u}_\ep\cdot\nabla)\left({\omega\over g'(\psi_\ep) \sqrt{1 - \epsilon^2}}\right)={ \partial_y \psi_\ep}$ and thus, ${\omega\over g'(\psi_\ep) \sqrt{1 - \epsilon^2}}=x$, which is, however,  not $2\pi$-periodic in $x$.
\end{remark}
By Lemma \ref{range BLo-B}, linear instability reduces to the condition
\[
n^-\!\left(L_{\ep,e}\big|_{\overline{R(B_\ep)}}\right)\ge 1.
\]
To study this quantity, we introduce the orthogonal projection \(\bar P_{\ep,e}\) from
\[
L^2_{\frac{1}{g'(\psi_\ep)},e}(\Omega_m)
=
\left\{
\omega\in L^2_{\frac{1}{g'(\psi_\ep)}}(\Omega_m)\;\middle|\;
\omega \text{ is even in } y
\right\}
\]
onto
\[
W_{\ep,e}
=
\left\{
\omega\in L^2_{\frac{1}{g'(\psi_\ep)},e}(\Omega_m)\;\middle|\;
(\omega,\varpi)_{L^2_{\frac{1}{g'(\psi_\ep)},e}}=0
\text{ for all }\varpi\in\overline{R(B_\ep)}
\right\}.
\]
Here \(\overline{R(B_\ep)}\subset X_{\ep,e}\), while \(\ker(B_\ep^*)\subsetneq W_{\ep,e}\).

This induces a projection \(\hat P_{\ep,e}\) from
\[
L^2_{g'(\psi_\ep),e}(\Omega_m)
=
\left\{
\psi\in L^2_{g'(\psi_\ep)}(\Omega_m)\;\middle|\;
\psi \text{ is even in } y
\right\}
\]
onto
\[\hat W_{\ep,e}=\left\{\psi\;\middle|\;\psi={\omega\over g'(\psi_\ep)}, \omega\in W_{\ep,e}\right\}\]
by
\[
\hat P_{\ep,e}
=
S_{L^2_{g'(\psi_\ep),e}(\Omega_m)}
\bar P_{\ep,e}
S_{L^2_{\frac{1}{g'(\psi_\ep)},e}(\Omega_m)}.
\]
As in \cite{lin2004some}, this projection has the form
\begin{align}\label{def-hat-P-ep-e}
(\hat P_{\ep,e}\psi)\big|_{\Gamma_i(\rho)}
=
\frac{\displaystyle \oint_{\Gamma_i(\rho)} \frac{\psi}{|\nabla\psi_\ep|}}
{\displaystyle \oint_{\Gamma_i(\rho)} \frac{1}{|\nabla\psi_\ep|}}
\end{align}
for \(\psi\in L^2_{g'(\psi_\ep),e}(\Omega_m)\), where \(\rho\) lies in the range of \(\psi_\ep\) and \(\Gamma_i(\rho)\) is a branch of the level set \(\{\psi_\ep=\rho\}\). Since \(\tilde X_{\ep,e}\subset L^2_{g'(\psi_\ep),e}(\Omega_m)\), we define
\[
\hat A_{\ep,e}
=
-\Delta-g'(\psi_\ep)(I-\hat P_{\ep,e})
:
\tilde X_{\ep,e}\to \tilde X_{\ep,e}^*.
\]

Then we have the following lemma.
\begin{lemma}\label{L e-hat A} The number of unstable modes of \eqref{sep-hamiltonian} is
$$n^-\left( L_{\ep,e} |_{\overline{R(B_\ep)}} \right) = n^-\left(\hat{A}_{\ep,e}\right).$$
Consequently, if $n^-\left(\hat{A}_{\ep,e}\right)>0$, then $\omega_\ep$ is linearly unstable for $2m\pi$-periodic perturbations.
\end{lemma}
\begin{proof} Since $\hat{P}_{\ep,e}$  commutes with $f(\psi_\ep)\cdot$ for any function $f$,  $\omega \in \overline{R(B_\ep)}$ if and only if $\hat{P}_{\ep,e} \frac{\omega}{g'(\psi_\ep)} = 0$. Note that $\bar{P}_{\ep,e}$ is orthogonal under the inner product of $  L^2_{\frac{1}{g'(\psi_\ep)},e}(\Omega_m)$.
For $\omega \in \overline{R(B_\ep)}\subset X_{\ep,e}$, there exists $\psi \in \tilde{X}_{\ep, e}$ such that $-\Delta\psi=\omega$ and
\begin{align*}
&\langle L_{\ep,e} \omega, \omega \rangle
= \iint_{\Omega_m} \left(\frac{\omega^2}{g'(\psi_\ep)} - \omega \psi\right) dxdy \\
= &\iint_{\Omega_m}\left({1\over\sqrt{g'(\psi_\ep)}}\bar{P}_{\ep,e}\left( {\omega} - g'(\psi_\ep)\psi \right)+{1\over\sqrt{g'(\psi_\ep)}}(I-\bar{P}_{\ep,e})\left( {\omega} - g'(\psi_\ep)\psi \right) \right)^2 dxdy \\
&- \iint_{\Omega_m}\left(g'(\psi_\ep)\psi^2 -|\nabla \psi|^2\right) dxdy \\
=& \iint_{\Omega_m} \left(\left( \frac{\omega}{\sqrt{g'(\psi_\ep)}} - \sqrt{g'(\psi_\ep)} (I - \hat{P}_{\ep,e}) \psi \right)^2 + g'(\psi_\ep) (\hat{P}_{\ep,e} \psi)^2  - g'(\psi_\ep)\psi^2 + |\nabla \psi|^2 \right)dxdy \\
 \geq& \iint_{\Omega_m} \left(|\nabla \psi|^2 - g'(\psi_\ep)\psi^2 + g'(\psi_\ep) (\hat{P}_{\ep,e}\psi)^2 \right)dxdy = \langle\hat{A}_{\ep,e} \psi, \psi\rangle.
\end{align*}

For $\psi\in\tilde{X}_{\ep, e}$, we have $\tilde{\omega} \triangleq g'(\psi_\ep)(I - \hat{P}_{\ep,e})\psi \in \overline{R(B_\ep)}$. Let $\tilde{\psi} = (-\Delta)^{-1}\tilde{\omega}$. Then
\begin{align}\nonumber
\langle\hat{A}_{\ep,e} \psi, \psi\rangle
&= \iint_{\Omega_m} \left( |\nabla \psi|^2 - g'(\psi_\ep)((I - \hat{P}_{\ep,e})\psi)^2\right) dxdy \\\nonumber
&= \iint_{\Omega_m} \left( |\nabla \psi|^2 - \frac{\tilde{\omega}^2}{g'(\psi_\ep)}\right)dxdy \\\nonumber
&= \iint_{\Omega_m} \left(|\nabla \psi|^2 - 2 \tilde{\omega} \psi + \frac{\tilde{\omega}^2}{g'(\psi_\ep)}\right)dxdy \\\nonumber
& \geq \iint_{\Omega_m}  \left(\frac{\tilde{\omega}^2}{g'(\psi_\ep)} - |\nabla \tilde{\psi}|^2\right) dxdy = \langle L_{\ep,e}\tilde{\omega}, \tilde{\omega} \rangle,
\end{align}
where we used $\langle\tilde \omega, \hat{P}_{\ep,e}\psi\rangle=0$.
From the two inequalities above, we have $n^{\leq0}\left( L_{\ep,e} |_{\overline{R(B_\ep)}} \right) = n^{\leq0}\left(\hat{A}_{\ep,e}\right)$.
Similar to (11.60) in \cite{lin2022instability}, we have $\dim \ker \left(L_{\ep,e} |_{\overline{R(B_\ep)}} \right)=\dim \ker (\hat{A}_{\ep,e})$. Thus,
 $n^{-}\left( L_{\ep,e} |_{\overline{R(B_\ep)}} \right) = n^{-}\left(\hat{A}_{\ep,e}\right)$.
\end{proof}

 To study the linear instability of the Kelvin--Stuart vortex $\omega_\ep$ for multi-periodic perturbations, we will construct a specific test function $\psi \in \tilde{X}_{\ep, e}$ such that
 \begin{align*}
\langle\hat{A}_{\ep,e} \psi, \psi \rangle=  b_{\ep, 1}(\psi) + b_{\ep, 2}(\psi) < 0,
\end{align*}
where
\begin{align*}
  b_{\ep, 1}(\psi) = \iint_{\Omega_m} \left(|\nabla \psi|^2  - g'(\psi_\ep) \psi^2 \right)dxdy
\end{align*}
and
\begin{align*}
 b_{\ep, 2}(\psi) =  \iint_{\Omega_m} g'(\psi_\ep)( \hat{P}_{\ep,e}\psi)^2 dxdy =  \int_{\min \psi_\ep}^{\infty} g'(\rho) \sum_{i=1}^{n_\rho} \frac{\left|\oint_{\Gamma_i (\rho)} \frac{\psi}{|\nabla \psi_\ep|}\right|^2}{\oint_{\Gamma_i (\rho)} \frac{1}{|\nabla \psi_\ep|}} d\rho.
 \end{align*}
Here,  $\{ \Gamma_i (\rho), i = 1, \cdots, n_\rho\}$ is the set of all the disjoint closed level curves in the level set $\{(x, y)\in \Omega_m |\psi_\ep(x,y) = \rho\}$, where $\rho\in[\min\psi_\ep,\infty)$.
Then  by Lemma \ref{L e-hat A} we have $n^-\left( L_{\ep,e} |_{\overline{R(B_\ep)}} \right) \geq 1$, and the linear instability follows from Lemma \ref{indice-theorem-sep}.

\subsection{Proof of multi-periodic instability   (even multiple case)}
In this subsection, we prove the linear instability of the  Kelvin--Stuart vortex $\omega_\ep$ for $4k\pi$-periodic perturbations. We take the test function
\begin{align}\label{test-even}
\tilde{\psi}_\ep(x, y)  =\tilde{\Psi}_\ep(\theta_\ep, \gamma_\ep)= \cos\left(\frac{\theta_\ep}{2}\right)(1-\gamma_\ep^2)^{1\over4}
\end{align}
with $(\theta_\ep, \gep) \in \tilde{\Omega}_{2k} = \mathbb{T}_{4k\pi} \times [-1, 1]$. Then $\tilde{\Psi}_\ep\in\tilde Y_{\ep,e}\Longrightarrow\tilde{\psi}_\ep\in\tilde X_{\ep,e}$. By Theorem \ref{sol to eigenvalue problem varepsilon=0-pde}, $\tilde{\psi}_\ep(x, y)$ is exactly an eigenfunction of the principal  eigenvalue  $\lambda={3\over 8}$ for
\eqref{eigenvalue problem-ep-original-m}, and thus,
$$-(\Delta +g'(\psi_\ep)) \tilde{\psi}_\ep = - \frac 5 8 g'(\psi_\ep)\tilde{\psi}_\ep.$$
Then
\begin{align}\nonumber
b_{\ep, 1}(\tilde{\psi}_\ep)
= & \int_{-\infty}^{+\infty} \int_0^{4k\pi} \left(|\nabla \tilde{\psi}_\ep|^2 - g'(\psi_\epsilon)\tilde{\psi}_\ep^2 \right)dx dy
=  -{5\over8}\int_{-\infty}^{+\infty} \int_0^{4k\pi} g'(\psi_\ep)\tilde{\psi}_\ep^2dx dy \\\label{b1-even}
= & -{5\over4} \int_0^{4k\pi}  \cos^2\left(\frac{\theta_\ep}{2}\right)d \theta_\ep\int_{-1}^{1} (1-\gamma_\ep^2)^{1\over2} d\gep
=  -{5\over4}k\pi^2.
\end{align}
$b_{\ep, 2}(\tilde{\psi}_\ep)$ vanishes by symmetry as seen in the next lemma.
\begin{lemma}\label{b2-even}
$$ b_{\ep, 2}(\tilde{\psi}_\ep) =  \int_{\min \psi_\ep}^{\max \psi_\ep} g'(\rho) \sum_{i=1}^{n_\rho} \frac{\left|\oint_{\Gamma_i (\rho)} \frac{\tilde{\psi}_\ep}{|\nabla \psi_\ep|}\right|^2}{\oint_{\Gamma_i (\rho)} \frac{1}{|\nabla \psi_\ep|}} d\rho = 0.$$
\end{lemma}
\begin{proof} Since $\tilde{\psi}_\ep$ is   `odd'   symmetrical about  $\{x= (2j-1)\pi\}$ along any trajectory of  the steady veloctiy, $1\leq j\leq 2k$, we have $\hat{P}_{\ep,e}\tilde{\psi}_\ep \equiv0$ on $\mathbb{T}_{4k\pi}\times\mathbb{R}$, and thus, $b_{\ep, 2}(\tilde{\psi}_\ep)=0$.
\end{proof}
Now we get linear instability of $\omega_\ep$ for perturbations with even multiples of the period.
\begin{Theorem}\label{multi-even}
Let $\ep\in[0,1)$. Then the steady state $\omega_\ep$ is linearly unstable for $4k\pi$-periodic perturbations, where $k\geq1$ is an integer.
\end{Theorem}
\begin{proof}
With the test function $\tilde{\psi}_\ep$ defined in \eqref{test-even}, by \eqref{b1-even} and Lemma \ref{b2-even}, we have
$$\langle\hat{A}_{\ep,e} \tilde{\psi}_\ep, \tilde{\psi}_\ep \rangle = -{5\over4}k\pi^2 < 0.$$
Then we have
$n^-\left( L_{\ep,e} |_{\overline{R(B_\ep)}} \right) = n^-\left(\hat{A}_{\ep,e}\right)\geq1$ by Lemma \ref{L e-hat A}.  The conclusion follows from  Lemma  \ref{indice-theorem-sep}.
\end{proof}

\subsection{Proof of multi-periodic instability  (odd multiple case)}
In this subsection, we study linear instability of the steady state $\omega_\ep$ for   $(4k+2)\pi$-periodic  perturbations, where $k\geq1$ is an integer.
We divide our discussion into two cases in terms of the $\ep$ values.
\medskip

\noindent{\bf{Case 1. Test functions for $\ep\in[0,{4\over5}]$.}}
\medskip

In this case,
 we take the test function to be
 \begin{align}\nonumber
\hat{\psi}_{1,\ep}(x,y)= &\hat\Psi_{1,\ep}(\theta_\ep,\gamma_\ep)\\\label{test-odd}
=  &\left\{ \begin{array}{ll}
         \sin\left(\frac{\theta_\ep}{3}\right)(1-\gamma_\ep^2)^{1\over6} & \mbox{if $(\theta_\ep,\gamma_\ep) \in [0, 6\pi]\times[-1,1]$},\\
         \sin\left(\theta_\ep\right)(1-\gamma_\ep^2)^{1\over2} & \mbox{if $(\theta_\ep,\gamma_\ep) \in (6\pi, (4k + 2)\pi]\times[-1,1]$}.\\
        \end{array} \right.
 \end{align}
 To show that $ \hat{\psi}_{1,\ep}\in \tilde X_{\ep,e}$, it suffices  to prove that $\hat\Psi_{1,\ep}\in\tilde Y_{\ep,e}$, where $\tilde Y_{\ep,e}$ is defined in  \eqref{def-Y-ep-m}. Note that $\hat\Psi_{1,\ep}\in C^0(\tilde \Omega_{\ep,2k+1})$. By Theorem \ref{sol to eigenvalue problem varepsilon=0-pde}, $\sin\left(\frac{\theta_\ep}{3}\right)(1-\gamma_\ep^2)^{1\over6}$ is an eigenfunction of the principal eigenvalue $\lambda={2\over 9}$ for  \eqref{eigenvalue problem-ep-new-m} with $m=3$. By Theorems \ref{associate_ep0} and \ref{associate_ep-new-variable-original-variable},
$\sin\left(\theta_\ep\right)(1-\gamma_\ep^2)^{1\over2}$ is an eigenfunction of the principal eigenvalue $\lambda=1$ for   \eqref{eigenvalue problem-ep-new}. Thus,
 \begin{align*}
 &\|\hat\Psi_{1,\ep}\|_{\tilde Y_{\ep,e}}^2=\left(\int_{-1}^{1} \int_0^{6\pi}
+ \int_{-1}^{1} \int_{6\pi}^{(4k+2)\pi}\right)\left({1\over1-\gamma_\ep^2}|\partial_{\theta_\ep}\hat\Psi_{1,\ep}|^2+(1-\gamma_\ep^2)|\partial_{\gamma_\ep}\hat\Psi_{1,\ep}|^2\right)d \theta_\ep d\gamma_\ep\\
=&{4\over 9}\int_{-1}^{1} \int_0^{6\pi}\sin^2\left({1\over 3}\theta_{\ep}\right)(1-\gamma_{\ep}^2)^{1\over 3}d\theta_\ep d\gamma_\ep+
2(k-1)\times 2 \int_{-1}^{1} \int_{0}^{2\pi}\sin^2(\theta_\ep)(1-\gamma_{\ep}^2)d\theta_\ep d\gamma_\ep\\
\leq &
{8\over3}\pi+{16\over3}(k-1)\pi<\infty,
 \end{align*}
 and moreover,
 \begin{align*}
 \int_{0}^{(4k+2)\pi}\hat\Psi_{1,\ep}(\theta_\ep,0)d\theta_\ep=
 \int_0^{6\pi}\sin\left({1\over 3}\theta_{\ep}\right)d\theta_\ep +
 \int_{6\pi}^{(4k+2)\pi}\sin(\theta_\ep)d\theta_\ep=0.
 \end{align*}
Again by Theorems \ref{associate_ep0}, \ref{associate_ep-new-variable-original-variable} and \ref{sol to eigenvalue problem varepsilon=0-pde},
\begin{align}\nonumber
b_{\ep, 1}(\hat{\psi}_{1,\ep})
& = \left(\int_{-\infty}^{+\infty} \int_0^{6\pi}
+ \int_{-\infty}^{+\infty} \int_{6\pi}^{(4k+2)\pi}\right) \left(|\nabla \hat{\psi}_{1,\ep}|^2 - g'(\psi_\epsilon)\hat{\psi}_{1,\ep}^2\right) dx dy\\\nonumber
&=\int_{-\infty}^{+\infty} \int_0^{6\pi}
 \left(|\nabla \hat{\psi}_{1,\ep}|^2 - g'(\psi_\epsilon)\hat{\psi}_{1,\ep}^2\right) dx dy\\\nonumber
 &=-{7\over 9}\int_{-\infty}^{+\infty} \int_0^{6\pi}
  g'(\psi_\epsilon)\hat{\psi}_{1,\ep}^2 dx dy\\\nonumber
 &= -{14\over 9}\int_{0}^{6\pi}\sin^2\left({1\over 3}\theta_{\ep}\right)d\theta_\ep\int_{-1}^1(1-\gamma_\ep^2)^{1\over3}d\gamma_\ep\\\label{b-ep-1-hat-psi}
 &\leq -{14\over 9}\times 3\pi\times {42\over 25}=-{196\pi\over 25}\leq -24.61,
\end{align}
where we used the fact that $\int_{-1}^1(1-\gamma_\ep^2)^{1\over3}d\gamma_\ep\geq {42\over 25}$.
By \eqref{steadyv}, $(2j\pi,0)$ and $((2j+1)\pi,0)$ are critical points of $\psi_\ep$  on $\mathbb{T}_{(4k+2)\pi}\times \mathbb{R}$, where $j=0,\cdots,2k$. The  Hessian matrix of $\psi_\ep$ is
\begin{align*}
\left( \begin{array}{cc} {-\ep^2-\ep\cos(x)\cosh(y)\over (\cosh(y)+\ep\cos(x))^2} & {\ep\sin(x)\sinh(y)\over (\cosh(y)+\ep\cos(x))^2} \\ {\ep\sin(x)\sinh(y)\over (\cosh(y)+\ep\cos(x))^2} & {1+\ep\cosh(y)\cos(x)\over (\cosh(y)+\ep\cos(x))^2} \end{array} \right).
\end{align*}
Then $(2j\pi,0)$ is a saddle point of $\psi_\ep$, and $((2j+1)\pi,0)$ is the minimal  point of $\psi_\ep$, since $\psi_\ep(x,y)\to\infty$ as $y\to \pm\infty$ for $x\in\mathbb{T}_{2\pi}$ and $j=0,\cdots,2k$.
Recall that $\rho_0$ is defined in
\eqref{def-rho0}.
Then $\min \psi_\ep = \psi_\ep((2j+1)\pi, 0) = - \rho_0$.
For $\rho\in[-\rho_0, \rho_0]$, the streamlines are in the trapped regions and  the level set $\Gamma(\rho) = \{(x, y) \in \Omega_{2k+1} | \psi_\ep(x,y) = \rho\} $ has $n_\rho = 2k+1$ closed level curves, i.e.
\begin{align}\label{Gamma-rho}
\Gamma(\rho) =  \bigcup_{i=1}^{n_\rho} \Gamma_i(\rho),
\end{align}
 where $\Gamma_i(\rho)$ corresponds to a periodic orbit inside the $i$-th cat's-eye trapped region.
Since
$\sin\left({1\over 3}\theta_\ep\right)$ is  `odd'   symmetrical about the point $( 3\pi,0)$ and $\sin\left(\theta_\ep\right)$ is  `odd'   symmetrical about the points $( 6\pi+(2j-1)\pi,0)$ for $j=1,\cdots,2k-2$,  we have
$(\hat{P}_{\ep,e} \hat{\psi}_{1,\ep})(x,y)=0$ for $(x,y)$ in the untrapped regions of $\mathbb{T}_{(4k+2)\pi}\times \mathbb{R}$ and the $2$nd, $j$-th trapped regions for $4\leq j\leq 2k+1$, where $k\geq2$.
Now, we compute the projection term for $(x,y)$ in  the $1$st and $3$rd  trapped regions, denoted by $D_{\rm{in},1}$ and $D_{\rm{in},3}$.
Using $x$ as the parameter in the $1$st trapped region,
we  represent the upper  separatrix to be $y(x)=\cosh^{-1}(1+\ep-\ep\cos(x)), x\in[0,2\pi]$ and  the lower  separatrix to be $y(x)=-\cosh^{-1}(1+\ep-\ep\cos(x)), x\in[0,2\pi]$.
Then
\begin{align}\nonumber
b_{\ep, 2}(\hat{\psi}_{1,\ep})
& = \iint_{D_{\rm{in},1}} g'(\psi_\ep) |\hat{P}_{\ep,e} \hat{\psi}_{1,\ep}|^2 dxdy + \iint_{D_{\rm{in},3}} g'(\psi_\ep) |\hat{P}_{\ep,e} \hat{\psi}_{1,\ep}|^2 dxdy\\\nonumber
&=2\iint_{D_{\rm{in},1}} g'(\psi_\ep) |\hat{P}_{\ep,e} \hat{\psi}_{1,\ep}|^2 dxdy
=2\int_{-\rho_0}^{\rho_0} g'(\rho)  \frac{\left|\oint_{\Gamma_1 (\rho)} \frac{\hat{\psi}_{1,\ep}}{|\nabla \psi_\ep|}\right|^2}{\oint_{\Gamma_1 (\rho)} \frac{1}{|\nabla\psi_\ep|}} d\rho\\\nonumber
&\leq2\int_{-\rho_0}^{\rho_0} g'(\rho) \oint_{\Gamma_1 (\rho)} \frac{|\hat{\psi}_{1,\ep}|^2}{|\nabla \psi_\ep|}d\rho=2 \iint_{D_{\rm{in},1}} g'(\psi_\ep) | \hat{\psi}_{1,\ep}|^2 dxdy\\\nonumber
&=2\int_0^{2\pi}\int_{-\cosh^{-1}(1+\ep-\ep\cos(x))}^{\cosh^{-1}(1+\ep-\ep\cos(x))}g'(\psi_\ep)
\sin^2\left(\frac{\theta_\ep}{3}\right)(1-\gamma_\ep^2)^{1\over3}dydx\\\label{b-ep-2-hat-psi-1-ep}
&\triangleq b_{\ep,3}(\hat{\psi}_{1,\ep}).
\end{align}
To study the monotonicity of $ b_{\ep,3}(\hat{\psi}_{1,\ep})$ with respect to $\ep\in[0,1)$, we need the following lemma.
\begin{lemma}\label{D-theta-gamma-ep-nest}
Let
\begin{align*}
D_{xy,\ep}=&D_{\rm{in},1}=\{(x,y)|{-\cosh^{-1}(1+\ep-\ep\cos(x))}\leq y\leq {\cosh^{-1}(1+\ep-\ep\cos(x))}, x\in\mathbb{T}_{2\pi}\}\\
D_{\theta_\ep\gamma_\ep,\ep}=&\{(\theta_\ep,\gamma_{\ep})|\theta_\ep=\theta_\ep(x,y),\gamma_{\ep}=\gamma_{\ep}(x,y),(x,y)\in D_{xy,\ep}\}
\end{align*}
for $\ep\in[0,1)$. Then as subsets of $\mathbb{T}_{2\pi}\times[-1,1]$, we have
 \begin{align}\label{D1D2-constant}
 D_{\theta_{\ep_1}\gamma_{\ep_1},\ep_1}\subset D_{\theta_{\ep_2}\gamma_{\ep_2},\ep_2} \quad\text{ for }\quad0\leq \ep_1\leq \ep_2<1.
 \end{align}
\end{lemma}

\begin{proof} It suffices to consider the case $y\geq0\Longleftrightarrow\gamma_\ep\geq0$, since $D_{xy,\ep}$ (resp. $D_{\theta_\ep\gamma_\ep,\ep}$) is symmetric with respect to the line $y=0$ (resp. $\gamma_\ep=0$). Instead of using  $(\theta_\ep,\gamma_{\ep})$ directly, we choose the equivalent variables $(\xi_\ep,\eta_\ep)$ and define
\begin{align*}
D_{\xi_\ep\eta_\ep,\ep}=&\{(\xi_\ep,\eta_\ep)|
\eta_\ep = \sqrt{1-\gamma_\ep^2} \sin(\theta_\ep),
\xi_\ep = \sqrt{1-\gamma_\ep^2} \cos(\theta_\ep),
(\theta_\ep,\gamma_{\ep})\in D_{\theta_\ep\gamma_\ep,\ep}\}.
\end{align*}
To prove \eqref{D1D2-constant}, it is sufficient to show that as subsets of the closed unit disk $D_1=\{(\xi_{\ep},\eta_{\ep})|\xi_{\ep}^2+\eta_{\ep}^2\leq1\}$,
\begin{align}\label{D1D2-constant-eta-xi}
D_{\xi_{\ep_1}\eta_{\ep_1},\ep_1}\subset D_{\xi_{\ep_2}\eta_{\ep_2},\ep_2}\quad\text{ for }\quad0\leq \ep_1\leq \ep_2<1.
\end{align}
In the original variables, $D_{xy,\ep}$ consists of the level curves $\{\psi_\ep=\rho\}$ for $\rho\in\bigg[\ln\left(\sqrt{1-\ep\over 1+\ep}\right),$ $\ln\left(\sqrt{1+\ep\over 1-\ep}\right)\bigg]$. In the variables $(\xi_\ep,\eta_\ep)$, we study the level curves of $\omega_\ep$ for convenience. By the expression  \eqref{omega-xi-eta-gamma-ep} of $\omega_\ep$ in $(\xi_\ep,\eta_\ep)$, $D_{\xi_\ep\eta_\ep,\ep}$ consists of the level curves
\begin{align}\label{elliptical}
\left\{(\xi_\ep,\eta_\ep)\bigg|\frac{(\xi_\ep - \ep)^2}{1-\ep^2} + \eta_\ep^2=-c\right\}\bigcap D_1
\end{align}
for $c\in\left[c_\ep,1/c_\ep\right]$, where $c_\ep=-{1+\ep\over 1-\ep}$.
This is a family of ellipses, with the parameters $c$ ranging from $c_\ep$ to $1/c_\ep$, intersecting with the closed unit disk $D_1$. For fixed $c\in\left[c_\ep,1/c_\ep\right]$, the center, semi-major and semi-minor axes of the ellipse are $(\ep,0)$,  $\sqrt{-c}$ and $\sqrt{-c(1-\ep^2)}$. To study the nested relationship \eqref{D1D2-constant-eta-xi}, we use the variables $\xi,\eta\in[-1,1]$, which are independent of $\ep$. Note that as a subset of the closed unit disk $D_1$,
 the curve \eqref{elliptical} is the same  if we replace the variables $(\xi_\ep,\eta_\ep)$ by $(\xi,\eta)$. Thus, $D_{\xi_\ep\eta_\ep,\ep}$ can be written as
\begin{align*}
D_{\xi_\ep\eta_\ep,\ep}=\bigcup_{c\in\left[c_\ep,1/c_\ep\right]} \left(\Gamma_{c,\ep}\cap D_1\right)=
\left\{(\xi,\eta)\bigg|-1/c_\ep\leq\frac{(\xi - \ep)^2}{1-\ep^2} + \eta^2\leq -c_\ep\right\}\bigcap D_1,
\end{align*}
where
\begin{align*}
\Gamma_{c,\ep}=\left\{(\xi,\eta)\bigg|\frac{(\xi - \ep)^2}{1-\ep^2} + \eta^2=-c\right\}.
\end{align*}
To prove \eqref{D1D2-constant-eta-xi},
 we divide our discussions into two steps.
\vspace{0.5mm}

\noindent{\bf{Step 1.}} For $\ep\in[0,1)$, we prove that
\begin{align}\label{out-in-nest}
\Gamma_{1/c_\ep,\ep}\text{  is  enclosed by } S_1, \text{ and }{S}_1\text{  is  enclosed by } \Gamma_{c_\ep,\ep},
\end{align}
where $c_\ep=-{1+\ep\over 1-\ep}$ and $S_1=\{(\xi,\eta)|\xi^2+\eta^2=1\}$ is the unit circle.
\eqref{out-in-nest} means that $\xi^2+\eta^2\leq1$ for $(\xi,\eta)\in\Gamma_{1/c_\ep,\ep}$ and $\frac{(\xi - \ep)^2}{1-\ep^2} + \eta^2\leq-c_\ep$ for $(\xi,\eta)\in S_1$.
See Figure \ref{fig3th:bdy22} for the curves  $\Gamma_{1/c_\ep,\ep}$, $S_1$ and $\Gamma_{c_\ep,\ep}$ with $\ep=0.5$.
Moreover, $\Gamma_{1/c_\ep,\ep}\cap S_1=\{(1,0)\}$ and ${S}_1\cap\Gamma_{c_\ep,\ep}=\{(-1,0)\}$ for $\ep>0$, while $\Gamma_{1/c_\ep,\ep}= S_1=\Gamma_{c_\ep,\ep}$ for $\ep=0$.
\begin{figure}[ht]
    \centering
	\includegraphics[width=0.7\textwidth]{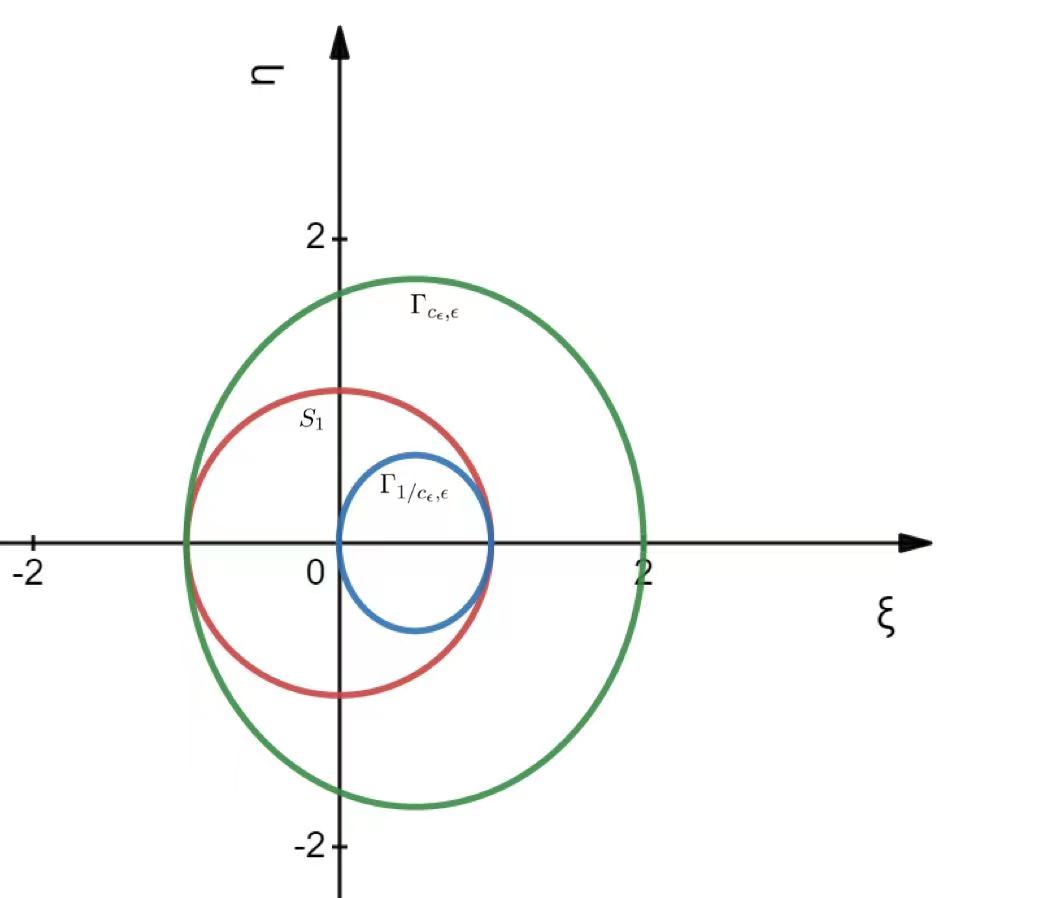}
	\caption{The curves $\Gamma_{1/c_\ep,\ep}$, $S_1$ and $\Gamma_{c_\ep,\ep}$ with $\ep=0.5$}
	\label{fig3th:bdy22}
\end{figure}


$\Gamma_{1/c_\ep,\ep}$ is given by the ellipse
\begin{align}\label{inner boundary}
\frac{(\xi - \ep)^2}{(1-\ep)^2} + {\eta^2\over {1-\ep\over 1+\ep}}=1.
\end{align}
Since the center and semi-minor axis of the ellipse \eqref{inner boundary} are $(\ep,0)$ and $1-\ep$, the right vertex of the ellipse is always $(1,0)$.
Here, we only need to consider $\eta\geq0$ since $D_{\xi_{\ep}\eta_{\ep},\ep}$ is symmetric with respect to the line $\eta=0$.
For $(\xi,\eta)\in \Gamma_{1/c_\ep,\ep}$ with $\eta\geq0$, we rewrite $\eta$ by $\eta_{1/c_\ep,\ep}(\xi)$ to indicate its dependence on  $\ep$, $c_\ep$ and $\xi$. Then $\eta_{1/c_\ep,\ep}(\xi)^2={1-\ep\over 1+\ep}-\frac{(\xi - \ep)^2}{1-\ep^2}$ for $\xi\in[2\ep-1,1]$. For  $(\xi,\eta)\in S_1$, we rewrite $\eta$ by $\eta_{S_1}(\xi)$ to indicate its dependence on $\xi$. Then $\eta_{S_1}(\xi)^2=1-\xi^2$ for $\xi\in[-1,1]$.
To prove that $\Gamma_{1/c_\ep,\ep}\text{  is  enclosed by } S_1$ and
$\Gamma_{1/c_\ep,\ep}\cap S_1=\{(1,0)\}$  for $\ep>0$, it suffices to show that $\eta_{S_1}(\xi)^2>\eta_{1/c_\ep,\ep}(\xi)^2$ for $\xi\in[\ep,1)$.
Since the right vertex of both the ellipse $\Gamma_{1/c_\ep,\ep}$ and the unit circle $S_1$ is  $(1,0)$, it suffices to verify that $\left|\partial_{\xi}\left(\eta_{S_1}(\xi)^2\right)\right|>\left|\partial_{\xi}\left(\eta_{1/c_\ep,\ep}(\xi)^2\right)\right|$ for $\xi\in[\ep,1]$. In fact, direct computation gives
\begin{align*}
&\left|\partial_{\xi}\left(\eta_{1/c_\ep,\ep}(\xi)^2\right)\right|-\left|\partial_{\xi}\left(\eta_{S_1}(\xi)^2\right)\right|=2\left(\frac{\xi - \ep}{1-\ep^2}-\xi\right)
={-2\ep(1-\ep\xi)\over 1-\ep^2}<0
\end{align*}
 for $\xi\in[\ep,1]$ and $\ep>0$.

$\Gamma_{c_\ep,\ep}$ is given by the ellipse
\begin{align}\label{boundary2}
\frac{(\xi - \ep)^2}{(1+\ep)^2} + {\eta^2\over {1+\ep\over 1-\ep}}=1.
\end{align}
Since the center and semi-minor axis of the ellipse \eqref{boundary2} are $(\ep,0)$ and $1+\ep$, the left vertex of the ellipse is always $(-1,0)$.
Here we only consider $\eta\geq0$ by  symmetry.
For $(\xi,\eta)\in \Gamma_{c_\ep,\ep}$ with $\eta\geq0$, we rewrite $\eta$ by $\eta_{c_\ep,\ep}(\xi)$. Then $\eta_{c_\ep,\ep}(\xi)^2={1+\ep\over 1-\ep}-\frac{(\xi - \ep)^2}{1-\ep^2}$ for $\xi\in[-1,1+2\ep]$. For  $(\xi,\eta)\in S_1$, $\eta_{S_1}(\xi)^2=1-\xi^2$ for $\xi\in[-1,1]$.
To prove that $S_1\text{  is  enclosed by } \Gamma_{c_\ep,\ep}$ and ${S}_1\cap\Gamma_{c_\ep,\ep}=\{(-1,0)\}$ for $\ep>0$, it suffices to show that $\eta_{c_\ep,\ep}(\xi)^2>\eta_{S_1}(\xi)^2$ for $\xi\in(-1,0]$.
Since the left vertex of both the ellipse $\Gamma_{c_\ep,\ep}$ and the unit circle $S_1$ is  $(-1,0)$, it suffices to verify that $\left|\partial_{\xi}\left(\eta_{c_\ep,\ep}(\xi)^2\right)\right|>\left|\partial_{\xi}\left(\eta_{S_1}(\xi)^2\right)\right|$ for $\xi\in[-1,0]$. Indeed,
\begin{align*}
&\left|\partial_{\xi}\left(\eta_{c_\ep,\ep}(\xi)^2\right)\right|-\left|\partial_{\xi}\left(\eta_{S_1}(\xi)^2\right)\right|=2\left(\frac{ \ep-\xi}{1-\ep^2}+\xi\right)
={2\ep(1-\ep\xi)\over 1-\ep^2}>0
\end{align*}
 for $\xi\in[-1,0]$ and $\ep>0$.

By Step 1,
\begin{align*}
D_{\xi_\ep\eta_\ep,\ep}=\left\{(\xi,\eta)\bigg|\xi^2+\eta^2\leq 1\leq\frac{(\xi - \ep)^2}{(1-\ep)^2} + {\eta^2\over {1-\ep\over 1+\ep}} \right\}.
\end{align*}
In other words, the outer boundary of $D_{\xi_\ep\eta_\ep,\ep}$ is always the unit circle $S_1$ and the inner boundary of $D_{\xi_\ep\eta_\ep,\ep}$ is the ellipse $\Gamma_{1/c_\ep,\ep}$. For $\ep=0.5$, see Figure \ref{Figure4th:orig-and-tran} for the upper trapped region $\{(x,y)\in D_{xy,\ep}|y\geq0\}$ in $(x,y)$ coordinate and  the corresponding region $D_{\xi_\ep\eta_\ep,\ep}$ in $(\xi,\eta)$ coordinate separately.

\begin{figure}[ht]
    \centering
\includegraphics[width=0.48\textwidth]{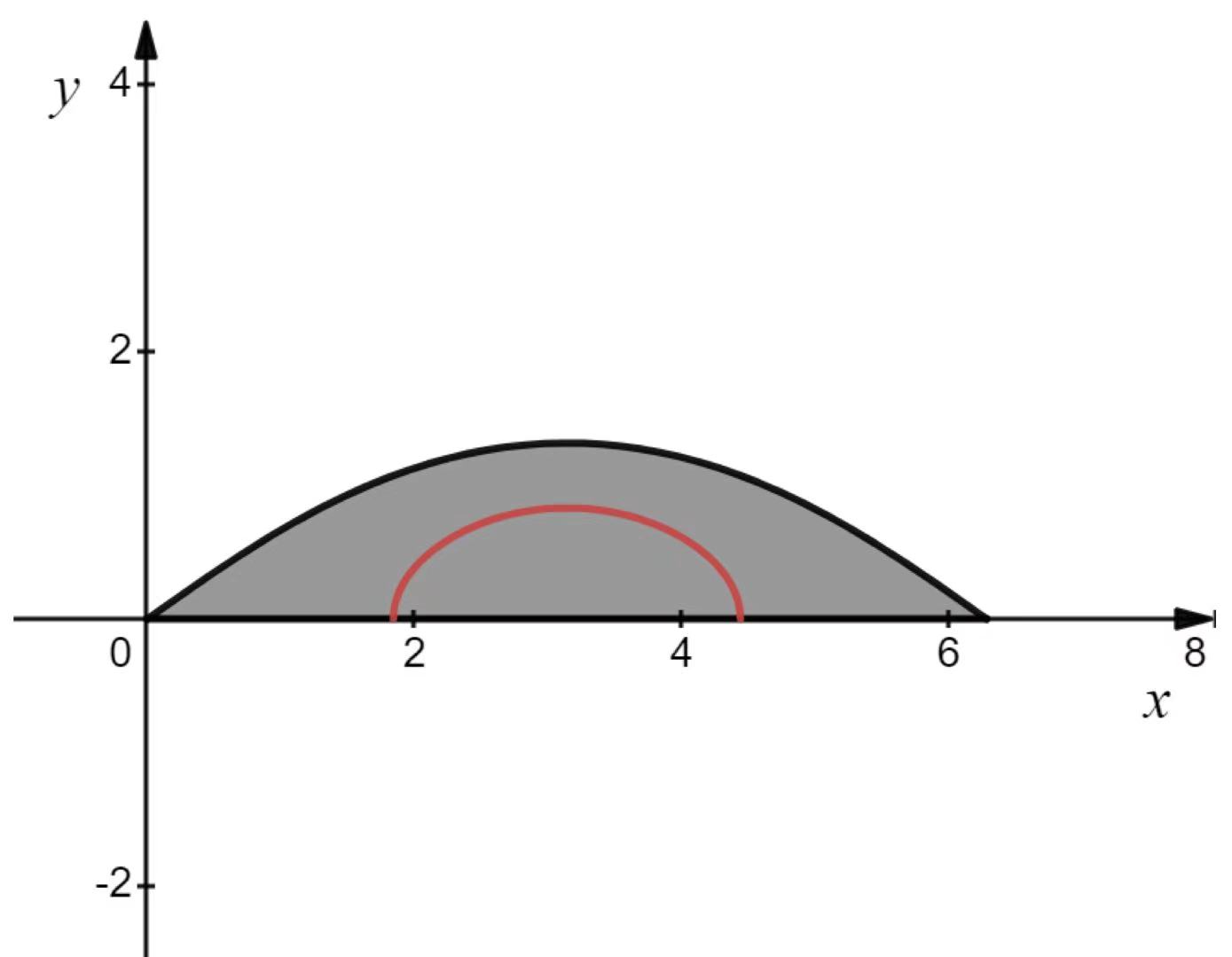}
\includegraphics[width=0.42\textwidth]{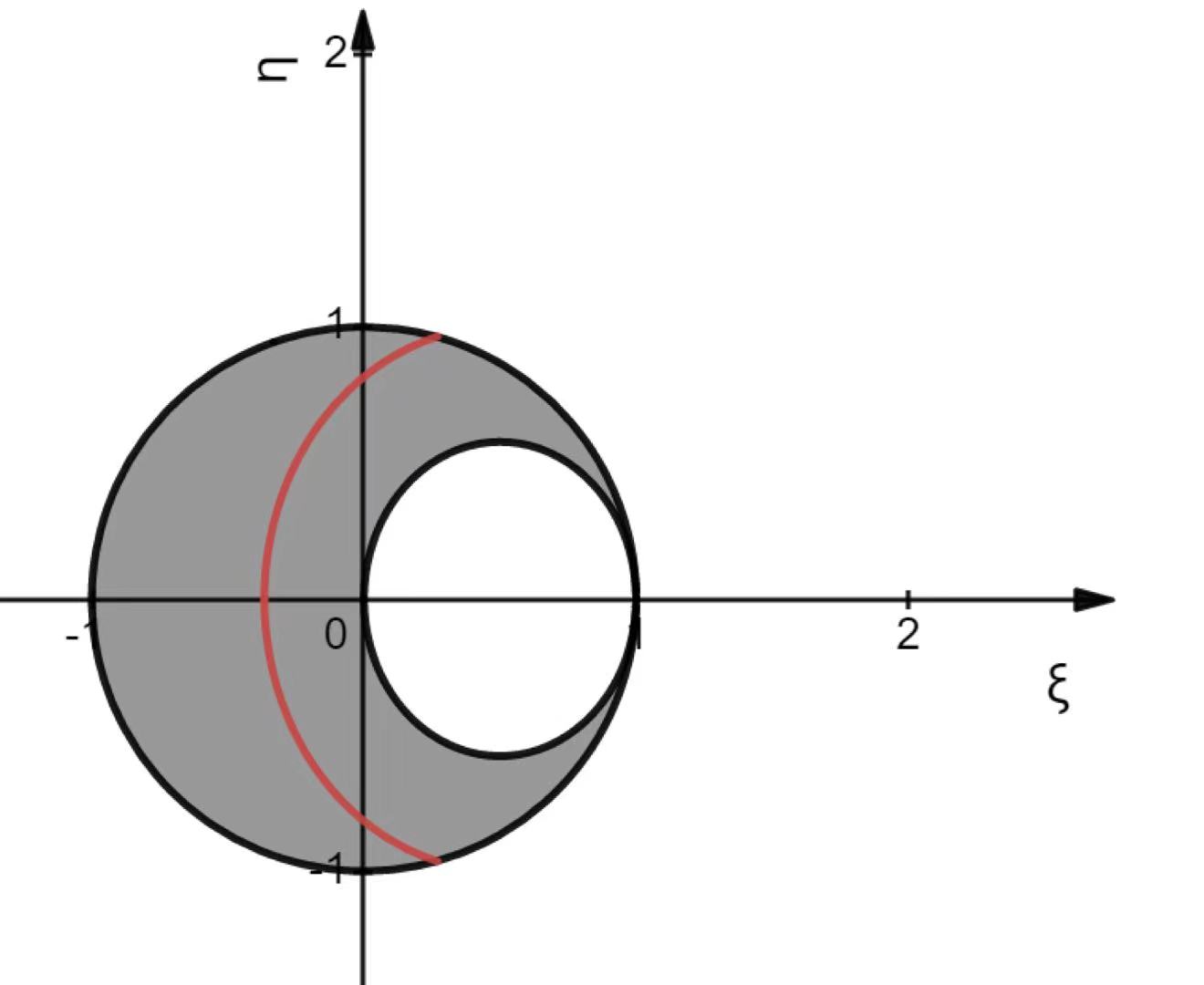}
\caption{Upper trapped region with $\ep=0.5$}
	\label{Figure4th:orig-and-tran}
\end{figure}
We point out the correspondence of  the streamlines and boundary of the upper trapped region  between the $(x,y)$ and $(\xi,\eta)$ coordinates.
\begin{itemize}
 \item For $\rho=\ln\left(\sqrt{1-\ep\over 1+\ep}\right)$, the streamline is the point $(\pi,0)$ in the $(x,y)$  coordinate, and is transformed to the  point $(-1,0)$ in the $(\xi,\eta)$ coordinate.
    \item    For $\rho=\ln\left(\sqrt{1+\ep\over 1-\ep}\right)$,  the upper
separatrix is transformed to the whole ellipse $\Gamma_{1/c_\ep,\ep}$ (the inner boundary of $D_{\xi_\ep\eta_\ep,\ep}$) in the $(\xi,\eta)$ coordinate.
\item For $\rho\in\left(\ln\left(\sqrt{1-\ep\over 1+\ep}\right),\ln\left(\sqrt{1+\ep\over 1-\ep}\right)\right)$, the upper part of the  streamline $\{\psi_\ep=\rho\}$ is transformed to the part of the ellipse $\Gamma_{-e^{-2\rho},\ep}\cap D_1$  in the $(\xi,\eta)$ coordinate, see the red curves in Figure \ref{Figure4th:orig-and-tran}.
    \item The boundary $\{y=0,x\in\mathbb{T}_{2\pi}\}$ in the $(x,y)$  coordinate is transformed to the unit circle $S_1$ (the outer boundary of $D_{\xi_\ep\eta_\ep,\ep}$) in the $(\xi,\eta)$ coordinate.
 \end{itemize}
 \vspace{0.5mm}

\noindent{\bf{Step 2.}} For $\ep\in[0,1)$, we prove the nested property for the inner boundary $\Gamma_{1/c_\ep,\ep}$ of $D_{\xi_{\ep}\eta_{\ep},\ep}$:
\begin{align}\label{Gamma12enclosed}
\Gamma_{1/c_{\ep_2},{\ep_2}}\text{  is  enclosed by }\Gamma_{1/c_{\ep_1},\ep_1} \quad\text{ if }\quad0\leq \ep_1< \ep_2<1.
\end{align}
See Figure \ref{fig5th:bdy1} for the curves  $\Gamma_{1/c_\ep,\ep}$ with $\ep=0.4, 0.5$.

By \eqref{inner boundary}, both the semi-major axis $\sqrt{1-\ep\over1+\ep}$  and semi-minor axis $1-\ep$ of $\Gamma_{1/c_{\ep},{\ep}}$ are decreasing on $\ep\in[0,1)$. Here
we only need to consider $\eta\geq0$ by symmetry.
Recall that $\eta_{1/c_\ep,\ep}(\xi)^2={1-\ep\over 1+\ep}-\frac{(\xi - \ep)^2}{1-\ep^2}, \xi\in[2\ep-1,1]$ for $(\xi,\eta_{1/c_\ep,\ep}(\xi))\in \Gamma_{1/c_{\ep},{\ep}}$.
\begin{figure}[ht]
    \centering
	\includegraphics[width=0.7\textwidth]{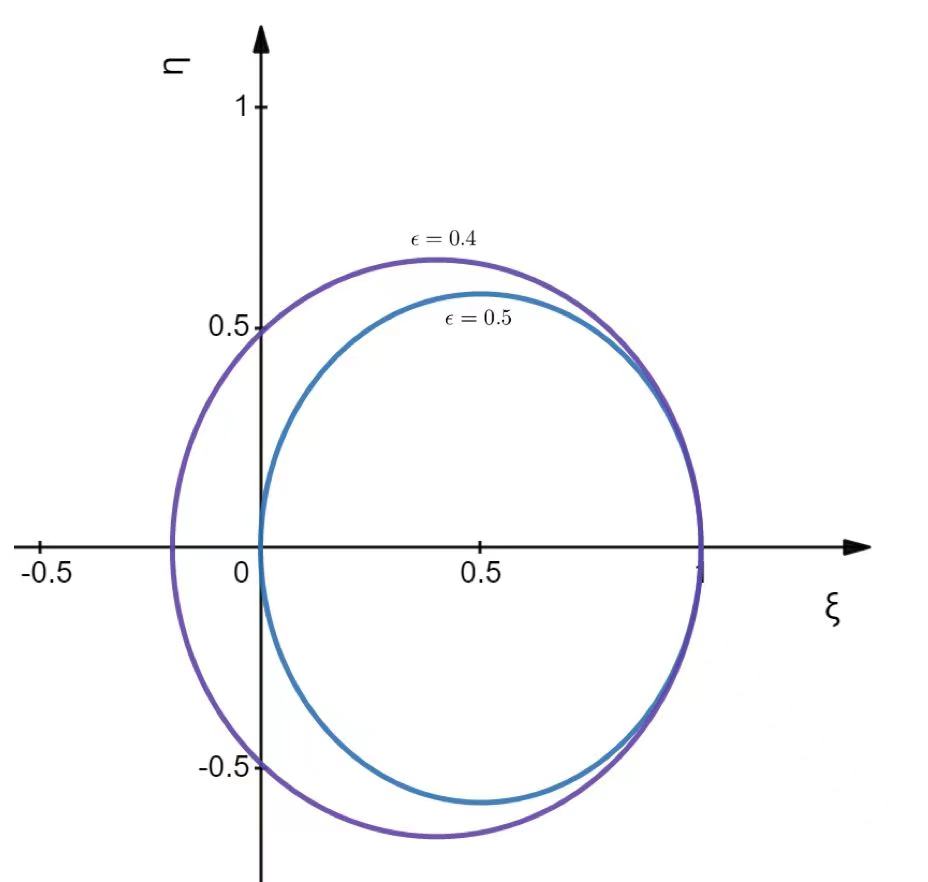}
	\caption{The curves  $\Gamma_{1/c_\ep,\ep}$ with $\ep=0.4, 0.5$}
	\label{fig5th:bdy1}
\end{figure}
To prove \eqref{Gamma12enclosed}, we will show that $\eta_{1/c_{\ep_1},\ep_1}(\xi)^2>\eta_{1/c_{\ep_2},\ep_2}(\xi)^2$ for $\xi\in[\ep_2,1)$.
Since the right vertex of the ellipse $\Gamma_{1/c_{\ep},{\ep}}$ is  $(1,0)$ for $\ep\in[0,1)$, it suffices to verify that $\left|\partial_{\xi}\left(\eta_{1/c_{\ep_1},\ep_1}(\xi)^2\right)\right|>\left|\partial_{\xi}\left(\eta_{1/c_{\ep_2},\ep_2}(\xi)^2\right)\right|$ for $\xi\in[\ep_2,1]$. In fact,
\begin{align*}
&\left|\partial_{\xi}\left(\eta_{1/c_{\ep_2},\ep_2}(\xi)^2\right)\right|-\left|\partial_{\xi}\left(\eta_{1/c_{\ep_1},\ep_1}(\xi)^2\right)\right|=2\left(\frac{\xi - \ep_2}{1-\ep_2^2}-\frac{\xi - \ep_1}{1-\ep_1^2}\right)\\
=&2{(\ep_2-\ep_1)\left((\ep_1+\ep_2)\xi-1-\ep_1\ep_2\right)\over (1-\ep_2^2)(1-\ep_1^2)}\leq 2{(\ep_2-\ep_1)\left(\ep_1+\ep_2-1-\ep_1\ep_2\right)\over (1-\ep_2^2)(1-\ep_1^2)}\\
=&2{(\ep_2-\ep_1)(\ep_1-1)(1-\ep_2)\over (1-\ep_2^2)(1-\ep_1^2)}<0
\end{align*}
 for $\xi\in[\ep_2,1]$ and $0\leq \ep_1< \ep_2<1$.

 \vspace{0.5mm}

By Step 2, we get \eqref{D1D2-constant-eta-xi}, which implies \eqref{D1D2-constant}.
\end{proof}
\begin{Corollary}\label{b3-increasing}
$ b_{\ep,3}(\hat{\psi}_{1,\ep})$ is non-decreasing on $\ep\in[0,1)$.
\end{Corollary}
\begin{proof}
By the definition of $ b_{\ep,3}(\hat{\psi}_{1,\ep})$ in \eqref{b-ep-2-hat-psi-1-ep} and Lemma \ref{D-theta-gamma-ep-nest}, we have
\begin{align*}
b_{\ep_1,3}(\hat{\psi}_{1,\ep_1})=
&2\iint_{D_{xy,\ep_1}}g'(\psi_{\ep_1})
\sin^2\left(\frac{\theta_{\ep_1}}{3}\right)(1-\gamma_{\ep_1}^2)^{1\over3}dxdy\\
=&4\iint_{D_{\theta_{\ep_1}\gamma_{\ep_1},\ep_1}}\sin^2\left(\frac{\theta}{3}\right)(1-\gamma^2)^{1\over3}d\theta d\gamma\\
\leq &4\iint_{D_{\theta_{\ep_2}\gamma_{\ep_2},\ep_2}}\sin^2\left(\frac{\theta}{3}\right)(1-\gamma^2)^{1\over3}d\theta d\gamma\\
=&2\iint_{D_{xy,\ep_2}}g'(\psi_{\ep_2})
\sin^2\left(\frac{\theta_{\ep_2}}{3}\right)(1-\gamma_{\ep_2}^2)^{1\over3}dxdy=b_{\ep_2,3}(\hat{\psi}_{1,\ep_2})
\end{align*}
for $0\leq \ep_1\leq \ep_2<1$.
\end{proof}
By   splitting  the trapped regions and taking  approximate summation for the integral in $b_{\ep,3}(\hat{\psi}_{1,\ep})|_{\ep={4\over5}}$, we have
\begin{align*}
b_{\ep,3}(\hat{\psi}_{1,\ep})|_{\ep={4\over5}}<24.38.
\end{align*}
It then follows from Corollary \ref{b3-increasing} that
 \begin{align}\label{b-ep-2-hat-psi}
 b_{\ep, 2}(\hat{\psi}_{1,\ep})<24.38 \quad\text{for}\quad\ep\in\left[0,{4\over5}\right].
 \end{align}
Combining \eqref{b-ep-1-hat-psi} and \eqref{b-ep-2-hat-psi}, we have
\begin{align}\label{test-odd-neg}
\langle\hat{A}_{\ep,e} \hat{\psi}_{1,\ep}, \hat{\psi}_{1,\ep} \rangle= b_{\ep, 1}(\hat{\psi}_{1,\ep}) + b_{\ep, 2}(\hat{\psi}_{1,\ep})<-24.61+24.38 =-0.23< 0.
\end{align}
\medskip

\noindent{\bf{Case 2. Test functions for $\ep\in\left({4\over5},1\right)$.}}
\medskip

Let
 \begin{align*}
&{\phi}_{2,\ep}(x,y)= \Phi_{2,\ep}(\theta_\ep,\gamma_\ep)\\\nonumber
=  &\left\{ \begin{array}{ll}
         \cos\left({1\over 2}\theta_\ep\right)(1-\gamma_\ep^2)^{1\over2} & \mbox{if $(\theta_\ep,\gamma_\ep) \in [0, 4k\pi]\times[-1,1]$},\\
         \cos\left(\theta_\ep\right)(1-\gamma_\ep^2)^{1\over2} & \mbox{if $(\theta_\ep,\gamma_\ep) \in \left((4k\pi,(4k+{1\over 2})\pi]\cup((4k+{3\over 2})\pi,(4k+2)\pi]\right)\times[-1,1]$},\\
         0& \mbox{if $(\theta_\ep,\gamma_\ep) \in ((4k+{1\over 2})\pi,(4k+{3\over 2})\pi]\times[-1,1]$}.
        \end{array} \right.
 \end{align*}
 Then
 \begin{align*}
 &\widehat{(\Phi_{2,\ep})}_0(0)={1\over (4k+2)\pi}\int_0^{(4k+2)\pi}\Phi_{2,\ep}(\theta_\ep,0)d\theta_{\ep}\\
 =&{1\over (4k+2)\pi}\left(\int_{4k\pi}^{(4k+{1\over2})\pi}+\int_{(4k+{3\over2})\pi}^{(4k+2)\pi}\right)\cos(\theta_{\ep})d\theta_{\ep}={1\over (2k+1)\pi}.
 \end{align*}
 We choose the test function
 \begin{align}\label{test-odd-2}
\hat{\psi}_{2,\ep}(x,y)= &\hat\Psi_{2,\ep}(\theta_\ep,\gamma_\ep)\triangleq \Phi_{2,\ep}(\theta_\ep,\gamma_\ep)-{1\over (2k+1)\pi}={\phi}_{2,\ep}(x,y)-{1\over (2k+1)\pi}
 \end{align}
 for $(\theta_\ep,\gamma_\ep) \in \mathbb{T}_{(4k+2)\pi}\times[-1,1]$.
Then $\hat\Psi_{2,\ep}\in C^0(\tilde \Omega_{2k+1})$ and
 \begin{align*}
 \|\hat\Psi_{2,\ep}\|_{\tilde Y_{\ep,e}}^2=&\left(\int_{-1}^{1} \int_0^{4k\pi}
+ \int_{-1}^{1} \int_{4k\pi}^{(4k+2)\pi}\right)\left({1\over1-\gamma_\ep^2}|\partial_{\theta_\ep}\hat\Psi_{2,\ep}|^2+(1-\gamma_\ep^2)|\partial_{\gamma_\ep}\hat\Psi_{2,\ep}|^2\right)d \theta_\ep d\gamma_\ep\\
=&\left(\int_{-1}^{1} \int_0^{4k\pi}
+ \int_{-1}^{1} \int_{4k\pi}^{(4k+2)\pi}\right)\left({1\over1-\gamma_\ep^2}|\partial_{\theta_\ep}\Phi_{2,\ep}|^2+(1-\gamma_\ep^2)|\partial_{\gamma_\ep}\Phi_{2,\ep}|^2\right)d \theta_\ep d\gamma_\ep\\
= &
k\pi+{1\over3}\pi<\infty.
 \end{align*}
 Moreover,
 \begin{align*}
 \int_{0}^{(4k+2)\pi}\hat\Psi_{2,\ep}(\theta_\ep,0)d\theta_\ep=\int_{0}^{(4k+2)\pi}\left(\Phi_{2,\ep}(\theta_\ep,0)-{1\over (2k+1)\pi}\right)d\theta_\ep=2-2
 =0.
 \end{align*}
Thus, $\hat\Psi_{2,\ep}\in\tilde Y_{\ep,e}$, which implies $ \hat{\psi}_{2,\ep}\in \tilde X_{\ep,e}$. Since $\hat{P}_{\ep,e}{1\over (2k+1)\pi}={1\over (2k+1)\pi}$, we have
\begin{align}\nonumber
\langle\hat{A}_{\ep,e} \hat{\psi}_{2,\ep}, \hat{\psi}_{2,\ep}\rangle
&= \iint_{\Omega_{2k+1}} \left( |\nabla \hat{\psi}_{2,\ep}|^2 - g'(\psi_\ep)((I - \hat{P}_{\ep,e})\hat{\psi}_{2,\ep})^2\right) dxdy\\\nonumber
&=\iint_{\Omega_{2k+1}}\left( |\nabla {\phi}_{2,\ep}|^2 - g'(\psi_\ep)((I - \hat{P}_{\ep,e}){\phi}_{2,\ep})^2\right) dxdy\\\label{Aep-psi-2-ep-2}
&=b_{\ep, 1}({\phi}_{2,\ep}) + b_{\ep, 2}({\phi}_{2,\ep}).
\end{align}
 By Corollary  \ref{kernel of  the operator A-ep and a decomposition of tilde Xep},
 $\cos\left(\theta_\ep\right)(1-\gamma_\ep^2)^{1\over2}\in \ker( A_\ep)$,
 and thus,
 \begin{align}\label{hat-psi-4kpi-4k+2pi-2}
-{1\over1-\gamma_\ep^2}\pa_{\theta_\ep}^2\Phi_{2,\ep}-\pa_{\gamma_\ep}\left((1-\gamma_\ep^2)\pa_{\gamma_\ep}\Phi_{2,\ep}\right)
=2\Phi_{2,\ep}
 \end{align}
 for $(\theta_\ep,\gamma_\ep) \in ((4k\pi,(4k+{1\over 2})\pi]\cup((4k+{3\over 2})\pi,(4k+2)\pi])\times[-1,1]$.
  By Lemma
\ref{sol to eigenvalue problem non-zero modes varepsilon=0 original}, $(1-\gamma_\ep^2)^{1\over2}$ is an
 eigenfunction of the eigenvalue $1$ for \eqref{eigenvalue problem2 non-zero modes varepsilon=0} with $k=1$. This, along with \eqref{laplacian}, gives
$$-(\Delta + g'(\psi_\ep)) {\phi}_{2,\ep} = - \frac 1 2 g'(\psi_\ep) \left( \frac 3 4 \frac{{\Phi}_{2,\ep}}{1-\gamma_\ep^2} \right),\;\;(x,y) \in [0, 4k\pi]\times \mathbb{R}.$$
Then
\begin{align}\nonumber
& \int_{-\infty}^{+\infty} \int_0^{4k\pi} \left(|\nabla {\phi}_{2,\ep}|^2 - g'(\psi_\epsilon){\phi}_{2,\ep}^2\right) dx dy  = \int_{-\infty}^{+\infty} \int_0^{4k\pi} - \frac 1 2 g'(\psi_\ep) \left( \frac 3 4 \frac{{\Phi}_{2,\ep}^2}{1-\gamma_\ep^2} \right)dx dy \\\label{b1-odd-term1-2}
= & - \int_{-1}^{1} \int_0^{4k\pi}  \left( \frac 3 4 \frac{{\Phi}_{2,\ep}^2}{1-\gamma_\ep^2} \right)d \theta_\ep d\gep
= -3k\pi.
\end{align}
Combining \eqref{hat-psi-4kpi-4k+2pi-2} and \eqref{b1-odd-term1-2}, we have
\begin{align}\nonumber
b_{\ep,1}(\phi_{2,\ep})
=&  \left( \int_{-\infty}^{+\infty} \int_0^{4k\pi}+\int_{-\infty}^{+\infty} \int_{4k\pi}^{(4k+2)\pi}\right)\left(
|\nabla {\phi}_{2,\ep}|^2 - g'(\psi_\epsilon){\phi}_{2,\ep}^2\right) dx dy\\\nonumber
=& -3k\pi+\left(\int_{-1}^1\int_0^{{\pi\over2}}+\int_{-1}^1\int_{3\pi\over2}^{2\pi}\right)\bigg(
{1\over1-\gamma_\ep^2}|\pa_{\theta_\ep}\Phi_{2,\ep}|^2\\\nonumber
&+(1-\gamma_\ep^2)|\pa_{\gamma_\ep}\Phi_{2,\ep}|^2-2|\Phi_{2,\ep}|^2\bigg)d\theta_\ep d\gamma_\ep\\\label{com-b-ep-1-phi-2-ep}
=&-3k\pi.
\end{align}
Since
$\cos\left({1\over 2}\theta_\ep\right)$ is  `odd'   symmetrical about the points $((2j-1)\pi,0)$  for $j=1,\cdots,2k$,  we have
$\hat{P}_{\ep,e} \hat{\psi}_{2,\ep}(x,y)=0$ for $(x,y)$ in the $j$-th trapped region of
$\mathbb{T}_{(4k+2)\pi}\times \mathbb{R}$, where $1\leq j\leq2k$.
Next, we compute the projection term for $(x,y)$ in  the $(2k+1)$-th trapped region, denoted by $D_{{\rm{in}},2k+1}$.
Using $x$ as the parameter,
we  represent the upper and lower  separatrix to be $y(x)=\cosh^{-1}(1+\ep-\ep\cos(x)), x\in[4k\pi,(4k+2)\pi]$ and  $y(x)=-\cosh^{-1}(1+\ep-\ep\cos(x)), x\in[4k\pi,(4k+2)\pi]$, respectively.
Then
\begin{align*}
&\iint_{D_{{\rm{in}},2k+1}} g'(\psi_\ep) |\hat{P}_{\ep,e} {\phi}_{2,\ep}|^2 dxdy
=\int_{-\rho_0}^{\rho_0} g'(\rho)  \frac{\left|\oint_{\Gamma_{2k+1} (\rho)} \frac{{\phi}_{2,\ep}}{|\nabla \psi_\ep|}\right|^2}{\oint_{\Gamma_{2k+1} (\rho)} \frac{1}{|\nabla\psi_\ep|}} d\rho\\
\leq&\int_{-\rho_0}^{\rho_0} g'(\rho) \oint_{\Gamma_{2k+1} (\rho)} \frac{|{\phi}_{2,\ep}|^2}{|\nabla \psi_\ep|}d\rho= \iint_{D_{{\rm{in}},2k+1}} g'(\psi_\ep) | {\phi}_{2,\ep}|^2 dxdy\\
\leq &\iint_{\Omega_{2k+1}\setminus\Omega_{2k}}g'(\psi_\ep)| {\phi}_{2,\ep}|^2dxdy
=2\int_{-1}^1\left(\int_0^{\pi\over 2}+\int_{3\pi\over2}^{2\pi}\right)\cos^2\left(\theta_\ep\right)(1-\gamma_\ep^2)d\theta_\ep d\gamma_\ep\\
=&{4\over3} \pi,
\end{align*}
where $\rho_0$ and $\Gamma_{2k+1} (\rho)$ are defined in \eqref{def-rho0} and \eqref{Gamma-rho}.
Now, we compute the projection term for $(x,y)$ in  the untrapped region, denoted by $D_{c}$.
\begin{align*}
&\iint_{D_{c}} g'(\psi_\ep) |\hat{P}_{\ep,e} {\phi}_{2,\ep}|^2 dxdy
=(2k+1)\left(\iint_{\Omega_{2k+1}\setminus (\Omega_{2k}\cup D_{{\rm{in}},2k+1})} g'(\psi_\ep) |\hat{P}_{\ep,e} {\phi}_{2,\ep}|^2 dxdy\right)\\
\leq&(2k+1)\left(\iint_{\Omega_{2k+1}\setminus (\Omega_{2k}\cup D_{{\rm{in}},2k+1})} g'(\psi_\ep) | {\phi}_{2,\ep}|^2 dxdy\right)\\
\leq &(2k+1)\left(\iint_{\Omega_{2k+1}\setminus (\Omega_{2k}\cup D_{{\rm{in}},2k+1})} g'(\psi_\ep)\cos^2\left(\theta_\ep\right)(1-\gamma_\ep^2)dxdy\right)\\
= &(2k+1)\left({8\over3} \pi-\iint_{D_{{\rm{in}},2k+1}} g'(\psi_\ep) \cos^2\left(\theta_\ep\right)(1-\gamma_\ep^2)dxdy\right)\\
=&(2k+1)\left({8\over3} \pi-\int_0^{2\pi}\int_{-\cosh^{-1}(1+\ep-\ep\cos(x))}^{\cosh^{-1}(1+\ep-\ep\cos(x))} g'(\psi_\ep) \cos^2\left(\theta_\ep\right)(1-\gamma_\ep^2) dydx\right)\\
\triangleq&(2k+1)\left({8\over3} \pi-b_{\ep,4}({\phi}_{2,\ep})\right).
\end{align*}
Thus,
\begin{align}\nonumber
 b_{\ep, 2}({\phi}_{2,\ep})=&\iint_{D_{{\rm{in}},2k+1}} g'(\psi_\ep) |\hat{P}_{\ep,e} {\phi}_{2,\ep}|^2 dxdy+\iint_{D_{c}} g'(\psi_\ep) |\hat{P}_{\ep,e} {\phi}_{2,\ep}|^2 dxdy\\\label{est-b-ep-2-phi-2-ep}
 \leq &{4\over3} \pi+(2k+1)\left({8\over3} \pi-b_{\ep,4}({\phi}_{2,\ep})\right).
 \end{align}
 \begin{Corollary}\label{b4-increasing}
 $b_{\ep,4}({\phi}_{2,\ep})$ is non-decreasing on $\ep\in[0,1)$.
\end{Corollary}
\begin{proof}
By the definition of $ b_{\ep,4}(\phi_{2,\ep})$ and Lemma \ref{D-theta-gamma-ep-nest}, we have
\begin{align*}
b_{\ep_1,4}(\phi_{2,\ep_1})=
&\iint_{D_{xy,\ep_1}}g'(\psi_{\ep_1}) \cos^2\left(\theta_{\ep_1}\right)(1-\gamma_{\ep_1}^2)dxdy\\
=&2\iint_{D_{\theta_{\ep_1}\gamma_{\ep_1},\ep_1}}\cos^2\left(\theta\right)(1-\gamma^2)d\theta d\gamma\\
\leq &2\iint_{D_{\theta_{\ep_2}\gamma_{\ep_2},\ep_2}}\cos^2\left(\theta\right)(1-\gamma^2)d\theta d\gamma\\
=&\iint_{D_{xy,\ep_2}}g'(\psi_{\ep_2})
\cos^2\left(\theta_{\ep_2}\right)(1-\gamma_{\ep_2}^2)dxdy=b_{\ep_2,4}(\phi_{2,\ep_2})
\end{align*}
for $0\leq \ep_1\leq \ep_2<1$.
\end{proof}
Since
\begin{align*}
b_{\ep,4}({\phi}_{2,\ep})|_{\ep={4\over5}}>6.94,
\end{align*}
by Corollary \ref{b4-increasing} we have $\min_{\ep\in[{4\over5},1)}b_{\ep,4}({\phi}_{2,\ep})>6.94.$
Then it follows from \eqref{est-b-ep-2-phi-2-ep} that
\begin{align}\label{est-b-ep-2-phi-2-ep-2}
 b_{\ep, 2}({\phi}_{2,\ep})
 \leq &{4\over3} \pi+(2k+1)\left({8\over3} \pi-6.94\right),\;\ep\in \left[{4\over5},1\right).
 \end{align}
By \eqref{Aep-psi-2-ep-2}, \eqref{com-b-ep-1-phi-2-ep} and \eqref{est-b-ep-2-phi-2-ep-2}, we have
 \begin{align}\nonumber
 &\langle\hat{A}_{\ep,e} \hat{\psi}_{2,\ep}, \hat{\psi}_{2,\ep}\rangle
=b_{\ep, 1}({\phi}_{2,\ep}) + b_{\ep, 2}({\phi}_{2,\ep})\leq -3k\pi+{4\over3} \pi+(2k+1)\left({8\over3} \pi-6.94\right)\\\label{A-ep-e-psi-2-neg}
=& \left({7\over 3}\pi-13.88\right)k+4\pi-6.94\leq {19\over 3}\pi-20.82<0
 \end{align}
 for $k\geq1$ and $\ep\in \left({4\over5},1\right)$.
 \medskip

 Combining Case 1  and Case 2, we obtain linear instability of $\omega_\ep$ for perturbations with odd multiples of the period.

\begin{Theorem}\label{multi-odd}
Let $\ep \in [0, 1)$. Then the steady state $\omega_\ep$ is linearly unstable for $(4k + 2)\pi$-periodic perturbations, where $k\geq1$ is an integer.
\end{Theorem}
\begin{proof}
For $\ep\in\left[0,{4\over5}\right]$, we define
 the test function to be $\hat{\psi}_{1,\ep}$   in \eqref{test-odd}.  By \eqref{test-odd-neg}, we have
$\langle\hat{A}_{\ep,e} \hat{\psi}_{1,\ep}, \hat{\psi}_{1,\ep} \rangle < 0.$
 For $\ep\in\left({4\over5},1\right)$, we define
 the test function to be $\hat{\psi}_{2,\ep}$   in \eqref{test-odd-2}.  By \eqref{A-ep-e-psi-2-neg}, we have
$\langle\hat{A}_{\ep,e} \hat{\psi}_{2,\ep}, \hat{\psi}_{2,\ep} \rangle < 0.$ Thus, $n^-\left( L_{\ep,e} |_{\overline{R(B_\ep)}} \right) = n^-\left(\hat{A}_{\ep,e}\right)\geq1$ for $\ep\in[0,1)$ by Lemma \ref{L e-hat A}. Then linear instability is obtained by applying Lemma  \ref{indice-theorem-sep}.
\end{proof}
\begin{remark}\label{number of unstable eigenvalues for multi-periodic perturbations}


$(1)$ For $\ep\in\left[0,{4\over5}\right]$, we use
 the test function $\hat{\psi}_{1,\ep}$ to get a negative direction of
$\hat{A}_{\ep,e}.$ A conjecture is that $\hat{\psi}_{1,\ep}$ is always a negative direction of
$\hat{A}_{\ep,e}$ for $\ep\in\left[0,1\right)$. The difficulty to prove or disprove this conjecture  is how to accurately compute or estimate the projection term in a rigorous way.

$(2)$
For $\ep=0$, the number of unstable eigenvalues of the linearized vorticity operator is $2(m-1)$. Indeed, on the one hand,
since
\begin{align*}
\langle\tilde A_{0,e}\psi,\psi\rangle=&\iint_{\Omega_m}\left(|\nabla\psi|^2-g'(\psi_0)\psi^2\right)dxdy+{\left(\iint_{\Omega_m}g'(\psi_0)\widehat\psi_0dxdy\right)^2\over \iint_{\Omega_m}g'(\psi_0)dxdy}\\
\leq &\iint_{\Omega_m}\left(|\nabla\psi|^2-g'(\psi_0)\psi^2\right)dxdy+{\iint_{\Omega_m}g'(\psi_0)\widehat\psi_0^2dxdy}\\
=&\iint_{\Omega_m}\left(|\nabla\psi|^2-g'(\psi_0)\psi^2\right)dxdy+{\iint_{\Omega_m}g'(\psi_0)(\hat P_{0,e}\psi)^2dxdy}=\langle\hat A_{0,e}\psi,\psi\rangle
 \end{align*}
 for $\psi\in\tilde X_{0,e}$, we have $n^-(\hat A_{0,e})\leq n^-(\tilde A_{0,e}).$
By Corollary \ref{A-L-dec-e}, $n^-\left(\hat{A}_{0,e}\right)\leq n^-\left(\tilde A_{0,e}\right)=2(m-1)$. On the other hand,
since $\hat W_{0,e}=\{\phi(y)\in L_{{ g'(\psi_0)},e}^2(\Omega_m)\}$
 and  $\hat{P}_{0,e}\psi=0$ for  $\psi\in\tilde X_{0,e-}$,
we have $\hat{A}_{0,e}|_{X_{0,e-}}=\tilde A_{0,e}|_{X_{0,e-}}$ and thus, $n^-\left(\hat{A}_{0,e}\right)=2(m-1)$. The conclusion is then a consequence of
Lemmas \ref{L e-hat A} and \ref{indice-theorem-sep}.
This  suggests that the number of unstable eigenvalues of the linearized vorticity operator is $2(m-1)$ for $\ep\ll1$.
\end{remark}
\section{Modulational instability}\label{modulational}
In this section, we study the linear stability of $\omega_\ep$ with respect to perturbations of the form
\begin{align}u(x, y) = \widetilde{u}(x, y)e^{i\alpha x},\nonumber\\
\omega(x, y) = \widetilde{\omega}(x, y)e^{i\alpha x},\label{perturbation-modulaitonal-form}\\
\psi(x, y) = \widetilde{\psi}(x, y)e^{i\alpha x},\nonumber\end{align}
where $\alpha \in (0, \frac 1 2]$, and $\widetilde{u}, \widetilde{\omega}, \widetilde{\psi}$ are complex-valued and defined on the domain $\Omega = \mathbb{T}_{2\pi} \times \mathbb{R}$.
\subsection{Complex Hamiltonian formulation}
Recall that the   linearized vorticity  operator has the form $ J_\epsilon L_\epsilon$, where
$J_\epsilon = -g'(\psi_\epsilon)\vec{u}_\epsilon\cdot\nabla$ and
 $L_\epsilon = \frac {1} {g'(\psi_\epsilon)} - (-\Delta)^{-1}$. We seek solutions of the form \eqref{perturbation-modulaitonal-form} for the linearized equations, where $\widetilde{\omega} \in L^2_{\frac{1}{g'(\psi_\ep)}}(\Omega)$.  Then we have $J_\epsilon L_\epsilon(e^{i\alpha x}\widetilde{\omega})=e^{i\alpha x}J_{\epsilon,\alpha} L_{\epsilon,\alpha}\widetilde{\omega}$, where
\begin{align}\label{def-J-ep-al}
J_{\epsilon, \alpha} =& g'(\psi_\epsilon)\vec{u}_\epsilon\cdot\nabla_\alpha:L^2_{g'(\psi_\ep)}(\Omega) \supset D(J_{\epsilon, \alpha}) \to L^2_{\frac{1}{g'(\psi_\ep)}}(\Omega),\\\label{def-L-ep-al}
 L_{\epsilon, \alpha} =& \frac {1} {g'(\psi_\epsilon)} - (-\Delta_\alpha)^{-1}:L^2_{\frac{1}{g'(\psi_\ep)}}(\Omega)\to L^2_{g'(\psi_\ep)}(\Omega),\end{align}
and
\begin{align}\label{nabla-alpha-Delta-alpha}
 \nabla_\alpha = (\partial_x + i\alpha, \partial_y)^T,\quad
 \Delta_\alpha  = (i\alpha + \partial_x)^2 + \partial_{yy}.
 \end{align}
To make it rigorous, we need to clarify  the solvability of the  $\alpha$-Poisson equation.
\begin{lemma}
For any $\widetilde{\omega} \in L^2_{\frac{1}{g'(\psi_\ep)}}(\Omega)$, the $\alpha$-Poisson equation
\begin{align}\label{a-Poisson}- \Delta_\alpha \widetilde{\psi} = \widetilde{\omega}\end{align}
has a unique weak solution $\widetilde{\psi}$ in the Hilbert space
$$H^1_\alpha(\Omega) := \{ \phi | \| \nabla_\alpha \phi \|^2_{L^2(\Omega)}< \infty  \}$$
equipped with the inner product
$$(\phi_1, \phi_2)_{H^1_\alpha(\Omega)} = \iint_\Omega \nabla_\alpha \phi_1 \cdot \overline{\nabla_\alpha \phi_2} dxdy.$$
\end{lemma}
\begin{remark}
Since $\mathbb{Z}\ni k\neq\alpha\in (0,{1\over2}]$, we have $c_0(k^2+\alpha^2)\leq (k+\alpha)^2$ for some $c_0>0$. Then
\begin{align*}
c_1\|  \phi \|_{H^1(\Omega)}^2\leq \| \nabla_\alpha \phi \|^2_{L^2(\Omega)}=\sum_{k\in\mathbb{Z}}\left((k+\alpha)^2\|\widehat{\phi}_k\|_{L^2(\mathbb{R})}^2 +\|\widehat{\phi}'_k\|_{L^2(\mathbb{R})}^2\right)\leq  c_2\|  \phi \|^2_{H^1(\Omega)}
\end{align*}
 for some $c_1, c_2>0$. Thus,  $H^1_\alpha(\Omega)\cong H^1(\Omega)$.
\end{remark}
\begin{proof}
 For $\widetilde{\omega} \in L^2_{\frac{1}{g'(\psi_\ep)}}(\Omega)$, we have
$$\iint_\Omega \phi \widetilde{\omega} dxdy \leq \iint_\Omega  \frac{|\widetilde{\omega}|^2}{g'(\psi_\ep)} dxdy \iint_\Omega g'(\psi_\ep) |\phi|^2 dxdy \leq C \|\widetilde{\omega}\|^2_{L^2_{\frac{1}{g'(\psi_\ep)}}(\Omega)} \|\phi\|^2_{H_\alpha^1(\Omega)}, \quad \phi \in H_\alpha^1(\Omega). $$
By the Riesz Representation Theorem, for any $\widetilde{\omega} \in L^2_{\frac{1}{g'(\psi_\ep)}}(\Omega)$, there exists a unique $\widetilde{\psi} \in H_\alpha^1(\Omega)$ such that
$$\iint_\Omega \widetilde{\omega} \phi dxdy = \langle \widetilde{\omega}, \phi \rangle = (\widetilde{\psi}, \phi)_{H_\alpha^1(\Omega)},\quad \phi \in H_\alpha^1(\Omega).$$
\end{proof}
For $\widetilde{\omega} \in L^2_{\frac{1}{g'(\psi_\ep)}}(\Omega)$, we denote $(- \Delta_\alpha)^{-1} \widetilde{\omega} \in H_\alpha^1(\Omega)$ to be the weak solution of the $\alpha$-Poisson equation
\eqref{a-Poisson}.
The linearized vorticity equation for $\widetilde{\omega}$ is  formulated as
 \begin{align}\label{complex Ham-modu}
 \partial_t \widetilde{\omega} = J_{\ep, \alpha} L_{\ep, \alpha} \widetilde{\omega}.
 \end{align}
The steady state $\omega_\ep$ is said to be linearly modulationally unstable for $\alpha\in(0,{1\over 2}]$ if the operator $J_{\ep, \alpha} L_{\ep, \alpha}$
has an unstable eigenvalue $\lambda$ with $Re(\lambda)>0$.

For $\widetilde{\omega} \in L^2_{\frac{1}{g'(\psi_\ep)}}(\Omega)$, let $\widetilde{\psi}=(-\Delta_\alpha)^{-1}\widetilde{\omega}\in H_\alpha^1(\Omega)$, then
\begin{align*}
\|\widetilde{\psi}\|_{H_\alpha^1(\Omega)}^2=\iint_\Omega \widetilde{\omega} \overline{\widetilde{\psi}} dxdy
 \leq C\|\omega\|_{L^2_{\frac{1}{g'(\psi_\ep)}}(\Omega)} \|\widetilde{\psi}\|_{H_\alpha^1(\Omega)}.
\end{align*}
Thus, $\|\widetilde{\psi}\|_{H_\alpha^1(\Omega)}\leq C \|\widetilde{\omega}\|_{L^2_{\frac{1}{g'(\psi_\ep)}}(\Omega)}$. Let $\widetilde\omega_i\in L^2_{\frac{1}{g'(\psi_\ep)}}(\Omega)$ and $\widetilde{\psi}_i=(-\Delta_\alpha)^{-1}\widetilde{\omega}_i\in H_\alpha^1(\Omega)$ for $i=1,2$. Then
\begin{align}\label{L-alpha-bounded}
\langle L_{\ep,\alpha} \widetilde{\omega}_1, \widetilde{\omega}_2 \rangle=
 \langle  \widetilde{\omega}_1, L_{\ep,\alpha}\widetilde{\omega}_2 \rangle
 \leq C\|\widetilde{\omega}_1\|_{L^2_{\frac{1}{g'(\psi_\ep)}}(\Omega)}\|\widetilde{\omega}_2\|_{L^2_{\frac{1}{g'(\psi_\ep)}}(\Omega)}.
\end{align}
Thus, $
 \langle L_{\ep,\alpha}\cdot,\cdot\rangle$
is bounded and symmetric  on $L^2_{\frac{1}{g'(\psi_\ep)}}(\Omega)$.

\subsection{Exact solutions to the associated eigenvalue problems  for the  modulational case}

Define
 \begin{align*}
 \tilde{A}_{\ep,\alpha}=-\Delta_\alpha-g'(\psi_\ep): H_\alpha^1(\Omega) \rightarrow H_\alpha^{1}(\Omega)^*,
\end{align*}
where the negative $\alpha$-Laplacian operator is understood in the weak sense.
Then
$
 \langle\tilde{A}_{\ep,\alpha}\cdot,\cdot\rangle
$
defines a  bounded and symmetric bilinear form on $H_\alpha^{1}(\Omega)$. Noting that $\iint_{\Omega}g'(\psi_\ep)|\psi|^2dxdy\leq \|\psi\|_{H_\alpha^1(\Omega)}^2$ for $\psi\in H_\alpha^1(\Omega)$, a
similar argument to Lemma \ref{equal-indices0} implies
\begin{align*}
\dim\ker(L_{\ep,\alpha})=\dim\ker(\tilde{A}_{\ep,\alpha}) \quad {\rm{and}} \quad n^-(L_{\ep,\alpha})=n^-(\tilde{A}_{\ep,\alpha}).
\end{align*}

Since  $H_\alpha^{1}(\Omega)$ is compactly embedded in $L_{g'(\psi_\ep)}^2(\Omega)$,
 we can
inductively define $\lambda_n$, $n\geq1$, as follows:
\begin{align}\nonumber
\lambda_n(\ep,\alpha)=& \inf_{\widetilde{\psi} \in H_\alpha^{1}(\Omega), (\widetilde{\psi}, \widetilde{\psi}_{i})_{L_{g'(\psi_\ep)}^2(\Omega)} = 0, i = 1, 2, \cdots, n-1}{\iint_\Omega|\nabla_\alpha\widetilde{\psi}|^2dxdy\over\iint_\Omega g'(\psi_\ep)|\widetilde{\psi}|^2dxdy}\\\nonumber
=&\min_{\widetilde{\psi} \in H_\alpha^{1}(\Omega), (\widetilde{\psi}, \widetilde{\psi}_{i})_{L_{g'(\psi_\ep)}^2(\Omega)} = 0, i = 1, 2, \cdots, n-1}{\|\widetilde{\psi}\|_{H_\alpha^{1}(\Omega)}^2\over\|\widetilde{\psi}\|_{L_{g'(\psi_\ep)}^2(\Omega)}^2},
\end{align}
where the infimum for $\lambda_i(\ep,\alpha)$ is attained at  $\widetilde\psi_{i} \in H_\alpha^{1}(\Omega)$ and $\|\widetilde{\psi}_i\|_{L_{g'(\psi_\ep)}^2(\Omega)} = 1$, $1\leq i \leq n-1$.
A direct computation of the first variation of $$G_{\ep,\alpha}(\widetilde{\psi})={\|\widetilde{\psi}\|_{H_\alpha^{1}(\Omega)}^2\over\|\widetilde{\psi}\|_{L_{g'(\psi_\ep)}^2(\Omega)}^2}$$
 at $\widetilde{\psi}_{n}$ gives
the corresponding Euler-Lagrangian equation
\begin{align}\label{elip02-alpha}
-\Delta_\alpha \widetilde\psi = \lambda g'(\psi_\ep)\widetilde\psi, \quad \widetilde\psi \in H_\alpha^{1}(\Omega).
\end{align}
To solve the associated eigenvalue problem \eqref{elip02-alpha}, at the first glance we try to use the new variables $(\theta_\ep,\gamma_\ep)$ directly, the transformed equation is however involved and difficult to handle. Instead, we consider  the full perturbation $\psi=\widetilde\psi e^{i\alpha x}$ and by  \eqref{elip02-alpha} it  satisfies
\begin{align}\label{elip02-alpha-full perturbation}
-\Delta(\widetilde\psi e^{i\alpha x}) = \lambda g'(\psi_\ep)(\widetilde\psi e^{i\alpha x}), \quad \widetilde\psi \in H_\alpha^{1}(\Omega).
\end{align}
Note that the full perturbation $\psi$ can also be written as $\widetilde\Psi(\theta_\ep,\gamma_\ep) e^{i\alpha \theta_\ep}$ in the new variables. This motivates us to introduce the following transformation
\begin{align}\label{transformation-modu}
\widetilde\Psi(\theta_\ep,\gamma_\ep)=\widetilde\psi(x,y) e^{i\alpha (x-\theta_\ep)}.
\end{align}
Since
$\widetilde\Psi(\theta_\ep+2\pi,\gamma_\ep)=e^{i\alpha(x(\theta_\ep+2\pi,\gamma_\ep)-\theta_\ep-2\pi)}
\widetilde\psi(x(\theta_\ep+2\pi,\gamma_\ep),y(\theta_\ep+2\pi,\gamma_\ep))=e^{i\alpha( x-\theta_\ep)}
\widetilde\psi(x,y)=\widetilde\Psi(\theta_\ep,\gamma_\ep)$, we know that  $\widetilde\Psi$ is $2\pi$-periodic  in $\theta_\ep$.
Moreover,
\begin{align*}
 \|\widetilde\psi\|_{{H}_\alpha^1(\Omega)}^2=\iint_{\tilde \Omega}\left({1\over1-\gamma_\ep^2}(|\widetilde\Psi_{\theta_\ep}+i\alpha\widetilde\Psi|^2)+(1-\gamma_\ep^2)|\widetilde\Psi_{\gamma_\ep}|^2\right)d \theta_\ep d\gamma_\ep\triangleq\|\widetilde\Psi\|_{Y_{\ep,\alpha}}^2,
\end{align*}
where $Y_{\ep,\alpha}=\{\Psi|\|\Psi\|_{Y_{\ep,\alpha}}<\infty\}$. By \eqref{elip02-alpha-full perturbation},
$\widetilde\Psi$ satisfies the eigenvalue problem
\begin{align}\label{eigenvalue problem-ep-new-alpha}
-\pa_{\gamma_\ep}\left((1-\gamma_\ep^2)\pa_{\gamma_\ep}\widetilde\Psi\right)-{1\over1-\gamma_\ep^2}(\pa_{\theta_\ep}+i\alpha)^2\widetilde\Psi
=2\lambda\widetilde\Psi, \quad \widetilde\Psi \in Y_{\ep,\alpha}.
\end{align}
Since $\widetilde\Psi$ is $2\pi$-periodic  in $\theta_\ep$, we separate it into the Fourier modes.
For the $k$-mode with $k\in\mathbb{Z}$, the eigenvalue problem \eqref{eigenvalue problem-ep-new-alpha}
is
\begin{equation}\label{eigenvalue problem2 non-zero modes varepsilon=0alpha}
-((1-\gamma_\ep^2)\varphi')'+{(k+\alpha)^2\over1-\gamma_\ep^2}\varphi =2\lambda\varphi \quad \text{on}\quad (-1,1),\quad\varphi\in \hat Y_1^\ep,
\end{equation}
where $
\hat Y_1^\ep
$ is defined in \eqref{Y-k-ep-def}.
To
 solve the eigenvalue problem \eqref{eigenvalue problem2 non-zero modes varepsilon=0alpha},
we use the transformation
\begin{align}\label{modulational-transformation}
\varphi=(1-\gamma_\ep^2)^{|k+\alpha|\over2}\phi.
\end{align}
Then \eqref{eigenvalue problem2 non-zero modes varepsilon=0alpha} is transformed to
\begin{equation}\label{eigenvalue problem2 non-zero modes varepsilon=0-transform-alpha}
(1-\gamma_\ep^2)\phi''-2\left(|k+\alpha|+1\right)\gamma_\ep\phi'+\left(-(k+\alpha)^2-|k+\alpha|+2\lambda\right)\phi =0 \quad \text{on}\quad (-1,1),
\end{equation}
where $\varphi\in W_{k+\alpha}=\{\phi|(1-\gamma_\ep^2)^{|k+\alpha|\over2}\phi\in\hat Y_1^\ep\}$.
Let
\begin{align*}
\beta=|k+\alpha|+{1\over2},\quad \lambda={1\over2}\left(n+|k+\alpha|\right)\left(n+|k+\alpha| +1\right)
\end{align*}
in \eqref{Gegenbauer differential equation} and \eqref{eigenvalue problem2 non-zero modes varepsilon=0-transform-alpha}, respectively.
Then the  equation \eqref{eigenvalue problem2 non-zero modes varepsilon=0-transform-alpha} and  the Gegenbauer differential equation \eqref{Gegenbauer differential equation} coincide. All the solutions of  \eqref{Gegenbauer differential equation} in $ L_{\hat g_\beta}^2(-1,1)$ are given by
Gegenbauer polynomials
$
C_n^\beta(\gamma_\ep), n\geq0$, in \eqref{Gegenbauer polynomials}. Since $\beta>{1\over2}$,  similar to \eqref{1-gammacnykep} we have $(1-\gamma_\ep^2)^{|k+\alpha|\over2}C_n^\beta\in\hat Y_1^\ep$ for $n\geq0$.
Thus,
\begin{align*}
&\varphi_{n,k+\alpha}(\gamma_\ep)\triangleq(1-\gamma_\ep^2)^{|k+\alpha|\over2}C_n^\beta(\gamma_\ep)\in\hat Y_1^\ep,\quad\lambda=\lambda_{n,k+\alpha}\triangleq{1\over2}\left(n+|k+\alpha|\right)\left(n+|k+\alpha| +1\right)
\end{align*}
solve \eqref{eigenvalue problem2 non-zero modes varepsilon=0alpha} for $n\geq0$.
Since $\beta>-{1\over2},$ $\{C_n^\beta\}_{n=0}^\infty$ is a complete and  orthogonal basis of $ L_{\hat g_\beta}^2(-1,1)$. This, along with the fact that   $\hat Y_1^\ep$ is embedded in $ L^2(-1,1)$, implies that
 $\{\varphi_{n,k+\alpha}$ $\}_{n=0}^\infty$ is a complete and  orthogonal basis of $\hat Y_1^\ep$ under the inner product of $ L^2(-1,1)$.
Now, we solve the eigenvalue problem \eqref{eigenvalue problem2 non-zero modes varepsilon=0alpha} for the $k$-mode, $k\in\mathbb{Z}$.
\begin{lemma}\label{sol to eigenvalue problem non-zero modes varepsilon=0-alpha-k} Fix $\alpha\in(0,{1\over2}]$ and $k\in\mathbb{Z}.$ Then
all the eigenvalues  of the eigenvalue problem \eqref{eigenvalue problem2 non-zero modes varepsilon=0alpha}  are $\lambda_{n,k+\alpha}={1\over2}\left(n+|k+\alpha|\right)\left(n+|k+\alpha| +1\right)$, $n\geq 0$. For $n\geq0$, the eigenspace associated to $\lambda_{n,{k+\alpha}}$ is $\text{span}\{\varphi_{n,k+\alpha}(\gamma_\ep)\}=\text{span}\{(1-\gamma_\ep^2)^{|k+\alpha|\over2}C_n^{|k+\alpha|+{1\over2}}(\gamma_\ep)\}$.
\end{lemma}

Thus, we get the solutions of
the eigenvalue problem \eqref{eigenvalue problem-ep-new-alpha}.
\begin{Theorem}\label{sol to eigenvalue problem varepsilon=0-pde-alpha} Fix $\alpha\in(0,{1\over2}]$.

$(1)$
All the eigenvalues  of the eigenvalue problem \eqref{eigenvalue problem-ep-new-alpha} are
\begin{align}\label{km-eigenvalues-alpha}
{1\over2}\alpha\left(\alpha+1\right),\quad{1\over2}\left(n\pm\alpha\right)\left( n\pm\alpha+1\right), &\quad  n\geq1.
\end{align}
 For  $n\geq0$,
the eigenspace associated to the eigenvalue  ${1\over2}\left(n+\alpha\right)\left( n+\alpha+1\right)$ is spanned by
\begin{align*}
&(1-\gamma_\ep^2)^{\alpha\over2}C_n^{\alpha+{1\over2}}(\gamma_\ep),\\
 &(1-\gamma_\ep^2)^{j+\alpha\over2}C_{n-j}^{{j+\alpha+{1\over2}}}(\gamma_\ep)e^{ij\theta_\ep},\;\;1\leq j\leq n.
 \end{align*}
 For  $n\geq1$,
the eigenspace associated to the eigenvalue  ${1\over2}\left(n-\alpha\right)\left( n-\alpha+1\right)$ is spanned by
\begin{align*}
 &(1-\gamma_\ep^2)^{j-\alpha\over2}C_{n-j}^{{j-\alpha+{1\over2}}}(\gamma_\ep)e^{-ij\theta_\ep},\;\;1\leq j\leq n.
 \end{align*}

$(2)$ All the eigenvalues  of the associated eigenvalue problem \eqref{elip02-alpha} are
given by \eqref{km-eigenvalues-alpha}.  For $n\geq0$,
the eigenspace associated to the eigenvalue  ${1\over2}\left(n+\alpha\right)\left( n+\alpha+1\right)$ is spanned by
\begin{align*}
&(1-\gamma_\ep^2)^{\alpha\over2}C_n^{\alpha+{1\over2}}(\gamma_\ep)e^{i\alpha(\theta_\ep-x)},\\\nonumber
 &(1-\gamma_\ep^2)^{j+\alpha\over2}C_{n-j}^{{j+\alpha+{1\over2}}}(\gamma_\ep)e^{ij\theta_\ep}e^{i\alpha(\theta_\ep-x)},\;\;1\leq j\leq n.
 \end{align*}
For  $n\geq1$,
the eigenspace associated to the eigenvalue  ${1\over2}\left(n-\alpha\right)\left( n-\alpha+1\right)$ is spanned by
\begin{align*}
 &(1-\gamma_\ep^2)^{j-\alpha\over2}C_{n-j}^{{j-\alpha+{1\over2}}}(\gamma_\ep)e^{-ij\theta_\ep}e^{i\alpha(\theta_\ep-x)},\;\;1\leq j\leq n.
 \end{align*}

 In particular,
the multiplicity of ${1\over2}\left(n+\alpha\right)\left( n+\alpha+1\right)$ is $n+1$ for  $n\geq0$, and the multiplicity of ${1\over2}\left(n-\alpha\right)\left( n-\alpha+1\right)$ is $n$ for  $n\geq1$.
\end{Theorem}

As an application, we give
 the explicit negative directions of  $\tilde A_{\ep,\alpha}$ and $L_{\ep,\alpha}$, confirm that the two operators are non-degenerate, as well as  provide   decompositions of
$H_\alpha^1(\Omega)$ and $L^2_{\frac{1}{g'(\psi_\ep)}}(\Omega)$  associated to  the two operators, respectively.

\begin{Corollary}\label{A-L-dec-e-alpha}
 Let $\alpha\in(0,{1\over2}]$. Then

$(1)$  the negative subspaces of  $H_\alpha^1(\Omega)$ and $L^2_{\frac{1}{g'(\psi_\ep)}}(\Omega)$  associated to $\tilde A_{\ep,\alpha}$ and $L_{\ep,\alpha}$ are
 \begin{align*} H_{\alpha-}^1(\Omega)&=\textup{span}\left\{(1-\gamma_\ep^2)^{\alpha\over2}e^{i\alpha(\theta_\ep-x)},
 (1-\gamma_\ep^2)^{1-\alpha\over2}e^{-i\theta_\ep}e^{i\alpha(\theta_\ep-x)}\right\},\\
 L^2_{\frac{1}{g'(\psi_\ep)}-}(\Omega)&=\textup{span}\left\{g'(\psi_\ep)(1-\gamma_\ep^2)^{\alpha\over2}e^{i\alpha(\theta_\ep-x)},
 g'(\psi_\ep)(1-\gamma_\ep^2)^{1-\alpha\over2}e^{-i\theta_\ep}e^{i\alpha(\theta_\ep-x)}\right\},
\end{align*}
respectively, where $\gamma_\ep=\gamma_\ep(x,y)$ and $\theta_\ep=\theta_\ep(x,y)$.
Thus,  $\dim H_{\alpha-}^1(\Omega)=\dim L^2_{\frac{1}{g'(\psi_\ep)}-}(\Omega)=2$.

$(2)$ $\ker (\tilde A_{\ep,\alpha})=\{0\}$ and $\ker (L_{\ep,\alpha})=\textup{span}\{0\}$.

$(3)$ Let  $ H_{\alpha+}^1(\Omega)= H_{\alpha}^1(\Omega) \ominus H_{\alpha-}^1(\Omega)$ and $L^2_{\frac{1}{g'(\psi_\ep)}+}(\Omega)=L^2_{\frac{1}{g'(\psi_\ep)}}(\Omega) \ominus L^2_{\frac{1}{g'(\psi_\ep)}-}(\Omega)$. Then
\begin{align*}
\langle \tilde A_{\ep,\alpha} \widetilde\psi,\widetilde\psi\rangle \geq \left(1-{2\over (\alpha+1)(\alpha+2)}\right) \| \widetilde\psi\|_{H_{\alpha}^1(\Omega)}^2, \quad\forall \widetilde\psi\in H_{\alpha+}^1(\Omega),
\end{align*}
and there exists $\delta>0$ such that
\begin{align*}
\langle L_{\ep,\alpha} \widetilde\omega,\widetilde\omega\rangle \geq \delta \| \widetilde\omega\|_{L^2_{\frac{1}{g'(\psi_\ep)}}(\Omega)}^2, \quad \forall\; \widetilde\omega\in L^2_{\frac{1}{g'(\psi_\ep)}+}(\Omega).
\end{align*}
\end{Corollary}
\begin{proof}
The proof is essentially due to the following three facts based on Theorem \ref{sol to eigenvalue problem varepsilon=0-pde-alpha}. First, the only eigenvalues, which are less than $1$, of \eqref{elip02-alpha} are ${1\over 2}\alpha(\alpha+1)$ and ${1\over 2}(1-\alpha)(2-\alpha)$. Second, $1$ is not an eigenvalue of \eqref{elip02-alpha}.
Finally, the minimal eigenvalue, which is larger than $1$, is ${1\over2}(1+\alpha)(2+\alpha)$.
\end{proof}
\subsection{A modulational  instability criterion}
Noting that $J_{\ep, \alpha}$ and  $L_{\ep, \alpha}$ are complex operators, we reformulate the linear modulational  problem in the real operators so that we can apply the index formula \eqref{index-formula-neg} for
the real separable Hamiltonian systems.

Let
\begin{align}\label{omega-triangle functions}
\omega(x,y)=\cos(\alpha x)\omega_1(x,y)+\sin(\alpha x)\omega_2(x,y),
\end{align}
where $\omega_1,\omega_2\in L^2_{\frac{1}{g'(\psi_\ep)}}(\Omega)$ are real-valued functions. We decompose
\begin{align*}
(-\Delta_\alpha)^{-1}=(-\Delta_\alpha)_1^{-1}+i(-\Delta_\alpha)_2^{-1},\quad(-\Delta_{-\alpha})^{-1}=(-\Delta_{\alpha})_1^{-1}-i(-\Delta_{\alpha})_2^{-1},
\end{align*}
where
\begin{align*}
(-\Delta_\alpha)_1^{-1}={1\over2} \left((-\Delta_\alpha)^{-1}+(-\Delta_{-\alpha})^{-1}\right),\quad
(-\Delta_{\alpha})_2^{-1}=-{i\over2} \left((-\Delta_\alpha)^{-1}-(-\Delta_{-\alpha})^{-1}\right).
\end{align*}
Here, $(-\Delta_\alpha)_1^{-1}$ is self-dual and $(-\Delta_{\alpha})_2^{-1}$ is anti-self-dual.
Since $\overline{(-\Delta_\alpha)^{-1}}=(-\Delta_{-\alpha})^{-1}$, $(-\Delta_\alpha)_1^{-1}$ and
$(-\Delta_{\alpha})_2^{-1}$ map real functions to real ones. By
\begin{align}\label{omega-i-1-2}
\omega={e^{i\alpha x}\over2}(\omega_1-i\omega_2)+{e^{-i\alpha x}\over 2}(\omega_1+i\omega_2),
\end{align}
we have
\begin{align}\nonumber
(-\Delta)^{-1}\omega=&\cos(\alpha x)\left((-\Delta_\alpha)_1^{-1}\omega_1+(-\Delta_\alpha)_2^{-1}\omega_2\right)\\\label{omega-neg-lap-triangle functions}
&+\sin(\alpha x)\left((-\Delta_\alpha)_1^{-1}\omega_2-(-\Delta_{\alpha})_2^{-1}\omega_1\right),
\end{align}
and
\begin{align}\nonumber
g'(\psi_\epsilon)\vec{u}_\epsilon\cdot\nabla \omega=&\cos(\alpha x)(g'(\psi_\epsilon)\vec{u}_\epsilon\cdot\nabla \omega_1+\alpha g'(\psi_\epsilon)u_{\ep,1}\omega_2)\\\label{g-der-nabla-triangle functions}
&+\sin(\alpha x)(g'(\psi_\epsilon)\vec{u}_\epsilon\cdot\nabla\omega_2-\alpha g'(\psi_\epsilon)u_{\ep,1}\omega_1).
\end{align}
We define the operators
\begin{align*}
\hat{J}_{\ep,\alpha}=&\left( \begin{array}{cc} g'(\psi_\epsilon)\vec{u}_\epsilon\cdot\nabla & \alpha g'(\psi_\epsilon)u_{\ep,1}\\ -\alpha g'(\psi_\epsilon)u_{\ep,1} & g'(\psi_\epsilon)\vec{u}_\epsilon\cdot\nabla \end{array} \right): \left(L^2_{g'(\psi_\ep)}(\Omega)\right)^2 \supset D(\hat{J}_{\epsilon, \alpha}) \to \left(L^2_{\frac{1}{g'(\psi_\ep)}}(\Omega)\right)^2,\\
 \hat {L}_{\epsilon, \alpha} =&
\left( \begin{array}{cc} \frac {1} {g'(\psi_\epsilon)} - (-\Delta_\alpha)_1^{-1} & -(-\Delta_\alpha)_2^{-1} \\ (-\Delta_\alpha)_2^{-1} & \frac {1} {g'(\psi_\epsilon)} - (-\Delta_\alpha)_1^{-1}  \end{array} \right)
 :\left(L^2_{\frac{1}{g'(\psi_\ep)}}(\Omega)\right)^2\to \left(L^2_{g'(\psi_\ep)}(\Omega)\right)^2.
\end{align*}
Then they are real operators,  $\hat{J}_{\ep,\alpha}$ is anti-self-dual and $\hat {L}_{\epsilon, \alpha}$ is self-dual. By \eqref{omega-triangle functions}, \eqref{omega-neg-lap-triangle functions} and \eqref{g-der-nabla-triangle functions}, $J_\epsilon L_\epsilon$ and $\hat{J}_{\ep,\alpha} \hat {L}_{\epsilon, \alpha}$ are related by
\begin{align*}
J_\epsilon L_\epsilon\omega=(\cos(\alpha x),\;\sin(\alpha x))\hat{J}_{\ep,\alpha} \hat {L}_{\epsilon, \alpha}\left( \begin{array}{cc} \omega_1 \\ \omega_2  \end{array} \right).
\end{align*}
By \eqref{omega-i-1-2}-\eqref{g-der-nabla-triangle functions},
 the  complex operators $J_{\ep,\alpha}, L_{\ep,\alpha}$ and the real  operators  $\hat{J}_{\ep,\alpha}, \hat {L}_{\epsilon, \alpha}$ are related by
\begin{align}\label{hat-J-L-alpha1}
\hat{J}_{\ep,\alpha}=&M^{-1}\left( \begin{array}{cc}J_{\ep,\alpha}  &0 \\0 & J_{\ep,-\alpha} \end{array} \right)M,\;\;
\hat{L}_{\ep,\alpha}=M^{-1}\left( \begin{array}{cc}L_{\ep,\alpha}  &0 \\0 & L_{\ep,-\alpha} \end{array} \right)M,\\\label{hat-J-L-alpha2}
\hat{J}_{\ep,\alpha}\hat{L}_{\ep,\alpha}=&M^{-1}\left( \begin{array}{cc}J_{\ep,\alpha}L_{\ep,\alpha}  &0 \\0 & J_{\ep,-\alpha}L_{\ep,-\alpha} \end{array} \right)M,
\end{align}
where
\begin{align*}
M={1\over2}\left( \begin{array}{cc}1  &-i \\1 & i \end{array} \right).
\end{align*}
By \eqref{def-J-ep-al}-\eqref{nabla-alpha-Delta-alpha}, we have
\begin{align}\label{L-JL-conjugate}
\overline{L_{\ep,\alpha}}=L_{\ep,-\alpha},\quad  \overline{J_{\ep,\alpha}L_{\ep,\alpha}}=J_{\ep,-\alpha}L_{\ep,-\alpha}.
\end{align}
By  \eqref{hat-J-L-alpha1} and \eqref{L-JL-conjugate}, we have
\begin{align*}
n^-(\hat L_{\ep,\alpha})=n^-( L_{\ep,\alpha})+n^-( L_{\ep,-\alpha})=2n^-(L_{\ep,\alpha}).
\end{align*}
For the real operator $\hat{J}_{\ep,\alpha}\hat{L}_{\ep,\alpha}$, let $k_{r,\ep,\alpha}, k_{c,\ep,\alpha}, k_{i,\ep,\alpha}^{\leq0}, k_{0,\ep,\alpha}^{\leq0}$ be the indices defined similarly as in Lemma
\ref{theorem-index}.
For the complex operator $J_{\ep,\alpha}L_{\ep,\alpha}$,
let $\tilde k_{r,\ep,\alpha}$ be  the sum of algebraic multiplicities of positive eigenvalues of $J_{\ep,\alpha}L_{\ep,\alpha}$,
 $\tilde k_{c,\ep,\alpha}$
be the sum of algebraic multiplicities of eigenvalues of $J_{\ep,\alpha}L_{\ep,\alpha}$ in the first
and the fourth quadrants,
$\tilde k_{i,\ep,\alpha}^{\leq0}$ be  the total number of non-positive dimensions
of $\langle L_{\ep,\alpha}\cdot,\cdot\rangle$  restricted to the  generalized eigenspaces of non-zero
pure imaginary eigenvalues of $J_{\ep,\alpha}L_{\ep,\alpha}$,
  and $\tilde k_{0,\ep,\alpha}^{\leq0}$ be the number of non-positive directions of $\langle L_{\ep,\alpha}\cdot, \cdot \rangle$ restricted to the generalized kernel of $J_{\ep,\alpha}L_{\ep,\alpha}$ modulo $\ker L_{\ep,\alpha}$.
By  \eqref{hat-J-L-alpha2}-\eqref{L-JL-conjugate}, we have
\begin{align}\label{index-k-tilde-k}
k_{r,\ep,\alpha}=2\tilde k_{r,\ep,\alpha}, \;\; k_{c,\ep,\alpha}=\tilde k_{c,\ep,\alpha},\;\;k_{i,\ep,\alpha}^{\leq0}=\tilde k_{i,\ep,\alpha}^{\leq0},\;\;k_{0,\ep,\alpha}^{\leq0}=2\tilde k_{0,\ep,\alpha}^{\leq0}.
\end{align}
Applying Lemma \ref{theorem-index} to the real operators  $\hat{J}_{\ep,\alpha}$ and $\hat{L}_{\ep,\alpha}$, by Corollary \ref{A-L-dec-e-alpha} we have
\begin{align}\label{index-formula-real-operator}
k_{r,\ep,\alpha}+ 2k_{c,\ep,\alpha}+2k_{i,\ep,\alpha}^{\leq0}+k_{0,\ep,\alpha}^{\leq0}=2n^-\left(\hat{L}_{\ep,\alpha}\right)=4.
\end{align}
Combining \eqref{index-k-tilde-k} and \eqref{index-formula-real-operator}, we get the index formula for the complex operators  $J_{\ep,\alpha}$ and $L_{\ep,\alpha}$:
\begin{align*}
\tilde k_{r,\ep,\alpha}+ \tilde k_{c,\ep,\alpha}+\tilde k_{i,\ep,\alpha}^{\leq0}+\tilde k_{0,\ep,\alpha}^{\leq0}=n^-\left(L_{\ep,\alpha}\right)=2.
\end{align*}
To study the linear modulational instability, one may try to prove that $\tilde k_{i,\ep,\alpha}^{\leq0}+\tilde k_{0,\ep,\alpha}^{\leq0}\leq 1$, it is however difficult to compute the two indices for the eigenvalues  of $J_{\ep,\alpha}L_{\ep,\alpha}$ in the imaginary axis.
Here, we use the separable Hamiltonian structure of the real operator $\hat{J}_{\ep,\alpha}\hat{L}_{\ep,\alpha}$.
Define two spaces
\begin{align*}
X_{\alpha, e} =& \left\{ \left( \begin{array}{c} \omega_1 \\ \omega_2 \end{array} \right) \in \left(L^2_{\frac{1}{g'(\psi_\ep)}}(\Omega)\right)^2 \bigg| \text{both } \omega_1 \text{ and } \omega_2 \text{ are even  in }y \right\},\\
X_{\alpha, o} =& \left\{  \left( \begin{array}{c} \omega_1 \\ \omega_2 \end{array} \right) \in \left( L^2_{\frac{1}{g'(\psi_\ep)}}(\Omega)\right)^2\bigg| \text{both } \omega_1 \text{ and } \omega_2 \text{ are odd in }y \right\}.
\end{align*}
Then $X_{\alpha, e}$ and  $X_{\alpha, o}$ are  Hilbert spaces.
The dual space of  $X_{\alpha, o}$ (resp. $X_{\alpha, e}$) restricted to the class of odd (resp. even) functions is denoted by $X_{\alpha, o}^*$ (resp. $X_{\alpha, e}^*$).
Let
\begin{align*}\hat{B}_\alpha = \hat{J}_{\ep,\alpha}|_{X_{\alpha, o}^*},\;\; \hat L_{\alpha,o} = \hat {L}_{\epsilon, \alpha}|_{X_{\alpha, o}},\;\;
\hat{L}_{\alpha,e} = \hat {L}_{\epsilon, \alpha}|_{X_{\alpha, e}}.
 \end{align*}
 Then
 \begin{align*}\hat{B}_\alpha: X_{\alpha, o}^*\supset D(B_\alpha) \rightarrow X_{\alpha, e},\;\;
 \hat L_{\alpha,o}:X_{\alpha, o} \rightarrow X_{\alpha, o}^*,\;\;
 \hat L_{\alpha,e} : X_{\alpha, e} \rightarrow X_{\alpha, e}^*.
 \end{align*}
The dual operator of $\hat{B}_\alpha$ is
$$\hat{B}_\alpha'= \left( \begin{array}{cc} -g'(\psi_\epsilon)\vec{u}_\epsilon\cdot\nabla & -\alpha g'(\psi_\epsilon)u_{\ep,1}\\ \alpha g'(\psi_\epsilon)u_{\ep,1} & -g'(\psi_\epsilon)\vec{u}_\epsilon\cdot\nabla \end{array} \right) : X_{\alpha, e}^* \supset D(B'_\alpha) \rightarrow X_{\alpha, o}.$$
We decompose $\left( \omega_1, \omega_2  \right)^T \in \left(L^2_{\frac{1}{g'(\psi_\ep)}}(\Omega)\right)^2$ as $ \left( \omega_{1,e}, \omega_{2,e}, \omega_{1,o}, \omega_{2,o} \right)^T $ such that $\left( \omega_1, \omega_2  \right)^T = \left( \omega_{1,e}, \omega_{2,e})+(\omega_{1,o},  \omega_{2,o} \right)^T$, where
$\vec{\omega}_{e}\triangleq\left( \omega_{1,e}, \omega_{2,e} \right)^T \in X_{\alpha, e}$ and $\vec{\omega}_{o}\triangleq\left(  \omega_{1,o}, \omega_{2,o} \right)^T  \in X_{\alpha, o}$. Then the linearized equation $\partial_t(\omega_1,\omega_2)^T=\hat{J}_{\ep,\alpha} \hat {L}_{\epsilon, \alpha}(\omega_1,\omega_2)^T$ can be written as  the following separable Hamiltonian system
\begin{align}\label{sep-hamiltonian-alpha}
\partial_t \left( \begin{array}{c}\vec{\omega}_{e} \\ \vec{\omega}_{o} \end{array} \right) = \left( \begin{array}{cc} 0 & \hat{B}_\alpha \\ -\hat{B}'_\alpha & 0 \end{array} \right)\left( \begin{array}{cc} \hat{L}_{\alpha,e} & 0 \\ 0 & \hat{L}_{\alpha,o} \end{array} \right) \left( \begin{array}{c} \vec{\omega}_{e} \\ \vec{\omega}_{o}\end{array} \right).
\end{align}
To apply the index formula  \eqref{index-formula-neg},
 we need to verify
{\textbf{(G1-4)}}  in Lemma \ref{indice-theorem-sep} for \eqref{sep-hamiltonian-alpha}. {\textbf{(G1)}} can be verified in a similar way as for \eqref{sep-hamiltonian}. Using \eqref{hat-J-L-alpha1},
{\textbf{(G2-4)}} can be verified by \eqref{L-alpha-bounded} and Corollary \ref{A-L-dec-e-alpha}. Then by Lemma \ref{indice-theorem-sep}, the number of unstable modes  for \eqref{sep-hamiltonian-alpha} is
$k_{r,\ep,\alpha}=n^-\left(\hat{L}_{\alpha,e}|_{\overline{R(\hat{B}_\alpha)}} \right)$ and $k_{c,\ep,\alpha}=0$. By \eqref{index-k-tilde-k} and  \eqref{hat-J-L-alpha1},
we have
\begin{align*}
2\tilde k_{r,\ep,\alpha}=k_{r,\ep,\alpha}=n^-\left(\hat{L}_{\alpha,e}|_{\overline{R(\hat{B}_\alpha)}} \right)=2n^-\left(L_{\alpha,e}|_{\overline{R(B_\alpha)}} \right)\Longrightarrow\tilde k_{r,\ep,\alpha}=n^-\left(L_{\alpha,e}|_{\overline{R(B_\alpha)}} \right),
\end{align*}
and
\begin{align}\label{index-formula-for-modulational-instabilitykc}\tilde k_{c,\ep,\alpha}=k_{c,\ep,\alpha}=0,
\end{align}
 where
 \begin{align}\label{L-alpha-e B-alpha}
 &L_{\alpha,e}=L_{\ep,\alpha}|_{L^2_{\frac{1}{g'(\psi_\ep)},e}(\Omega)},\quad B_\alpha=J_{\ep,\alpha}|_{L^2_{{g'(\psi_\ep)},o}(\Omega)}.
 \end{align}
Here, we recall that
$L^2_{\frac{1}{g'(\psi_\ep)},e}(\Omega)=\left\{\omega\in L^2_{\frac{1}{g'(\psi_\ep)}}(\Omega)\;\bigg|\;\omega \text{ is even in }y\right\},$
$L^2_{\frac{1}{g'(\psi_\ep)},o}(\Omega)=\bigg\{\omega\in$ $ L^2_{\frac{1}{g'(\psi_\ep)}}(\Omega)\;\bigg|\;\omega  \text{ is odd in } y\bigg\},$
$L^2_{{g'(\psi_\ep)},e}(\Omega)=\bigg\{\omega\in L^2_{{g'(\psi_\ep)}}(\Omega)\;\bigg|\;\omega \text{ is even in }y\bigg\}$
and $L^2_{{g'(\psi_\ep)},o}(\Omega)$ $=\bigg\{\omega\in L^2_{{g'(\psi_\ep)}}(\Omega)\;\bigg|\;\omega \text{ is odd in }y\bigg\}$.

In summary, we have the following criterion for modulational instability of $\omega_\ep$.
\begin{lemma}\label{modulational case:unstable modes}
The number of unstable  modes of $J_{\ep,\alpha}L_{\ep,\alpha}$ is $n^-\left(L_{\alpha,e}|_{\overline{R(B_\alpha)}} \right)$, where $L_{\alpha,e}$ and $B_\alpha$ are defined in  \eqref{L-alpha-e B-alpha}. Consequently, if $n^-\left(L_{\alpha,e}|_{\overline{R(B_\alpha)}} \right)\geq1$, then $\omega_\ep$ is linearly modulationally unstable.
\end{lemma}

Let $ L_e^2(\Omega)=\{\phi\in L^2(\Omega)|\phi \text{ is even in } y\}$. Since the dual space of $ L_e^2(\Omega)$ is restricted to the class of even functions, we have $ L_e^2(\Omega)=(L_e^2(\Omega))^*$. To study $n^-\left(L_{\alpha,e}|_{\overline{R(B_\alpha)}} \right)$,  we define
$\bar{P}_{\alpha,e}$ to be the orthogonal   projection of the space $(L_e^2(\Omega))^*=L_e^2(\Omega)$ on $\ker(\vec{u}_\ep\cdot\nabla_\alpha)$. For $\widetilde{\psi}\in\ker(\vec{u}_\ep\cdot\nabla_\alpha)$, we have
$(\vec{u}_\ep\cdot\nabla)(\widetilde{\psi}e^{i\alpha x})=0$ and thus, $\widetilde{\psi}e^{i\alpha x}|_{\Gamma(\rho)}\equiv c_0$, where $\Gamma(\rho)$ is a connected closed curve of the level set $\{\psi_\ep=\rho\}$.  Recall that $\rho_0$ is defined in \eqref{def-rho0}.
For
$\rho\in[\rho_0,\infty)$, $\Gamma(\rho)$ is in the un-trapped regions. Since  $\widetilde{\psi}(0,y)=c_0=\widetilde{\psi}(2\pi,y)e^{2\alpha\pi i  }$ and $\widetilde{\psi}(0,y)=\widetilde{\psi}(2\pi,y)$, we have
\begin{align}\label{untrapped region-tilde-psi-e-ialphax}
\widetilde{\psi}e^{i\alpha x}|_{\Gamma(\rho)}\equiv c_0=0,
\end{align}
 and thus, $\widetilde{\psi}\equiv0$ in  the un-trapped regions.
For
$\rho\in[-\rho_0,\rho_0)$,  the level set $\{\psi_\ep=\rho\}$ is in the trapped region and it is exactly  one closed  curve $\Gamma(\rho)$.
Let $(X(s;x_0,y_0), Y(s;x_0,y_0))$ be the solution to the equation
\begin{align}\label{characteristic equation}
\begin{cases}
\dot{X}(s)=\partial_y\psi_\ep(X(s), Y(s)),\\
\dot{Y}(s)=-\partial_x\psi_\ep(X(s), Y(s)),
\end{cases}
\end{align}
with the initial data $X(0)=x_0,Y(0)=y_0$, where $(x_0,y_0)\in\Gamma(\rho)$. Then $\psi_{\ep}$ is conserved  along $\Gamma(\rho)$.
Let  $l_\rho$ be the arc length variable on $\Gamma(\rho)$
and $L_\rho(\ep)$ be the length of  $\Gamma(\rho)$.
 Along the trajectory, the particle solves
\begin{align*}
{d l_\rho(s)\over ds}=|\nabla\psi_\ep |(X(s;x_0,y_0),Y(s;x_0,y_0))
\end{align*}
 and the period of the particle motion is
\begin{align*}
T_\ep(\rho)=\int_0^{L_\rho(\ep)}{1\over|\nabla\psi_\ep|}dl_\rho.
\end{align*}
Define the action and angle variables by
\begin{align*}
I_\ep(\rho)&={1\over 2\pi}\int_{-\rho_0}^\rho\left(\int_0^{L_{\tilde\rho}(\ep)}{1\over|\nabla\psi_\ep|}dl_{\tilde \rho}\right)d\tilde \rho,\;\;
\theta_\ep={2\pi\over T_\ep(\rho)} \int_{0}^{l_\rho}{1\over |\nabla\psi_\ep|}dl_{\tilde \rho}.
\end{align*}
 Then $I_\ep$ is increasing on $\rho\in[-\rho_0,\rho_0)$ and  $0\leq \theta_\ep\leq  2\pi$. We define the inverse map of $I_\ep(\rho)$ by $\rho(I_\ep)$.
Define the frequency by
\begin{align*}
\vartheta_\ep(I_\ep)={2\pi\over T_\ep(\rho(I_\ep))}.
\end{align*}
The action-angle transform $(x,y)\rightarrow(I_\ep,\theta_\ep)$ is a smooth diffeomorphism with Jacobian $-1$.
 The  characteristic equation \eqref{characteristic equation} takes the form
\begin{align*}
\begin{cases}
\dot{I}_\ep=0,\\
\dot{\theta}_\ep=\vartheta_\ep(I_\ep).
\end{cases}
\end{align*}
The transport operator $\vec{u}_{\ep}\cdot\nabla$ takes the form
\begin{align*}
\vec{u}_{\ep}\cdot\nabla=\partial_y\psi_\ep\partial_x-\partial_x\psi_\ep\partial_y=
\vartheta_\ep(I_\ep)\partial_{\theta_\ep}.
\end{align*}
Thus, $\ker(\vartheta_\ep(I_\ep)\partial_{\theta_\ep})=\{f(I_\ep):f(I_\ep)\in L^2(\Omega)\;\text{and}\; f(I_\ep(\rho))=0\text{ for }\rho\in[\rho_0,\infty)\}=\{h(\psi_\ep):h(\psi_\ep)\in L^2(\Omega)\;\text{and}\; h(\psi_\ep)=0\text{ for }\psi_\ep\geq\rho_0\}=\ker(\vec{u}_{\ep}\cdot\nabla)$. Thus, $\ker(\vec{u}_{\ep}\cdot\nabla_\alpha)=\{h(\psi_\ep)e^{-i\alpha x}:h(\psi_\ep)\in L^2(\Omega)\;\text{and}\; h(\psi_\ep)=0\text{ for }\psi_\ep\geq\rho_0\}$. Let $\phi\in L_e^2(\Omega)$.
For any $\varphi=h(\psi_\ep)e^{-i\alpha x}\in \ker(\vec{u}_{\ep}\cdot\nabla_\alpha)$, we have
\begin{align*}
(\phi-\bar{P}_{\alpha,e}\phi,\varphi)_{L^2(\Omega)}=&\iint_{\Omega}(\phi-\bar{P}_{\alpha,e}\phi)\overline{ h(\psi_\ep)}e^{i\alpha x} dxdy\\
=&\int_{-\rho_0}^{\rho_0}\left(\oint_{\Gamma(\rho)}{(\phi-\bar{P}_{\alpha,e}\phi)\overline{ h(\psi_\ep)}e^{i\alpha x}\over|\nabla\psi_\ep|}\right)d\rho\\
=&\int_{-\rho_0}^{\rho_0}\overline{ h(\rho)}\left(\oint_{\Gamma(\rho)}{\phi e^{i\alpha x}\over|\nabla\psi_\ep|}-(\bar{P}_{\alpha,e}\phi e^{i\alpha x})|_{\Gamma(\rho)}\oint_{\Gamma(\rho)}{1\over|\nabla\psi_\ep|}
\right)d\rho=0,
\end{align*}
where we used $\bar{P}_{\alpha,e}\phi e^{i\alpha x}$ takes constant on $\Gamma(\rho)$ since $\bar{P}_{\alpha,e}\phi \in \ker(\vec{u}_{\ep}\cdot\nabla_\alpha)$.
This gives
\begin{align*}
(\bar{P}_{\alpha,e}\phi) |_{\Gamma(\rho)}=
\left\{ \begin{array}{lll} {\oint_{\Gamma(\rho)}{\phi e^{i\alpha x}\over|\nabla\psi_\ep|}\over\oint_{\Gamma(\rho)}{1\over|\nabla\psi_\ep|}}e^{-i\alpha x} &\mbox{ for } & \rho \in [-\rho_0, \rho_0 ),\\
 0 &\mbox{ for } & \rho\in [\rho_0, \infty).
 \end{array} \right.
\end{align*}
 It induces  a   projection
$\hat{P}_{\alpha,e}$ of $(L_{{1\over g'(\psi_\ep)},e}^2(\Omega))^*=L_{g'(\psi_\ep),e}^2(\Omega)$  on $\ker (B_\alpha')$ by $\hat{P}_{\alpha,e}=(S_e')^{-1}\bar{P}_{\alpha,e} S_e'$, where
$S_e: L_e^2(\Omega) \rightarrow L^2_{\frac{1}{g'(\psi_\ep)},e}(\Omega),  S_e\omega = g'(\psi_\ep)^{1/2}\omega$
defines an isometry. The dual space $(L_{{1\over g'(\psi_\ep)},e}^2(\Omega))^*$ is restricted to the class of even functions.
Noting that $L^2_{{g'(\psi_\ep)},e}(\Omega)= (L^2_{\frac{1}{g'(\psi_\ep)},e}$ $(\Omega))^*$, we  define the operator
$$\hat{A}_{\alpha,e} = - \Delta_\alpha - g'(\psi_\ep)(I - \hat{P}_{\alpha,e}): L^2_{{g'(\psi_\ep)},e}(\Omega) \rightarrow L^2_{\frac{1}{g'(\psi_\ep)},e}(\Omega).$$
Similar to Lemma \ref{L e-hat A}, we can estimate $n^-\left( L_{\alpha,e} |_{\overline{R(B_\alpha)}} \right)$ by studying the negative directions of $\langle\hat{A}_{\alpha,e}\cdot,\cdot\rangle$.
\begin{lemma}\label{L e-hat A-alpha}
$$n^-\left( L_{\alpha,e} |_{\overline{R(B_\alpha)}} \right) = n^-\left(\hat{A}_{\alpha,e}\right).$$
In particular, the number of unstable  modes of $J_{\ep,\alpha}L_{\ep,\alpha}$ is $n^-\left(\hat{A}_{\alpha,e}\right)$. If $n^-\left(\hat{A}_{\alpha,e}\right)\geq1$, then $\omega_\ep$ is linearly modulationally unstable.
\end{lemma}
\subsection{Proof of modulational instability}
 To study the linear modulational instability of the Kelvin--Stuart vortex $\omega_\ep$, we construct the test function to be
  \begin{align}\label{test-function-modulational-instability}\widetilde\psi_{\ep,\alpha} =(1-\gamma_\ep^2)^{\alpha\over2}e^{i\alpha(\theta_\ep-x)}\in L^2_{{g'(\psi_\ep)},e}(\Omega),
   \end{align}
 which is an eigenfunction of the eigenvalue ${1\over2}\alpha\left( \alpha+1\right)$ for the associated eigenvalue problem \eqref{elip02-alpha} in Theorem \ref{sol to eigenvalue problem varepsilon=0-pde-alpha}, and confirm that
 \begin{align*}
\langle\hat{A}_{\alpha,e} \widetilde\psi_{\ep,\alpha}, \widetilde\psi_{\ep,\alpha} \rangle=  b_{\alpha, 1}(\widetilde\psi_{\ep,\alpha}) + b_{\alpha, 2}(\widetilde\psi_{\ep,\alpha}) < 0,
\end{align*}
where
\begin{align}\label{func-b1-alpha}
  b_{\alpha, 1}(\widetilde\psi_{\ep,\alpha}) = \iint_{\Omega} \left(|\nabla_\alpha \widetilde\psi_{\ep,\alpha}|^2  - g'(\psi_\ep) |\widetilde\psi_{\ep,\alpha}|^2 \right)dxdy,
\end{align}
\begin{align}\label{func-b2-alpha}
 b_{\alpha, 2}(\widetilde\psi_{\ep,\alpha}) =  \iint_{\Omega} g'(\psi_\ep)( \hat{P}_{\alpha,e}\widetilde\psi_{\ep,\alpha})^2 dxdy =  \int_{-\rho_0}^{\rho_0} g'(\rho) \frac{\left|\oint_{\Gamma (\rho)} \frac{\widetilde\psi_{\ep,\alpha} e^{i \alpha x}}{|\nabla \psi_\ep|}\right|^2}{\oint_{\Gamma (\rho)} \frac{1}{|\nabla \psi_\ep|}} d\rho,
 \end{align}
 and $\rho_0 $ is defined in \eqref{def-rho0}.
Here,  $\Gamma (\rho)=\{\psi_\ep=\rho\}$ for $\rho\in[-\rho_0,\rho_0)$. Since $\widetilde\psi_{\ep,\alpha} $ is an eigenfunction of the eigenvalue ${1\over2}\alpha\left( \alpha+1\right)$ for \eqref{elip02-alpha}, we have
\begin{align}\label{func-b1-alpha-2}
  b_{\alpha, 1}(\widetilde\psi_{\ep,\alpha}) = 2\pi(\alpha(\alpha+1)-2)\int_{-1}^1(1-\gamma_\ep^2)^\alpha d\gamma_\ep.
\end{align}
To compute $ b_{\alpha, 2}(\widetilde\psi_{\ep,\alpha})$, we  convert the curve integrals to  definite integrals.
Note that $\Gamma(\rho) = \{(x, y) \in \Omega| \psi_\ep(x,y) = \rho\} $ is a closed level curve in the trapped region for $\rho\in(-\rho_0, \rho_0]$.
We divide $\Gamma(\rho)$ into two parts, namely, the upper part
\begin{align*}
\Gamma_{+} (\rho)
& = \{(x, y) \in \mathbb{T}_{2\pi} \times \mathbb{R}  \;|\;  \psi_\ep(x,y) = \rho,  y \geq 0\},
\end{align*}
and the lower part
\begin{align*}
\Gamma_{-} (\rho) = \{(x, y) \in \mathbb{T}_{2\pi} \times \mathbb{R}\;|\;  \psi_\ep(x,y) = \rho,  y < 0\}.
\end{align*}
 Using $x$ as the parameter,
we  represent $\Gamma_{+}(\rho)$ and $\Gamma_{-}(\rho)$ as follows:
$$\vec{r}_{+} (x) = (x, \cosh^{-1}(\sqrt{1-\ep^2} e^{\rho} - \ep \cos(x))), \quad x \in [ x_0, 2\pi - x_0], $$
and
$$\vec{r}_{-} (x) = (x, -\cosh^{-1}(\sqrt{1-\ep^2} e^{\rho} - \ep \cos(x))), \quad x \in ( x_0, 2\pi - x_0), $$
respectively. Here, $x_0 = \arccos\left( \frac{\sqrt{1-\ep^2} e^\rho - 1}{\ep} \right) $ is the point on $[0, \pi]$ such that $\psi_\ep(x_0,0)=\rho$.
Moreover, we have
\begin{align}\label{dr+}
\left|\frac{d \vec{r}_{\pm}(x)}{dx}\right| = \sqrt{ 1 + \left(\frac{\ep \sin(x)}{\sinh(y(x))}\right)^2},
\end{align}
where
\begin{align}\label{sinhyx}
 \sinh(y(x))=\sqrt{( \sqrt{1-\ep^2} e^{\rho} - \ep \cos(x) )^2 - 1} \end{align}
 and
\begin{align*}
 y(x)=\cosh^{-1}(\sqrt{1-\ep^2} e^{\rho} - \ep \cos(x)).
 \end{align*}
Noting that $\sinh(y(x_0))=\sinh(y(2\pi-x_0))=0$, $\left|\frac{d \vec{r}_{\pm}(x)}{dx}\right|$ is singular near $x_0$ and $2\pi-x_0$. To avoid the singularity, one might represent $\Gamma(\rho)$ in terms of the parameter $y$ near the two points $(x_0,0)$ and $(2\pi-x_0,0)$ if necessary.
Then we  represent $\left| \nabla \psi_\ep \right| $ and $\widetilde{\psi}_{\ep,\alpha}$ on $\Gamma_{+}(\rho)$ and $\Gamma_{-}(\rho)$ in terms of the parameter $x$. Since
$\psi_\ep(x,y) = \rho$, we have $\cosh(y) + \ep \cos(x) = e^\rho \sqrt{1-\ep^2}$. So
\begin{align}\label{tilde psi-p-x}
\left| \nabla \psi_\ep \right|
= & \left| \left( - \frac{\ep \sin(x)}{e^\rho \sqrt{1-\ep^2}}, \frac{\sinh(y)}{e^\rho \sqrt{1-\ep^2}} \right) \right|
=  \frac{\sqrt{\ep^2 \sin^2(x) + \sinh^2(y)}}{e^\rho \sqrt{1-\ep^2}}.
\end{align}
By \eqref{dr+}-\eqref{tilde psi-p-x}, we have
\begin{align}\nonumber
&\oint_{\Gamma (\rho)} \frac{1}{|\nabla \psi_\ep|}=2\oint_{\Gamma_+ (\rho)} \frac{1}{|\nabla \psi_\ep|}=2\int_{x_0}^{2\pi-x_0}\frac{1}{|\nabla \psi_\ep|}\left|\frac{d \vec{r}_{+}(x)}{dx}\right|dx\\\label{Gamma-rho-nabla-psi1}
=&2\int_{x_0}^{2\pi-x_0}\frac{e^\rho \sqrt{1-\ep^2}}{ \sinh(y(x))}dx
=2e^\rho \sqrt{1-\ep^2}\int_{x_0}^{2\pi-x_0}\frac{1}{ \sqrt{(e^\rho \sqrt{1-\ep^2}-\ep \cos(x))^2-1} }dx
\end{align}
and
\begin{align}\nonumber
&\oint_{\Gamma (\rho)} \frac{\widetilde\psi_{\ep,\alpha} e^{i \alpha x}}{|\nabla \psi_\ep|}=2\oint_{\Gamma_+ (\rho)} \frac{\widetilde\psi_{\ep,\alpha} e^{i \alpha x}}{|\nabla \psi_\ep|}=2\int_{x_0}^{2\pi-x_0}\frac{e^\rho \sqrt{1-\ep^2}(1-\gamma_\ep^2)^{\alpha\over2}e^{i\alpha\theta_\ep}}{ \sinh(y(x))}dx\\\label{Gamma-rho-nabla-tilde-psi2}
=&2e^\rho \sqrt{1-\ep^2}\int_{x_0}^{2\pi-x_0}\frac{(1-\gamma_\ep^2)^{\alpha\over2}(\cos(\alpha\theta_\ep)+i\sin(\alpha\theta_\ep))}{ \sqrt{(e^\rho \sqrt{1-\ep^2}-\ep \cos(x))^2-1} }dx,
\end{align}
where $x_0 = \arccos\left( \frac{\sqrt{1-\ep^2} e^\rho - 1}{\ep} \right)$,
\begin{align*}
1-\gamma_\ep^2&=1-\sinh^2(y)e^{-2\rho}=1-\left((e^\rho \sqrt{1-\ep^2}-\ep \cos(x))^2-1\right)e^{-2\rho}
\end{align*}
and
\begin{align}
\label{1-gamma-ep-2-expression}
\cos(\theta_\ep) &= \frac{\xi_\ep}{\sqrt{1-\gep^2}} = \frac{\ep + \sqrt{1-\ep^2} \cos(x) e^{-\rho} }{\sqrt{1 - \left(\left( \sqrt{1-\ep^2} e^{\rho} - \ep \cos(x) \right)^2 - 1\right) e^{-2\rho}}}.
\end{align}
 Note that \eqref{func-b1-alpha-2}, \eqref{func-b2-alpha} and \eqref{Gamma-rho-nabla-psi1}-\eqref{Gamma-rho-nabla-tilde-psi2} give the explicit expression of $\langle\hat{A}_{\alpha,e} \widetilde\psi_{\ep,\alpha}, \widetilde\psi_{\ep,\alpha} \rangle=  b_{\alpha, 1}(\widetilde\psi_{\ep,\alpha}) + b_{\alpha, 2}(\widetilde\psi_{\ep,\alpha})$.
 \begin{figure}[ht]
    \centering
	\includegraphics[width=0.56\textwidth]{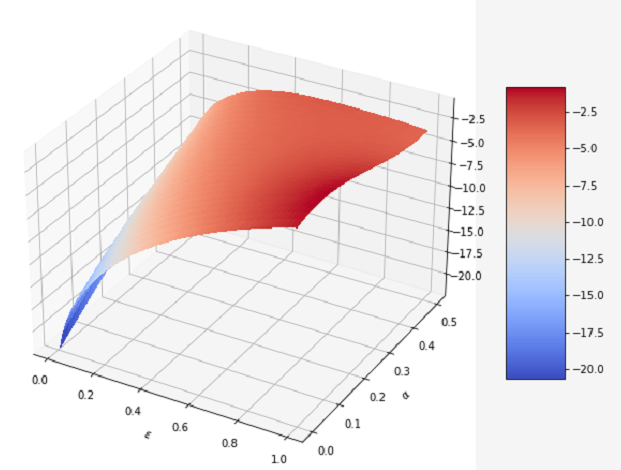}
	\caption{The value of $\langle\hat{A}_{\alpha,e} \widetilde\psi_{\ep,\alpha}, \widetilde\psi_{\ep,\alpha} \rangle$}
	\label{fig:3rdFig}
\end{figure}
 The integrals in the expression are computable, and  we compute $\langle\hat{A}_{\alpha,e} \widetilde\psi_{\ep,\alpha}, \widetilde\psi_{\ep,\alpha} \rangle$ as a real-valued function of $(\alpha,\epsilon)$ by Python.
The values of  $\langle\hat{A}_{\alpha,e} \widetilde\psi_{\ep,\alpha}, \widetilde\psi_{\ep,\alpha} \rangle$ are given in Figure \ref{fig:3rdFig}, and it reveals
 that
 \begin{align}\label{test-function-modulational-instability-quadratic form}
 \max_{\alpha\in(0,{1\over2}],\ep\in[0,1)}\langle\hat{A}_{\alpha,e} \widetilde\psi_{\ep,\alpha}, \widetilde\psi_{\ep,\alpha} \rangle=\langle\hat{A}_{\alpha,e} \widetilde\psi_{\ep,\alpha}, \widetilde\psi_{\ep,\alpha} \rangle|_{\alpha=0.01,\ep=0.99}=-0.78<0.
 \end{align}
Now, we are in a position to prove linear modulational instability for the family of steady states $\omega_\ep$, $\ep\in[0,1)$.

\begin{proof}[Proof of Theorem \ref{main result3-modulation-unstable}]
With the test function  $\widetilde\psi_{\ep,\alpha}$ defined in \eqref{test-function-modulational-instability},  we infer from \eqref{test-function-modulational-instability-quadratic form} that
$\langle\hat{A}_{\alpha,e} \widetilde\psi_{\ep,\alpha}, \widetilde\psi_{\ep,\alpha} \rangle<0$ for $\alpha\in(0,{1\over2}]$ and $\ep\in[0,1)$. Thus,
the number of unstable  modes of $J_{\ep,\alpha}L_{\ep,\alpha}$ is $n^-\left( L_{\alpha,e} |_{\overline{R(B_\alpha)}} \right) = n^-\left(\hat{A}_{\alpha,e}\right)\geq1$  by Lemma \ref{L e-hat A-alpha}. This proves  linear modulational instability of $\omega_\ep$.
\end{proof}
\begin{remark} \label{modulational-remark}

For the hyperbolic tangent shear flow ($\ep=0$), the trapped region vanishes and by \eqref{untrapped region-tilde-psi-e-ialphax}, we have $\ker(\vec{u}_0\cdot\nabla_\alpha)=\{0\}$ for $\alpha\in(0,{1\over2}]$. Thus,
 $ {\overline{R(B_\alpha)}}={L^2_{\frac{1}{g'(\psi_0)},e}(\Omega)}$. By Corollary $\ref{A-L-dec-e-alpha}$, $n^-( L_{\alpha,e} )|_{\ep=0}=n^-( L_{\ep,\alpha})|_{\ep=0}=2$. We infer from Lemma $\ref{modulational case:unstable modes}$ that for any modulational parameter $\alpha\in(0,{1\over2}]$, the number of   unstable modes in  the shear case is $2$. This also indicates that for fixed $\alpha\in(0,{1\over2}]$, the number of   unstable modes for the Kelvin--Stuart vortex $\omega_\ep$ with $\ep\ll1$ is $2$.

\end{remark}
Finally, we give the relations between multi-periodic instability and modulational instability.
\begin{lemma}\label{multi-periodic-modulation} Let $\ep\in[0,1)$.
$(1)$ If the steady state $\omega_\ep$ is linearly  $2m\pi$-periodic unstable for some $m\geq2$, then there exists an integer  $1\leq\hat l\leq m-1$ such that $\omega_\ep$ is linearly modulationally unstable
for $\alpha={\hat l\over m}$.

$(2)$ If the steady state $\omega_\ep$ is linearly modulationally unstable
for some rational number $\alpha={p\over q}\in (0,{1\over2}]$ with $p,q\in\mathbb{Z}^+$, then $\omega_\ep$ is linearly  $2q\pi$-periodic unstable.
\end{lemma}
\begin{proof} (1) Let ${\lambda}_* $  be an unstable eigenvalue of $J_{\ep,m} L_{\ep,m} $ with an eigenfunction ${\omega}_*\in X_{\ep,m}$. Then
$${\omega}_*(x, y) = \sum_{k \in \mathbb{Z}} e^{\frac{ikx}{m}} \widehat{\omega}_{*,k}(y)  = \sum_{l = 0}^{m-1} e^{\frac{ilx}{m}} \omega_{*, l} (x, y), $$
where
$$\omega_{*, l} (x, y) = \sum_{n \in \mathbb{Z}} e^{inx} \widehat{\omega}_{*,mn + l}(y) \in L^2_{\frac {1} {g'(\psi_\epsilon)}}(\Omega),\quad 0\leq l\leq m-1.$$
Since $J_{\ep,m} L_{\ep,m}{\omega}_*={\lambda}_*{\omega}_*$, we have
\begin{align*}
J_\ep L_\ep {\omega}_{*, 0} + \sum_{l = 1}^{m-1} e^{\frac{ilx}{m}} J_{\ep, \frac{l}{m}}L_{\ep, \frac{l}{m}} {\omega}_{*, l} = {\lambda}_* \left({\omega}_{*, 0} + \sum_{l = 1}^{m-1} e^{\frac{ilx}{m}} {\omega}_{*, l}\right).
\end{align*}
By induction,
\begin{align*}
J_\ep L_\ep {\omega}_{*, 0} = {\lambda}_* {\omega}_{*, 0}\quad\text{and}\quad
J_{\ep, \frac{l}{m}}L_{\ep, \frac{l}{m}} {\omega}_{*, l}  = {\lambda}_* {\omega}_{*, l}  \quad\text{for} \quad l = 1, \cdots, m-1.
\end{align*}
By Theorem \ref{main result1-co-periodic perturbations},
 $\omega_\ep$ is spectrally stable for co-periodic perturbations. This, along with $Re({\lambda}_*)>0$,  implies that
${\omega}_{*, 0}\equiv 0$. Thus, there exists $1 \leq  \hat{l} \leq m - 1$ such that
$ {\omega}_{*,\hat{l}} \not\equiv 0$ and
$$J_{\ep, \frac{\hat{l}}{m}}L_{\ep, \frac{\hat{l}}{m}} {\omega}_{*, \hat{l}} = {\lambda}_*{\omega}_{*, \hat{l}},$$
which gives modulational instability of $\omega_\ep$ for $\alpha = \frac{\hat{l}}{m}$.

For $\alpha={p\over q}$, let $\lambda_\alpha$ be an unstable eigenvalue of $J_{\ep, \alpha}L_{\ep, \alpha}$ with an eigenfunction $\omega_\alpha$. Then  $e^{i\alpha x}\omega_\alpha$ is $2q\pi$-periodic in $x$ and
\begin{align}\label{modu-multi}J_{\ep,q}L_{\ep,q}(e^{i\alpha x}\omega_\alpha )=e^{i\alpha x}J_{\ep, \alpha}L_{\ep, \alpha}\omega_\alpha=\lambda_\alpha e^{i\alpha x}\omega_\alpha.
\end{align}
By \eqref{index-formula-for-modulational-instabilitykc}, $\lambda_\alpha$ is real-valued. By separating the real and imaginary parts in \eqref{modu-multi}, we know that $\lambda_\alpha$ is an unstable eigenvalue of $J_{\ep,q}L_{\ep,q}$.
\end{proof}

\begin{remark}
Motivated by the test function \eqref{test-even} for  $4\pi$-periodic perturbations, we give an alternative test function
 \begin{align*}
 \widetilde\phi_{\ep,{1\over2}} =\left({1+e^{-i\theta_\ep}\over 2}\right)(1-\gamma_\ep^2)^{1\over4}e^{{i\over2}(\theta_\ep-x)}\in L^2_{{g'(\psi_\ep)},e}(\Omega)
 \end{align*}
 for $\ep\in[0,1)$ and $\alpha={1\over2}$. The advantage of $ \widetilde\phi_{\ep,{1\over2}}$ is that  $b_{\alpha, 2}(\widetilde\phi_{\ep,\alpha})|_{\alpha={1\over2}}=0$ since $ \widetilde\phi_{\ep,{1\over2}} e^{{i\over2}x}=\cos\left({1\over2}\theta_\ep\right)(1-\gamma_\ep^2)^{1\over 4}$ is `odd' symmetrical about $\{x=\pi\}$ along any trajectory of the steady velocity.
 By \eqref{b1-even}, we have $b_{\alpha, 1}(\widetilde\phi_{\ep,\alpha})|_{\alpha={1\over2}}=-{5\over8}\pi^2$.
 Here, $b_{\alpha, 1}$ and $b_{\alpha, 2}$ are defined in \eqref{func-b1-alpha}-\eqref{func-b2-alpha}. Thus,
 $\langle\hat{A}_{\alpha,e} \widetilde\phi_{\ep,\alpha}, \widetilde\phi_{\ep,\alpha} \rangle|_{\alpha={1\over2}}=-{5\over8}\pi^2<0$ for $\ep\in[0,1)$.

By Lemma \ref{L e-hat A-alpha}, we show the  linear modulational instability of $\omega_\ep$ for $\alpha={1\over2}$ without computer assistance.
 By Lemma $\ref{multi-periodic-modulation}$ $(2)$,  again we rigorously prove   that $\omega_\ep$ is linearly unstable
 for $4k\pi$-periodic perturbations and $\ep\in[0,1)$.
\end{remark}

\section{Nonlinear orbital stability for co-periodic perturbations}\label{Sec-Nonlinear orbital stability for co-periodic perturbations}

In this section, we prove  nonlinear orbital stability for the Kelvin--Stuart vortices  $\omega_\ep$, $\ep\in(0,1)$.

\subsection{The pseudoenergy-Casimir  functional and the distance functional}
First, we separate the perturbed stream function $\tilde \psi=\psi_\ep+\psi$ in a combination of the steady part $\psi_{\ep}(x, y)$ and the perturbation part  $\psi(x,y)$,
where $\psi_{\ep}(x, y)=\ln \left(\frac{\cosh (y) + \epsilon \cos (x)}{\sqrt{1-\epsilon^2}}\right)$. Correspondingly, the perturbed  velocity and vorticity can be written as $\vec{u}_\ep+\vec{u}$ and $\tilde\omega=\omega_\ep + \omega$, respectively. Now, the nonlinear vorticity equation \eqref{vor} takes the form
\begin{equation}\label{vorr}
\partial_t \omega + \{\omega_\ep + \omega, \psi_\ep + \psi\} = 0.
\end{equation}
By Proposition 4.4 in \cite{Milisic-Razafison13}, the Green function $G(x,y)$ solving
\begin{align*}
-\Delta \phi=\delta(0,0) \quad \text{on}\quad \Omega
\end{align*}
is
\begin{align}\label{green function}
G(x,y)=-{1\over 4\pi}\ln(\cosh(y)-\cos(x)),
\end{align}
which can also be obtained by   \eqref{catseye}-\eqref{steadyw} for the point vortex case ($\ep=-1$). Note that the total  energy ${1\over2}\iint_{\Omega}|\vec{u}_\ep+\vec{u}|^2dxdy$ is not finite since $\vec{u}_\ep\to(\pm1,0)$ as $y\to \pm \infty$.
Motivated by \cite{Majda-Bertozzi02}, we introduce an alternative bounded  functional called the pseudoenergy:
\begin{align}\label{def-pseudoenergy}
PE(\tilde \omega)={1\over2}\iint_{\Omega}(G\ast\tilde \omega)\tilde \omega dxdy,
\end{align}
where $\tilde\omega\in Y_{non}$ defined in \eqref{def-X-non-ep} and $G\ast\tilde \omega$ is the usual convolution of $G$ and $\tilde \omega$ on $\Omega$.
By Proposition 4.4 in  \cite{Milisic-Razafison13}, $G=G_1+G_2$, where $G_1\in L^1\cap L^2(\Omega)$ and $G_2(x,y)=-{1\over4\pi} |y|$.
Then
\begin{align}\nonumber
|PE(\tilde \omega)|\leq& \left|{1\over2}\iint_{\Omega}(G_1\ast\tilde \omega)\tilde \omega dxdy\right|
+\left|{1\over2}\iint_{\Omega}(G_2\ast\tilde \omega)\tilde \omega dxdy\right|\\\nonumber
\leq &{1\over2}\|G_1\ast\tilde \omega\|_{L^2(\Omega)}\|\tilde \omega\|_{L^2(\Omega)}+{1\over8\pi}\iint_{\Omega}\left(\iint_{\Omega}(|y|+|\tilde y|)\tilde \omega(\tilde x,\tilde y)d\tilde x d\tilde y\right)\tilde \omega(x,y) dxdy\\\label{PE-finite}
\leq&{1\over2}\|G_1\|_{L^1(\Omega)}\|\tilde \omega\|_{L^2(\Omega)}^2+{1\over 4\pi}\|y\tilde\omega\|_{L^1(\Omega)}\|\tilde\omega\|_{L^1(\Omega)}<\infty
\end{align}
for $\tilde\omega\in Y_{non}$.
The relative  pseudoenergy (for the perturbation part) is
\begin{align*}
E_\ep(\omega)=PE(\tilde \omega)-PE(\omega_{\ep})={1\over2}\iint_{\Omega}\left((G\ast\tilde \omega)\tilde \omega -(G\ast\omega_{\ep}) \omega_{\ep}\right) dxdy,
\end{align*}
where $\omega=\tilde \omega-\omega_{\ep}$.
To study the nonlinear stability of $\omega=0$,  we construct a Lyapunov functional for the evolved system \eqref{vorr}. Since  $\omega_\epsilon = g(\psi_\epsilon) = - e^{-2\psi_\epsilon}$, we have $\psi_\epsilon=g^{-1}(\omega_\epsilon)=-{1\over 2}\ln(-\omega_\ep)$.
Define $h(s)={1\over2}(s-s\ln(-s))$ for $s<0$. Then $h'(\omega_\ep)=-{1\over2}\ln(-\omega_\ep)=\psi_\ep$.
Following  Arnol$'$d \cite{Arnold65,Arnold69}, we use the pseudoenergy-Casimir (PEC) functional for the perturbation of vorticity
\begin{align}\nonumber H_\ep(\omega)&= \iint_{\Omega}h(\omega_\ep + \omega )dxdy - E_\ep(\omega)\\\nonumber
&={1\over2}\iint_{\Omega}\left(((\omega_\ep + \omega)-(\omega_\ep + \omega)\ln(-\omega_\ep - \omega))-(G\ast\tilde \omega)\tilde \omega +(G\ast\omega_{\ep}) \omega_{\ep}\right)dxdy.
\end{align}
Then $\omega=0$ is a critical point of $H_\ep$ since
$$H_\ep'(0)= h'(\omega_\epsilon)  - \psi_\epsilon = 0,$$
where $H_\ep'$ is the first variation of the functional $H_\ep$.
The space of the perturbed vorticity is defined  in \eqref{def-X-non-ep}
and the space of vorticity perturbations is denoted by
\begin{align*}
X_{non,\ep}=\{\omega=\tilde \omega-\omega_\ep|\tilde \omega\in Y_{non}\}.
\end{align*}
The  PEC functional is well-defined in $X_{non,\ep}$ since
 $-\tilde \omega\ln(-\tilde \omega)\in L^1(\Omega)$ by Lemma \ref{tilde-omega0-kappa-properties} (8).
Note that the steady state  $\omega_\ep$ is pointwise negative, and  in the analysis of nonlinear stability, we consider the perturbed vorticity in the same fashion.
We prove in the Appendix the existence of weak solutions to the nonlinear 2D Euler equation with initial vorticity in $Y_{non}$. Now, we prove the existence and uniqueness of weak solutions to the Poisson equation.
\begin{lemma}\label{well-poseness-Poisson-equation-nonlinear-case}
For $\ep\in[0,1)$ and $\omega \in X_{non,\ep}$, the Poisson equation
$$-\Delta \psi = \omega$$
has a unique weak solution in $\tilde{X}_\ep$, which is defined in \eqref{tilde-X0} for $\ep=0$ and \eqref{tilde-X-e} for $\ep\in(0,1)$.
\end{lemma}
\begin{proof}
For $\phi \in \tilde{X}_\ep$, similar to \eqref{psi-m-dec} we split it into the shear part  $ \widehat\phi_0 $ and the non-shear part $\phi_{\neq0}$. Then
$\|\widehat{\phi}_0\|_{\dot{H}^1(\mathbb{R})}\leq \|\phi\|_{\tilde{X}_\ep}$ and   $\|\phi_{\neq0}\|_{H^1(\Omega)}\leq C\|\phi_{\neq0}\|_{\tilde{X}_\ep}.$
Since $\iint_{\Omega}\omega dxdy=0$, we have
\begin{align*}\iint_{\Omega} \omega\widehat{\phi}_0 dxdy& = \iint_{\Omega} \omega \left(\widehat{\phi}_0(y)-\widehat{\phi}_0(0)\right) dxdy \leq
\|\phi\|_{\tilde{X}_\ep}\iint_{\Omega} |\omega| \sqrt{|y|} dxdy\\
&\leq\|\phi\|_{\tilde{X}_\ep}\left(\iint_\Omega|\omega_\ep|\sqrt{|y|}dxdy+\|y\tilde \omega \|_{L^1(\Omega)}^{1\over2}\|\tilde \omega\|_{L^1(\Omega)}^{1\over2}\right)
\leq C\|\phi\|_{\tilde{X}_\ep}
\end{align*}
and
\begin{align*}
\iint_{\Omega} \omega \phi dxdy
& = \iint_{\Omega} \omega \widehat\phi_0 dxdy + \iint_{\Omega} \omega \phi_{\neq0} dxdy \\
& \leq C\|\phi\|_{\tilde{X}_\ep} + \|\omega\|_{L^2(\Omega)}  \|\phi_{\neq0}\|_{L^2(\Omega)}\leq C\|\phi\|_{\tilde{X}_\ep}.
\end{align*}
By the Riesz Representation Theorem, there exists a unique $\psi \in \tilde{X}_\ep$ such that
$$\iint_{\Omega} \omega \phi dxdy = \iint_{\Omega}\nabla \psi \cdot \nabla\phi dxdy,\quad \phi \in \tilde{X}_\ep.$$
\end{proof}

For $\omega=\tilde \omega-\omega_{\ep}$, we give the relation between $G\ast\omega$ and the weak solution $\psi$  in Lemma \ref{well-poseness-Poisson-equation-nonlinear-case}.
\begin{lemma}\label{G-ast-omega-psi-constant}
$G\ast\omega-\psi$ is a constant for $\omega=\tilde \omega-\omega_{\ep}$, where $\ep\in[0,1)$, $\tilde \omega\in Y_{non}$ and $\psi\in\tilde{X}_\ep$ is the weak solution of $-\Delta \psi = \omega$.
\end{lemma}
\begin{proof}
Since $G=G_1+G_2$, $G_1\in L^1\cap L^2(\Omega)$ and $G_2(x,y)=-{1\over4\pi} |y|$, we have
\begin{align}\label{G-omega-conv}
|(G\ast\omega)(x, y)|\leq \|G_1\|_{L^2(\Omega)}\|\omega\|_{L^2(\Omega)}+{1\over4\pi}\left|\iint_{\Omega}|y-\tilde y|(\tilde\omega-\omega_{\ep})(\tilde x,\tilde y)d\tilde xd\tilde y\right|.
\end{align}
Let $B_R=\{x\in\mathbb{T}_{2\pi},y\in[-R,R]\}$. Note that $\iint_{\Omega}(\tilde \omega-\omega_{\ep})dxdy=0$ and $\tilde \omega-\omega_{\ep}\in L^1(\Omega)$. For any $\kappa>0$, there exists $R_{\kappa}>0$ such that
\begin{align*}
\left|\iint_{B_{R_\kappa}}(\tilde\omega-\omega_{\ep})dxdy\right|<\kappa\quad\text{and}\quad
\iint_{B_{R_\kappa}^c}|\tilde\omega-\omega_{\ep}|dxdy<\kappa.
\end{align*}
Thus, for $| y|>R_{\kappa}$, we have
\begin{align}\nonumber
&\left|\iint_{\Omega}|y-\tilde y|(\tilde\omega-\omega_{\ep})(\tilde x,\tilde y)d\tilde xd\tilde y\right|\\\nonumber
\leq&\left|\iint_{B_{R_\kappa}}(y-\tilde y)(\tilde\omega-\omega_{\ep})(\tilde x,\tilde y)d\tilde xd\tilde y\right|+\iint_{B_{R_\kappa}^c}|y-\tilde y||(\tilde\omega-\omega_{\ep})(\tilde x,\tilde y)|d\tilde xd\tilde y\\\nonumber
\leq &\kappa| y|+\|y(\tilde\omega-\omega_{\ep})\|_{L^1(B_{R_\kappa})} +\kappa| y|+\|y(\tilde\omega-\omega_{\ep})\|_{L^1(B_{R_\kappa}^c)}\\\label{G-omega-conv2}
\leq &2\kappa | y|+C.
\end{align}
Combining \eqref{G-omega-conv} and \eqref{G-omega-conv2}, we have for $| y|>R_{\kappa}$,
\begin{align}\label{G-omega-conv-sum}
|(G\ast\omega)( x, y)|\leq {\kappa\over2\pi} | y|+C.
\end{align}
Since $\psi=\widehat\psi_0+\psi_{\neq0}\in \tilde X_{\ep}$, we have
\begin{align}\label{widehat-psi-property}
|\widehat\psi_0(y)|\leq \|\widehat\psi_0'\|_{L^2(\mathbb{R})}|y|^{1\over2}+|\widehat\psi_0(0)|\leq C|y|^{1\over2}+C \text{ and }\psi_{\neq0}\in  H^1(\Omega),
\end{align}
where $\widehat\psi_0$ and $\psi_{\neq0}$  are the shear part  and  the non-shear part of $\psi$, respectively.
 Since $
-\Delta(G\ast \omega-\psi)=0$, we have $G\ast \omega-\psi=\sum_{j\neq0} e^{ijx}(d_{1j}e^{jy}+d_{2j}e^{-jy})+c_1y+c_2$, where  $d_{1j}, d_{2j}, c_1, c_2\in \mathbb{R}$ for $j\neq0$. By \eqref{G-omega-conv-sum}-\eqref{widehat-psi-property}, $d_{1j}, d_{2j}, c_1=0$ for $j\neq0$, and thus, $G\ast \omega-\psi=c_2$.
\end{proof}

Note that $\lim_{y\to\pm\infty}\partial_y\psi_{\ep}(x,y)=\pm1$ for fixed $x\in \mathbb{T}_{2\pi}$.
By a similar argument to \eqref{v-mu term1}, we have $\lim_{y\to\pm\infty}(\partial_y G* \omega_\ep)(x,y)=\pm1$ for fixed $x\in \mathbb{T}_{2\pi}$, and thus,
$G\ast\omega_{\ep}-\psi_{\ep}$ is a constant.
Since $\iint_{\Omega}(G\ast\omega_{\ep}) \tilde \omega dxdy=\iint_{\Omega}(G\ast\tilde \omega) \omega_{\ep} dxdy$, by Lemma \ref{G-ast-omega-psi-constant} we have
\begin{align*}
E_\ep(\omega)=&PE(\tilde \omega)-PE(\omega_{\ep})={1\over2}\iint_{\Omega}\left((G\ast\tilde \omega)\tilde \omega -(G\ast\omega_{\ep}) \omega_{\ep}\right) dxdy\\
=&{1\over2} \iint_{\Omega}\left((G\ast\tilde \omega)\tilde \omega -(G\ast\omega_{\ep}) \tilde \omega \right) dxdy+{1\over2} \iint_{\Omega}(G\ast\omega_{\ep}) (\tilde \omega-\omega_{\ep}) dxdy\\
=&{1\over2} \iint_{\Omega}(G\ast\tilde \omega)(\tilde \omega-\omega_{\ep}) dxdy+{1\over2} \iint_{\Omega}\psi_{\ep}  \omega dxdy\\
=&{1\over2} \iint_{\Omega}(\psi_{\ep}+\psi)\omega dxdy+{1\over2} \iint_{\Omega}\psi_{\ep}  \omega dxdy= \iint_{\Omega}\psi_{\ep}\omega dxdy+{1\over2}\iint_{\Omega}|\nabla \psi|^2dxdy,
\end{align*}
where we used $\iint_{\Omega}\omega dxdy=0$, $\omega=\tilde \omega-\omega_{\ep}$ and $\psi$ is the weak solution of $-\Delta \psi = \omega$ in $\tilde{X}_\ep$.

Since $h'(\omega_\ep)=\psi_\ep$,  we have
\begin{align*}
H_\ep(\omega) - H_\ep(0)
& =  \iint_{\Omega} f_{\omega_\ep}(\omega) dxdy-\frac 1 2\iint_\Omega  |\nabla\psi|^2 dxdy,
\end{align*}
where
\begin{align*}
f_{\omega_\ep}(\omega) = h(\omega_\ep + \omega ) - h(\omega_\ep)  - \psi_\ep \omega
\end{align*}
for  $\omega \in X_{non,\ep}$.
Define the distance functionals
\begin{align}\nonumber
d_{1}(\tilde\omega,\omega_\ep)&=\iint_{\Omega}f_{\omega_\ep}(\omega)dxdy,\quad d_2(\tilde\omega,\omega_\ep)=\iint_{\Omega}(G*\omega)\omega dxdy=\iint_{\Omega}|\nabla\psi|^2dxdy,\\\label{distance-euler case}
d(\tilde\omega,\omega_\ep)&=d_{1}(\tilde\omega,\omega_\ep)+d_{2}(\tilde\omega,\omega_\ep),
\end{align}
where $\tilde\omega\in Y_{non}$ is the perturbed vorticity. By Lemma \ref{well-poseness-Poisson-equation-nonlinear-case},  $d_2(\tilde\omega,\omega_\ep)$ is well-defined for $\tilde\omega\in Y_{non}$.
By Lemma \ref{tilde-omega0-kappa-properties} (7), we have $\psi_\ep\tilde \omega\in{L^1(\Omega)}$ for $\tilde\omega\in Y_{non}$, and thus,
 by Taylor's formula we have
\begin{align}\nonumber
0\leq&\int_0^1\iint_\Omega{(1-r)\big(\tilde \omega-\omega_{\ep}\big)^2\over 2|\omega^r|}dxdydr=d_1(\tilde \omega,\omega_{\ep})\\\nonumber
=&\iint_{\Omega}\left({1\over2}(\tilde \omega-\tilde \omega\ln(-\tilde \omega))-{1\over 2}\omega_\ep-\psi_\ep\tilde \omega \right) dxdy\\
\leq&\|\tilde \omega\|_{L^1(\Omega)}+\|\tilde \omega\|_{L^2(\Omega)}^2+\|\omega_\ep\|_{L^1(\Omega)}+\|\psi_\ep\tilde \omega\|_{L^1(\Omega)}<\infty,\label{d1-well-def}
\end{align}
where $\omega^r=r\tilde \omega+(1-r)\omega_{\ep}
$ for $r\in[0,1]$. Here, we used $s\ln s\leq s^2$ for $s>0$. Thus, $d_1(\tilde\omega,\omega_\ep)$ is well-defined for $\tilde\omega\in Y_{non}$.
\begin{remark}
 For $a\in[1,2)$ and $\tilde \omega\in Y_{non}$, we have
 \begin{align}\label{La-control-by-d}
 \|\tilde \omega-\omega_{\ep}\|_{L^a(\Omega)}\leq (3\sqrt{2\pi})^{{2\over a}-1}(\|\tilde \omega\|_{L^2(\Omega)}+\|\omega_{\ep}\|_{L^2(\Omega)})^{2-{2\over a}}d_1(\tilde\omega,\omega_\ep)^{{1\over a}-{1\over2}}.
 \end{align}
 For $\tilde \omega\in Y_{non}\cap L^3(\Omega)$, we have
 \begin{align}\nonumber
 &\|\tilde \omega-\omega_{\ep}\|_{L^2(\Omega)}\\\label{L2-control-by-d}
 \leq& {\sqrt{6}\over2}\big(\|\tilde \omega\|_{L^3(\Omega)}^3+\|\omega_\ep\|_{L^\infty(\Omega)}\|\tilde \omega\|_{L^2(\Omega)}^2 +\|\omega_\ep\|_{L^\infty(\Omega)}^2\|\tilde \omega\|_{L^1(\Omega)}+\|\omega_\ep\|_{L^3(\Omega)}^3\big)^{1\over4} d_1(\tilde\omega,\omega_\ep)^{{1\over4}}.
 \end{align}
  In fact, for $a=1$ and $\tilde \omega\in Y_{non}$, we have
\begin{align}\nonumber
&\|\tilde \omega-\omega_{\ep}\|_{L^1(\Omega)}={3\over2}\int_0^1\iint_{\Omega}\sqrt{1-r}|\tilde \omega-\omega_{\ep}|dxdydr\\\nonumber
\leq&{3\over2}\left(\int_0^1\iint_\Omega{(1-r)\big|\tilde \omega-\omega_{\ep}\big|^2\over 2|\omega^r|}dxdydr\right)^{1\over2}\left(\int_0^1\iint_\Omega2|\omega^r| dxdydr\right)^{1\over2}\\\label{L1-control-by-d-pf}
=&3\sqrt{2\pi}d_1(\tilde\omega,\omega_\ep)^{1\over2}.
\end{align}
For $a\in(1,2)$ and $\tilde \omega\in Y_{non}$, we have
\begin{align*}
&\|\tilde \omega-\omega_{\ep}\|_{L^a(\Omega)}\leq \|\tilde \omega-\omega_{\ep}\|_{L^1(\Omega)}^{{2\over a}-1}\|\tilde \omega-\omega_{\ep}\|_{L^2(\Omega)}^{2-{2\over a}},
\end{align*}
which, combined with \eqref{L1-control-by-d-pf}, yields \eqref{La-control-by-d}.
For $\tilde \omega\in Y_{non}\cap L^3(\Omega)$,  by a similar argument to \eqref{L1-control-by-d-pf}, we have
 \begin{align}\label{vorticityL2-d12}
 &\|\tilde \omega-\omega_{\ep}\|_{L^2(\Omega)}^2\leq{3\over2}\left(\iint_\Omega|\tilde \omega^3-\tilde\omega^2 \omega_\ep-\tilde \omega\omega_\ep^2+\omega_\ep^3|dxdy\right)^{1\over2} d_1(\tilde\omega,\omega_\ep)^{{1\over2}},
 \end{align}
which gives \eqref{L2-control-by-d}.
\end{remark}

\subsection{The dual functional and its regularity}
We   try to study the Taylor expansion of $H_\ep$ near $\omega=0$ directly, and use the positiveness of  $L_\ep$ in a finite co-dimensional subspace of $X_\ep$. However,
 $\|\omega\|_{L^3}$ cannot be controlled by $\|\omega\|_{L_{1\over g'(\psi_\ep)}^2}$ in general. Our approach is  to transform $H_\ep$ to its dual functional and then study the Taylor expansion of the dual functional.
We observe that
\begin{align}\nonumber&H_\ep(\omega) - H_\ep(0)
 = d_{1}(\tilde\omega,\omega_\ep)-{1\over2}d_2(\tilde\omega,\omega_\ep)\\\nonumber
 =&{1\over2}\iint_{\Omega}|\nabla\psi|^2dxdy-\iint_{\Omega}(\psi\omega-f_{\omega_\ep}(\omega))dxdy\\\label{H-omega-H-0}
 \geq &  \iint_{\Omega} \left(\frac 1 2 |\nabla \psi|^2 - f_{\omega_\ep}^*(\psi)\right) dxdy
 \end{align}
 for $\omega \in X_{non,\ep}$, where $f_{\omega_\ep}^*$ is the Legendre transformation of $f_{\omega_\ep}$.
This gives  a lower bound of $d_{1}(\tilde\omega,\omega_\ep)-{1\over2}d_2(\tilde\omega,\omega_\ep)$. Then  we compute the pointwise expression of  $f_{\omega_\ep}^*$.
\begin{lemma}\label{Legendre transformation}
Let $\ep\in[0,1)$, $(x,y)\in\Omega$ and  $f_{\omega_\ep(x,y)}(z)= h(\omega_\ep(x,y) + z ) - h(\omega_\ep(x,y))  -  h'(\omega_\ep(x,y))z $ for  $z\in(-\infty,-\omega_\ep(x,y))$.
Then
 the Legendre transformation of $f_{\omega_\ep(x,y)}$ is
\begin{align*}
f_{\omega_\ep(x,y)}^*(s)= - \frac 1 2 \omega_\ep(x,y) (e^{-2s} + 2s - 1),\quad s\in\mathbb{R}.
\end{align*}
\end{lemma}
\begin{proof}
By its definition of  the Legendre transformation, $f_{\omega_\ep(x,y)}^*(s)=\sup\limits_{z<-\omega_\ep(x,y)}(sz-f_{\omega_\ep(x,y)}(z)),$ $ s\in\mathbb{R}$.
Let $F_{\omega_\ep(x,y),s}(z)=sz-f_{\omega_\ep(x,y)}(z)$ for $z<-\omega_\ep(x,y)$. Then
\begin{align*}
F_{\omega_\ep(x,y),s}'(z)= s - h'(\omega_\ep(x,y) +z) +h'(\omega_\ep(x,y)) =s+{1\over2}\ln|\omega_\ep(x,y)+z|+\psi_\ep(x,y).
\end{align*}
Thus, there exists a unique $z_{\omega_\ep(x,y)}(s)\triangleq\omega_\ep(x,y)(e^{-2s}-1)\in(-\infty,-\omega_\ep(x,y))$ such that $F_{\omega_\ep(x,y),s}'(z_{\omega_\ep(x,y)}(s))=0$ and $F_{\omega_\ep(x,y),s}''(z)={1\over 2(\omega_\ep(x,y)+z)}<0$ for $z\in(-\infty,-\omega_\ep(x,y))$, which implies
\begin{align*}
&f_{\omega_\ep(x,y)}^*(s)=F_{\omega_\ep(x,y),s}(z_{\omega_\ep(x,y)}(s))\\
 =& (s+\psi_\ep(x,y)) \omega_\ep(x,y)(e^{-2s}-1) - h(\omega_\ep(x,y)e^{-2s} ) + h(\omega_\ep)  \\
 =& - \frac 1 2 \omega_\ep(x,y) (e^{-2s} + 2s - 1),\quad s\in\mathbb{R}.
\end{align*}
\end{proof}
By \eqref{H-omega-H-0} and Lemma \ref{Legendre transformation}, we have
\begin{align*}
 d_{1}(\tilde\omega,\omega_\ep)-{1\over2}d_2(\tilde\omega,\omega_\ep)
 \geq &  \iint_{\Omega} \left(\frac 1 2 |\nabla \psi|^2 +\frac 1 2 \omega_\ep (e^{-2\psi} + 2\psi - 1)\right) dxdy.
 \end{align*}
To apply the Taylor formula of the  functional
\begin{align}\nonumber
\mathscr{B}_\ep( \psi)\triangleq &  \iint_{\Omega} \left(\frac 1 2 |\nabla \psi|^2 +\frac 1 2 \omega_\ep (e^{-2\psi} + 2\psi - 1)\right) dxdy\\\label{def-functional-B}
=  & \iint_{\Omega} \left(\frac 1 2 |\nabla \psi|^2 -\frac 1 4 g'(\psi_\ep)(e^{-2\psi} + 2\psi - 1)\right) dxdy,\quad \psi\in \tilde X_\ep,
\end{align}
we first study its regularity. To this end, we need the following inequalities.
\begin{lemma}\label{Orlicz-type inequlity-lemma}
For $\ep\in[0,1)$ and $a\in\mathbb{R}$, we have
\begin{align}\label{Orlicz-type inequlity}
\iint_\Omega g'(\psi_\ep) e^ {a\psi} dxdy \leq\iint_\Omega g'(\psi_\ep) e^ {|a\psi|} dxdy \leq   C e^{Ca^2\|\psi\|_{\tilde{X}_\ep}^2},\quad\psi \in \tilde{X}_\ep.
\end{align}
In particular, for $p\in\mathbb{Z}^+$,
\begin{align*}\iint_\Omega g'(\psi_\ep) |\psi|^p dxdy \leq p!\iint_\Omega g'(\psi_\ep) e^{|\psi|} dxdy \leq  Cp! e^{C \|\psi\|_{\tilde{X}_\ep}^2},\quad\psi \in \tilde{X}_\ep.\end{align*}
\end{lemma}
\begin{proof} We first prove \eqref{Orlicz-type inequlity} for $\ep=0$.
Applying the similar decomposition \eqref{psi-m-dec} to $\psi \in \tilde{X}_\ep$, we have $\psi = \widehat\psi_0 +\psi_{\neq0}$, where
 $\psi_{\neq0} \in H^1(\Omega)$. Since $$|a\widehat\psi_0(y)| \leq|a|\|\widehat\psi_0'\|_{L^2(\mathbb{R})}|y|^{\frac 1 2}\leq |a|\|\psi\|_{\tilde{X}_0}|y|^{\frac 1 2}\leq{a^2\over4}\|\psi\|_{\tilde{X}_0}^2+|y|,$$ we have
\begin{align}\label{0-mode-or-estimate}
\sqrt{g'(\psi_0)} e^{|a\widehat\psi_0(y)|} \leq \sqrt{g'(\psi_0)} e^{{a^2\over4}\|\psi\|_{\tilde{X}_0}^2 } e^{|y|} \leq C e^{{a^2\over4}\|\psi\|_{\tilde{X}_0}^2 }. \end{align}
Without loss of generality, assume that $\|\psi_{\neq0}\|_{\tilde{X}_0} \neq 0$. It follows from Subsection 8.26 in \cite{Adams75} that $H^1(\Omega)$ is embedded in the Orlicz space $L_{A_0}(\Omega)$ with $A_0(t) = e^{t^2} - 1$.
Since $\psi_{\neq0} \in H^1(\Omega)$, we have
   $\psi_{\neq0} \in L_{A_0}(\Omega)$ and
$\|\psi_{\neq0}\|_{L_{A_0}(\Omega)} \leq C \|\psi_{\neq0}\|_{H^1(\Omega)} \leq C \|\psi\|_{\tilde{X}_0}.$
Let $k_0 = \|\psi_{\neq0}\|_{L_{A_0}(\Omega)} + \|\psi_{\neq0}\|_{\tilde{X}_0}$. Then $k_0 \leq C\|\psi\|_{\tilde{X}_0}$. By the definition of the norm $\|\cdot\|_{L_{A_0}(\Omega)}$ (see (13) in Chapter VIII), we have
\begin{align*}
\|\psi_{\neq0}\|_{L_{A_0}(\Omega)}
&= \inf \left \{ k > 0 \bigg| \iint_{\Omega} \left( e^{\left (\frac{|\psi_{\neq0}|}{k}\right)^2} - 1\right)dx dy \leq 1 \right \},
\end{align*}
and thus, there exists $k_1\in[\|\psi_{\neq0}\|_{L_{A_0}(\Omega)}, k_0)$ such that
\begin{align}\label{Orlicz2}
\iint_{\Omega} \left( e^{\left(\frac{|\psi_{\neq0}|}{k_0}\right)^2} - 1 \right)dx dy \leq \iint_{\Omega} \left( e^{\left(\frac{|\psi_{\neq0}|}{k_1}\right)^2} - 1 \right)dx dy \leq  1.\end{align}
By \eqref{0-mode-or-estimate}, \eqref{Orlicz2} and the fact that $k_0 \leq C\|\psi\|_{\tilde{X}_0}$, we have
\begin{align*}
& \iint_\Omega g'(\psi_0)e^{|a\psi|}  dxdy
 \leq   \iint_\Omega \sqrt{g'(\psi_0)}e^{|a\widehat\psi_0|} \sqrt{g'(\psi_0)} e^{|a\psi_{\neq0}|} dxdy \\
 \leq & Ce^{\frac{a^2}{4} \|\psi\|^2_{\tilde{X}_0}} \iint_\Omega \sqrt{g'(\psi_0)}e^{\left| \frac{\psi_{\neq0}}{k_0}\right|^2} e^{\frac{a^2}{4}k_0^2}  dxdy \\
 = & Ce^{\frac{a^2}{4} \left(\|\psi\|^2_{\tilde{X}_0}+k_0^2\right)} \iint_\Omega \sqrt{g'(\psi_0)} \left( e^{\left| \frac{\psi_{\neq0}}{k_0}\right|^2} - 1\right)  dxdy +  Ce^{\frac{a^2}{4} \left(\|\psi\|^2_{\tilde{X}_0}+k_0^2\right)} \iint_\Omega \sqrt{g'(\psi_0)}  dxdy \\
\leq & Ce^{Ca^2 \|\psi\|^2_{\tilde{X}_0}}  \iint_\Omega \left( e^{\left| \frac{\psi_{\neq0}}{k_0}\right|^2} - 1\right)  dxdy +  Ce^{Ca^2 \|\psi\|^2_{\tilde{X}_0}} \\
\leq & Ce^{Ca^2 \|\psi\|^2_{\tilde{X}_0}}.
\end{align*}

Now, we consider the case $\epsilon\in(0,1)$. By
\eqref{Orlicz-type inequlity} for $\epsilon=0$, we have $\iint_{\tilde\Omega} e^{a\Psi}dxd\gamma_0\leq C e^{Ca^2\|\Psi\|_{\tilde Y_0}^2}$ for $\Psi\in\tilde Y_0$ in the new variables $(x,\gamma_0=\tanh(y))$. Then $\iint_{\tilde \Omega} e^{a\Psi}d\theta_\ep d\gamma_\ep\leq C e^{Ca^2\|\Psi\|_{\tilde Y_\ep}^2}$ for $\Psi\in\tilde Y_\ep$ in the new variables $(\theta_\ep,\gamma_\ep)$ for $\epsilon\in(0,1)$. Thus, \eqref{Orlicz-type inequlity} holds true  for $\epsilon\in(0,1)$.
\end{proof}

With the help of Lemma \ref{Orlicz-type inequlity-lemma}, we prove the  required  $C^2$ regularity of $\mathscr{B}_\ep$.

\begin{lemma}\label{B-C2}
$\mathscr{B}_\ep\in C^2(\tilde X_\ep)$, and for $ \psi\in \tilde X_\ep$,
\begin{align*}
\mathscr{B}_\ep'(\psi) &= -\Delta\psi +\frac{1}{2}g'(\psi_\ep)(e^{-2\psi}-1),\\
\langle \mathscr{B}_\ep''(\psi)\phi,\varphi \rangle&=  \iint_{\Omega}\left(\nabla\phi\cdot\nabla\varphi- g'(\psi_\ep) e^{-2\psi}\phi\varphi\right)dxdy,\quad \phi,\varphi\in\tilde X_\ep,
\end{align*}
 where   $\mathscr{B}_\ep$ is defined in \eqref{def-functional-B} and $\ep\in[0,1)$.
\end{lemma}
\begin{proof}
Let $\psi\in \tilde{X}_\ep$. For  $\phi \in \tilde{X}_\ep$, by Lemmas \ref{poincare1}, \ref{poincare1ep} and \ref{Orlicz-type inequlity-lemma} we have
\begin{align*}
|\partial_\lambda \mathscr{B}_\ep(\psi + \lambda \phi)|_{\lambda = 0}|
= & \iint_\Omega  \left(-\Delta \psi+{1\over2}  g'(\psi_\ep) (e^{-2\psi}-1)\right)\phi dxdy\\
\leq &\|\psi\|_{\tilde{X}_\ep}\|\phi\|_{\tilde{X}_\ep}+ C \left(\iint_\Omega    g'(\psi_\ep) (e^{-4\psi}-2e^{-2\psi}+1)dxdy\right)^{1\over2}\|\phi \|_{\tilde{X}_\ep}\\
\leq &\left(\|\psi\|_{\tilde{X}_\ep}+ C \left(Ce^{C\|\psi\|_{\tilde{X}_\ep}^2}+C\right)^{1\over2}\right)\|\phi \|_{\tilde{X}_\ep}.
\end{align*}
Thus,  $\mathscr{B}_\ep$  is G$\hat{\text{a}}$teaux differentiable at $\psi\in  \tilde{X}_\ep$.
To show that $\mathscr{B}_\ep\in C^1(\tilde X_\ep)$, we choose $\{\psi_n\}_{n=1}^\infty\in \tilde X_\ep$ such that $\psi_n\to\psi$ in $\tilde{X}_\ep$, and prove that for fixed $\phi\in\tilde X_\ep$,
\begin{align*}
\partial_\lambda \mathscr{B}_\ep(\psi_n + \lambda \phi)|_{\lambda = 0}\to\partial_\lambda \mathscr{B}_\ep(\psi + \lambda \phi)|_{\lambda = 0}
\end{align*}
as $n\to\infty$.
In fact, there exists $N>0$ such that $\|\psi_n\|_{\tilde X_\ep}\leq \|\psi\|_{\tilde X_\ep}+1$ for $n\geq N$, and  by Lemmas \ref{poincare1}, \ref{poincare1ep} and \ref{Orlicz-type inequlity-lemma} we have for $n\geq N$,
\begin{align*}
&|\partial_\lambda \mathscr{B}_\ep(\psi_n + \lambda \phi)|_{\lambda = 0}-\partial_\lambda \mathscr{B}_\ep(\psi + \lambda \phi)|_{\lambda = 0}| \\
= & \left|\iint_{\Omega}\left(\nabla(\psi_n-\psi)\cdot\nabla\phi +{1\over 2} g'(\psi_\ep)(e^{-2\psi_n} - e^{-2\psi})\phi\right) dxdy\right|\\
\leq&\|\psi_n-\psi\|_{\tilde X_\ep}\|\phi\|_{\tilde X_\ep}+\left|\int_0^1  \iint_{\Omega} g'(\psi_\ep) e^{-2(s\psi_n + (1-s)\psi)}  (\psi_n - \psi)  \phi dxdyds\right| \\
\leq&\|\psi_n-\psi\|_{\tilde X_\ep}\|\phi\|_{\tilde X_\ep}+\|\psi_n-\psi\|_{\tilde X_\ep}\|\phi\|_{L^4_{g'(\psi_\ep)}}\int_0^1 \left( \iint_{\Omega} g'(\psi_\ep) e^{-8(s\psi_n + (1-s)\psi)}  dxdy\right)^{1\over4}ds\\
\leq &  \|\psi_n-\psi\|_{\tilde X_\ep}\|\phi\|_{\tilde X_\ep}+
\|\psi_n-\psi\|_{\tilde X_\ep}
\left(Ce^{C\|\phi\|_{\tilde X_\ep}^2}\right)^{1\over4}
\int_0^1\left(Ce^{C\|s\psi_n + (1-s)\psi\|_{\tilde X_\ep}^2}\right)^{1\over4}ds\\
\leq &\left(\|\phi\|_{\tilde X_\ep}+C_{\|\phi\|_{\tilde X_\ep}}C_{\|\psi\|_{\tilde X_\ep}}\right)\|\psi_n-\psi\|_{\tilde X_\ep}\to 0\quad \text{as}\quad n\to\infty.
\end{align*}
This proves that $\mathscr{B}_\ep\in C^1(\tilde X_\ep)$.
Then we show that the 2-th order G$\hat{\text{a}}$teaux derivative of $\mathscr{B}_\ep$ exists at $\psi\in  \tilde{X}_\ep$. For
$\phi \in  \tilde{X}_\ep$ and $\varphi\in\tilde{X}_\ep$, by Lemma  \ref{Orlicz-type inequlity-lemma} we have
\begin{align*}
&\left|\partial_\tau\partial_\lambda \mathscr{B}_\ep(\psi + \lambda\phi+\tau\varphi)|_{\lambda =\tau= 0}\right|
= \left|\iint_{\Omega}\left(\nabla\phi\cdot\nabla\varphi- g'(\psi_\ep) e^{-2\psi}\phi\varphi\right)dxdy\right|\\
\leq&\|\phi\|_{\tilde{X}_\ep}\|\varphi\|_{\tilde{X}_\ep}+\left(\iint_{\Omega} g'(\psi_\ep) e^{-4\psi} dxdy\right)^{1\over2}
\|\phi\|_{L_{g'(\psi_\ep)}^4}\|\varphi\|_{L_{g'(\psi_\ep)}^4}\\
\leq&\|\phi\|_{\tilde{X}_\ep}\|\varphi\|_{\tilde{X}_\ep}+Ce^{C\left(\|\psi\|_{\tilde{X}_\ep}^2+\|\phi\|_{\tilde{X}_\ep}^2+\|\varphi\|_{\tilde{X}_\ep}^2\right)},
\end{align*}
which implies that   $\mathscr{B}_\ep$  is 2-order G$\hat{\text{a}}$teaux differentiable at $\psi\in  \tilde{X}_\ep$.
To show that $\mathscr{B}_\ep\in C^2(\tilde X_\ep)$, we use  $\{\psi_n\}_{n=1}^\infty\in \tilde X_\ep$  as above, and
  for  $\phi,\varphi\in\tilde X_\ep$ and $n\geq N$,
\begin{align*}
&|\partial_\tau\partial_\lambda \mathscr{B}_\ep(\psi_n + \lambda \phi+\tau\varphi)|_{\lambda =\tau= 0}-\partial_\tau\partial_\lambda \mathscr{B}_\ep(\psi + \lambda \phi+\tau\varphi)|_{\lambda =\tau= 0}|\\
=&\left|2\int_0^1\iint_{\Omega}g'(\psi_\ep) e^{-2(s\psi_n + (1-s)\psi)}(\psi_n-\psi)\phi\varphi dxdyds\right|\\
\leq&C\|\psi_n-\psi\|_{\tilde X_\ep}\|\phi\|_{L_{g'(\psi_\ep)}^6}\|\varphi\|_{L_{g'(\psi_\ep)}^6}\int_0^1\left(\iint_{\Omega}g'(\psi_\ep) e^{-12(s\psi_n + (1-s)\psi)}dxdy\right)^{1\over6}ds\\
\leq&C\|\psi_n-\psi\|_{\tilde X_\ep}
\left(Ce^{C\|\phi\|_{\tilde X_\ep}^2}\right)^{1\over6}\left(Ce^{C\|\varphi\|_{\tilde X_\ep}^2}\right)^{1\over6}
\int_0^1\left(Ce^{C\|s\psi_n + (1-s)\psi\|_{\tilde X_\ep}^2}\right)^{1\over6}ds\\
\leq &C_{\|\phi\|_{\tilde X_\ep}}C_{\|\varphi\|_{\tilde X_\ep}}C_{\|\psi\|_{\tilde X_\ep}}\|\psi_n-\psi\|_{\tilde X_\ep}\to 0\quad \text{as}\quad n\to\infty.
\end{align*}
This proves that $\mathscr{B}_\ep\in C^2(\tilde X_\ep)$.
\end{proof}

\begin{remark}
In view of Lemma \ref{Orlicz-type inequlity-lemma}, one can use a similar argument in the proof of Lemma \ref{B-C2} to show that  $\mathscr{B}_\ep\in C^\infty(\tilde X_\ep)$.
\end{remark}

By Lemma \ref{B-C2}, we have $\mathscr{B}_\ep'(0) = 0$, and
\begin{align*}
\langle \mathscr{B}_\ep''(0)\psi_1,\psi_2 \rangle&=  \iint_{\Omega}\left(\nabla\psi_1\cdot\nabla\psi_2- g'(\psi_\ep) \psi_1\psi_2\right)dxdy,\quad \psi_1,\psi_2\in\tilde X_\ep.
\end{align*}
Recall that
$A_\ep =-\Delta -g'(\psi_\ep):\tilde{X}_\ep \rightarrow \tilde{X}_\ep^*$ for $\ep\in[0,1)$. Then
\begin{align}\label{B-ep-A-ep}
\langle \mathscr{B}_\ep''(0)\psi_1,\psi_2 \rangle=\langle A_\ep \psi_1,\psi_2 \rangle,\quad \psi_1,\psi_2\in\tilde X_\ep.
\end{align}
By Corollaries \ref{kernel of  the operator A0 and a decomposition of tilde X0} and \ref{kernel of  the operator A-ep and a decomposition of tilde Xep}, we have
\begin{align*}
\ker ( A_\ep)={\rm{span}}\left\{\eta_\ep(x,y), \gamma_\ep(x,y), \xi_\ep(x,y)\right\}
\end{align*}
and
\begin{align}\label{A-ep-positive-lower-bound}
\langle  A_\ep \psi,\psi\rangle \geq C_0 \| \psi\|_{\tilde X_\ep}^2, \quad \quad \psi\in \tilde X_{\ep+}=\tilde X_\ep \ominus\ker ( A_\ep)
\end{align}
for some $C_0>0$ independent of $\ep\in[0,1)$.

\subsection{Removal of the kernel generated  by translations and  parameter variation}
Let us first consider the 3 dimensional orbit
$$\Gamma = \{\omega_{\ep_1}(x+x_1, y + y_1)| \ep_1 \in (0, 1), x_1 \in \mathbb{T}_{2\pi}, y_1 \in \mathbb{R} \}.$$
To prove the nonlinear orbital stability of the steady states,  we need to carefully study the translations of the steady states in the $x$ and $y$ directions, as well as  the  variation of the parameter $\ep$, so that the perturbation of the stream function is perpendicular to the three kernel functions of $ A_\ep$.
\begin{lemma}\label{imp-vertical condition} Let $\ep_0\in(0,1)$. Then  there exists $\delta=\delta(\ep_0)>0$ such that for any  $(x_0,y_0)\in\Omega$ and  $\tilde \omega\in Y_{non}$ with $d_2(\tilde \omega,\omega_{\ep_0}(x+x_0,y+y_0))=\|\tilde\psi-\psi_{\ep_0}(x+x_0,y+y_0)\|_{\dot{H}^1(\Omega)}^2\leq \delta$, there exist $(\tilde x_0,\tilde y_0)\in\Omega$ and $\tilde \epsilon_0\in(a(\ep_0),b(\ep_0))$, depending continuously  on $(x_0,y_0)\in\Omega$ and  $\tilde \omega$, such that
\begin{align*}
\iint_{\Omega}\nabla\left(\tilde \psi(x,y)-\psi_{\tilde \ep_0}(x+\tilde x_0,y+\tilde y_0)\right)\cdot\nabla\eta_{\tilde \ep_0}\left(x+\tilde x_0,y+\tilde y_0\right)dxdy=0,\\
\iint_{\Omega}\nabla\left(\tilde \psi(x,y)-\psi_{\tilde \ep_0}(x+\tilde x_0,y+\tilde y_0)\right)\cdot\nabla\gamma_{\tilde \ep_0}\left(x+\tilde x_0,y+\tilde y_0\right)dxdy=0,\\
\iint_{\Omega}\nabla\left(\tilde \psi(x,y)-\psi_{\tilde \ep_0}(x+\tilde x_0,y+\tilde y_0)\right)\cdot\nabla\xi_{\tilde \ep_0}\left(x+\tilde x_0,y+\tilde y_0\right)dxdy=0,
\end{align*}
and
 \begin{align*}
|x_0-\tilde x_0|+|y_0-\tilde y_0|+|\ep_0-\tilde \ep_0|\leq C(\ep_0)\sqrt{\delta}
\end{align*}
for some $a(\ep_0)\in (0,\ep_0)$ and $b(\ep_0)\in(\ep_0,1)$, where $\tilde\psi=G*\tilde\omega$.
\end{lemma}

\begin{proof} For $\tilde \omega\in Y_{non}$, since $\tilde \psi-\psi_{\ep_0}=G*(\tilde \omega-\omega_{\ep_0})-c$ for  some constant $c$, by Lemma \ref{G-ast-omega-psi-constant} we have $\tilde \psi-\psi_{\ep_0}\in {\dot{H}^1(\Omega)}$.
 For $x_0=y_0=0$, we
define
 the map $S=(S_1,S_2,S_3)$ from $Y_{non}\times \mathbb{T}_{2\pi}\times\mathbb{R}\times (0,1)$ to $\mathbb{R}^3$ by
\begin{align*}
S_1(\tilde\omega,x_1,y_1,\ep_1)
=&\iint_{\Omega}\nabla\left(\tilde \psi(x,y)-\psi_{\ep_1}(x+x_1,y+y_1)\right)\cdot\nabla\eta_{\ep_1}\left(x+x_1,y+y_1\right)dxdy,\\
S_2(\tilde\omega,x_1,y_1,\ep_1)
=&\iint_{\Omega}\nabla\left(\tilde \psi(x,y)-\psi_{\ep_1}(x+x_1,y+y_1)\right)\cdot\nabla\gamma_{\ep_1}\left(x+x_1,y+y_1\right)dxdy,\\
S_3(\tilde\omega,x_1,y_1,\ep_1)
=&\iint_{\Omega}\nabla\left(\tilde \psi(x,y)-\psi_{\ep_1}(x+x_1,y+y_1)\right)\cdot\nabla\xi_{\ep_1}\left(x+x_1,y+y_1\right)dxdy.
\end{align*}
Note that $S(\omega_{\ep_0},0,0,\ep_0)=(0,0,0)$ and
\begin{align*}
&{\partial(S_1,S_2,S_3)\over \partial(x_1,y_1,\ep_1)}\bigg|_{\tilde \omega=\omega_{\ep_0},x_1=0,y_1=0,\ep_1=\ep_0}\\
=&\left| \begin{array}{cccc} -\iint_\Omega\nabla\partial_x\psi_{\ep}\cdot\nabla\eta_{\ep}dxdy & -\iint_\Omega\nabla\partial_y\psi_{\ep}\cdot\nabla\eta_{\ep}dxdy
&-\iint_\Omega\nabla\partial_\ep\psi_{\ep}\cdot\nabla\eta_{\ep}dxdy\\ -\iint_\Omega\nabla\partial_x\psi_{\ep}\cdot\nabla\gamma_{\ep}dxdy& -\iint_\Omega\nabla\partial_y\psi_{\ep}\cdot \nabla\gamma_{\ep}dxdy &
 -\iint_\Omega\nabla\partial_\ep\psi_{\ep}\cdot\nabla\gamma_{\ep}dxdy\\
 -\iint_\Omega\nabla\partial_x\psi_{\ep}\cdot\nabla\xi_{\ep}dxdy& -\iint_\Omega\nabla\partial_y\psi_{\ep}\cdot\nabla\xi_{\ep}dxdy  &-\iint_\Omega\nabla\partial_\ep\psi_{\ep}\cdot \nabla\xi_{\ep}dxdy
\end{array} \right|_{\ep=\ep_0}.
\end{align*}
By \eqref{three-kers1}-\eqref{three-kers3}, \eqref{def-eta-ep}-\eqref{def-xi-ep} and Proposition \ref{prop1}, we have
\begin{align*}
\iint_\Omega\nabla\partial_x\psi_{\ep}\cdot\nabla\eta_{\ep}dxdy&={-\ep\over \sqrt{1-\ep^2}}\iint_\Omega|\nabla\eta_\ep|^2dxdy=
{-\ep\over \sqrt{1-\ep^2}}\int_{-1}^1\int_{0}^{2\pi}(1-\eta_\ep^2)d\theta_\ep d\gamma_\ep\\
&={-\ep\over \sqrt{1-\ep^2}}\int_{-1}^1\int_{0}^{2\pi}\left(\gamma_\ep^2\sin^2(\theta_\ep)+\cos^2(\theta_\ep)\right)d\theta_\ep d\gamma_\ep={-\ep\over \sqrt{1-\ep^2}}{8\over3}\pi,\\
\iint_\Omega\nabla\partial_y\psi_{\ep}\cdot\nabla\eta_{\ep}dxdy&={1\over \sqrt{1-\ep^2}}\iint_\Omega\nabla\gamma_\ep\cdot \nabla\eta_\ep dxdy=
{-1\over \sqrt{1-\ep^2}}\int_{-1}^1\int_{0}^{2\pi}\gamma_\ep\eta_\ep d\theta_\ep d\gamma_\ep\\
&={-1\over \sqrt{1-\ep^2}}\int_{-1}^1\int_{0}^{2\pi}\gamma_\ep(1-\gamma_\ep^2)^{1\over2}\sin(\theta_\ep) d\theta_\ep d\gamma_\ep=0,\\
\iint_\Omega\nabla\partial_\ep\psi_{\ep}\cdot\nabla\eta_{\ep}dxdy&={1\over 1-\ep^2}\iint_\Omega\nabla\xi_\ep\cdot \nabla\eta_\ep dxdy={-1\over 1-\ep^2}\int_{-1}^1\int_{0}^{2\pi}\xi_\ep\eta_\ep d\theta_\ep d\gamma_\ep\\
&={-1\over 1-\ep^2}\int_{-1}^1\int_{0}^{2\pi}(1-\gamma_\ep^2)\sin(\theta_\ep)\cos(\theta_\ep) d\theta_\ep d\gamma_\ep=0,\\
\iint_\Omega\nabla\partial_y\psi_{\ep}\cdot \nabla\gamma_{\ep}dxdy&={1\over \sqrt{1-\ep^2}}\iint_\Omega|\nabla\gamma_\ep|^2 dxdy={1\over \sqrt{1-\ep^2}}\int_{-1}^1\int_{0}^{2\pi}(1-\gamma_\ep^2)d\theta_\ep d\gamma_\ep\\
&={1\over \sqrt{1-\ep^2}}{8\over3}\pi,\\
\iint_\Omega\nabla\partial_\ep\psi_{\ep}\cdot\nabla\gamma_{\ep}dxdy&={1\over 1-\ep^2}\iint_\Omega\nabla\xi_\ep \cdot\nabla\gamma_{\ep} dxdy={-1\over 1-\ep^2}\int_{-1}^1\int_{0}^{2\pi}\xi_\ep \gamma_{\ep} d\theta_\ep d\gamma_\ep\\
&={-1\over 1-\ep^2}\int_{-1}^1\int_{0}^{2\pi}(1-\gamma_\ep^2)^{1\over2}\cos(\theta_\ep) \gamma_{\ep} d\theta_\ep d\gamma_\ep=0,\\
\iint_\Omega\nabla\partial_\ep\psi_{\ep}\cdot \nabla\xi_{\ep}dxdy&={1\over 1-\ep^2}\iint_\Omega\nabla\xi_\ep \cdot\nabla\xi_{\ep} dxdy={1\over 1-\ep^2}\int_{-1}^1\int_{0}^{2\pi}(1-\xi_\ep^2) d\theta_\ep d\gamma_\ep\\
&={1\over 1-\ep^2}\int_{-1}^1\int_{0}^{2\pi}\left(\gamma_\ep^2\cos^2(\theta_\ep)+\sin^2(\theta_\ep)\right) d\theta_\ep d\gamma_\ep={1\over 1-\ep^2}{8\over3}\pi.
\end{align*}
Then
\begin{align*}
\iint_\Omega\nabla\partial_x\psi_{\ep}\cdot\nabla\gamma_{\ep}dxdy=
 \iint_\Omega\nabla\partial_x\psi_{\ep}\cdot\nabla\xi_{\ep}dxdy=\iint_\Omega\nabla\partial_y\psi_{\ep}\cdot\nabla\xi_{\ep}dxdy =0.
 \end{align*}
Thus,
\begin{align*}
{\partial(S_1,S_2,S_3)\over \partial(x_1,y_1,\ep_1)}\bigg|_{\tilde \omega=\omega_{\ep_0},x_1=0,y_1=0,\ep_1=\ep_0}
=&\left| \begin{array}{cccc} {\ep_0\over \sqrt{1-\ep_0^2}}{8\over3}\pi &0
&0\\ 0& {-1\over \sqrt{1-\ep_0^2}}{8\over3}\pi &
0\\
0&0  &{-1\over 1-\ep_0^2}{8\over3}\pi
\end{array} \right|\\
=&{\ep_0\over (1-\ep_0^2)^2}\left({8\over3}\pi\right)^3\neq0.
\end{align*}
By the Implicit Function Theorem, there exists $\delta=\delta(\ep_0)>0$ such that
for any
$\tilde \omega\in Y_{non}$ with $d_2(\tilde \omega,\omega_{\ep_0})\leq \delta$, there exist $\tilde x_0=\tilde x_0(\tilde\omega)\in\mathbb{T}_{2\pi}$, $\tilde y_0=\tilde y_0(\tilde\omega)\in\mathbb{R}$ and $\tilde \ep_0=\tilde \epsilon_0(\tilde\omega)\in(a(\ep_0),b(\ep_0))\subset(0,1)$, depending continuously on $\tilde \omega$, such that
 $S_i(\tilde\omega,\tilde x_0(\tilde\omega),$ $\tilde y_0(\tilde\omega),\tilde \ep_0(\tilde\omega))=0$ for $i=1,2,3$.

Define a mapping $:\chi \mapsto \mathcal{T}\chi$ by
\begin{align*}
(\mathcal{T}\chi)(\tilde \omega):=\chi(\tilde \omega)-\left({\partial(S_1,S_2,S_3)\over \partial(x_1,y_1,\ep_1)}\bigg|_{\tilde \omega=\omega_{\ep_0},x_1=0,y_1=0,\ep_1=\ep_0}\right)^{-1}\vec{S}(\tilde \omega,\chi(\tilde \omega)^T),
\end{align*}
where $\chi\in C(\bar{B}_{d_2}(\omega_{\ep_0},\delta),\Omega\times (0,1))$, $\bar{B}_{d_2}(\omega_{\ep_0},\delta)$ is the closed ball in $Y_{non}$ centred at $\omega_{\ep_0}$ with semi-radius $\delta$ under the distance $d_2$, and  $\vec{S}=(S_1,S_2,S_3)^T$. The distance between $\chi_1$ and $\chi_2$ is given by  $\rho(\chi_1,\chi_2)=\max_{\tilde \omega\in\bar{B}_{d_2}(\omega_{\ep_0},\delta)}|\chi_1(\tilde \omega)-\chi_2(\tilde \omega)|.$
It is  standard that $\mathcal{T}$ is a contracting mapping with rate $\mu\in(0,1)$ on $\mathcal{H}=\{\chi\in C(\bar{B}_{d_2}(\omega_{\ep_0},\delta),\Omega\times (0,1))|\chi(\omega_{\ep_0})=(0,0,\ep_0)^T,|\chi(\tilde \omega)-(0,0,\ep_0)^T|\leq\nu\}$ for some $\nu>0$, and moreover, $\chi^*$, which is defined by $\chi^*(\tilde \omega)=(\tilde x_0(\tilde \omega),\tilde y_0(\tilde \omega),\tilde \ep_0(\tilde \omega))^T$ on $\bar{B}_{d_2}(\omega_{\ep_0},\delta)$, is the unique fixed point of $\mathcal{T}$. Then $\rho(\chi,\chi^*)= \rho(\chi,\mathcal{T}\chi^*)\leq \rho(\chi,\mathcal{T}\chi)+\rho(\mathcal{T}\chi,\mathcal{T}\chi^*)\leq \rho(\chi,\mathcal{T}\chi)+\mu\rho(\chi,\chi^*)$ for $\chi\in\mathcal{H}$, which implies that $\rho(\chi,\chi^*)\leq {1\over 1-\mu}\rho(\chi,\mathcal{T}\chi)$. By  choosing $\chi_0\equiv(0,0,\ep_0)^T$,  for any $\tilde \omega\in\bar{B}_{d_2}(\omega_{\ep_0},\delta)$ we have
\begin{align*}
&|\tilde x_0(\tilde \omega)|+|\tilde y_0(\tilde \omega)|+|\tilde \ep_0(\tilde \omega)-\ep_0|\leq \rho(\chi_0,\chi^*)\leq {1\over 1-\mu}\rho(\chi_0,\mathcal{T}\chi_0)\\
\leq& {C\over 1-\mu}\left\|\left({\partial(S_1,S_2,S_3)\over \partial(x_1,y_1,\ep_1)}\bigg|_{\tilde \omega=\omega_{\ep_0},x_1=0,y_1=0,\ep_1=\ep_0}\right)^{-1}\right\|\max_{\tilde \omega\in\bar{B}_{d_2}(\omega_{\ep_0},\delta)}|\vec{S}(\tilde \omega,(0,0,\ep_0))|\leq C(\ep_0)\sqrt{\delta},
\end{align*}
where $\|\cdot\|$ is a norm on $\mathbb{R}^{3\times 3}$.

Let $x_0\neq 0$ or $y_0\neq0$. For any $\tilde \omega\in Y_{non}$ with $d_2(\tilde \omega,\omega_{\ep_0}(x+x_0,y+y_0))=\|\tilde\psi(x,y)-\psi_{\ep_0}(x+x_0,y+y_0)\|_{\dot{H}^1(\Omega)}^2\leq \delta$, we define $\tilde \psi_1(x,y)=\tilde \psi(x-x_0,y-y_0)$ and $\tilde \omega_1=-\Delta\tilde \psi_1$. Then $d_2(\tilde \omega_1,\omega_{\ep_0})=\|\tilde\psi_1-\psi_{\ep_0}\|_{\dot{H}^1(\Omega)}^2\leq \delta$, and thus, there exist $\tilde x_0(\tilde\omega_1)\in\mathbb{T}_{2\pi}$, $\tilde y_0(\tilde\omega_1)\in\mathbb{R}$ and $\tilde \epsilon_0(\tilde\omega_1)\in(a(\ep_0),b(\ep_0))$ such that
\begin{align*}
S_i(\tilde\omega_1,\tilde x_0(\tilde\omega_1), \tilde y_0(\tilde\omega_1),\tilde \ep_0(\tilde\omega_1))=S_i(\tilde\omega,x_0+\tilde x_0(\tilde\omega_1), y_0+\tilde y_0(\tilde\omega_1),\tilde\ep_0(\tilde\omega_1))=0
\end{align*}
 for $i=1,2,3$. The conclusion follows from setting $\tilde x_0=x_0+\tilde x_0(\tilde \omega_1), \tilde y_0=y_0+\tilde y_0(\tilde \omega_1)$ and $\tilde\ep_0=\tilde \ep_0(\tilde \omega_1)$.
\end{proof}

Moreover, we prove that the following functional is not locally flat on the family of steady states $\omega_{\ep}, \ep\in[0,1)$. This is useful to control the distance between  the evolved solution and the given steady state  in the $\ep$ direction.

\begin{lemma}\label{intOmega2} As a function of $\ep$,
\begin{align}\label{additional functional-2deuler}I(\omega_\ep)\triangleq\iint_{\Omega} (-\omega_\ep)^{3\over2} dxdy\end{align}
cannot be a constant on any subinterval of $(-1, 1)$, where  $\omega_\ep = - \frac{1 - \ep^2}{(\cosh(y) + \ep \cos(x))^2}$.
\end{lemma}
\begin{proof}
By \eqref{Jacobian of the transformation-ep}, we have  $$\frac{\partial (\theta_\ep, \gamma_\ep)}{\partial (x, y)} =\frac{1}{2} g'(\psi_\ep) = - \omega_\ep,$$
and thus,
\begin{align*}
\iint_{\Omega} (-\omega_\ep)^{3\over2}  dxdy = \int_{-1}^1 \int_{0}^{2\pi} (-\omega_\ep)^{1\over2}  d \theta_\ep d\gamma_\ep.
\end{align*}
By \eqref{omega-xi-eta-gamma-ep}, we have
\begin{align*}
-\omega_\ep
 = \eta_\ep^2 + \frac{1}{1-\ep^2} (\xi_\ep - \ep)^2.
\end{align*}
Recall that $
\eta_\ep = \sqrt{1-\gamma_\ep^2} \sin(\theta_\ep)$ and $
\xi_\ep = \sqrt{1-\gamma_\ep^2} \cos(\theta_\ep)$. Then
we have
\begin{align*}
I(\omega_\ep)&=\iint_{\Omega} (-\omega_\ep)^{3\over2}  dxdy
 = \int_{-1}^1 \int_{0}^{2\pi} (-\omega_\ep)^{1\over2}  d \theta_\ep d\gamma_\ep \\
& = \int_{-1}^1 \int_{0}^{2\pi} \left(\eta_\ep^2 + \frac{1}{1-\ep^2} (\xi_\ep - \ep)^2\right)^{1\over2} d \theta_\ep d\gamma_\ep \\
& = \int_{-1}^1 \int_{0}^{2\pi} \left((1-\gamma_\ep^2) \sin^2(\theta_\ep) + \frac{1}{1-\ep^2} \left( \sqrt{1-\gamma_\ep^2} \cos(\theta_\ep) -\ep \right)^2\right)^{1\over2}  d \theta_\ep d\gamma_\ep \\
& \geq\frac{1}{\sqrt{1-\ep^2}}\int_{-1}^1 \int_{0}^{2\pi} \left| \sqrt{1-\gamma_\ep^2} \cos(\theta_\ep) -\ep \right|  d \theta_\ep d\gamma_\ep
\\
&\to\infty \quad\text{as}\quad \ep\to\pm1^{\mp}.
\end{align*}
Since $I(\omega_\ep)$, as a function of $\ep$, is real-analytic on $(-1,1)$, $I(\omega_\ep)$ cannot be a constant on any subinterval of $(-1, 1)$.
\end{proof}

\subsection{Proof of nonlinear orbital stability for co-periodic perturbations}
Now, we are in a position to prove Theorem \ref{main result4-nonlinear orbital stability}.

\begin{proof}[Proof of  Theorem \ref{main result4-nonlinear orbital stability}]

We prove in the Appendix the existence of weak solutions to the 2D Euler equation with initial vorticity $\tilde\omega_0\in Y_{non}$. The first step is to construct a smooth approximate solution sequence. For $\mu>0$, let $\tilde\omega_0^\mu$ be the mollified initial vorticity defined by \eqref{tilde-omega0-kappa-def}. In Lemma \ref{lem-construction of an approximate solution sequence}, we show that the initial velocity
\[
\vec v_0^\mu=K\ast \tilde\omega_0^\mu
\]
generates a global smooth solution $\vec v^\mu(t)$ to the 2D Euler equation with $\vec v^\mu(t)\in H^q(\Omega)$ for every $q\ge 3$. The family $\{\vec v^\mu\}$ forms an approximate solution sequence with  $L^1$ and $L^2$ vorticity control; see Definition \ref{Approximate solution sequence for the 2D Euler equation}. We then prove in Lemma \ref{convergence of an approximate solution sequence} and Theorem \ref{existence of weak solution to the 2D Euler equation-thm} that $\vec v^\mu\to \vec v$ in $L^1\cap L^2(\Omega_{R,T})$ for every $R,T>0$, and that the limit $\vec v$ is a weak solution of the 2D Euler equation with initial vorticity $\tilde\omega_0\in Y_{non}$, where $\Omega_{R,T}=[0,T]\times B_R$ and $B_R=\{(x,y)\in\mathbb T_{2\pi}\times[-R,R]\}$. With this existence result in hand, we divide the proof of nonlinear orbital stability of $\omega_{\ep_0}$ into two steps.
\vspace{0.5mm}

\noindent{\bf{Step 1.}} Prove the nonlinear orbital stability for the smooth approximate solution $\omega^\mu(t)=\curl(\vec{v}^{\mu}(t))$. More precisely,
for any $\kappa>0$, there exists $\tilde\delta=\tilde\delta(\ep_0,\kappa)>0$ (independent of $\mu$) such that if
\begin{align}\nonumber&\inf_{(x_0,y_0)\in\Omega} d(\tilde \omega^\mu(0),\omega_{\ep_0}(x+x_0,y+y_0))\\
+&\inf_{(x_0,y_0)\in\Omega}\|\tilde \omega^\mu(0)-\omega_{\ep_0}(x+x_0,y+y_0)\|_{L^2(\Omega)}<\tilde\delta(\ep_0,\kappa),\label{initial-vorticity-app-small}
\end{align}
 then for any $t\geq0$, we have
\begin{align}\label{onlinear orbital stability-app}
\inf_{(x_0,y_0)\in\Omega}d(\tilde \omega^\mu(t),\omega_{\ep_0}(x+x_0,y+y_0))<\kappa.
\end{align}

By Lemma \ref{tilde-omega0-kappa-properties} (8), $\tilde \omega^\mu(0)\in Y_{non}$. It follows from Corollary \ref{y-tilde-omega-pseudoenergy-conserved} (1) that $\tilde \omega^\mu(t)\in Y_{non}$ for $t>0$. Thus, we infer from Lemma \ref{well-poseness-Poisson-equation-nonlinear-case} and \eqref{d1-well-def} that $d(\tilde \omega^\mu(t),\omega_{\ep_0}(x+x_0,y+y_0))$ is well-defined for $t>0$.
By Lemma \ref{imp-vertical condition},  there exists $\delta_0(\ep_0)>0$ such that for any  $(x_0,y_0)\in\Omega$ and  $\tilde \omega\in Y_{non}$ with $d_2(\tilde \omega,\omega_{\ep_0}(x+x_0,y+y_0))< \delta_0(\ep_0)$, there exist $(\tilde x_0,\tilde y_0)\in\Omega$ and $\tilde\epsilon_0\in(a(\ep_0),b(\ep_0))$, depending continuously on $\tilde\omega, x_0$ and $y_0$, such that
\begin{align}\label{app-lemma-imp-vertical condition}
\tilde \psi\left(x-\tilde x_0,y-\tilde y_0\right)-\psi_{\tilde\ep_0}(x,y)\perp\ker \left( A_{\tilde\ep_0}\right)\quad \text{in}\quad \dot{H}^1(\Omega)
\end{align}
 and
$
|x_0-\tilde x_0|+|y_0-\tilde y_0|+|\ep_0-\tilde \ep_0|\leq C(\ep_0)\sqrt{\delta_0(\ep_0)}
$
for some $a(\ep_0)\in (0,\ep_0)$ and $b(\ep_0)\in(\ep_0,1)$.
For any $\kappa>0$, let $\tilde\delta=\tilde\delta(\ep_0,\kappa)<\min\big\{{\kappa^2\over8C_1C_2(\ep_0)^2C_3(\ep_0)^2},$ ${\delta_0(\ep_0)\over2},1\big\}$, where $C_1, C_2(\ep_0), C_3(\ep_0)>1$ are  determined by \eqref{de-c1}, \eqref{I-omegaep1t0-omegaep0} and \eqref{de-c3}.
For the initial data $\tilde \omega^\mu(0)$ satisfying
\eqref{initial-vorticity-app-small},
there exist $(x_0^\mu(0),y_0^\mu(0))\in\Omega$ and $(x_*^\mu(0),y_*^\mu(0))\in\Omega$ such that
\begin{align}\label{initial data distance} d(\tilde \omega^\mu(0),\omega_{\ep_0}(x+x_0^\mu(0),y+y_0^\mu(0)))<\tilde\delta(\ep_0,\kappa),
\end{align}
and
\begin{align}
\|\tilde \omega^\mu(0)-\omega_{\ep_0}(x+x_*^\mu(0),y+y_*^\mu(0))\|_{L^2(\Omega)}<\tilde\delta(\ep_0,\kappa).\label{initial data}\end{align}

For $t\geq0$, we claim that if there exists $( x_0^\mu(t),y_0^\mu(t))\in\Omega$ such that $d(\tilde \omega^\mu(t),\omega_{\ep_0}(x+x_0^\mu(t),y+y_0^\mu(t)))<\delta_0(\ep_0)$, then there exist $(x_1^\mu(t),y_1^\mu(t))\in\Omega$ and $\ep_1^\mu(t)\in(a(\ep_0),b(\ep_0))$ such that
\begin{align}\label{a prior estimate}
d(\tilde \omega^\mu(t),\omega_{\ep_1^\mu(t)}(x+x_1^\mu(t),y+y_1^\mu(t)))<{\kappa^2\over4C_2(\ep_0)^2C_3(\ep_0)^2}.
\end{align}
In fact, by applying \eqref{app-lemma-imp-vertical condition} to  $\tilde \omega^\mu(t)$,  we can choose $(x_1^\mu(t),y_1^\mu(t))\in\Omega$ and $\ep_1^\mu(t)\in(a(\ep_0),b(\ep_0))$, depending continuously on $t$, such that
$\tilde \psi^\mu(x-x_1^\mu(t),$ $y-y_1^\mu(t))-\psi_{\ep_1^\mu(t)}(x,y)\perp\ker \left( A_{\ep_1^\mu(t)}\right)$
in $ \tilde X_{\ep_1^\mu(t)}$,
and
\begin{align}\label{translation distance}
|x_0^{\mu}(t)-x_1^{\mu}(t)|+|y_0^{\mu}(t)-y_1^{\mu}(t)|+|\ep_0-\ep_1^{\mu}(t)|\leq C(\ep_0)\sqrt{\delta_0(\ep_0)}.
\end{align}
By \eqref{initial data distance} and Lemma \ref{imp-vertical condition}, $\sqrt{\delta_0(\ep_0)}$ in \eqref{translation distance} can be replaced by $\sqrt{\tilde\delta(\ep_0,\kappa)}$ for $t=0$. By adding a constant if necessary, we have $\tilde \psi^\mu(x-x_1^\mu(t),$ $y-y_1^\mu(t))-\psi_{\ep_1^\mu(t)}(x,y)\in
\tilde X_{\ep_1^\mu(t)}$. Noting that if the constant is omitted, then the proof is the same since $\iint_\Omega\psi\omega dxdy=\iint_\Omega(\psi-c)\omega dxdy$ in  \eqref{H-omega-H-0} for any $c\in\mathbb{R}$ due to $\iint_\Omega\omega dxdy=0$. So in this proof, we write $\tilde \psi^\mu(x-x_1^\mu(t),$ $y-y_1^\mu(t))-\psi_{\ep_1^\mu(t)}(x,y)\in
\tilde X_{\ep_1^\mu(t)}$ in the sense that a constant difference is allowed.
By taking $\tilde\delta(\ep_0,\kappa)>0$ smaller, we infer from \eqref{translation distance} for $t=0$ that  $ d(\omega_{\ep_0}(x+x_0^\mu(0),y+y_0^\mu(0)),\omega_{\ep}(x+x_1^\mu(0),y+y_1^\mu(0)))<{\kappa^2\over8C_1C_2(\ep_0)^2C_3(\ep_0)^2}$, which along with \eqref{initial data distance}, implies
\begin{align*}
&d(\tilde \omega^\mu(0),\omega_{\ep}(x+x_1^\mu(0),y+y_1^\mu(0)))\\
\leq& d(\tilde \omega^\mu(0),\omega_{\ep_0}(x+x_0^\mu(0),y+y_0^\mu(0)))\\
&+d( \omega_{\ep_0}(x+x_0^\mu(0),y+y_0^\mu(0)), \omega_{\ep}(x+x_1^\mu(0), y+ y_1^\mu(0)))\\
\leq &{\kappa^2\over8C_1C_2(\ep_0)^2C_3(\ep_0)^2}+{\kappa^2\over8C_1C_2(\ep_0)^2C_3(\ep_0)^2}={\kappa^2\over4C_1C_2(\ep_0)^2C_3(\ep_0)^2},
\end{align*}
where $\ep=\ep_0$ or $\ep_1^\mu(0)$.
Take $\tau\in(0,1)$ small enough such that  $\left(  (1-\tau) C_0- \frac 1 2 \tau\right)>\tau$, where
$C_0>0$ is given in \eqref{A-ep-positive-lower-bound}.
By \eqref{H-omega-H-0}-\eqref{def-functional-B}, \eqref{B-ep-A-ep}-\eqref{A-ep-positive-lower-bound} and Lemma \ref{B-C2}, we have
\begin{align}\nonumber
& d(\tilde \omega^\mu(0),\omega_{\ep_1^\mu(0)}(x+x_1^\mu(0),y+y_1^\mu(0)))\\\nonumber
\geq&H_{\ep_1^\mu(0)}(\tilde \omega^\mu(0)-\omega_{\ep_1^\mu(0)}(x+x_1^\mu(0),y+y_1^\mu(0))) - H_{\ep_1^\mu(0)}(0)\\\nonumber
 =&H_{\ep_1^\mu(t)}(\tilde \omega_{tran}^\mu(t)-\omega_{\ep_1^\mu(t)}) - H_{\ep_1^\mu(t)}(0)\\\nonumber
 =&
 \tau d_1(\tilde \omega_{tran}^\mu(t),\omega_{\ep_1^\mu(t)}) - \frac 1 2 \tau d_2(\tilde \omega_{tran}^\mu(t),\omega_{\ep_1^\mu(t)}) \\\nonumber
 &+ (1-\tau) \left( d_1(\tilde \omega_{tran}^\mu(t),\omega_{\ep_1^\mu(t)}) - \frac 1 2  d_2(\tilde \omega_{tran}^\mu(t),\omega_{\ep_1^\mu(t)})\right) \\\nonumber
 \geq& \tau d_1(\tilde \omega_{tran}^\mu(t),\omega_{\ep_1^\mu(t)}) - \frac 1 2 \tau d_2(\tilde \omega_{tran}^\mu(t),\omega_{\ep_1^\mu(t)})
 + (1-\tau) \mathscr{B}_{\ep_1^\mu(t)}(\tilde \psi_{tran}^\mu(t)-\psi_{\ep_1^\mu(t)})\\\nonumber
 =&\tau d_1(\tilde \omega_{tran}^\mu(t),\omega_{\ep_1^\mu(t)}) - \frac 1 2 \tau d_2(\tilde \omega_{tran}^\mu(t),\omega_{\ep_1^\mu(t)})\\\nonumber
& + (1-\tau) \left(\langle A_{\ep_1^\mu(t)}(\tilde \psi_{tran}^\mu(t)-\psi_{\ep_1^\mu(t)}),(\tilde \psi_{tran}^\mu(t)-\psi_{\ep_1^\mu(t)})\rangle+o(d_2(\tilde \omega_{tran}^\mu(t),\omega_{\ep_1^\mu(t)}))\right)\\\nonumber
\geq&\tau d_1(\tilde \omega_{tran}^\mu(t),\omega_{\ep_1^\mu(t)})+\left(  (1-\tau) C_0- \frac 1 2 \tau\right) d_2(\tilde \omega_{tran}^\mu(t),\omega_{\ep_1^\mu(t)})\\\nonumber
&+o(d_2(\tilde \omega_{tran}^\mu(t),\omega_{\ep_1^\mu(t)}))\\\nonumber
\geq&\tau d(\tilde \omega_{tran}^\mu(t),\omega_{\ep_1^\mu(t)})+o(d(\tilde \omega_{tran}^\mu(t),\omega_{\ep_1^\mu(t)}))\\\label{t control by 0-Taylor of dual functional}
=&\tau d(\tilde \omega^\mu(t),\omega_{\ep_1^\mu(t)}(x+x_1^\mu(t),y+y_1^\mu(t)))
+o(d(\tilde \omega^\mu(t),\omega_{\ep_1^\mu(t)}(x+x_1^\mu(t),y+y_1^\mu(t)))),
\end{align}
where $\tilde \omega_{tran}^\mu(t)\triangleq\tilde \omega^\mu(t,x-x_1^\mu(t),y-y_1^\mu(t))$, $\tilde \psi_{tran}^\mu(t)\triangleq\tilde \psi^\mu(t,x-x_1^\mu(t),y-y_1^\mu(t))$, and we used the fact that $H_{\ep}(\tilde \omega^\mu(t)-\omega_{\ep}(x+x_1,y+y_1)) - H_{\ep}(0)$ is conserved for all $t, x_1, y_1$ and $\ep$. Here, the conservation for $t$ and $\ep$ can be deduced from Corollary \ref{y-tilde-omega-pseudoenergy-conserved} (2) and \eqref{H independent ep}, respectively. Then for $\kappa>0$ sufficiently small, by \eqref{t control by 0-Taylor of dual functional} and the continuity of $d(\tilde \omega^\mu(t),\omega_{\ep_1^\mu(t)}(x+x_1^\mu(t),y+y_1^\mu(t)))$ on $t$ we have
\begin{align}\nonumber
&d(\tilde \omega^\mu(t),\omega_{\ep_1^\mu(t)}(x+x_1^\mu(t),y+y_1^\mu(t)))\\\label{de-c1}
\leq& C_1 d(\tilde \omega^\mu(0),\omega_{\ep_1^\mu(0)}(x+x_1^\mu(0),y+y_1^\mu(0)))< {\kappa^2\over4C_2(\ep_0)^2C_3(\ep_0)^2},
\end{align}
where  $C_1={2\over\tau}>1$. This proves \eqref{a prior estimate}.

For any $\kappa\in(0,\min\{\delta_0(\ep_0), 1\})$, suppose that \eqref{onlinear orbital stability-app} is not true. Then there exists  $t_0>0$ such that $\inf_{(x_0,y_0)\in\Omega}d(\tilde \omega^\mu(t),\omega_{\ep_0}(x+x_0,y+y_0))<\kappa$ for $0\leq t<t_0$ and
\begin{align}\label{suppose inf=kappa}
\inf_{(x_0,y_0)\in\Omega}d(\tilde \omega^\mu(t_0),\omega_{\ep_0}(x+x_0,y+y_0))=\kappa.
\end{align}
Since $\kappa<\delta_0(\ep_0)$, there exists $( x_0^\mu(t),y_0^\mu(t))\in\Omega$, depending continuously on $t$, such that $ d(\tilde \omega^\mu(t),\omega_{\ep_0}(x+x_0^\mu(t),y+y_0^\mu(t)))<\delta_0(\ep_0)$ for $0\leq t\leq t_0$.
 By \eqref{a prior estimate},
 there exist $(x_1^\mu(t),y_1^\mu(t))\in\Omega$ and $\ep_1^\mu(t)\in(a(\ep_0),b(\ep_0))$ such that
\begin{align}\label{d-t0-ep1}
d(\tilde \omega^\mu(t),\omega_{\ep_1^\mu(t)}(x+x_1^\mu(t),y+y_1^\mu(t)))<{\kappa^2\over4C_2(\ep_0)^2C_3(\ep_0)^2}<{\kappa\over2},\quad 0\leq t\leq t_0.
\end{align}
We then show that
\begin{align}\label{ep1t0ep0}
d(\omega_{\ep_1^\mu(t_0)},\omega_{\ep_0})<{\kappa\over 2}.
\end{align}
Assume that \eqref{ep1t0ep0} is true. Then
\begin{align*}
&d(\tilde \omega^\mu(t_0),\omega_{\ep_0}(x+x_1^\mu(t_0),y+y_1^\mu(t_0)))\\
\leq& d(\tilde \omega^\mu(t_0),\omega_{\ep_1^\mu(t_0)}(x+x_1^\mu(t_0),y+y_1^\mu(t_0)))\\
&+d(\omega_{\ep_1^\mu(t_0)}(x+x_1^\mu(t_0),y+y_1^\mu(t_0)),\omega_{\ep_0}(x+x_1^\mu(t_0),y+y_1^\mu(t_0))) \\
<&{\kappa\over2}+{\kappa\over2}=\kappa.
\end{align*}
This contradicts \eqref{suppose inf=kappa}.

The rest is to prove \eqref{ep1t0ep0}. By the continuity of $
d(\omega_{\ep},\omega_{\ep_0})$ on $\ep$, it suffices to show that $|\ep_1^\mu(t_0)-\ep_0|<\delta_1(\ep_0)$ for some  $\delta_1(\ep_0)>0$ small enough.
Note that
$|\ep_1^{\mu}(0)-\ep_0|\leq C(\ep_0)\sqrt{\tilde\delta(\ep_0,\kappa)}$ by \eqref{translation distance} for $t=0$, and $\ep_1^{\mu}(t)$ is continuous on $t\in[0,t_0]$.
 By Lemma \ref{intOmega2} and taking $\tilde\delta(\ep_0,\kappa)>0$ smaller, we only need to prove that
\begin{align}\label{I-omegaep1t0-omegaep0}
|I(\omega_{\ep_1^\mu(t)})-I(\omega_{\ep_0})|<{\kappa\over C_2(\ep_0)},\quad 0\leq t\leq t_0
\end{align}
for some $C_2(\ep_0)>1$ large enough, where $I(\tilde \omega)=\iint_{\Omega}(-\tilde \omega)^{3\over2}dxdy$ for $\tilde \omega\in Y_{non}$. In fact, by Taylor's formula, we have
\begin{align}\nonumber
&d_1(\tilde \omega^\mu(t),\omega_{\ep_1^\mu(t)}(x+x_1^\mu(t),y+y_1^\mu(t)))\\\nonumber
=&\iint_{\Omega}\bigg(h(\tilde \omega^\mu(t))-h(\omega_{\ep_1^\mu(t)}(x+x_1^\mu(t),y+y_1^\mu(t)))\\\nonumber
&-h'(\omega_{\ep_1^\mu(t)}(x+x_1^\mu(t),y+y_1^\mu(t)))
(\tilde \omega^\mu(t)-\omega_{\ep_1^\mu(t)}(x+x_1^\mu(t),y+y_1^\mu(t)))\bigg)dxdy\\\nonumber
=&\int_0^1\iint_\Omega{(1-r)\big(\tilde \omega^\mu(t)-\omega_{\ep_1^\mu(t)}(x+x_1^\mu(t),y+y_1^\mu(t))\big)^2\over 2|\omega^{\mu,r}(t)|}dxdydr\\\label{d1 estimates}
\geq&\iint_\Omega{\big(\tilde \omega^\mu(t)-\omega_{\ep_1^\mu(t)}(x+x_1^\mu(t),y+y_1^\mu(t))\big)^2\over 4|\tilde \omega^\mu(t)+\omega_{\ep_1^\mu(t)}(x+x_1^\mu(t),y+y_1^\mu(t))|}dxdy,
\end{align}
where $0\leq t\leq t_0$ and $\omega^{\mu,r}(t,x,y)=r\tilde \omega^\mu(t,x,y)+(1-r)\omega_{\ep_1^\mu(t)}(x+x_1^\mu(t),y+y_1^\mu(t))
$ for $r\in[0,1]$.
Noting that $I(\tilde \omega^\mu(t))$ is conserved for all $t$, by \eqref{d1 estimates} and \eqref{d-t0-ep1}  we have
\begin{align}\nonumber
&|I(\tilde \omega^\mu(0))-I(\omega_{\ep_1^\mu(t)})|=|I(\tilde \omega^\mu(t))-I(\omega_{\ep_1^\mu(t)}(x+x_1^\mu(t),y+y_1^\mu(t)))|\\\nonumber
=&\left|\iint_\Omega\left((-\tilde \omega^\mu(t))^{3\over2}-(-\omega_{\ep_1^\mu(t)}(x+x_1^\mu(t),y+y_1^\mu(t)))^{3\over2}\right)dxdy\right|\\\nonumber
=&{3\over2}\left|\int_0^1\iint_\Omega|\omega^{\mu,r}(t)|^{1\over2}(\tilde \omega^\mu(t)-\omega_{\ep_1^\mu(t)}(x+x_1^\mu(t),y+y_1^\mu(t)))dxdydr\right|\\\nonumber
\leq&{3\over2}\bigg|\iint_\Omega|\tilde \omega^\mu(t)+\omega_{\ep_1^\mu(t)}(x+x_1^\mu(t),y+y_1^\mu(t))|^{1\over2}\cdot\\\nonumber
&|\tilde \omega^\mu(t)-\omega_{\ep_1^\mu(t)}(x+x_1^\mu(t),y+y_1^\mu(t))|dxdy\bigg|\\\nonumber
\leq&{3\over2}\left(\iint_\Omega{(\tilde \omega^\mu(t)-\omega_{\ep_1^\mu(t)}(x+x_1^\mu(t),y+y_1^\mu(t)))^2\over 4|\tilde \omega^\mu(t)+\omega_{\ep_1^\mu(t)}(x+x_1^\mu(t),y+y_1^\mu(t))|}dxdy\right)^{1\over2}\cdot\\\nonumber
&\left(\iint_\Omega4|\tilde \omega^\mu(t)+\omega_{\ep_1^\mu(t)}(x+x_1^\mu(t),y+y_1^\mu(t))|^2dxdy\right)^{1\over2}\\\nonumber
\leq &{3\sqrt{2}}d_1(\tilde \omega^\mu(t),\omega_{\ep_1^\mu(t)}(x+x_1^\mu(t),y+y_1^\mu(t)))^{1\over2}
\left(\|\tilde \omega^\mu(t)\|_{L^2(\Omega)}^2+\|\omega_{\ep_1^\mu(t)}\|_{L^2 (\Omega)}^2\right)^{1\over2}\\\nonumber
\leq&{3\sqrt{2}}d_1(\tilde \omega^\mu(t),\omega_{\ep_1^\mu(t)}(x+x_1^\mu(t),y+y_1^\mu(t)))^{1\over2}
\left(\|\tilde \omega^\mu(0)\|_{L^2(\Omega)}^2+\|\omega_{\ep_1^\mu(t)}\|_{L^2 (\Omega)}^2\right)^{1\over2}\\\nonumber
\leq&C_3(\ep_0)d_1(\tilde \omega^\mu(t),\omega_{\ep_1^\mu(t)}(x+x_1^\mu(t),y+y_1^\mu(t)))^{1\over2}\\\label{Iomega0Iomegat1}
< &{\kappa\over2C_2(\ep_0)},\quad 0\leq t\leq t_0,
\end{align}
where
\begin{align}\label{de-c3}
C_3(\ep_0)=&{3}\sqrt{2}\bigg(\left(1+\|\omega_{\ep_0}\|_{L^2(\Omega)}\right)^2+
\max_{\ep\in [a(\ep_0),b(\ep_0)]}\|\omega_{\ep}\|_{L^2 (\Omega)}^2
\bigg)^{1\over2}>1,
\end{align}
and we used
\begin{align}&\|\tilde \omega^\mu(0)\|_{L^2(\Omega)}\leq\|\tilde \omega^\mu(0)-\omega_{\ep_0}(x+x_*^\mu(0),y+y_*^\mu(0))\|_{L^2(\Omega)}+ \|\omega_{\ep_0}\|_{L^2(\Omega)}\nonumber\\\label{uniform-L2-bound-tilde-omega0}
\leq& \tilde\delta(\ep_0,\kappa)+\|\omega_{\ep_0}\|_{L^2(\Omega)}\leq 1+\|\omega_{\ep_0}\|_{L^2(\Omega)}\end{align}  due to \eqref{initial data}.
Similar to \eqref{d1 estimates}-\eqref{Iomega0Iomegat1}, we have
\begin{align}\nonumber
&|I(\tilde \omega^\mu(0))-I(\omega_{\ep_0})|=|I(\tilde \omega^\mu(0))-I(\omega_{\ep_0}(x+x_1^\mu(0),y+y_1^\mu(0)))|\\\label{Itildeomega0Iomegaep0}
\leq& C_3(\ep_0)d_1(\tilde \omega^\mu(0),\omega_{\ep_0}(x+x_1^\mu(0),y+y_1^\mu(0)))^{1\over2}
\leq {\kappa\over2\sqrt{C_1}C_2(\ep_0)}< {\kappa\over2C_2(\ep_0)},
\end{align}
where we used \eqref{initial data distance}.
Combining \eqref{Iomega0Iomegat1} and \eqref{Itildeomega0Iomegaep0}, we have
\begin{align*}
|I(\omega_{\ep_1^\mu(t)})-I(\omega_{\ep_0})|\leq |I(\tilde \omega^\mu(0))-I(\omega_{\ep_1^\mu(t)})|+|I(\tilde \omega^\mu(0))-I(\omega_{\ep_0})|
<{\kappa\over C_2(\ep_0)}
\end{align*}
for $0\leq t\leq t_0$.
This proves \eqref{I-omegaep1t0-omegaep0}.

\vspace{0.5mm}

\noindent{\bf{Step 2.}} Prove the nonlinear orbital stability \eqref{onlinear orbital stability-goal} for the weak solution $\tilde \omega(t)$ by taking limits.

For any $\kappa>0$, let $\delta(\ep_0,\kappa)={1\over3}\tilde\delta\left(\ep_0,{1\over2}\kappa\right)$ and $\tilde \omega(0)\in Y_{non}$ such that
$$\inf_{(x_0,y_0)\in\Omega} d(\tilde \omega(0),\omega_{\ep_0}(x+x_0,y+y_0))+\inf_{(x_0,y_0)\in\Omega}\|\tilde \omega(0)-\omega_{\ep_0}(x+x_0,y+y_0)\|_{L^2(\Omega)}<\delta(\ep_0,\kappa).$$
Then there exist $(\tilde x_1, \tilde y_1), (\tilde x_2,\tilde y_2)\in\Omega$ such that
\begin{align}\label{tilde-omega0-ometa-ep0-distance}
 d(\tilde \omega(0),\omega_{\ep_0}(x+\tilde x_1,y+\tilde y_1))+\|\tilde \omega(0)-\omega_{\ep_0}(x+\tilde x_2,y+\tilde y_2)\|_{L^2(\Omega)}<\delta(\ep_0,\kappa).
\end{align}
By Lemma \ref{tilde-omega0-kappa-properties} (8),  $-\tilde \omega^\mu(0)\ln(-\tilde \omega^\mu(0))\to-\tilde \omega(0)\ln(-\tilde\omega(0))$ in $L^1(\Omega)$. Moreover, $\tilde \omega^\mu(0)\to \tilde \omega(0)$ in $L^1\cap L^2(\Omega)$ and $\psi_{\ep_0}\tilde \omega^\mu(0)\to\psi_{\ep_0}\tilde \omega(0)$ in $L^1(\Omega)$ by Lemma \ref{tilde-omega0-kappa-properties} (4) and (7). Since $\psi_{(\tilde x_1,\tilde y_1)}(0,x,y)=(-\Delta)^{-1}(\tilde \omega(0,x-\tilde x_1,y-\tilde y_1)-\omega_{\ep_0}(x,y))\in \dot{H}^1(\Omega)$ by Lemma \ref{well-poseness-Poisson-equation-nonlinear-case}, we have $\psi_{(\tilde x_1,\tilde y_1)}^\mu(0)=\hat J_{\mu}\star\psi_{(\tilde x_1,\tilde y_1)}(0)\in\dot{H}^1(\Omega)$ and $ \nabla\psi_{(\tilde x_1,\tilde y_1)}^\mu(0)\to\nabla\psi_{(\tilde x_1,\tilde y_1)}(0)$ in $(L^2(\Omega))^2$, where $\star$  is defined in \eqref{convolution-R2-def}. Thus,
\begin{align*}
&\iint_\Omega\bigg(|h(\tilde \omega^\mu(0))-h(\tilde \omega(0))|+|\psi_{\ep_0}(x+\tilde x_1,y+\tilde y_1)(\tilde \omega^\mu(0)-\tilde \omega(0))|\\
&+2|\nabla\psi_{(\tilde x_1,\tilde y_1)}^\mu(0)-\nabla\psi_{(\tilde x_1,\tilde y_1)}(0)|^2\bigg)dxdy
+
\|\tilde \omega^\mu(0)- \tilde \omega(0)\|_{L^2(\Omega)}\to0
\end{align*}
as $\mu\to0^+$. This, along with \eqref{tilde-omega0-ometa-ep0-distance}, implies
\begin{align*}
&\inf_{(x_0,y_0)\in\Omega} d(\tilde \omega^\mu(0),\omega_{\ep_0}(x+x_0,y+y_0))+\inf_{(x_0,y_0)\in\Omega}\|\tilde \omega^\mu(0)-\omega_{\ep_0}(x+x_0,y+y_0)\|_{L^2(\Omega)}\\
\leq &d(\tilde \omega^\mu(0),\omega_{\ep_0}(x+\tilde x_1,y+\tilde y_1))+\|\tilde \omega^\mu(0)-\omega_{\ep_0}(x+\tilde x_2,y+\tilde y_2)\|_{L^2(\Omega)}\\
\leq&\iint_\Omega\bigg(|h(\tilde \omega^\mu(0))-h(\tilde \omega(0))|+|\psi_{\ep_0}(x+\tilde x_1,y+\tilde y_1)(\tilde \omega^\mu(0)-\tilde \omega(0))|\\
&+2|\nabla\psi_{(\tilde x_1,\tilde y_1)}^\mu(0)-\nabla\psi_{(\tilde x_1,\tilde y_1)}(0)|^2\bigg)dxdy
+d_1(\tilde \omega(0),\omega_{\ep_0}(x+\tilde x_1,y+\tilde y_1))\\
&+2d_2(\tilde \omega(0),\omega_{\ep_0}(x+\tilde x_1,y+\tilde y_1))+\|\tilde \omega^\mu(0)- \tilde \omega(0)\|_{L^2(\Omega)}+
\|\tilde \omega(0)-\omega_{\ep_0}(x+\tilde x_2,y+\tilde y_2)\|_{L^2(\Omega)}\\
\leq &3\delta(\ep_0,\kappa)=\tilde\delta\left(\ep_0,{1\over2}\kappa\right)
\end{align*}
for $\mu>0$ sufficiently small.
For fixed  $t\geq0$, by applying  Step 1, there exists $(x_1^\mu(t),y_1^\mu(t))\in\Omega$ such that
\begin{align}\label{onlinear orbital stability-apply1}
d(\tilde \omega_{tran}^\mu(t),\omega_{\ep_0})=d(\tilde \omega^\mu(t),\omega_{\ep_0}(x+x_1^\mu(t),y+y_1^\mu(t)))<{1\over2}\kappa
\end{align}
for $\mu>0$ sufficiently small.

Then we claim  that there exists $C(\ep_0,\tilde\omega(0))>0$ (independent of $\mu$) such that  $|y_1^\mu(t)|<C(\ep_0,\tilde\omega(0))$ for $\mu>0$ sufficiently small. Indeed,
by Corollary \ref{y-tilde-omega-pseudoenergy-conserved} (1) and Lemma \ref{tilde-omega0-kappa-properties} (6), we have
\begin{align}\label{y-bounded-first term}
\left|\iint_\Omega y\tilde\omega^\mu(t)dxdy\right|=\left|\iint_\Omega y\tilde\omega^\mu(0)dxdy\right|\leq\|y\tilde\omega^\mu(0)\|_{L^1(\Omega)}\leq \|y\tilde\omega(0)\|_{L^1(\Omega)}+1
\end{align}
 for $\mu>0$ small enough. For $|y|>\ln(4)$, we have
 \begin{align*}
 \psi_{\ep_0}(x,y)= \ln \left(\frac{\cosh (y) + \epsilon_0 \cos (x)}{\sqrt{1-\epsilon_0^2}} \right)\geq\ln \left(\frac{\cosh (y) -1}{\sqrt{1-\epsilon_0^2}} \right)\geq\ln \left({e^{|y|}\over4\sqrt{1-\epsilon_0^2}} \right)>0,
\end{align*}
and thus,
\begin{align}\label{y-bounded-second term}
|y|\leq \psi_{\ep_0}(x,y)+C_4(\ep_0), \quad y\in \mathbb{R},
\end{align}
where
 \begin{align*}
 C_4(\ep_0)=\left|\ln\left(4\sqrt{1-\epsilon_0^2}\right)\right|+\ln(4)+\max\limits_{x\in\mathbb{T}_{2\pi},y\in[-\ln(4),\ln(4)]}|\psi_{\ep_0}(x,y)|.
 \end{align*}
By \eqref{y-bounded-first term}-\eqref{y-bounded-second term}, \eqref{d1-well-def} and \eqref{onlinear orbital stability-apply1}, we have
\begin{align}\nonumber
|4\pi y_1^\mu(t)|=&\left|\iint_\Omega(y-y_1^\mu(t))\tilde\omega_{tran}^\mu(t)dxdy-\iint_\Omega y\tilde\omega_{tran}^\mu(t)dxdy\right|\\\nonumber
\leq&\|y\tilde\omega(0)\|_{L^1(\Omega)}+1-\iint_\Omega\psi_{\ep_0}\tilde\omega_{tran}^\mu(t)dxdy
+C_4(\ep_0)\|\tilde\omega^\mu(t)\|_{L^1(\Omega)}\\\nonumber
\leq&\|y\tilde\omega(0)\|_{L^1(\Omega)}+1+
d_1(\tilde \omega_{tran}^\mu(t),\omega_{\ep_0})\\\nonumber
&+\iint_{\Omega}\left({1\over2}(-\tilde \omega^\mu(t)+\tilde \omega^\mu(t)\ln(-\tilde \omega^\mu(t)))+{1\over 2}\omega_{\ep_0} \right) dxdy+C_4(\ep_0)\|\tilde\omega^\mu(t)\|_{L^1(\Omega)}\\\nonumber
\leq&\|y\tilde\omega(0)\|_{L^1(\Omega)}+1+{\kappa\over2}+\left({1\over2}+C_4(\ep_0)\right)(\|\tilde\omega(0)\|_{L^1(\Omega)}+1)\\\nonumber
&+{1\over2}(\|\tilde \omega(0)\ln(-\tilde \omega(0))\|_{L^1(\Omega)}+1)+{1\over2}\|\omega_{\ep_0}\|_{L^1(\Omega)}\triangleq 4\pi C(\ep_0,\tilde\omega(0))
\end{align}
for $\mu>0$ small enough, where we used
\begin{align*}
\|\tilde\omega^\mu(t)\|_{L^1(\Omega)}=&\|\tilde\omega^\mu(0)\|_{L^1(\Omega)}\leq \|\tilde\omega(0)\|_{L^1(\Omega)}+1,\\
\|\tilde \omega^\mu(t)\ln(-\tilde \omega^\mu(t))\|_{L^1(\Omega)}=&\|\tilde \omega^\mu(0)\ln(-\tilde \omega^\mu(0))\|_{L^1(\Omega)}\leq
\|\tilde \omega(0)\ln(-\tilde \omega(0))\|_{L^1(\Omega)}+1
\end{align*}
by Lemma \ref{tilde-omega0-kappa-properties} (4) and (8).

Up to a subsequence,  $x_1^\mu(t)\to x_1(t)$ and $y_1^\mu(t)\to y_1(t)$ for some $(x_1(t), y_1(t))\in\Omega$ as $\mu\to0^+$. We denote $\tilde \omega_{tran}(t)\triangleq\tilde \omega(t,x-x_1(t),y-y_1(t))$.
 By \eqref{tilde-omega-kappa-weak convergence L1L2}, we have
\begin{align*}
&\left|\iint_\Omega\left(\tilde \omega_{tran}^\mu(t)-\tilde \omega_{tran}(t)\right)\varphi(x,y)dxdy\right|\\
=&\bigg|\iint_\Omega\bigg(\tilde \omega^\mu(t)(\varphi(x+x_1^\mu(t),y+y_1^\mu(t))-\varphi(x+x_1(t),y+y_1(t)))+\\
&(\tilde \omega^\mu(t)-\tilde \omega(t))\varphi(x+x_1(t),y+y_1(t))\bigg)dxdy\bigg|\\
\leq&\|\tilde \omega^\mu(t)\|_{L^2(\Omega)}\|\varphi(x+x_1^\mu(t),y+y_1^\mu(t))-\varphi(x+x_1(t),y+y_1(t))\|_{L^2(\Omega)}\\
&+\bigg|\iint_\Omega
(\tilde \omega^\mu(t)-\tilde \omega(t))\varphi(x+x_1(t),y+y_1(t))dxdy\bigg|\to0\text{ as } \mu\to0^+
\end{align*}
for $\varphi\in L^2(\Omega)$,
where we used  $\|\tilde \omega^\mu(t)\|_{L^2(\Omega)}\leq C$ uniformly for $\mu>0$ small enough by Lemma \ref{lem-construction of an approximate solution sequence}.
Thus,
\begin{align}\label{tilde-omega-mu-translation-convergence L2}
\tilde \omega_{tran}^\mu(t)\rightharpoonup\tilde \omega_{tran}(t) \text{ in } L^2(\Omega).
\end{align}
Since $h(s)={1\over2}(s-s\ln(-s))$ is convex on $(-\infty,0]$, $\tilde \omega(t)\leq0$ a.e. on $\Omega$ by Corollary \ref{vorticity L123}, and $\psi_\ep \in L^2(B_R)$ for any $R>0$, it follows from  Theorem 1.1, Remark (iii) in \cite{Dacorogna} (see also \cite{Morrey1966})  and  \eqref{tilde-omega-mu-translation-convergence L2}  that
\begin{align}\nonumber
&\iint_{B_R}\left(h(\tilde \omega_{tran}(t)) - h(\omega_{\ep_0})  - \psi_{\ep_0} (\tilde \omega_{tran}(t)-\omega_{\ep_0})\right)dxdy\\\nonumber
\leq& \liminf_{\mu\to0^+}\iint_{B_R}(h(\tilde \omega_{tran}^\mu(t)) - h(\omega_{\ep_0}) - \psi_{\ep_0} (\tilde \omega_{tran}^\mu(t)
-\omega_{\ep_0}))dxdy\\\label{BRtrand1control}
\leq& \liminf_{\mu\to0^+} d_1(\tilde \omega_{tran}^\mu(t),\omega_{\ep_0}),
\end{align}
where $B_R=\mathbb{T}_{2\pi}\times [-R,R]$. By \eqref{limit-for-approximate solution-L2-t},  $x_1^\mu(t)\to x_1(t)$ and $y_1^\mu(t)\to y_1(t)$, we have
\begin{align}\label{BRtrand2control}
&\|\nabla\psi_{tran}(t)\|_{L^2(B_R)}=\lim_{\mu\to0^+}
\|\nabla\psi_{tran}^\mu(t)\|_{L^2(B_R)}
\leq\lim_{\mu\to0^+} d_2(\tilde \omega_{tran}^\mu(t),\omega_{\ep_0})
\end{align}
 for any $R>0$, where
$\psi_{tran}^\mu(t)\triangleq(-\Delta)^{-1}(\tilde \omega^\mu(t,x-x_1^\mu(t),y-y_1^\mu(t))-\omega_{\ep_0})$ and $\psi_{tran}(t)\triangleq(-\Delta)^{-1}(\tilde \omega(t,x-x_1(t),y-y_1(t))-\omega_{\ep_0})$. Taking $R\to\infty$ in
\eqref{BRtrand1control}-\eqref{BRtrand2control}, up to a subsequence, we have
\begin{align*}
d(\tilde \omega(t),\omega_{\ep_0}(x+x_1(t),y+y_1(t)))=d(\tilde \omega_{tran}(t),\omega_{\ep_0})\leq \lim_{\mu\to0}d(\tilde \omega_{tran}^\mu(t),\omega_{\ep_0})\leq {1\over 2}\kappa<\kappa,
\end{align*}
where we used \eqref{onlinear orbital stability-apply1} in the second inequality.
\end{proof}

\begin{remark}\label{rem-weaken-condition-main-result4}
As the proof of Theorem \ref{main result4-nonlinear orbital stability} shows, the term
\[
\inf_{(x_0,y_0)\in\Omega}\|\tilde \omega_0-\omega_{\ep_0}(x+x_0,y+y_0)\|_{L^2(\Omega)}
\]
in the initial assumption \eqref{initial-condition-main result4-nonlinear orbital stability-2} is introduced only to guarantee a uniform \(L^2\)-bound for the initial vorticity of the approximate solutions; see \eqref{uniform-L2-bound-tilde-omega0}. Accordingly, the condition \eqref{initial-condition-main result4-nonlinear orbital stability-2} can be replaced by
\[
\inf_{(x_0,y_0)\in\Omega} d(\tilde \omega_0,\omega_{\ep_0}(x+x_0,y+y_0))<\delta,
\]
together with the additional assumption \(\|\tilde \omega_0\|_{L^2(\Omega)}\le C\) for some constant \(C>0\).
\end{remark}

\begin{remark}\label{rem-La-control-main-result4}
Theorem \ref{main result4-nonlinear orbital stability} also yields quantitative control of the vorticity in $L^a(\Omega)$ for $a\in[1,2)$. Specifically,
under the conditions of this theorem, by \eqref{La-control-by-d} and \eqref{onlinear orbital stability-goal} we have
\begin{align*}
\inf_{(x_0,y_0)\in\Omega}\|\tilde \omega(t)-\omega_{\ep_0}(x+x_0,y+y_0)\|_{L^a(\Omega)}<(3\sqrt{2\pi})^{{2\over a}-1}(\delta+2\|\omega_{\ep_0}\|_{L^2(\Omega)})^{2-{2\over a}}\kappa^{{1\over a}-{1\over2}},
\end{align*}
where $a\in[1,2)$.
Under an additional $L^3$-bound $M$ on the initial vorticity,
 Theorem \ref{main result4-nonlinear orbital stability} establishes $L^2$-norm  control of the vorticity.
In fact, it  follows from \eqref{L2-control-by-d} and  \eqref{onlinear orbital stability-goal} that
\begin{align*}
\inf_{(x_0,y_0)\in\Omega}\|\tilde \omega(t)-\omega_{\ep_0}(x+x_0,y+y_0)\|_{L^2(\Omega)}<C\kappa^{{1\over4}},
\end{align*}
where $C>0$ is a constant depending on $\|\omega_{\ep_0}\|_{L^\infty}, \|\omega_{\ep_0}\|_{L^2}, \|\omega_{\ep_0}\|_{L^3}, \delta$ and $M$.
\end{remark}
\begin{remark}
Another standard approach to nonlinear stability is variational: one tries to characterize the equilibrium as a global minimizer of a suitable Lyapunov functional and then exploit this minimizing property. For the Kelvin--Stuart vortices, a natural candidate is the pseudoenergy-Casimir functional
\[
H(\tilde\omega)
=
\iint_{\Omega}
\left(
\frac12\tilde\omega-\frac12\tilde\omega\ln(-\tilde\omega)
\right)\,dx\,dy
-
\frac12
\iint_{\Omega}
(G*\tilde\omega)\tilde\omega\,dx\,dy
\]
defined on the  space $Y_{non}$ in \eqref{def-X-non-ep}. A direct computation gives
\[
H'(\omega_\ep)=0,
\]
and hence
\begin{align}\label{H independent ep}
\frac{d}{d\ep}H(\omega_\ep)
=
\langle H'(\omega_\ep),\partial_\ep\omega_\ep\rangle
=
0,
\end{align}
where we used $\iint_\Omega \partial_\ep\omega_\ep\,dx\,dy=0$.

Our proof above shows that, up to spatial translations, each $\omega_\ep$ with $\ep\in(0,1)$ is a local minimizer of $H$ on $Y_{non}$; see \eqref{t control by 0-Taylor of dual functional}. Suppose now that $\omega_{\ep_0}$ is a global minimizer of $H$ for some $\ep_0\in(0,1)$. Then \eqref{H independent ep} would imply that every member of the family $\omega_\ep$, $\ep\in(0,1)$, is also a global minimizer of $H$. In particular, for any fixed $\ep$, the minimizer $\omega_\ep$ would fail to be isolated, and this creates a serious obstruction to a direct variational approach. Note that this non-isolation is not caused by spatial translations, but by variation of the parameter $\ep$ itself. A further difficulty is that $\omega_\ep$ becomes singular as $\ep\to1^-$, so the lack of compactness is too severe for one to expect convergence of minimizing sequences by standard variational methods.

\end{remark}
\section{Numerical results}\label{Numerical Results}

The numerical analysis consists of two parts. The first part is to approximate an eigenvalue with a corresponding  eigenfunction for the eigenvalue problem \eqref{elip02} in the co-periodic case, which motivates us to compute the first few eigenvalues with corresponding eigenfunctions for the $0$-mode in \eqref{eigen value-function}.
The second part shows that the number of unstable eigenvalues decreases as $\ep$ increases in the modulational case.

 \subsection{An eigenfunction of the associated eigenvalue problem for the co-periodic case}\label{eigenfunction-motivation}
We simulate the eigenvalues and eigenfunctions of the operator $\tilde{A}_\ep$  by means of the spectral method in the co-periodic case.
We discretize the space $\tilde{X}_\ep$ with the following basis functions
 $$ \mathcal{B} =\left\{\psi_{n,k}(x,y) | n \in \mathbb{N}, k \in \mathbb{Z}\right\},$$
where
$$\psi_{n,k}(x,y) = \left\{ \begin{array}{cc} \frac{1}{\sqrt{2\pi}} \int_0^y H_n(\hat{y})d\hat{y}, & k = 0,  \\ \frac{1}{\sqrt{\pi}} H_n(y)\cos(kx), & k > 0, \\
\frac{1}{\sqrt{\pi}} H_n(y)\sin(kx), & k < 0,\end{array} \right.$$
$
H_n(y)= \frac{e^{-y^2/2}}{\pi^{1/4}\sqrt{2^nn!}}\hat{H}_n(y)$ and $\hat{H}_n(y)=(-1)^ne^{y^2}\frac{d^n}{dy^n}e^{-y^2}$, $n \in \mathbb{N}$,
are the Hermite functions and the Hermite polynomials, respectively.
Note that $\{H_n(y)|n \in \mathbb{N}\}$  form an orthonormal basis of $L^2(\mathbb{R})$. Moreover,
 $\{ \psi_{n,0}(y) = \frac{1}{\sqrt{2\pi}} \int_0^y H_n(\hat{y}) d\hat{y}| n\in \mathbb{N} \}$ is orthonormal in the sense that
\begin{align}\label{psi-nk-orthonormal}
(\psi_{n_1,0}, \psi_{n_2,0})_{\dot{H}^1(\Omega)} = \iint_{\Omega} \nabla \psi_{n_1,0} \cdot \nabla \psi_{n_2,0} dxdy = \delta_{n_1, n_2}.
\end{align}
 For any $\psi_{n_1, k_1},  \psi_{n_2, k_2} \in \mathcal{B}$, we have
 \begin{align*}
 \langle\tilde{A}_\ep \psi_{n_1, k_1}, \psi_{n_2, k_2}\rangle
 =& \iint_\Omega \nabla \psi_{n_1, k_1} \cdot {\nabla \psi_{n_2, k_2}} dxdy - \iint_\Omega g'(\psi_\ep)\psi_{n_1, k_1} {\psi_{n_2, k_2}} dxdy \\
 &  + \frac{1}{8\pi} \iint_\Omega g'(\psi_\ep)\psi_{n_1, k_1} dxdy \iint_\Omega g'(\psi_\ep)\psi_{n_2, k_2} dxdy.
 \end{align*}
 We  use the above  equality to find a finite dimensional matrix, which approximates the operator $\tilde{A}_\ep$, and obtain the spectral information of $\tilde{A}_\ep$ by studying the eigenvalues and eigenvectors of the approximate matrix.

The procedure to discretize the problem is summarized as follows:
\begin{enumerate}
 \item Choose a positive integer $N$.
 \item Truncate the basis $\mathcal{B}$ to $\mathcal{B}_N = \left\{\psi_{n,k}(x,y) | 0\leq n \leq2N,  -N \leq k \leq N\right\}$.
 \item Compute the $(2N + 1)^2 \times (2N+1)^2$ matrix $\mathbf{\tilde{A}}_\ep$ using
 $$(\mathbf{\tilde{A}}_\ep)_{(n_1, k_1), (n_2, k_2)} = \langle\tilde{A}_\ep \psi_{n_1, k_1}, \psi_{n_2, k_2}\rangle \text{ for } \psi_{n_1, k_1},  \psi_{n_2, k_2} \in \mathcal{B}_N.$$
 \item Calculate the eigenvalues $\lambda_i$ and eigenvectors $v_i$ of $\mathbf{\tilde{A}}_\ep$.
 \item Use the eigenvectors $v_i$ in (4) and the truncated basis $\mathcal{B}_N$ in (2) to compute the approximated eigenfunctions $f_i$ of $\tilde{A}_\ep$.
 \end{enumerate}

 We pick $N = 7$ and take different values for $\epsilon \in [0, 1)$. Then we  compute the $225\times225$ dimensional matrix $\mathbf{\tilde{A}}_\ep$ to approximate $\tilde{A}_\ep$ and calculate its eigenvalues. We summarize the first 10  eigenvalues of $\mathbf{\tilde{A}}_\ep$ in Table \ref{tbl:Tbl1}.
\begin{table}[ht]
\centering
\caption{The first 10  eigenvalues of $\mathbf{\tilde{A}}_\ep$}
\label{tbl:Tbl1}
\resizebox{0.8\textwidth}{!}{
\begin{tabular}{| l | r  r  r  r  r  r  r  r  r |}
\hline
  $\epsilon$ &    0.0 &    0.1 &    0.2 &    0.3 &    0.4 &    0.5 &    0.6 &    0.7 &    0.8 \\
\hline
$\lambda_1$ & 0.0000 & 0.0000 & 0.0000 & 0.0000 & 0.0001 & 0.0001 & 0.0002 & 0.0007 & 0.0041 \\
$\lambda_2$ & 0.0001 & 0.0001 & 0.0001 & 0.0001 & 0.0001 & 0.0002 & 0.0006 & 0.0024 & 0.0118 \\
$\lambda_3$ & 0.0001 & 0.0001 & 0.0001 & 0.0001 & 0.0001 & 0.0003 & 0.0008 & 0.0032 & 0.0169 \\
$\lambda_4$ & 0.6667 & 0.6682 & 0.6728 & 0.6807 & 0.6926 & 0.7094 & 0.7329 & 0.7662 & 0.8163 \\
$\lambda_5$ & 0.8336 & 0.8334 & 0.8329 & 0.8324 & 0.8322 & 0.8331 & 0.8361 & 0.8432 & 0.8588 \\
$\lambda_6$ & 0.9016 & 0.9018 & 0.9023 & 0.9034 & 0.9051 & 0.9078 & 0.9122 & 0.9192 & 0.9314 \\
$\lambda_7$ & 0.9367 & 0.9369 & 0.9375 & 0.9386 & 0.9404 & 0.9430 & 0.9468 & 0.9525 & 0.9612 \\
$\lambda_8$ & 0.9601 & 0.9603 & 0.9609 & 0.9620 & 0.9636 & 0.9659 & 0.9691 & 0.9733 & 0.9792 \\
$\lambda_9$ & 0.9738 & 0.9740 & 0.9745 & 0.9753 & 0.9766 & 0.9783 & 0.9806 & 0.9836 & 0.9875 \\
$\lambda_{10}$ & 0.9850 & 0.9851 & 0.9854 & 0.9860 & 0.9868 & 0.9879 & 0.9894 & 0.9912 & 0.9934 \\
\hline
\end{tabular}
}
\end{table}
Even though the accuracy is affected for large $\epsilon$ values due to the singularity of the steady state at $\epsilon = 1$,
we could observe some interesting patterns from the numerical results.
\begin{itemize}
\item The eigenvalues $\lambda_i$ do not have a clear dependence on $\epsilon$.
\item For all $\epsilon$ values, $\mathbf{\tilde{A}}_\ep$ has three zero eigenvalues.
\item When $\ep = 0$, the first 3 eigenfunctions $f_1, f_2, f_3$ correspond to the three kernel functions of $\tilde{A}_0$, i.e.
$$f_1(x,y) = \tanh(y),\quad f_2(x,y) = \frac{\cos(x)}{\cosh(y)},\quad f_3(x,y) = \frac{\sin(x)}{\cosh(y)}.$$

\item The $4$-th eigenvalue $\lambda_4$ is a good approximation of the number $\frac 2 3$.
\item When $\epsilon = 0$, the $4$-th eigenfunction $f_4$ only depends on $y$ and has a bell shaped curve that matches the curve of $\tanh^2(y)$ perfectly after some linear transformation, see Figure \ref{fig:9thFig-co2}.
\begin{figure}[ht]
    \centering
	\includegraphics[width=0.85\textwidth]{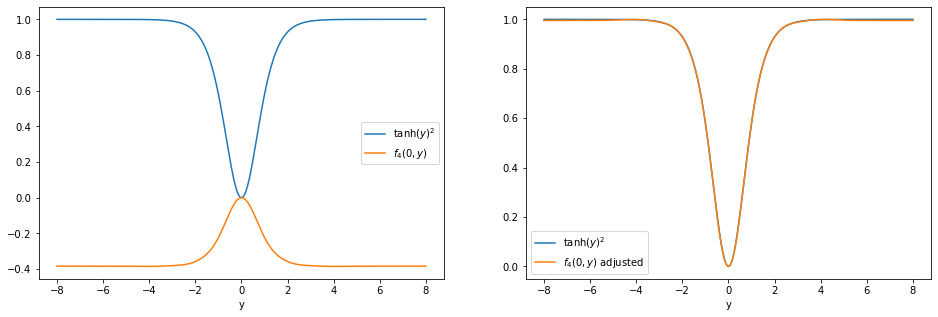}
	\caption{The 4-th eigenfunction $f_4$ of $\mathbf{\tilde{A}}_0$}
	\label{fig:9thFig-co2}
\end{figure}
\end{itemize}
The above observations  give a hint that
\begin{align}\label{eigen4}
\mathbf{\tilde{A}}_0 \vec{v}_4& = \lambda_4 \vec{v}_4 =  \frac 2 3 \vec{v}_4,\\\nonumber
v_{4,n,k}&=0\quad\text{for}\quad k\neq0\Longrightarrow
f_4=\sum\limits_{n=0}^{2N}\sum\limits_{k=-N}^Nv_{4,n,k}\psi_{n,k}=\sum\limits_{n=0}^{2N}v_{4,n,0} \psi_{n,0},
\end{align}
and $f_4$ might be $\tanh^2(y)$,
where
$\vec{v}_4=(v_{4,n,k})_{0\leq n\leq 2N,-N\leq k\leq N}$. By \eqref{psi-nk-orthonormal}, we have
$\|\vec{v}_4\|_{l^2}=\iint_{\Omega}|\nabla f_4|^2dxdy=\iint_\Omega(-\Delta f_4)f_4dxdy$.
By \eqref{eigen4}, $f_4$ approximately satisfies  $$\tilde{A}_0 f_4 = (-\Delta - g'(\psi_0)(I-P_0)) f_4 =  \frac 2 3 (-\Delta f_4),$$
which implies
$$-\Delta f_4 = 3g'(\psi_0)(I-P_0)f_4,$$
where $g'(\psi_0) = 2\sech^2(y)$.
This is exactly true when $f_4(x,y) = \tanh^2(y)$ since
$$-\Delta \tanh^2(y) = 2\sech^2(y)(3\tanh^2(y)-1) = 3g'(\psi_0)\left(\tanh^2(y) - \frac 1 3\right)$$
and $$P_0(\tanh^2(y)) = \frac{\int_0^{2\pi} \int_{-\infty}^{+\infty} g'(\psi_0) \tanh^2(y) dydx}{8\pi} = \frac1 2 \int_{-\infty}^{+\infty} \sech^2(y) \tanh^2(y)dy  = \frac 1 3.$$
By the above numerical simulation, $\tanh^2(y)$ is  an eigenfunction of the  eigenvalue $\lambda=3$ for \eqref{mode0}.  Recall that  $\tanh(y)$ is  an eigenfunction of the  eigenvalue $\lambda=1$ for \eqref{mode0}. Observing the form of these two eigenfunctions, our intuition is that all the eigenfunctions are possibly polynomials of $\tanh(y)$. This motivates us to compute
the first few eigenvalues and eigenfunctions as in \eqref{eigen value-function}, and inspires us to try
the change of variable $\gamma=\tanh(y)$ for the hyperbolic tangent shear flow. It is surprising and lucky to relate the eigenvalue problem \eqref{mode0} to the Legendre differential equations after the change of variable.

\subsection{The number of unstable modes in the modulational case}\label{The number of unstable modes in the modulational case}
In Section \ref{modulational}, we study the linear modulational instability analytically.  In this subsection, we obtain an interesting numerical phenomenon that there exists $\ep_0\in(0,1)$ such that  the number of unstable modes changes from $2$ to $1$ once $\ep$ passes through  $\ep_0$ increasingly for $\alpha={1\over 2}$ or ${1\over3}$.

To avoid solving the Poisson equation, we analyze the problem using the stream functions and solve the following generalized eigenvalue problem
 \begin{equation}\label{modeig}
 M_{\ep\alpha}\widetilde{\psi}=\sigma (-\Delta_\alpha) \widetilde{\psi}, \quad \widetilde{\psi} \in H^1(\Omega),
 \end{equation}
 where
$M_{\ep\alpha} = J_{\ep,\alpha}L_{\ep,\alpha} (-\Delta_\alpha)$, $J_{\ep,\alpha}$, $L_{\ep,\alpha}$ and $\Delta_\alpha$ are defined in \eqref{def-J-ep-al}-\eqref{nabla-alpha-Delta-alpha}.
The study of modulational instability is equivalent to the study of the generalized eigenvalue problem in \eqref{modeig}. We use a spectral method to discretize this problem and study the resulting generalized eigenvalue problem with two approximation matrices. We take the basis
$$\tilde{\mathcal{B}} = \{\tilde{\psi}_{n,k}(x,y) | n \in \mathbb{N}, k \in \mathbb{Z}\},$$
where
$\tilde{\psi}_{n,k}(x,y) = \frac {1} {\sqrt{2\pi}}e^{ikx}H_n(y).$
We know that $\tilde{\mathcal{B}}$ is an orthonormal basis of $H^1(\Omega)$ and for any $\tilde{\psi}_{n_1,k_1}, \tilde{\psi}_{n_2,k_2} \in \tilde{\mathcal{B}}$,
\begin{align*}
\langle M_{\ep\alpha} \tilde{\psi}_{n_1,k_1} ,\tilde{\psi}_{n_2,k_2}\rangle
  =  \iint_\Omega M_{\ep\alpha} \tilde{\psi}_{n_1,k_1}(x,y) \overline{\tilde{\psi}_{n_2,k_2}(x,y)} dxdy
  \end{align*}
and
\begin{align*}
\langle-\Delta_\alpha \tilde{\psi}_{n_1,k_1} ,\tilde{\psi}_{n_2,k_2}\rangle
  = & \iint_\Omega -\Delta_\alpha \tilde{\psi}_{n_1,k_1}(x,y) \overline{\tilde{\psi}_{n_2,k_2}(x,y)} dxdy.\\
\end{align*}

 \subsubsection{Algorithm}
The procedure to discretize the problem is summarized as follows:
\begin{enumerate}
 \item Choose a positive integer $N$.
 \item Truncate the basis $\tilde{\mathcal{B}}$ to $\tilde{\mathcal{B}}_N = \left\{\tilde{\psi}_{n,k}(x,y) | 0\leq n \leq2N,  -N \leq k \leq N\right\}$.
 \item Compute the $(2N + 1)^2 \times (2N+1)^2$ matrices $\mathbf{M}_{\ep\alpha}$, $\mathbf{D}_{\alpha}$ with the entries
 $$(\mathbf{M}_{\ep\alpha})_{(n_1, k_1), (n_2, k_2)} = (M_{\ep\alpha} \tilde{\psi}_{n_1, k_1}, \tilde{\psi}_{n_2, k_2}) $$
 and
 $$(\mathbf{D}_{\alpha})_{(n_1, k_1), (n_2, k_2)} = (-\Delta_{\alpha} \tilde{\psi}_{n_1, k_1}, \tilde{\psi}_{n_2, k_2}) $$
 for $\tilde{\psi}_{n_1, k_1},  \tilde{\psi}_{n_2, k_2} \in \tilde{\mathcal{B}}_N.$
 \item Solve $\sigma$ from the generalized eigenvalue problem
\begin{align}\label{eigen-modu-app}
\mathbf{M}_{\ep\alpha}^* =\sigma \mathbf{D}_{\alpha}^*.
\end{align}
 \end{enumerate}
Here, $\mathbf{M}_{\ep\alpha}^*$ is the  conjugate transpose of $\mathbf{M}_{\ep\alpha}$.
 \subsubsection{Results}
 We pick $N = 7$ and take different values for $\epsilon \in (0, 1)$ and $\alpha \in (0, \frac 1 2]$. Then  we compute the $225 \times 225$ dimensional matrices $\mathbf{M}_{\ep\alpha}$, $\mathbf{D}_{\alpha}$ and calculate the generalized eigenvalues $\sigma$.
\begin{figure}[ht]
    \centering
	\includegraphics[scale = 0.4]{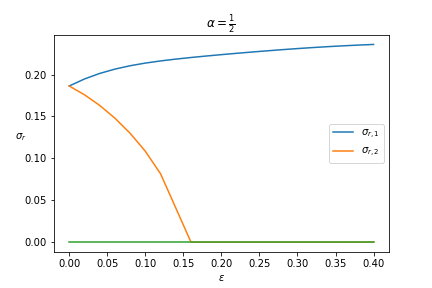}
        \includegraphics[scale = 0.4]{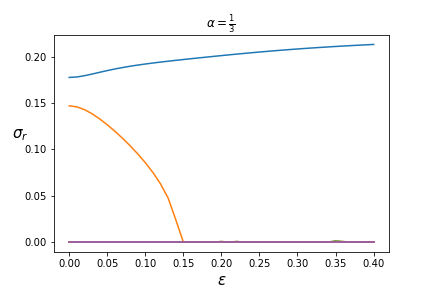}
	\caption{Positive real parts of the generalized eigenvalues of \eqref{eigen-modu-app}  }
	\label{fig:eleventhFig}
\end{figure}

Our numerical results provide an interesting piece of information.
Figure \ref{fig:eleventhFig} shows the correspondence between the positive real parts of the unstable eigenvalues and $\ep$ for  $\alpha={1\over2}, {1\over3}$. When $\alpha = \frac{1}{2}$, as $\ep$ grows from $0$ to $0.4$, there are two unstable directions with the same positive growth rates $0.186$ in the beginning, and then one of them decreases to $0$ at $\ep = 0.16$ while the other slowly increases up to $0.235$. This result compares well with the result in Figure 3 of \cite{pierrehumbert1982two}. Similarly, when $\alpha = \frac{1}{3}$, there are two unstable directions with positive growth rates. One of them decreases to $0$ at $\ep = 0.14$ and the other slowly increases up to $0.210$. This indicates that the number of unstable eigenvalues changes from $2$ to $1$ as $\ep$ grows far from $0$.
From the analytical perspective, the area of  the trapped region of the cat's-eye flow is getting larger and the effect of  the projection term is increasing as $\ep$ grows. Thus, the value of  the quadratic form
$ b_{\alpha, 2}$ in \eqref{func-b2-alpha} increases,  which leads to a decrease in the number of negative directions of $L_{\alpha,e} |_{\overline{R(B_\alpha)}}  $ as well as the unstable eigenvalues.

If we take $\alpha$ close to $0$, then the numerical simulations could only give us one unstable eigenvalue for $\ep$ small enough. Indeed, there are exactly $2$  unstable eigenvalues in this case  by Remark \ref{modulational-remark}.
We explain why numerically there is only one unstable eigenvalue for $\ep$ small enough. Note that we use the Hermite functions as the basis of $\tilde X_\ep$, and these functions decay very fast (with a Gaussian rate $e^{-y^2/2}$) near $\pm\infty$.
As one of the negative directions of $\tilde A_{\ep,\alpha}$  is
$(1-\gamma_\ep^2)^{\alpha\over2}e^{i\alpha(\theta_\ep-x)}$ decaying like $\sech^{\alpha}(y)$ near $\pm\infty$ by Corollary
\ref{A-L-dec-e-alpha},
the eigenfunction corresponding to the unstable eigenvalue with smaller growth rate may decay more slowly for $\alpha\ll1$, and our numerical simulations only detect the low-frequency part of the eigenfunctions (we pick $N=7$). If we take $N$ to be larger than 20, then the computational cost increases dramatically.

\section{Stability and instability of Kelvin--Stuart magnetic islands}\label{Stability and instability of magnetic islands0-sec}
The Kelvin--Stuart cat's-eye profiles form a family of static equilibria of the planar ideal MHD equations.
 The equilibria are given by the magnetic island solutions $(\omega=0,\phi_{\ep})$, where $\phi_\ep$ is given in \eqref{Kelvin--Stuart cat's eyes-mhd-m-p}. In this section, we prove spectral stability and conditional nonlinear orbital stability for co-periodic perturbations, and coalescence instability of
 the Kelvin--Stuart magnetic islands $(\omega=0,\phi_{\ep})$.

 For the steady magnetic potential $\phi_{\ep}(x,y)=\ln \left(\frac{\cosh (y) + \epsilon \cos (x)}{\sqrt{1-\epsilon^2}} \right)$, we have
 \begin{align}\label{steady-mhd}
\phi_{\ep}=G*J^{\ep}-\ln\sqrt{1-\epsilon^2},
 \end{align}
 where $G$ is defined in \eqref{green function}.
 In fact, since
\begin{align*}
 &(G*J^{\ep})(x,y)-|y|={1\over 4\pi}\iint_{\Omega}\ln(\cosh(y-\tilde y)-\cos(x-\tilde x)){1\over 2}g'(\psi_\ep(\tilde x,\tilde y))d\tilde x d\tilde y-|y|\\
 =&
 {1\over 4\pi}\int_{-1}^1\int_0^{2\pi}\ln{\cosh(y-\tilde y)-\cos(x-\tilde x)\over e^{|y|}}d\tilde\theta_\ep d\tilde\gamma_{\ep}\to\ln{1\over2}
 \end{align*}
  and
  $\ln(\cosh(y)+\ep\cos(x))-|y|=\ln{\cosh(y)+\ep\cos(x)\over e^{|y|}}\to\ln{1\over2}$
  as $y\to\pm\infty$, we infer from $-\Delta(G*J^{\ep})=-\Delta\ln(\cosh(y)+\ep\cos(x))= J^{\ep}$ that
\begin{align*}
G*J^{\ep}(x,y)=\ln \left(\cosh (y) + \epsilon \cos (x) \right),
\end{align*}
where $\tilde\theta_\ep=\theta_\ep (\tilde x,\tilde y)$ and $\tilde \gamma_{\ep}=\gamma_{\ep}(\tilde x,\tilde y)$.
\subsection{Spectral stability for co-periodic perturbations}
We consider the co-periodic perturbations of the magnetic island solutions $(\omega=0,\phi_{\ep})$ for $\ep\in[0,1)$.
Linearizing \eqref{mhd} around $(\omega=0,\phi_{\ep})$, we have
\begin{align}\label{linearized mhd}
\left\{ \begin{array}{lll} \partial_t \phi=-\{\phi_\ep,\psi\},\\
 \partial_t \omega=-\{\phi_{\ep},(-\Delta-g'(\phi_{\ep}))\phi\}.
 \end{array} \right.
\end{align}
Unlike the linearized  2D Euler equation around the Kelvin--Stuart vortex,
the linearized equation \eqref{linearized mhd} has a different separable Hamiltonian structure
\begin{align}\label{linearized mhd-sep-hamiltonian}
\partial_t \left( \begin{array}{c} \phi \\ \omega \end{array} \right) = \left( \begin{array}{cc} 0 & D_\ep \\ -D_{\ep}' & 0 \end{array} \right)\left( \begin{array}{cc}-\Delta-g'(\phi_{\ep}) & 0 \\ 0 & (-\Delta)^{-1} \end{array} \right) \left( \begin{array}{c} \phi \\ \omega \end{array} \right),
\end{align}
where  $-\Delta-g'(\phi_{\ep}):\tilde W_{\ep}\to\tilde W_{\ep}^*$,
\begin{align*}
\tilde W_{\ep}=\left\{ \phi\in \dot{H}^1(\Omega) :\iint_\Omega g'(\phi_\ep)\phi dxdy=0 \right\},
\end{align*}
 $(-\Delta)^{-1}:\tilde X_\ep^*\to\tilde X_\ep$ is defined by
\begin{align*}(-\Delta)^{-1}\omega=I_{\tilde X_\ep}\omega,\quad\omega\in\tilde X_\ep^*,
\end{align*}
$I_{\tilde X_\ep}:\tilde X_\ep^*\to \tilde X_\ep$ is the isomorphism defined by the Riesz representation
theorem,
 $D_\ep=-\{\phi_{\ep},\cdot\}:\tilde X_\ep\supset D(D_\ep)\to\tilde W_{\ep}$, and $\tilde{X}_\ep$ is defined in \eqref{tilde-X0} for $\ep=0$ and \eqref{tilde-X-e} for $\ep\in(0,1)$.
 Since $\iint_{\Omega}g'(\phi_\ep)\phi(t) dxdy$ is conserved for the linearized equation \eqref{linearized mhd},  it is reasonable to consider the perturbation of the magnetic potential  to satisfy $\iint_\Omega g'(\phi_\ep)\phi dxdy=0$ in the space $\tilde W_{\ep}$.

Since $P_\ep\phi=0$ for $\phi\in \tilde W_\ep$, we have $-\Delta-g'(\phi_{\ep})=-\Delta-g'(\phi_{\ep})(I-P_\ep)=\tilde A_\ep: \tilde W_{\ep}\to\tilde W_{\ep}^*$, where $P_\ep$ takes the form  \eqref{P-ep}.
For any $\phi\in \tilde W_\ep$, there exist $\phi_*\in \tilde X_\ep$ and a constant $c_*$ such that $\phi-\phi_*=c_*$, and
\begin{align}\label{n-tilde A-w-x}
\langle\tilde A_\ep\phi,\phi\rangle=\langle\tilde A_\ep\phi_*,\phi_*\rangle.
\end{align}
Thus, the properties of the quadratic form $\langle\tilde A_\ep\cdot,\cdot\rangle|_{\tilde W_\ep}$ are equivalent to those of the quadratic form $\langle\tilde A_\ep\cdot,\cdot\rangle|_{\tilde X_\ep}$, which was studied in  Section \ref{co-periodic-linear}.

Now, we verify the assumptions $\textbf{(G1-4)}$ in Lemma \ref{indice-theorem-sep} for the separable Hamiltonian system
\eqref{linearized mhd-sep-hamiltonian}. By a similar argument as for $B_{\ep}$ and $B_{\ep}'$ in
\eqref{sep-hamiltonian}, we infer that $D_\ep$ and $D_\ep'$ are densely defined and closed. This verifies $\textbf{(G1)}$. Since
\begin{align*}
\langle (-\Delta)^{-1}\omega_1,\omega_2\rangle=\iint_{\Omega}(I_{\tilde X_\ep}\omega_1)\omega_2 dxdy=(\omega_1,\omega_2),
\end{align*}
we know that $(-\Delta)^{-1}$ is bounded and self-dual, $\ker((-\Delta)^{-1})=\{0\}$, $\langle(-\Delta)^{-1}\omega,\omega\rangle=\|\omega\|_{\tilde X_\ep^*}^2$ for $\omega\in\tilde X_\ep^*$, and thus,  $\textbf{(G2)}$ is verified.
 $\textbf{(G3-4)}$ are verified by \eqref{n-tilde A-w-x} and Corollaries \ref{kernel of  the operator tilde A0 and a decomposition of tilde X0}, \ref{kernel of  the operator tilde A-ep and a decomposition of tilde Xep}.
 By  Lemma \ref{indice-theorem-sep}, we obtain that
 \begin{align}\label{stab-criteria-co-mhd}
 (\omega=0,\phi_{\ep}) \text{ is spectrally stable if and only if } n^-\left(\tilde A_{\ep}|_{\overline{R(D_{\ep})}}\right)=0.
 \end{align}
Again by \eqref{n-tilde A-w-x} and Corollaries \ref{kernel of  the operator tilde A0 and a decomposition of tilde X0}, \ref{kernel of  the operator tilde A-ep and a decomposition of tilde Xep},  $\langle\tilde A_\ep\cdot,\cdot\rangle|_{\tilde W_\ep}\geq0$ and thus, $n^-\left(\tilde A_{\ep}|_{\overline{R(D_{\ep})}}\right)=0$  in the  co-periodic case for $\ep\in[0,1)$. This proves Theorem \ref{main result1-mhd-all1} (2).
\subsection{Proof of coalescence instability}
In this subsection, we prove coalescence instability of the magnetic island equilibria $(\omega=0,\phi_{\ep})$, that is, linear double-periodic instability of the whole family. The proof is based on the separable Hamiltonian structure of the linearized MHD equations together with our analysis of linear double-periodic instability for Kelvin--Stuart vortices in the 2D Euler case.
Let $\Omega_2 = \mathbb{T}_{4\pi} \times \mathbb{R}$.
The linearized equation around $(\omega=0,\phi_{\ep})$ is
\begin{align}\label{linearized mhd-sep-hamiltonian-2}
\partial_t \left( \begin{array}{c} \phi \\ \omega \end{array} \right) = \left( \begin{array}{cc} 0 & D_{\ep,2} \\ -D_{\ep,2}' & 0 \end{array} \right)\left( \begin{array}{cc}-\Delta-g'(\phi_{\ep}) & 0 \\ 0 & (-\Delta)^{-1} \end{array} \right) \left( \begin{array}{c} \phi \\ \omega \end{array} \right),
\end{align}
where  $-\Delta-g'(\phi_{\ep}):\tilde W_{\ep,2}\to\tilde W_{\ep,2}^*$,
\begin{align*}
\tilde W_{\ep,2}=\left\{ \phi \bigg| \|\nabla \phi\|_{L^2(\Omega_2)} < \infty\quad {\rm{ and }}\quad\iint_{\Omega_2} g'(\phi_\ep)\phi dxdy=0 \right\},
\end{align*}
$(-\Delta)^{-1}:
\tilde{X}_{\ep,2}^*\to\tilde{X}_{\ep,2}$ is defined by
\begin{align*}(-\Delta)^{-1}\omega=I_{\tilde{X}_{\ep,2}}\omega,\quad\omega\in \tilde{X}_{\ep,2}^*,
\end{align*}
 $D_{\ep,2}=-\{\phi_{\ep},\cdot\}:\tilde{X}_{\ep,2}\supset D(D_{\ep,2})\to\tilde W_{\ep,2}$, and $\tilde{X}_{\ep,2}$ is defined in
\eqref{space-tilde-Xep-m} for $m=2$. Similar to \eqref{linearized mhd-sep-hamiltonian},
$\textbf{(G1-2)}$ in Lemma \ref{indice-theorem-sep} can be verified for \eqref{linearized mhd-sep-hamiltonian-2}.
Note that $-\Delta\phi-g'(\phi_{\ep})\phi=-\Delta\phi-g'(\phi_{\ep})(I-P_{\ep,2})\phi=\tilde A_{\ep,2}\phi$ due to $P_{\ep,2}\phi=0$ for $\phi\in \tilde W_{\ep,2}$.
By Corollaries \ref{A-L-dec-o} and \ref{A-L-dec-e}, a similar argument to \eqref{n-tilde A-w-x} implies
$n^-(\tilde A_{\ep,2}|_{\tilde W_{\ep,2}})=2$, $\ker(\tilde A_{\ep,2}|_{\tilde W_{\ep,2}})=3$ and $\langle \tilde A_{\ep,2}\phi,\phi\rangle\geq C\|\phi\|_{\tilde W_{\ep,2}}^2$ for some $C>0$, where $\phi\in\tilde W_{\ep,2+}$.
   This verifies
  $\textbf{(G3-4)}$ in Lemma \ref{indice-theorem-sep} for  \eqref{linearized mhd-sep-hamiltonian-2}.
 By  Lemma \ref{indice-theorem-sep}, we have
 \begin{align}\label{stab-criteria-coalescence-mhd}
 (\omega=0,\phi_{\ep}) \text{ is coalescence unstable if and only if } n^-\left(\tilde A_{\ep,2}|_{\overline{R(D_{\ep,2})}}\right)>0.
 \end{align}
We take the test function $\tilde{\psi}_\ep$ defined in \eqref{test-even}, where $(\theta_\ep, \gep) \in \tilde{\Omega}_{2} = \mathbb{T}_{4\pi} \times [-1, 1]$ are given in \eqref{transf1}-\eqref{transf2}. Noting that
\begin{align*}
\iint_{\Omega_2}g'(\phi_\ep)\tilde{\psi}_\ep dxdy  =2\int_{-1}^1\int_{0}^{4\pi}\cos\left(\frac{\theta_\ep}{2}\right)(1-\gamma_\ep^2)^{1\over4} d\theta_\ep d\gamma_\ep=0,
\end{align*}
we have $\tilde{\psi}_\ep\in\tilde W_{\ep,2}$.
Since $\tilde{\psi}_\ep$ is `odd' symmetrical about  $\{x=\pi\}$ along any trajectory of the steady velocity, a similar argument to Lemma \ref{b2-even} implies that $
\tilde{\psi}_\ep\in \overline{R(D_{\ep,2})}$. It follows from \eqref{b1-even} that $\langle \tilde A_{\ep,2} \tilde{\psi}_\ep,\tilde{\psi}_\ep\rangle<0$, and thus, $n^-\left(\tilde A_{\ep,2}|_{\overline{R(D_{\ep,2})}}\right)>0$. This proves Theorem \ref{main result1-mhd-all1} (1).

\begin{remark}
It remains open to determine whether, for odd integers $m>1$, the Kelvin--Stuart magnetic island equilibrium $(\omega=0,\phi_\ep)$ is linearly unstable under $2m\pi$-periodic perturbations. In particular, the triple-periodic case is still unresolved. The double-periodic argument developed here does not seem to extend directly to odd-periodic perturbations. For example, in the triple-periodic case, the test function used for the corresponding 2D Euler instability does not apply, since it does not belong to $\overline{R(D_{\ep,3})}$.
\end{remark}
\subsection{Nonlinear orbital stability   for co-periodic perturbations}
Let $\tilde \omega$, $\tilde \psi$, $\tilde J$ and $\tilde \phi$ be  the perturbed vorticity, stream function, current density and  magnetic potential, respectively. The perturbations of vorticity, stream function, current density and magnetic potential are denoted by $\omega=\tilde \omega-0$,  $\psi=\tilde \psi-0$,
 $J=\tilde J-J^\ep$ and $\phi=\tilde \phi-\phi_\ep$, correspondingly. The perturbed stream function is determined by $
\tilde \psi=G*\tilde \omega$ for
\begin{align}\label{tilde Y-space}\tilde\omega\in \tilde Y=\left\{\tilde\omega\in L^1\cap L^3 (\Omega)\bigg|\iint_{\Omega}\tilde\omega dxdy=0,y\tilde\omega\in L^1(\Omega)\right\}.
\end{align}
Then $(\partial_y\tilde \psi(x,y),-\partial_x\tilde \psi(x,y))\to (0,0)$ as $y\to\pm\infty$ for $x\in\mathbb{T}_{2\pi}$, and $\vec{v}=(\partial_y\tilde \psi,-\partial_x\tilde \psi)$, where $\vec{v}$ is  the perturbed velocity field.
Since the perturbed magnetic field $\vec{B}$ is required to satisfy
\[
\vec{B}(x,y)\to (\pm1,0)\qquad \text{as } y\to\pm\infty,\quad x\in\mathbb{T}_{2\pi},
\]
it is natural to impose
\[
\iint_\Omega \tilde J\,dxdy=-4\pi,
\qquad
\iint_\Omega J\,dxdy=0.
\]

For \(\tilde J\in W_{non}\), define the perturbed magnetic potential by
\[
\tilde\phi = G*\tilde J-\ln\sqrt{1-\ep^2},
\]
where
\[
W_{non}
:=
\left\{
\tilde J\in L^1(\Omega)\cap L^3(\Omega)\;\middle|\;
y\tilde J\in L^1(\Omega),\;
\iint_\Omega \tilde J\,dx\,dy=-4\pi
\right\}.
\]
As in \eqref{v-mu term2}-\eqref{v-mu term1}, one has
\[
(\partial_y\tilde\phi(x,y),-\partial_x\tilde\phi(x,y))\to (\pm 1,0)
\qquad\text{as } y\to \pm\infty,
\]
for each \(x\in\mathbb T_{2\pi}\). We therefore represent the perturbed magnetic field by
$\vec B=(\partial_y\tilde\phi,-\partial_x\tilde\phi).$ Taking the curl of Faraday's law \(\partial_t\vec B=-\curl(\vec E)\), we obtain
\[
\partial_t\tilde J=-\Delta\{\tilde\psi,\tilde\phi\}.
\]
Convolving this identity with \(G\) yields
\[
\partial_t(G*\tilde J)=\{\tilde\psi,G*\tilde J\},
\]
and hence $\partial_t\tilde\phi=\{\tilde\psi,\tilde\phi\}.$ The constant \(-\ln\sqrt{1-\ep^2}\) is included in the definition of \(\tilde\phi\) so that the steady states $\phi_\ep = G*J^\ep-\ln\sqrt{1-\ep^2},$
given in \eqref{steady-mhd}, satisfy the same Liouville's equation \eqref{elip} for all \(\ep\in[0,1)\). Without this normalization, the function \(g\) in \eqref{elip} would depend on \(\ep\), which is inconvenient for the analysis.

Let $\hat h(s)=-{1\over2}e^{-2s}$. Then $\hat h'(\phi_\ep)=e^{-2\phi_\ep}=-g( \phi_\ep)=-J^{\ep}$, where $g(s)=-e^{-2s}$.
For $\tilde \omega\in\tilde Y$ and
\begin{align}\label{def-Z-non-ep}
\tilde \phi\in \tilde Z_{non,\ep}\triangleq \{ \tilde \phi=G*\tilde J-\ln\sqrt{1-\epsilon^2}|\tilde J\in W_{non}\},
\end{align}
motivated by \cite{holm1985nonlinear}, we define the energy-Casimir (EC) functional
\begin{align}\nonumber
\hat H(\tilde \omega,\tilde \phi)=&{1\over2}\iint_{\Omega}\tilde \psi \tilde\omega dxdy+{1\over2}\iint_{\Omega}(G\ast\tilde J)\tilde J  dxdy+\iint_{\Omega}\hat h(\tilde\phi) dxdy\\\label{EC-functional-mhd}
=&{1\over2}\iint_{\Omega}(G\ast\tilde\omega)\tilde\omega dxdy+{1\over2}\iint_{\Omega}(G\ast \tilde J) \tilde J dxdy-\iint_{\Omega}{1\over2}e^{-2\tilde\phi} dxdy.
\end{align}
Similar to \eqref{PE-finite}, we have $|\iint_{\Omega}(G\ast\tilde\omega)\tilde\omega dxdy|<\infty$ and $|\iint_{\Omega}(G\ast\tilde J)\tilde J  dxdy|<\infty$.
For $\tilde\omega\in  \tilde Y$,
 by  Lemma \ref{well-poseness-Poisson-equation-nonlinear-case},
 the Poisson equation
$-\Delta \psi =\tilde \omega\in \tilde Y$   has a  unique weak solution $\psi=I_{\tilde X_\ep} \tilde\omega$ in $\tilde{X}_\ep$, where $I_{\tilde X_\ep}:\tilde X_\ep^*\to \tilde X_\ep$ is the isomorphism defined by the Riesz representation
theorem.
By Lemma \ref{G-ast-omega-psi-constant},
$G\ast\tilde\omega-I_{\tilde X_\ep} \tilde\omega$ is a constant for $\tilde\omega\in \tilde Y$. Then \begin{align}\label{G-tildeomega-omega-positive}
\iint_{\Omega}(G\ast\tilde\omega)\tilde\omega dxdy=\iint_{\Omega}(I_{\tilde X_\ep} \tilde\omega)\tilde\omega  dxdy=\iint_{\Omega}|\nabla(I_{\tilde X_\ep} \tilde\omega)|^2 dxdy>0\end{align}
for $0\neq \tilde\omega\in\tilde Y$, where we used $\iint_{\Omega}\tilde\omega dxdy=0$.
For $\tilde \phi\in \tilde Z_{non,\ep}$, by \eqref{steady-mhd} we have $\tilde \phi-\phi_\ep=G*(\tilde J-J^{\ep})=G*J$.
 The space of perturbations of  magnetic potentials is
$ Z_{non,\ep}\triangleq\{\tilde \phi-\phi_\ep=G*J|\tilde \phi\in \tilde Z_{non,\ep}\}$. Similar to Lemmas \ref{well-poseness-Poisson-equation-nonlinear-case}-\ref{G-ast-omega-psi-constant},
there exist $\phi_*\in \tilde X_\ep$ and a constant $c_*$  such that $\phi-\phi_*=c_*$ for each $\phi=G*J\in Z_{non,\ep}$. Then for $\tilde \phi\in \tilde Z_{non,\ep}$, we have
 \begin{align*}
\iint_{\Omega}{1\over2}e^{-2\tilde\phi} dxdy=&\iint_{\Omega}{1\over2}e^{-2\phi_\ep}e^{-2\phi} dxdy={1\over4}\iint_{\Omega}g'( \phi_\ep)e^{-2(\phi_*+c_*)} dxdy
\leq C e^{C\|\phi_*\|_{\tilde{X}_\ep}^2}<\infty
\end{align*}
due to Lemma \ref{Orlicz-type inequlity-lemma} and $\phi_*\in \tilde X_\ep$. Thus, the EC functional \eqref{EC-functional-mhd} is well-defined. Then $\hat H'(0,\phi_{\ep})=-\Delta\phi_{\ep}+\hat h'(\phi_{\ep})=-\Delta\phi_{\ep}-g(\phi_{\ep})=0$ and
\begin{align}\nonumber
\hat H(\tilde \omega,\tilde \phi)-\hat H( 0,  \phi_{\ep})
=&{1\over2}\iint_{\Omega}(G\ast\omega)\omega dxdy+{1\over2}\iint_{\Omega}\left((G\ast\tilde J)\tilde J-(G\ast J^{\ep}) J^{\ep}\right)dxdy\\\nonumber
&+\iint_{\Omega}\left(\hat h(\tilde\phi)-\hat h(\phi_{\ep}) \right)dxdy\\\nonumber
=&{1\over2}\iint_{\Omega}(G\ast\omega)\omega dxdy+{1\over2}\iint_{\Omega}|\nabla\phi|^2dxdy\\\nonumber
&+\iint_{\Omega}\left(\hat h(\phi_\ep+\phi)-\hat h(\phi_\ep)-\hat h'(\phi_\ep)\phi \right)dxdy\\\nonumber
=&{1\over2}\iint_{\Omega}(G\ast\omega)\omega dxdy+\iint_{\Omega} \left(\frac 1 2 |\nabla \phi|^2 -\frac 1 4 g'(\phi_\ep)(e^{-2\phi} + 2\phi - 1)\right) dxdy\\\nonumber
=&{1\over2}\iint_{\Omega}(G\ast\omega)\omega dxdy\\\nonumber
&+\iint_{\Omega} \left(\frac 1 2 |\nabla \phi|^2 -\frac 1 4 g'(\phi_\ep)(e^{-2(\phi-P_\ep\phi)} + 2(\phi-P_\ep\phi) - 1)\right) dxdy\\\label{H-omega-phi}
&+\iint_{\Omega} \left(-\frac 1 2e^{-2\phi_\ep}(e^{-2\phi}-e^{-2(\phi-P_\ep\phi)} + 2P_\ep\phi)\right) dxdy,
\end{align}
where the expression of  $P_\ep$ is given in \eqref{P-ep}.
Define two  functionals by
\begin{align}\nonumber
S_\ep( \phi)\triangleq &
\iint_{\Omega} \left(\frac 1 2 |\nabla \phi|^2 -\frac 1 4 g'(\phi_\ep)(e^{-2(\phi-P_\ep\phi)} + 2(\phi-P_\ep\phi) - 1)\right) dxdy,\quad \phi\in \tilde X_\ep,\\\label{def-functional-S}
R_\ep( \phi)\triangleq &\iint_{\Omega} \left(-\frac 1 2e^{-2\phi_\ep}(e^{-2\phi}-e^{-2(\phi-P_\ep\phi)} + 2P_\ep\phi)\right) dxdy,\quad \phi\in Z_{non,\ep},
\end{align}
and the distance functionals by
\begin{align}\nonumber
&\hat d_{1}((\tilde\omega,\tilde \phi),(0,\phi_\ep))=\iint_{\Omega}(G\ast\omega)\omega dxdy,\quad
\hat d_{2}((\tilde\omega,\tilde \phi),(0,\phi_\ep))=\iint_{\Omega}|\nabla\phi|^2dxdy,\\\label{distance3-mhd}
&\hat d_{3}((\tilde\omega,\tilde \phi),(0,\phi_\ep))=-\iint_{\Omega}\left(\hat h(\phi_\ep+\phi)-\hat h(\phi_\ep)-\hat h'(\phi_\ep)\phi \right)dxdy,\\\label{distance-mhd}
&\hat d((\tilde\omega,\tilde \phi),(0,\phi_\ep))=\hat d_{1}((\tilde\omega,\tilde \phi),(0,\phi_\ep))+\hat d_{2}((\tilde\omega,\tilde \phi),(0,\phi_\ep))+\hat d_{3}((\tilde\omega,\tilde \phi),(0,\phi_\ep))
\end{align}
for $\tilde \omega\in \tilde Y$ and $\tilde \phi\in\tilde  Z_{non,\ep}$, where we used \eqref{G-tildeomega-omega-positive} and  the fact that  $e^{-2s} + 2s - 1>0$ for $s\neq 0$ to ensure that $\hat d_1$and $\hat d_3$ are well-defined, respectively.
Then we study the $C^2$ regularity of   $S_\ep$ and prove that the remainder term $R_\ep$ is a high order term of the distance $\hat d$. We need the following inequalities.
\begin{lemma}\label{e-p-c-e-x}
For $\ep\in(0,1)$, $a\in\mathbb{R}$ and $p\in\mathbb{Z}^+$, we have $|P_\ep\phi|\leq C\|\phi\|_{\tilde X_\ep}$,
\begin{align*}
&\iint_\Omega g'(\phi_\ep) e^ {a|\phi-P_\ep\phi|} dxdy \leq  C e^{C(a)\left(\|\phi\|_{\tilde{X}_\ep}+\|\phi\|_{\tilde{X}_\ep}^2\right)},\\
&\iint_\Omega g'(\phi_\ep) |\phi-P_\ep\phi|^p dxdy  \leq  C(p) e^{C\left(\|\phi\|_{\tilde{X}_\ep}+\|\phi\|_{\tilde{X}_\ep}^2\right)}
\end{align*}
for $\phi \in \tilde{X}_\ep$.
\end{lemma}
\begin{proof}
$|P_\ep\phi|\leq C\|\phi\|_{\tilde X_\ep}$ follows from
\eqref{projection-controlled by-X-ep} for $\phi \in \tilde{X}_\ep$.
By Lemma  \ref{Orlicz-type inequlity-lemma}, we have
\begin{align*}
\iint_\Omega g'(\phi_\ep) e^ {a|\phi-P_\ep\phi|} dxdy \leq& e^ {|a||P_\ep\phi|} \iint_\Omega g'(\phi_\ep) e^ {|a||\phi|} dxdy\leq  Ce^ {C|a|\|\phi\|_{\tilde X_\ep}+Ca^2\|\phi\|_{\tilde{X}_\ep}^2},\\
\iint_\Omega g'(\phi_\ep) |\phi-P_\ep\phi|^p dxdy \leq& p!\iint_\Omega g'(\phi_\ep) e^{|\phi-P_\ep\phi|} dxdy \leq  Cp! e^ {C\|\phi\|_{\tilde X_\ep}+C\|\phi\|_{\tilde{X}_\ep}^2},\quad \phi \in \tilde{X}_\ep.
\end{align*}
\end{proof}
The $C^2$ regularity of   $S_\ep$ is proved as follows.
\begin{Lemma}\label{C2-mhd}
$S_\ep\in C^2(\tilde X_\ep)$, $S_\ep'(0) = 0$ and
\begin{align*}
\langle S_\ep''(0)\phi_1,\phi_2 \rangle&=  \iint_{\Omega}\left(\nabla\phi_1\cdot\nabla\phi_2- g'(\phi_\ep) (\phi_1-P_\ep\phi_1)(\phi_2-P_\ep\phi_2)\right)dxdy=\langle \tilde A_\ep \phi_1,\phi_2 \rangle
\end{align*}
for $\phi_1,\phi_2\in\tilde X_\ep$,
where   $\tilde{A}_\ep$ is defined in \eqref{tilde-A-ep-A-ep} and $\ep\in(0,1)$.
\end{Lemma}
\begin{proof}
Let $\phi\in \tilde{X}_\ep$. For  $\psi \in \tilde{X}_\ep$, by Lemmas  \ref{poincare2ep} and \ref{e-p-c-e-x} we have
\begin{align*}
|\partial_\lambda S_\ep(\phi + \lambda \psi)|_{\lambda = 0}|
= & \iint_\Omega  \left(\nabla\phi\cdot\nabla\psi+{1\over2}  g'(\phi_\ep) (e^{-2(\phi-P_\ep\phi)}-1)(\psi-P_\ep\psi) \right)dxdy\\
\leq &\|\phi\|_{\tilde{X}_\ep}\|\psi\|_{\tilde{X}_\ep}+ C \left(\iint_\Omega    g'(\phi_\ep) (e^{-4(\phi-P_\ep\phi)}-2e^{-2(\phi-P_\ep\phi)}+1)dxdy\right)^{1\over2}\|\psi \|_{\tilde{X}_\ep}\\
\leq &\left(\|\phi\|_{\tilde{X}_\ep}+ C \left(C e^{C\left(\|\phi\|_{\tilde{X}_\ep}+\|\phi\|_{\tilde{X}_\ep}^2\right)}+C\right)^{1\over2}\right)\|\psi \|_{\tilde{X}_\ep}.
\end{align*}
Thus,  $S_\ep$  is G$\hat{\text{a}}$teaux differentiable at $\phi\in  \tilde{X}_\ep$.
Let $\{\phi_n\}_{n=1}^\infty\in \tilde X_\ep$ such that $\phi_n\to\phi$ in $\tilde{X}_\ep$,
and choose $N>0$ such that $\|\phi_n\|_{\tilde X_\ep}\leq \|\phi\|_{\tilde X_\ep}+1$ for $n\geq N$.   By Lemmas  \ref{poincare2ep} and \ref{e-p-c-e-x} we have for $n\geq N$ and $\psi\in\tilde X_\ep$,
\begin{align*}
&|\partial_\lambda S_\ep(\phi_n + \lambda \psi)|_{\lambda = 0}-\partial_\lambda S_\ep(\phi + \lambda \psi)|_{\lambda = 0}| \\
= & \left|\iint_{\Omega}\left(\nabla(\phi_n-\phi)\cdot\nabla\psi +{1\over 2} g'(\phi_\ep)(e^{-2(\phi_n-P_\ep\phi_n)} - e^{-2(\phi-P_\ep\phi)})(\psi-P_\ep\psi)\right) dxdy\right|\\
\leq&\|\phi_n-\phi\|_{\tilde X_\ep}\|\psi\|_{\tilde X_\ep}\\
&+\left|\int_0^1  \iint_{\Omega} g'(\phi_\ep) e^{-2(s(\phi_n-P_\ep\phi_n) + (1-s)(\phi-P_\ep\phi))}  (\phi_n - \phi-P_\ep(\phi_n-\phi))  (\psi-P_\ep\psi) dxdyds\right| \\
\leq&\|\phi_n-\phi\|_{\tilde X_\ep}\|\psi\|_{\tilde X_\ep}\\
&+\|\phi_n-\phi\|_{\tilde X_\ep}\|\psi-P_\ep\psi\|_{L^4_{g'(\phi_\ep)}}\int_0^1 \left( \iint_{\Omega} g'(\phi_\ep) e^{-8(s(\phi_n-P_\ep\phi_n) + (1-s)(\phi-P_\ep\phi))}  dxdy\right)^{1\over4}ds\\
\leq &  \|\phi_n-\phi\|_{\tilde X_\ep}\|\psi\|_{\tilde X_\ep}\\
&+\|\phi_n-\phi\|_{\tilde X_\ep}
 C e^{C\left(\|\psi\|_{\tilde{X}_\ep}+\|\psi\|_{\tilde{X}_\ep}^2\right)}
\int_0^1e^{C\left(\|s\phi_n + (1-s)\phi\|_{\tilde X_\ep}+\|s\phi_n + (1-s)\phi\|_{\tilde X_\ep}^2\right)}ds\\
\leq &\left(\|\psi\|_{\tilde X_\ep}+C_{\|\psi\|_{\tilde X_\ep}}C_{\|\phi\|_{\tilde X_\ep}}\right)\|\phi_n-\phi\|_{\tilde X_\ep}\to 0\quad \text{as}\quad n\to\infty.
\end{align*}
Thus, $S_\ep\in C^1(\tilde X_\ep)$.
 For
$\psi \in  \tilde{X}_\ep$ and $\varphi\in\tilde{X}_\ep$, by Lemma  \ref{e-p-c-e-x} we have
\begin{align*}
&\left|\partial_\tau\partial_\lambda S_\ep(\phi + \lambda\psi+\tau\varphi)|_{\lambda =\tau= 0}\right|\\
=& \left|\iint_{\Omega}\left(\nabla\psi\cdot\nabla\varphi- g'(\phi_\ep) e^{-2(\phi-P_\ep\phi)}(\psi-P_\ep\psi)(\varphi-P_\ep\varphi)\right)dxdy\right|\\
\leq&\|\psi\|_{\tilde{X}_\ep}\|\varphi\|_{\tilde{X}_\ep}+\left(\iint_{\Omega} g'(\phi_\ep) e^{-4(\phi-P_\ep\phi)} dxdy\right)^{1\over2}
\|\psi-P_\ep\psi\|_{L_{g'(\phi_\ep)}^4}\|\varphi-P_\ep\varphi\|_{L_{g'(\phi_\ep)}^4}\\
\leq&\|\psi\|_{\tilde{X}_\ep}\|\varphi\|_{\tilde{X}_\ep}+Ce^{C\left(\|\phi\|_{\tilde{X}_\ep}+\|\psi\|_{\tilde{X}_\ep}+\|\varphi\|_{\tilde{X}_\ep}+
\|\phi\|_{\tilde{X}_\ep}^2+\|\psi\|_{\tilde{X}_\ep}^2+\|\varphi\|_{\tilde{X}_\ep}^2\right)}.
\end{align*}
Let  $\{\phi_n\}_{n=1}^\infty\in \tilde X_\ep$ be defined as above.
  For  $\psi,\varphi\in\tilde X_\ep$ and $n\geq N$, we have
\begin{align*}
&|\partial_\tau\partial_\lambda S_\ep(\phi_n + \lambda \psi+\tau\varphi)|_{\lambda =\tau= 0}-\partial_\tau\partial_\lambda S_\ep(\phi + \lambda \psi+\tau\varphi)|_{\lambda =\tau= 0}|\\
=&\left|2\int_0^1\iint_{\Omega}g'(\phi_\ep) e^{-2(s(\phi_n-P_\ep\phi_n) + (1-s)(\phi-P_\ep\phi))}(\phi_n - \phi-P_\ep(\phi_n-\phi))(\psi-P_\ep\psi)(\varphi-P_\ep\varphi) dxdyds\right|\\
\leq&C\|\phi_n-\phi\|_{\tilde X_\ep}\|\psi-P_\ep\psi\|_{L_{g'(\phi_\ep)}^6}\|\varphi-P_\ep\varphi\|_{L_{g'(\phi_\ep)}^6}\\
&\int_0^1\left(\iint_{\Omega}g'(\phi_\ep) e^{-12(s(\phi_n-P_\ep\phi_n) + (1-s)(\phi-P_\ep\phi))}dxdy\right)^{1\over6}ds\\
\leq&C\|\phi_n-\phi\|_{\tilde X_\ep}
e^{C(\|\psi\|_{\tilde X_\ep}+\|\psi\|_{\tilde X_\ep}^2)}e^{C(\|\varphi\|_{\tilde X_\ep}+\|\varphi\|_{\tilde X_\ep}^2)}
\int_0^1\left(Ce^{C(\|s\phi_n + (1-s)\phi\|_{\tilde X_\ep}+\|s\phi_n + (1-s)\phi\|_{\tilde X_\ep}^2)}\right)^{1\over6}ds\\
\leq &C_{\|\psi\|_{\tilde X_\ep}}C_{\|\varphi\|_{\tilde X_\ep}}C_{\|\phi\|_{\tilde X_\ep}}\|\phi_n-\phi\|_{\tilde X_\ep}\to 0\quad \text{as}\quad n\to\infty.
\end{align*}
Thus, $S_\ep\in C^2(\tilde X_\ep)$.
\end{proof}
Next, we estimate the remainder term $R_\ep$.
\begin{lemma} \label{remainder term R}
For $\phi\in Z_{non,\ep}$ and $\left|\iint_{\Omega}(e^{-2\tilde \phi}-e^{-2\phi_\ep})dxdy\right|<1$, we have
\begin{align}\label{remainder term-estimate}
|R_\ep(\phi)|\leq O(\hat d_{3}((\tilde\omega,\tilde \phi),(0,\phi_\ep))^2)+C\left|\iint_{\Omega}(e^{-2\tilde \phi}-e^{-2\phi_\ep})dxdy\right|
\end{align}
as $\hat d_{3}((\tilde\omega,\tilde \phi),(0,\phi_\ep))\to0$.
\end{lemma}
\begin{proof}
By \eqref{P-ep} and \eqref{distance3-mhd}, we have
\begin{align*}
P_\ep\phi={\iint_\Omega\hat h'(\phi_\ep)\phi dxdy\over4\pi}={1\over 4\pi}\left(\hat d_{3}((\tilde\omega,\tilde \phi),(0,\phi_\ep))-{1\over2}\iint_\Omega(e^{-2\tilde \phi}-e^{-2\phi_\ep})dxdy\right)
\end{align*}
for $\phi\in Z_{non,\ep}$. Then we infer from the definition \eqref{def-functional-S} of $R_\ep$ that
\begin{align*}
|R_\ep(\phi)|=&\left|-{1\over 2}\iint_{\Omega}\left(e^{-2\tilde \phi}-e^{-2(\tilde \phi-P_\ep \phi)}+2e^{-2\phi_\ep}P_\ep\phi\right)dxdy\right|\\
\leq&\left|{1\over 2}(e^{2P_\ep \phi}-1-2P_\ep\phi)\iint_{\Omega}e^{-2\phi_\ep}dxdy\right|
+\left|{1\over 2}(e^{2P_\ep\phi}-1)\iint_{\Omega}(e^{-2\tilde \phi}-e^{-2\phi_\ep})dxdy\right|\\
\leq&(P_\ep\phi)^2O(1)+|P_\ep\phi|\left|\iint_{\Omega}(e^{-2\tilde \phi}-e^{-2\phi_\ep})dxdy\right|O(1)\\
\leq&O(\hat d_{3}((\tilde\omega,\tilde \phi),(0,\phi_\ep))^2)+C\left(\iint_{\Omega}(e^{-2\tilde \phi}-e^{-2\phi_\ep})dxdy\right)^2,
\end{align*}
which gives \eqref{remainder term-estimate}.
\end{proof}
Now, we  prove Theorem \ref{main result1-mhd-all}, that is,  the Kelvin--Stuart magnetic islands $(\omega=0,\phi_{\ep_0})$ are conditionally nonlinearly orbitally stable for co-periodic perturbations, where $\ep_0\in(0,1)$.

\begin{proof}By Lemma \ref{imp-vertical condition},  there exists $\delta_0(\ep_0)>0$ such that for any  $(x_0,y_0)\in\Omega$ and  $\tilde \phi$ with $\hat d_2((\tilde \omega,\tilde \phi),(0,\phi_{\ep_0}(x+x_0,y+y_0)))< \delta_0(\ep_0)$, there exist $(\tilde x_0,\tilde y_0)\in\Omega$ and $\tilde\epsilon_0\in(a(\ep_0),b(\ep_0))$, depending continuously on $\tilde\phi, x_0$ and $y_0$, such that
\begin{align}\label{app-lemma-imp-vertical condition-mhd}
\tilde \phi\left(x-\tilde x_0,y-\tilde y_0\right)-\phi_{\tilde\ep_0}(x,y)\perp\ker \left( \tilde A_{\tilde\ep_0}\right)\quad \text{in}\quad \dot{H}^1(\Omega)
\end{align}
 and
$
|x_0-\tilde x_0|+|y_0-\tilde y_0|+|\ep_0-\tilde \ep_0|\leq C(\ep_0)\sqrt{\delta_0(\ep_0)}
$
for some $a(\ep_0)\in (0,\ep_0)$ and $b(\ep_0)\in(\ep_0,1)$.
For  $\kappa>0$, let $\delta=\delta(\ep_0,\kappa)<\min\big\{{\kappa^4\over32C_1C_2(\ep_0)^4C_3(\ep_0)^4},$ ${\delta_0(\ep_0)\over2}\big\}$, where $C_1, C_2(\ep_0), C_3(\ep_0)>1$ are  determined by \eqref{de-c1-mhd}, \eqref{I-omegaep1t0-omegaep0-mhd} and \eqref{Iomega0Iomegat1-mhd}.
For the initial data $(\tilde \omega(0)=\tilde \omega_0,\tilde \phi(0)=\tilde \phi_0)$ satisfying
\eqref{initial data-mhd},
there exists $(x_0(0),y_0(0))\in\Omega$  such that
\begin{align}\nonumber
&\hat d((\tilde\omega(0),\tilde \phi(0)),(0,\phi_{\ep_0}(x+x_0(0),y+y_0(0))))+\left|\iint_{\Omega}(e^{-2\tilde \phi(0)}-e^{-2\phi_{\ep_0}})dxdy\right|\\\label{initial data-translation-infimum}
<&\delta(\ep_0,\kappa)\leq {\kappa^4\over32C_1C_2(\ep_0)^4C_3(\ep_0)^4}.
\end{align}

For \(t\ge 0\), we claim that if there exists \((x_0(t),y_0(t))\in \Omega\) such that
\[
\hat d\bigl((\tilde\omega(t),\tilde\phi(t)),(0,\phi_{\ep_0}(x+x_0(t),y+y_0(t)))\bigr)<\delta_0(\ep_0),
\]
then there exist \((x_1(t),y_1(t))\in \Omega\) and \(\ep_1(t)\in (a(\ep_0),b(\ep_0))\) such that
\begin{align}\label{a prior estimate-mhd}
\hat d\bigl((\tilde\omega(t),\tilde\phi(t)),(0,\phi_{\ep_1(t)}(x+x_1(t),y+y_1(t)))\bigr)
<\frac{\kappa^4}{16C_2(\ep_0)^4C_3(\ep_0)^4}.
\end{align}

Indeed, by \eqref{app-lemma-imp-vertical condition-mhd}, there exist \((x_1(t),y_1(t))\in\Omega\) and \(\ep_1(t)\in(a(\ep_0),b(\ep_0))\), depending continuously on \(t\), such that
\[
\tilde\phi(x-x_1(t),y-y_1(t))-\phi_{\ep_1(t)}(x,y)\perp \ker(\tilde A_{\ep_1(t)})
\qquad\text{in }\dot H^1(\Omega),
\]
and
\[
|x_0(t)-x_1(t)|+|y_0(t)-y_1(t)|+|\ep_0-\ep_1(t)|
\le C(\ep_0)\sqrt{\delta_0(\ep_0)}
\]
for \(t>0\), while at \(t=0\),
\begin{align}\label{initial-x0x1-y0y1-ep0ep1}
|x_0(0)-x_1(0)|+|y_0(0)-y_1(0)|+|\ep_0-\ep_1(0)|
\le C(\ep_0)\sqrt{\delta(\ep_0,\kappa)}.
\end{align}

Recall that
\[
\langle \tilde A_\ep\varphi,\varphi\rangle \ge C_0\|\varphi\|_{\tilde X_\ep}^2,
\qquad
\varphi\in \tilde X_{\ep,+}:=\tilde X_\ep\ominus \ker(\tilde A_\ep),
\]
where
\[
\ker(\tilde A_\ep)=\mathrm{span}\{\eta_\ep,\gamma_\ep,\xi_\ep\}.
\]
By choosing \(\delta(\ep_0,\kappa)>0\) smaller if necessary, it follows from
\eqref{initial-x0x1-y0y1-ep0ep1} and \eqref{initial data-translation-infimum} that
\[
\hat d\bigl((0,\phi_{\ep_0}(x+x_0(0),y+y_0(0))),(0,\phi_{\ep}(x+x_1(0),y+y_1(0)))\bigr)
< \frac{\kappa^4}{32C_1C_2(\ep_0)^4C_3(\ep_0)^4},
\]
and
\[
\hat d\bigl((\tilde\omega(0),\tilde\phi(0)),(0,\phi_\ep(x+x_1(0),y+y_1(0)))\bigr)
+\left|\iint_\Omega \bigl(e^{-2\tilde\phi(0)}-e^{-2\phi_{\ep_0}}\bigr)\,dx\,dy\right|
\le
\frac{\kappa^4}{16C_1C_2(\ep_0)^4C_3(\ep_0)^4}
\]
for \(\ep=\ep_0\) or \(\ep=\ep_1(0)\).

Finally, choose \(\tau\in(0,\tfrac12)\) sufficiently small so that
\[
-\frac12\tau +(1+\tau)C_0>\tau.
\]
By \eqref{H-omega-phi}-\eqref{def-functional-S} and Lemmas \ref{C2-mhd}-\ref{remainder term R} we have
\begin{align}\nonumber
& \hat d((\tilde \omega(0),\tilde \phi(0)),(0,\phi_{\ep_1(0)}(x+x_1(0),y+y_1(0)))\\\nonumber
\geq&\hat H(\tilde \omega(0),\tilde \phi(0))-\left(\hat H( 0, \phi_{\ep_1(0)}(x+x_1(0),y+y_1(0)))+4\pi\ln\sqrt{1-\ep_1(0)^2}\right)+4\pi\ln\sqrt{1-\ep_1(0)^2}\\\nonumber
 \geq& \hat H(\tilde \omega(t),\tilde \phi_{tran}(t))-\hat H( 0, \phi_{\ep_1(t)})-4\pi\ln\sqrt{1-\ep_1(t)^2}+4\pi\ln\sqrt{1-\ep_1(0)^2}\\\nonumber
=&{1\over2}\iint_\Omega (G*\tilde \omega(t))\tilde\omega(t) dxdy+{1\over2}\iint_\Omega (2(G*J^t)J^{\ep_1(t)}+(G*J^t)J^t) dxdy\\\nonumber
&+\iint_\Omega (\hat h(\phi_{\ep_1(t)}+\phi^t)-\hat h (\phi_{\ep_1(t)})) dxdy-4\pi\ln\sqrt{1-\ep_1(t)^2}+4\pi\ln\sqrt{1-\ep_1(0)^2}\\\nonumber
=&{1\over2}\iint_\Omega (G*\tilde \omega(t))\tilde\omega(t) dxdy+{1\over2}\iint_\Omega |\nabla\phi^t|^2 dxdy-4\pi\ln\sqrt{1-\ep_1(t)^2}+4\pi\ln\sqrt{1-\ep_1(0)^2}\\\nonumber
&+\iint_\Omega (\hat h(\phi_{\ep_1(t)}+\phi^t)-\hat h (\phi_{\ep_1(t)})-\hat h '(\phi_{\ep_1(t)})(G*J^t)) dxdy\\\nonumber
=&{1\over2}\iint_\Omega (G*\tilde \omega(t))\tilde\omega(t) dxdy+{1\over2}\iint_\Omega |\nabla\phi^t|^2 dxdy-4\pi\ln\sqrt{1-\ep_1(t)^2}+4\pi\ln\sqrt{1-\ep_1(0)^2}\\\nonumber
&+\iint_\Omega \left(\hat h(\phi_{\ep_1(t)}+\phi^t)-\hat h (\phi_{\ep_1(t)})-\hat h '(\phi_{\ep_1(t)})(\phi^t-\ln\sqrt{1-\ep_1(t)^2}+\ln\sqrt{1-\ep_0^2}) \right)dxdy\\\nonumber
=&{1\over2}\iint_\Omega (G*\tilde \omega(t))\tilde\omega(t) dxdy+{1\over2}\iint_\Omega |\nabla\phi^t|^2 dxdy-4\pi\ln\sqrt{1-\ep_0^2}+4\pi\ln\sqrt{1-\ep_1(0)^2}\\\nonumber
&+\iint_\Omega \left(\hat h(\phi_{\ep_1(t)}+\phi^t)-\hat h (\phi_{\ep_1(t)})-\hat h '(\phi_{\ep_1(t)})\phi^t \right)dxdy\\\nonumber
 =&\left({1\over2}\hat d_1+{1\over2}\hat d_2-\hat d_3\right)((\tilde \omega(t),\tilde \phi_{tran}(t)), ( 0, \phi_{\ep_1(t)}))-4\pi\ln\sqrt{1-\ep_0^2}+4\pi\ln\sqrt{1-\ep_1(0)^2}\\\nonumber
 =&{1\over2}\hat d_1((\tilde \omega(t),\tilde \phi_{tran}(t)), ( 0, \phi_{\ep_1(t)})) +
 \tau \left(\hat d_3-{1\over2}\hat d_2\right)((\tilde \omega(t),\tilde \phi_{tran}(t)), ( 0, \phi_{\ep_1(t)}))+\\\nonumber
 &(1+\tau) \left({1\over2}\hat d_2-\hat d_3\right)((\tilde \omega(t),\tilde \phi_{tran}(t)), ( 0, \phi_{\ep_1(t)}))-4\pi\ln\sqrt{1-\ep_0^2}+4\pi\ln\sqrt{1-\ep_1(0)^2} \\\nonumber
 =& \left( {1\over2}\hat d_1+\tau \left(\hat d_3-{1\over2}\hat d_2\right)\right)((\tilde \omega(t),\tilde \phi_{tran}(t)), ( 0, \phi_{\ep_1(t)}))
 + (1+\tau) S_{\ep_1(t)}( \phi^t-c_*(t))\\\nonumber
 &+(1+\tau)R_{\ep_1(t)}(\phi^t)-4\pi\ln\sqrt{1-\ep_0^2}+4\pi\ln\sqrt{1-\ep_1(0)^2}\\\nonumber
  \geq&\left( {1\over2}\hat d_1+\tau \left(\hat d_3-{1\over2}\hat d_2\right)\right)((\tilde \omega(t),\tilde \phi_{tran}(t)), ( 0, \phi_{\ep_1(t)}))+(1+\tau)\cdot
  \\\nonumber
 &\langle  \tilde A_{\ep_1(t)} (\phi^t-c_*(t)),\phi^t-c_*(t)\rangle+ o(\hat d_2((\tilde \omega(t),\tilde \phi_{tran}(t)),(0,\phi_{\ep_1(t)}))) \\\nonumber
 &-o(\hat d_{3}((\tilde \omega(t),\tilde \phi_{tran}(t)),(0,\phi_{\ep_1(t)})))-C\left|\iint_{\Omega}(e^{-2\tilde \phi_{tran}(t)}-e^{-2\phi_{\ep_1(t)}})dxdy\right|\\\nonumber
 &-4\pi\ln\sqrt{1-\ep_0^2}+4\pi\ln\sqrt{1-\ep_1(0)^2}\\\nonumber
 \geq&\left({1\over2} \hat d_1+\tau \hat d_3\right)((\tilde \omega(t),\tilde \phi_{tran}(t)), ( 0, \phi_{\ep_1(t)}))+\left(-{1\over2}\tau+(1+\tau)C_0\right)\hat d_2((\tilde \omega(t),\tilde \phi_{tran}(t)), ( 0, \phi_{\ep_1(t)}))\\\nonumber
 &
 + o(\hat d((\tilde \omega(t),\tilde \phi_{tran}(t)),(0,\phi_{\ep_1(t)})))-C\left|\iint_{\Omega}(e^{-2\tilde \phi(0)}-e^{-2\phi_{\ep_0}})dxdy\right|
 \\\nonumber
 &-4\pi\ln\sqrt{1-\ep_0^2}+4\pi\ln\sqrt{1-\ep_1(0)^2}\\\nonumber
\geq&\tau \hat d((\tilde \omega(t),\tilde \phi(t)), ( 0, \phi_{\ep_1(t)}(x+x_1(t),y+y_1(t))))\\\nonumber
&+o(\hat d((\tilde \omega(t),\tilde \phi(t)), ( 0, \phi_{\ep_1(t)}(x+x_1(t),y+y_1(t)))))-C\left|\iint_{\Omega}(e^{-2\tilde \phi(0)}-e^{-2\phi_{\ep_0}})dxdy\right|
\\\nonumber
 &-4\pi\ln\sqrt{1-\ep_0^2}+4\pi\ln\sqrt{1-\ep_1(0)^2},
\end{align}
where $\phi^t=\tilde \phi_{tran}(t)-\phi_{\ep_1(t)}$, $J^t=\tilde J_{tran}(t)-J^{\ep_1(t)}$,   $\tilde \phi_{tran}(t)=\tilde \phi(t;x-x_1(t),y-y_1(t))$, $\tilde J_{tran}(t)=\tilde J(t;x-x_1(t),y-y_1(t))$, $c_*(t)$ is chosen such that  $\phi^t-c_*(t)\in \tilde  X_{\ep_1(t)}$. Here, we used
$\tilde \phi(t)=G*\tilde J(t)-\ln\sqrt{1-\ep_0^2}$ for the initial data $\tilde \phi(0)=G*\tilde J(0)-\ln\sqrt{1-\ep_0^2}\in \tilde Z_{non,\ep_0}$,
\begin{align*}
&\tilde \phi_{tran}(t)=G*\tilde J_{tran}(t)-\ln\sqrt{1-\ep_0^2}\\
=&G*(J^{\ep_1(t)}+J^t)-\ln\sqrt{1-\ep_1(t)^2}+\ln\sqrt{1-\ep_1(t)^2}-\ln\sqrt{1-\ep_0^2}\\
=&\phi_{\ep_1(t)}+G*J^t+\ln\sqrt{1-\ep_1(t)^2}-\ln\sqrt{1-\ep_0^2},\\
\Longrightarrow\phi^t=&G*J^t+\ln\sqrt{1-\ep_1(t)^2}-\ln\sqrt{1-\ep_0^2},
\end{align*}
$S_{\ep_1(t)}(\phi^t)=S_{\ep_1(t)}(\phi^t-c_*(t))$, and $\hat H(0,\omega_\ep)+4\pi\ln\sqrt{1-\ep^2}$ is conserved for $\ep$, since
\begin{align*}
{d\over d\ep}\hat H(0,\phi_{\ep})=\iint_{\Omega}\partial_\ep(G* J^\ep)J^\ep dxdy=\iint_{\Omega}\partial_\ep(\phi_\ep+\ln\sqrt{1-\ep^2})J^\ep dxdy=-4\pi{d\over d\ep}\ln\sqrt{1-\ep^2}.
\end{align*}
Then for $\kappa>0$ sufficiently small, by assumption (ii) and  taking $\delta(\ep_0,\kappa)>0$ smaller, we have
\begin{align}\nonumber
&\hat d((\tilde \omega(t),\tilde \phi(t)), ( 0, \phi_{\ep_1(t)}(x+x_1(t),y+y_1(t))))\\\nonumber
\leq& C_1 \hat d((\tilde \omega(0),\tilde \phi(0)), ( 0, \phi_{\ep_1(0)}(x+x_1(0),y+y_1(0))))+ C_1 \left|\iint_{\Omega}(e^{-2\tilde \phi(0)}-e^{-2\phi_{\ep_0}})dxdy\right|\\\label{de-c1-mhd}
&+4\pi|\ln\sqrt{1-\ep_0^2}-\ln\sqrt{1-\ep_1(0)^2}|
< {\kappa^4\over16C_2(\ep_0)^4C_3(\ep_0)^4}
\end{align}
for some $C_1>1$.

For any $\kappa\in(0,\min\{\delta_0(\ep_0), 1\})$, suppose that \eqref{onlinear orbital stability-goal-mhd} is not true. Then there exist  $t_0>0$ and $( x_0(t),y_0(t))\in\Omega$, depending continuously on $t$, such that $  \hat d((\tilde\omega(t),\tilde \phi(t)),(0,\phi_{\ep_0}(x+x_0(t),y+y_0(t))))<\kappa<\delta_0(\ep_0)$ for $0\leq t< t_0$, and
\begin{align}\label{infimum-point-t0}
 \inf_{(x_0,y_0)\in\Omega}\hat d((\tilde\omega(t_0),\tilde \phi(t_0)),(0,\phi_{\ep_0}(x+x_0,y+y_0)))=\kappa.
\end{align}
 By \eqref{a prior estimate-mhd},
 there exist $(x_1(t),y_1(t))\in\Omega$ and $\ep_1(t)\in(a(\ep_0),b(\ep_0))$, depending continuously on $t$, such that
\begin{align}\label{d-t0-ep1-mhd}
\hat d((\tilde \omega(t),\tilde \phi(t)),( 0, \phi_{\ep_1(t)}(x+x_1(t),y+y_1(t))))<{\kappa^4\over16C_2(\ep_0)^4C_3(\ep_0)^4}<{\kappa\over2}
\end{align}
for $0\leq t\leq t_0$.
If we can prove that
$
\hat d((0,\phi_{\ep_1(t_0)}),(0,\phi_{\ep_0}))<{\kappa\over 2},
$
 then
$
\hat d((\tilde \omega(t_0),\tilde \phi(t_0)),( 0, \phi_{\ep_0}(x+x_1(t_0),y+y_1(t_0))))<\kappa,
$
which contradicts \eqref{infimum-point-t0}.

Now, we prove that $\hat d((0,\phi_{\ep_1(t_0)}),(0,\phi_{\ep_0}))<{\kappa\over 2}$.   By Lemma \ref{intOmega2}, \eqref{initial-x0x1-y0y1-ep0ep1} and taking $\delta(\ep_0,\kappa)>0$ smaller, it suffices to show that
\begin{align}\label{I-omegaep1t0-omegaep0-mhd}
\left|I\left(-e^{-2\phi_{\ep_1(t)}}\right)-I\left(-e^{-2\phi_{\ep_0}}\right)\right|<{\kappa\over C_2(\ep_0)}
\end{align}
for  some $C_2(\ep_0)>1$ large enough, where $0\leq t\leq t_0$ and $I(J)=\iint_{\Omega}(-J)^{3\over2}dxdy$. In fact,
\begin{align}\nonumber
&\hat d_3((\tilde \omega(t),\tilde \phi(t)),( 0, \phi_{\ep_1(t)}(x+x_1(t),y+y_1(t))))\\\nonumber
=&-\iint_{\Omega}\bigg(\hat h(\tilde \phi(t))-\hat h(\phi_{\ep_1(t)}(x+x_1(t),y+y_1(t)))\\\nonumber
&-\hat h'(\phi_{\ep_1(t)}(x+x_1(t),y+y_1(t)))
(\tilde \phi(t)-\phi_{\ep_1(t)}(x+x_1(t),y+y_1(t)))\bigg)dxdy\\\nonumber
=&\int_0^1\iint_\Omega2(1-r)e^{-2\phi^r(t)}\big(\tilde \phi(t)-\phi_{\ep_1(t)}(x+x_1(t),y+y_1(t))\big)^2dxdydr\\\nonumber
=&\int_0^1\iint_{\Omega}2(1-r)e^{-2\phi_{\ep_1(t)}}e^{-2r\phi^t}\left(\phi^t\right)^2dxdydr\\\nonumber
\geq&\int_0^1\iint_{\Omega}2(1-r)e^{-2\phi_{\ep_1(t)}}e^{-2\left|\phi^t\right|}\left(\phi^t\right)^2dxdydr\\\label{d1 estimates-mhd}
=&{1\over2}\iint_{\Omega}g'\left(\phi_{\ep_1(t)}\right)e^{-2\left|\phi^t\right|}\left(\phi^t\right)^2dxdy,
\end{align}
where $0\leq t\leq t_0$ and $\phi^r(t,x,y)=r\tilde \phi(t,x,y)+(1-r)\phi_{\ep_1(t)}(x+x_1(t),y+y_1(t))$ for $r\in[0,1]$.
Moreover, by Lemmas \ref{e-p-c-e-x}, \ref{poincare2ep}, \eqref{initial data-translation-infimum} and \eqref{d-t0-ep1-mhd} we have
\begin{align*}
&\iint_{\Omega}g'\left(\phi_{\ep_1(t)}\right)e^{7\left|\phi^t\right|}dxdy\\
\leq& e^{7\left|P_{\ep_1(t)}(\phi^t)\right|}\iint_{\Omega}g'\left(\phi_{\ep_1(t)}\right)e^{7\left|\phi^t-c_*(t)-P_{\ep_1(t)}(\phi^t-c_*(t))\right|}dxdy\\
\leq&Ce^{C|\iint_\Omega\hat h'(\phi_{\ep_1(t)})\phi^tdxdy|} e^{C\left(\|\phi^t\|_{\tilde{X}_\ep}+\|\phi^t\|_{\tilde{X}_\ep}^2\right)}\\
\leq &
Ce^{C\hat d_3((\tilde \omega(t),\tilde \phi(t)),( 0, \phi_{\ep_1(t)}(x+x_1(t),y+y_1(t))))+C\left|\iint_{\Omega}(e^{-2\tilde \phi(0)}-e^{-2\phi_{\ep_0}})dxdy\right|}\cdot\\
&e^{C\hat d_2((\tilde \omega(t),\tilde \phi(t)),( 0, \phi_{\ep_1(t)}(x+x_1(t),y+y_1(t))))^{1\over2}+C\hat d_2((\tilde \omega(t),\tilde \phi(t)),( 0, \phi_{\ep_1(t)}(x+x_1(t),y+y_1(t))))}\\
\leq &Ce^{C\kappa}e^{C\kappa^{1\over2}+C\kappa}\leq C,\\
&\iint_{\Omega}g'\left(\phi_{\ep_1(t)}\right)\left|\phi^t\right|^2dxdy\\
\leq&2\iint_{\Omega}g'\left(\phi_{\ep_1(t)}\right)\left|\phi^t-c_*(t)-P_{\ep_1(t)}(\phi^t-c_*(t))\right|^2dxdy+2\left|P_{\ep_1(t)}(\phi^t)\right|^2\iint_{\Omega}g'\left(\phi_{\ep_1(t)}\right)dxdy\\
\leq &\hat d_2((\tilde \omega(t),\tilde \phi(t)),( 0, \phi_{\ep_1(t)}(x+x_1(t),y+y_1(t))))\\
&+C\hat d_3((\tilde \omega(t),\tilde \phi(t)),( 0, \phi_{\ep_1(t)}(x+x_1(t),y+y_1(t))))^2+C\left|\iint_{\Omega}(e^{-2\tilde \phi(0)}-e^{-2\phi_{\ep_0}})dxdy\right|^2\leq C
\end{align*}
for $0\leq t\leq t_0$. Thus, by  \eqref{d-t0-ep1-mhd} and \eqref{d1 estimates-mhd}  we have
\begin{align}\nonumber
&\left|I\left(-e^{-2\tilde \phi(t)}\right)-I\left(-e^{-2\phi_{\ep_1(t)}}\right)\right|=\left|I\left(-e^{-2\tilde \phi(t)}\right)-I\left(-e^{-2\phi_{\ep_1(t)}(x+x_1(t),y+y_1(t))}\right)\right|\\\nonumber
=&\left|\iint_\Omega\left(e^{-3\tilde \phi(t)}-e^{-3\phi_{\ep_1(t)}(x+x_1(t),y+y_1(t))}\right)dxdy\right|\\\nonumber
=&3\bigg|\int_0^1\iint_\Omega e^{-3\phi^r(t)}\left(\tilde\phi(t)- \phi_{\ep_1(t)}(x+x_1(t),y+y_1(t)) \right)dxdydr\bigg|\\\nonumber
=&3\bigg|\int_0^1\iint_\Omega e^{-3\phi_{\ep_1(t)}}
e^{-3r\phi^t}\phi^t dxdydr\bigg|\\\nonumber
\leq&3\iint_\Omega e^{-3\phi_{\ep_1(t)}}
e^{3\left|\phi^t\right|}
\left|\phi^t \right|dxdy\\\nonumber
\leq&{3\over2}\left\|e^{-\phi_{\ep_1(t)}}\right\|_{L^\infty(\Omega)}
\iint_\Omega\left(\sqrt{2}e^{-\phi_{\ep_1(t)}}e^{{7\over2}\left|\phi^t\right|}\right)
\left(2^{1\over4}e^{-{1\over2}\phi_{\ep_1(t)}}e^{-{1\over2}\left|\phi^t\right|}\left|\phi^t\right|^{1\over2}\right)\\\nonumber
&\left(2^{1\over4}e^{-{1\over2}\phi_{\ep_1(t)}}\left|\phi^t\right|^{1\over2}\right)dxdy\\\nonumber
\leq &{3\over2}\left({1+b(\ep_0)\over1-b(\ep_0)}\right)^{1\over2}\left(\iint_{\Omega}g'\left(\phi_{\ep_1(t)}\right)e^{7\left|\phi^t\right|}dxdy\right)^{1\over2}
\left(\iint_{\Omega}g'\left(\phi_{\ep_1(t)}\right)e^{-2\left|\phi^t\right|}\left|\phi^t\right|^2dxdy\right)^{1\over4}\\\nonumber
&\left(\iint_{\Omega}g'\left(\phi_{\ep_1(t)}\right)\left|\phi^t\right|^2dxdy\right)^{1\over4}\\\nonumber
\leq&C_3(\ep_0)\hat d_3((\tilde \omega(t),\tilde \phi(t)),( 0, \phi_{\ep_1(t)}(x+x_1(t),y+y_1(t))))^{1\over4}\\\label{Iomega0Iomegat1-mhd}
< &{\kappa\over2C_2(\ep_0)},
\end{align}
where $0\leq t\leq t_0$ and we used $\left\|e^{-\phi_{\ep_1(t)}}\right\|_{L^\infty(\Omega)}\leq  \left({1+\ep_1(t)\over 1-\ep_1(t)}\right)^{1\over2}\leq \left({1+b(\ep_0)\over 1-b(\ep_0)}\right)^{1\over2}$.
Similar to \eqref{d1 estimates-mhd}-\eqref{Iomega0Iomegat1-mhd} and by the fact that $\hat d((\tilde\omega(0),\tilde \phi(0)),(0,\phi_{\ep_0}(x+x_1(0),y+y_1(0))))<{\kappa^4\over16C_1C_2(\ep_0)^4C_3(\ep_0)^4}$, we have
\begin{align}\nonumber
&\left|I\left(-e^{-2\tilde \phi(0)}\right)-I\left(-e^{-2\phi_{\ep_0}}\right)\right|=\left|I\left(-e^{-2\tilde \phi(0)}\right)-I\left(-e^{-2\phi_{\ep_0}(x+x_1(0),y+y_1(0))}\right)\right|\\\nonumber
\leq& C_3(\ep_0)\hat d_3((\tilde \omega(0),\tilde \phi(0)),( 0, \phi_{\ep_0}(x+x_1(0),y+y_1(0))))^{1\over4}\\\label{Itildeomega0Iomegaep0-mhd}
\leq& {\kappa\over2C_1^{1\over4}C_2(\ep_0)}< {\kappa\over2C_2(\ep_0)}.
\end{align}
By  \eqref{Iomega0Iomegat1-mhd}-\eqref{Itildeomega0Iomegaep0-mhd} and assumption (iii), we obtain \eqref{I-omegaep1t0-omegaep0-mhd}.
\end{proof}
\begin{appendix}

{\makeatletter
\renewcommand{\@seccntformat}[1]{} 
\makeatother
\section*{Appendix: Existence of weak solutions to 2D Euler equation with non-vanishing velocity at infinity}}

In the Appendix, we prove the existence of weak solutions to the 2D Euler equation with initial vorticity in \(Y_{non}\), defined in \eqref{def-X-non-ep}. Our approach is inspired by the work of Majda \cite{DiPerna-Majda87,Majda-Bertozzi02} for the whole plane \(\mathbb R^2\). We begin by constructing an approximate solution sequence through mollification of the initial data. We then analyze the corresponding approximate initial data and establish several basic properties of the sequence that are used  in the nonlinear analysis of Section \ref{Sec-Nonlinear orbital stability for co-periodic perturbations}. In place of the radial-energy decomposition used in \(\mathbb R^2\), we introduce a shear-energy decomposition adapted to the strip \(\Omega=\mathbb T_{2\pi}\times\mathbb R\) in order to prove global existence of the approximate solutions. Finally, we prove the \(L^1_{{loc}}\cap L^2_{loc}\) convergence of the approximate solution sequence and pass to the limit in the approximation parameter to obtain a weak solution with the prescribed initial vorticity.
{\setcounter{section}{1}

\setcounter{equation}{0}
\makeatletter
\global\c@section=1          
\global\c@equation=0         
\global\c@Theorem=0          
\makeatother
\subsection{Properties of the approximate initial data}

The definitions of a  weak solution and an approximate solution sequence for the 2D Euler equation are given as follows.

\begin{definition}
[Weak solution] A velocity field $\vec{u}(t,x,y)$ with initial data $\vec{u}_0$ is a weak solution of the 2D Euler equation if

$(\rm{i})$ $\vec{u}\in L^1(\Omega_{R,T})$ for any $T, R>0$,

$(\rm{ii})$ $u_iu_j\in L^1(\Omega_{R,T})$ for $i,j=1,2$,

$(\rm{iii})$ ${\mathrm{div}}(\vec{u})=0$ in the sense of distributions, i.e. $\iint_{\Omega}\nabla\varphi\cdot\vec{u} dxdy=0$ for any $\varphi\in C([0,T],C_0^1(\Omega))$,

$(\rm{iv})$ for any $\vec{\Phi}=(\Phi_1,\Phi_2)\in C^1([0,T], C_0^1(\Omega))$ with ${\mathrm{div}}(\vec{\Phi})=0$ in the sense of distributions,
\begin{align*}
\iint_{\Omega} (\vec{\Phi}\cdot\vec{u})(t,x,y)|_{t=0}^T dxdy=\int_0^T\iint_{\Omega}\left(\partial_t\vec{\Phi}\cdot \vec{u}+(\vec{u}\cdot\nabla)\vec{\Phi}\cdot\vec{u}\right)dxdydt,
\end{align*}
where $\Omega_{R,T}=[0,T]\times B_R$ and $B_R=\{x\in\mathbb{T}_{2\pi},y\in[-R,R]\}$.
\end{definition}
\begin{definition} [Approximate solution sequence for the 2D Euler equation]
\label{Approximate solution sequence for the 2D Euler equation} A sequence $\{\vec{u}^\mu\}$  is an approximate solution sequence for the 2D Euler equation if

$(\rm{i})$  $\vec{u}^\mu\in C([0,T],L_{\text{loc}}^2(\Omega))$, and   $\max_{0\leq t\leq T}\iint_{B_R}|\vec{u}^\mu(t,x,y)|^2dxdy\leq C(T, R) $ independent of $\mu$ for any $T, R>0$,

$(\rm{ii})$ ${\mathrm{div}}(\vec{u}^\mu)=0$ in the sense of distributions,

$(\rm{iii})$
$\lim_{\mu\to0}\int_0^T\iint_\Omega\left(\partial_t\vec{\Phi}\cdot \vec{u}^\mu+(\vec{u}^\mu\cdot\nabla)\vec{\Phi}\cdot\vec{u}^\mu\right)dxdydt=0$ for any $\vec{\Phi}\in C_0^\infty([0,T]\times\Omega)$ with ${\mathrm{div}}(\vec{\Phi})=0$.
\\
The approximate solution sequence  $\{\vec{u}^\mu\}$ is said to have $L^1$ vorticity control if, in addition,

 $(\rm{iv})$ $\max_{0\leq t\leq T}\iint_{\Omega}|\omega^\mu(t,x,y)|dxdy<C(T)$ for any $T>0$, where $\omega^\mu=\curl(\vec{u}^\mu)$.\\
The approximate solution sequence  $\{\vec{u}^\mu\}$ with $L^1$ vorticity control  is said to have $L^q$ vorticity control ($q>1$) if, in addition,

$(\rm{v})$ $\max_{0\leq t\leq T}\iint_{\Omega}|\omega^\mu(t,x,y)|^qdxdy<C(T)$ for any $T>0$.
\end{definition}

\begin{remark}
An approximate solution sequence  $\{\vec{u}^\mu\}$  for the 2D Euler equation satisfies
\begin{align*}
\|\varphi\vec{u}^\mu(t_1)-\varphi\vec{u}^\mu(t_2)\|_{H_{\text{loc}}^{-L}(\Omega)}\leq C|t_1-t_2|
\end{align*}
for $0\leq t_1,t_2\leq T$, $L>0$ and $\varphi\in C_0^\infty(\Omega)$, i.e. $\{\varphi\vec{u}^\mu\}$ is uniformly bounded in $Lip([0,T],$ $H_{\text{loc}}^{-L}(\Omega))$.
\end{remark}

To construct an approximate solution sequence $\{\vec{v}^{\mu}\}$ for the 2D Euler equation, we decompose the initial vorticity $\tilde\omega_0\in  Y_{non}$ into the shear part and the non-shear part:
\begin{align}\label{shear-energy decomposition}\tilde\omega_0(x,y)=\tilde\omega_{0,0}(y)+ \tilde\omega_{0,\neq0}(x,y),
 \end{align}
 where $\tilde\omega_{0,\neq0}(x,y)=\sum_{j\neq0}e^{ijx}\tilde\omega_{0,j}(y)$. Then $\iint_{\Omega}\tilde\omega_0 dxdy=2\pi\int_{-\infty}^\infty\tilde\omega_{0,0}dy =-4\pi$ and $\iint_{\Omega}\tilde\omega_{0,\neq0}$ $dxdy=0$.
By \eqref{green function}, we have $\tilde\psi_{0,\neq0}=G\ast \tilde\omega_{0,\neq0}$ solves $-\Delta\phi=\tilde\omega_{0,\neq0}$, and the non-shear initial velocity is defined by
 $\vec{v}_{0,\neq0}=\nabla^{\bot}\tilde \psi_{0,\neq0}=K\ast \tilde\omega_{0,\neq0}$, where \begin{align*}
K=\nabla^{\bot}G={1\over4\pi}\left({-\sinh(y)\over \cosh(y)-\cos(x)},{\sin(x)\over \cosh(y)-\cos(x)}\right).
\end{align*}
Since $\cosh(y)=1+{y^2\over2}+o(y^2)$ and $\cos(x)=1-{x^2\over2}+o(x^2)$, we have
\begin{align}\label{K-near-0}
|K(x,y)|\sqrt{x^2+y^2}={1\over4\pi}\sqrt{\cosh(y)+\cos(x)\over \cosh(y)-\cos(x)}\sqrt{x^2+y^2}\to{1\over2\pi}
\end{align}
as $(x,y)\to(0,0)$. On the other hand,
 \begin{align}\label{K-near-pm-infty}
 K(x,y)\to\left(\mp{1\over 4\pi},0\right) \text{ with exponential rate}
 \end{align}
 as $y\to\pm\infty$ uniformly for $x\in \mathbb{T}_{2\pi}$.

Note that \eqref{shear-energy decomposition}  gives a shear-energy decomposition in the sense that $\vec{v}_{0,\neq0}=K*\tilde\omega_{0,\neq0}\in L^2(\Omega)$. In fact,
let
 \begin{align}\nonumber
 \rho\in C_0^\infty(\mathbb{R})& \text{ with }\rho(y)=1 \text{ for } |y|\leq 1, \rho(y)=0 \text{ for } |y|>2,\\\nonumber
 \rho_s(x,y)&=\rho\left({y\over s}\right) \text{ for }(x,y)\in \Omega \text{ and }s>0, \\\nonumber
(1-\rho_{s})_{>0}&\equiv(1-\rho_{s}) \text{ for } y>0 \text{ and } (1-\rho_{s})_{>0}\equiv0 \text{ for } y\leq0,\\\label{def-cut off functions}
 (1-\rho_{s})_{<0}&\equiv(1-\rho_{s}) \text{ for } y<0 \text{ and } (1-\rho_{s})_{<0}\equiv0 \text{ for } y\geq0.
 \end{align}
  By Young's inequality, we have
\begin{align*}
\|\vec{v}_{0,\neq0}\|_{L^2(\Omega)}\leq& \|(\rho_1K)\ast\tilde\omega_{0,\neq0}\|_{L^2(\Omega)}+\left\|\left((1-\rho_1)_{>0}\left(K+\left({1\over 4\pi},0\right)\right)\right)\ast\tilde\omega_{0,\neq0}\right\|_{L^2(\Omega)}\\
&+\left\|\left((1-\rho_1)_{<0}\left(K-\left({1\over 4\pi},0\right)\right)\right)\ast\tilde\omega_{0,\neq0}\right\|_{L^2(\Omega)}\\
\leq&\bigg( \|\rho_1K\|_{L^1(\Omega)}+\left\|(1-\rho_1)_{>0}\left(K+\left({1\over 4\pi},0\right)\right)\right\|_{L^1(\Omega)}\\
&+\left\|(1-\rho_1)_{<0}\left(K-\left({1\over 4\pi},0\right)\right)\right\|_{L^1(\Omega)}\bigg)\|\tilde\omega_{0,\neq0}\|_{L^2(\Omega)}\leq C\|\tilde\omega_{0}\|_{L^2(\Omega)},
\end{align*}
 where we used \eqref{K-near-pm-infty}, $(1-\rho_1)_{>0}\ast\tilde\omega_{0,\neq0}=0$ and $(1-\rho_1)_{<0}\ast\tilde\omega_{0,\neq0}=0$.

For $\tilde\omega_0\in  Y_{non}$ and $\mu>0$,
we extend $\tilde\omega_0$ from $\Omega$ to $\mathbb{R}^2$ by setting $\tilde\omega_0(x,y)=\tilde \omega_0(x-2k\pi,y)$ for $(x,y)\in[2k\pi,(2k+2)\pi)\times \mathbb{R}$, where $k\in\mathbb{Z}$ and $k\neq0$.
Then we
 define the initial data of the approximate solution sequence by
\begin{align}\label{tilde-omega0-kappa-def}
\tilde\omega_{0}^{\mu}(x,y)=(\hat J_{\mu}\star\tilde\omega_{0})(x,y)
\end{align}
 for $(x,y)\in\Omega$ and $\mu\in(0,1)$,
 where
 \begin{align}\label{convolution-R2-def}
 (\hat J_{\mu}\star\tilde\omega_{0})(x,y)\triangleq\iint_{\mathbb{R}^2}\hat J_{\mu}(x-\tilde x,y-\tilde y)\tilde\omega_{0}(\tilde x,\tilde y)d\tilde x d\tilde y,
 \end{align}
  $\hat J_{\mu}(x,y)=\mu^{-2}\hat J\left({x\over \mu},{y\over \mu}\right)$, $\hat J\in C_0^\infty(\mathbb{R}^2)$ satisfies that $\hat J\geq0$, $\hat J(x,y)=0$ if $x^2+y^2\geq1$ and $\iint_{\mathbb{R}^2}\hat J(x,y) dxdy=1$.
 Here, we use the notation $\star$ to avoid the confusion with the usual convolution  $*$.
Note that $\hat J_{\mu}(x,y)=0$ if $\sqrt{x^2+y^2}\geq\mu$ and $\iint_{\mathbb{R}^2}\hat J_{\mu}(x,y)dxdy=1$. Moreover, $\hat J_{\mu}\star\varpi\in C^{\infty}(\mathbb{R}^2)$ if $\varpi\in L_{loc}^1(\Omega)$.
To study  the inheritance and convergence  of the  approximate  initial data $\tilde\omega_{0}^{\mu}$,
we give some basic properties of $\hat J_{\mu}\star\varpi$, which are elementary to the proof of  Theorem \ref{main result4-nonlinear orbital stability}.

\begin{lemma} \label{tilde-omega0-kappa-properties}
Let $\mu>0$ and $\varpi\in L_{loc}^1(\Omega)$.

$(1)$ $\hat J_{\mu}\star\varpi$ is $2\pi$-periodic in $x$.

$(2)$ If $\varpi<0$ on $\Omega$, then $\hat J_{\mu}\star\varpi<0$ on $\Omega$.

$(3)$ If $\iint_{\Omega}\varpi dxdy=c$, then $\iint_{\Omega}\hat J_{\mu}\star\varpi dxdy=c$.

$(4)$ If $\varpi\in L^p(\Omega)$ for $1\leq p<\infty$, then $\hat J_\mu\star\varpi\in L^p(\Omega)$, $\|\hat J_\mu\star\varpi\|_{L^p(\Omega)}\leq \|\varpi\|_{L^p(\Omega)}$ and $\hat J_\mu\star\varpi \to\varpi$ in ${L^p(\Omega)}$.

$(5)$ If $\varpi\in L^2(\Omega)$, then $\|\hat J_{\mu}\star\varpi\|_{H^q(\Omega)}\leq C(\mu,q)\|\varpi\|_{L^2(\Omega)}$ and $\|D^q\hat J_{\mu}\star\varpi\|_{L^\infty(\Omega)}=\|\hat J_{\mu}\star D^q\varpi\|_{L^\infty(\Omega)}\leq C(\mu,q)\|\varpi\|_{L^2(\Omega)}$ for $q\in\mathbb{Z}^+\cup\{0\}$.

$(6)$ If $\varpi, y\varpi\in L^1(\Omega)$, then $y(\hat J_{\mu}\star\varpi)\in L^1(\Omega)$ and $y(\hat J_{\mu}\star\varpi)\to y\varpi$ in $L^1(\Omega)$.

$(7)$ If $\varpi, y\varpi\in L^1(\Omega)$, then $\psi_\ep\varpi, \psi_\ep(\hat J_{\mu}\star\varpi)\in L^1(\Omega)$ and $\psi_\ep(\hat J_{\mu}\star\varpi)\to \psi_\ep\varpi$ in $L^1(\Omega)$ for $\ep\in[0,1)$.

$(8)$ If $\varpi\in Y_{non}$, then $\hat J_{\mu}\star\varpi\in Y_{non}$, $-\varpi\ln(-\varpi),-(\hat J_{\mu}\star\varpi)\ln(-(\hat J_{\mu}\star\varpi))\in L^1(\Omega)$ and
\begin{align}\label{xlnx-convergence}
-(\hat J_{\mu}\star\varpi)\ln(-(\hat J_{\mu}\star\varpi))\to-\varpi\ln(-\varpi)\quad \text{in} \quad L^1(\Omega),
\end{align}
where $Y_{non}$ is defined in \eqref{def-X-non-ep}.
\end{lemma}
\begin{proof}
We extend $\varpi$ from $\Omega$ to $\mathbb{R}^2$ as above.
Since
\begin{align*}
(\hat J_{\mu}\star\varpi)(x,y)=&\iint_{\mathbb{R}^2}\hat J_\mu(\tilde x,\tilde y)\varpi(x-\tilde x,y-\tilde y)d\tilde x d\tilde y=\iint_{\mathbb{R}^2}\hat J_\mu(\tilde x,\tilde y)\varpi(x+2\pi-\tilde x,y-\tilde y)d\tilde x d\tilde y\\
=&\hat J_{\mu}\star\varpi(x+2\pi,y)
\end{align*}
for $(x,y)\in\mathbb{R}^2$, (1) holds true. (2) is trivially verified.

 (3) follows from
 \begin{align*}
 \iint_\Omega\hat J_{\mu}\star\varpi dxdy=\iint_{\mathbb{R}^2}\hat J_\mu(\tilde x,\tilde y)\left(\iint_\Omega\varpi(x-\tilde x,y-\tilde y) dxdy\right)d\tilde x d\tilde y=c\iint_{\mathbb{R}^2}\hat J_\mu(\tilde x,\tilde y)d\tilde x d\tilde y=c.
 \end{align*}

Next, we prove (4). For $1<p<\infty$,
\begin{align}\nonumber
|(\hat J_{\mu}\star\varpi)(x,y)|\leq& \left(\iint_{\mathbb{R}^2}\hat  J_{\mu}(\tilde x,\tilde y)d\tilde x d\tilde y\right)^{1\over p'}\left(\iint_{\mathbb{R}^2}\hat J_{\mu}(\tilde x,\tilde y)|\varpi(x-\tilde x,x-\tilde y)|^p d\tilde xd\tilde y\right)^{1\over p}\\\label{hat-J-kappa-star-omega-estimate}
=&\left(\iint_{\mathbb{R}^2}\hat J_{\mu}(\tilde x,\tilde y)|\varpi(x-\tilde x,x-\tilde y)|^p d\tilde xd\tilde y\right)^{1\over p},
\end{align}
where $p'={p\over p-1}$. Then
\begin{align}\nonumber
\|\hat J_{\mu}\star\varpi\|_{L^p(\Omega)}^p\leq& \iint_\Omega\iint_{\mathbb{R}^2} \hat  J_{\mu}(\tilde x,\tilde y)|\varpi(x-\tilde x,y-\tilde y)|^p d\tilde xd\tilde y dxdy\\\label{J-omega-Lp}
=&\iint_{\mathbb{R}^2} \hat J_{\mu}(\tilde x,\tilde y)d\tilde xd\tilde y \iint_\Omega|\varpi(x-\tilde x,y-\tilde y)|^p  dxdy=\|\varpi\|_{L^p(\Omega)}^p.
\end{align}
For $p=1$, $\eqref{J-omega-Lp}$ follows directly from the definition of $\hat J_{\mu}\star\varpi$.
Let $\delta>0$ and $1\leq p<\infty$.
Choose $\varpi_1\in C_0(\Omega)$ such that $\|\varpi-\varpi_1\|_{L^p(\Omega)}<{\delta\over 3}$. By \eqref{J-omega-Lp}, we have $\|\hat J_{\mu}\star\varpi-\hat J_{\mu}\star\varpi_1\|_{L^p(\Omega)}<{\delta\over 3}$. Since $|\hat J_{\mu}\star\varpi_1(x,y)-\varpi_1(x,y)|\leq \sup_{\sqrt{(x-\tilde x)^2+(y-\tilde y)^2}\leq \mu}|\varpi_1(\tilde x,\tilde y)-\varpi_1(x,y)|$, $\varpi_1$ is uniformly continuous on $\Omega$ and $\text{supp}(\varpi_1)$ is compact, we have $\|\hat J_{\mu}\star\varpi_1-\varpi_1\|_{L^p(\Omega)}\leq {\delta\over 3}$ for $\mu$ sufficiently small. Thus, $\|\hat J_{\mu}\star\varpi-\varpi\|_{L^p(\Omega)}\leq \delta$.

To  prove (5), we
 denote $D^j\hat J=\hat J^j$ for  $0\leq j\leq q$. Since
\begin{align*}
(D^j\hat J_{\mu}\star\varpi)(x,y)=\mu^{-j-2}\iint_{\mathbb{R}^2}\hat J^j\left({x-\tilde x\over \mu},{y-\tilde y\over \mu}\right)\varpi(\tilde x,\tilde y)d\tilde x d\tilde y,
\end{align*}
we have
\begin{align}\nonumber
|(D^j\hat J_{\mu}\star\varpi)(x,y)|^2
\leq &\mu^{-2j}\left(\mu^{-2}\iint_{\mathbb{R}^2}\hat J^j\left({x-\tilde x\over \mu},{y-\tilde y\over \mu}\right)d\tilde x d\tilde y\right)\\\nonumber
&\left(\mu^{-2}\iint_{\mathbb{R}^2}\hat J^j\left({x-\tilde x\over \mu},{y-\tilde y\over \mu}\right)\varpi(\tilde x,\tilde y)^2d\tilde xd\tilde y\right)\\\label{DjhatJkappaomegaxy}
\leq &{C_j\over \mu^{2j}}\mu^{-2}\iint_{\mathbb{R}^2}\hat J^j\left({x-\tilde x\over \mu},{y-\tilde y\over \mu}\right)\varpi(\tilde x,\tilde y)^2d\tilde xd\tilde y.
\end{align}
Thus,
\begin{align*}
\sum_{0\leq j\leq q}\|D^j\hat J_{\mu}\star\varpi\|_{L^2(\Omega)}^2\leq &\sum_{0\leq j\leq q}{C_j\over \mu^{2j}}\mu^{-2}\iint_{\mathbb{R}^2}\hat J^j\left({\tilde x\over \mu},{\tilde y\over \mu}\right)\left(\iint_\Omega\varpi(x-\tilde x,y-\tilde y)^2dxdy\right)d\tilde x d\tilde y\\
\leq &\sum_{0\leq j\leq q}{C_j\over \mu^{2j}}\|\varpi\|_{L^2(\Omega)}^2\leq C(\mu,q)\|\varpi\|_{L^2(\Omega)}^2.
\end{align*}
Since $\hat J^{q}\left({x-\tilde x\over \mu},{y-\tilde y\over \mu}\right)=0$ for $\sqrt{(x-\tilde x)^2+(y-\tilde y)^2}\geq\mu$ and $\hat J^{q}\in C_0^\infty(\mathbb{R}^2)$, by \eqref{DjhatJkappaomegaxy} for $j=q$ we have
$|(D^q\hat J_{\mu}\star\varpi)(x,y)|
\leq C(\mu,q)\|\varpi\|_{L^2(\Omega)}$ for any $(x,y)\in\Omega$ and $\mu>0$ sufficiently small.

Then we prove (6). Noting that
\begin{align}\nonumber
&\|y(\hat J_{\mu}\star\varpi)\|_{L^1(\Omega)}\leq\iint_{\mathbb{R}^2}\hat J_\mu(\tilde x,\tilde y)\iint_{\Omega}|y\varpi(x-\tilde x,y-\tilde y)|dxdyd\tilde x d\tilde y\\\nonumber
\leq&\iint_{\mathbb{R}^2}\hat J_\mu(\tilde x,\tilde y)\iint_{\Omega}(|y-\tilde y|+|\tilde y|) |\varpi(x-\tilde x,y-\tilde y)|dxdyd\tilde x d\tilde y\\\nonumber
\leq&\|y\varpi\|_{L^1(\Omega)}+\|\varpi\|_{L^1(\Omega)}\iint_{\mathbb{R}^2}\hat J_\mu(\tilde x,\tilde y)|\tilde y|d\tilde x d\tilde y,
\end{align}
we have $y(\hat J_{\mu}\star\varpi)\in L^1(\Omega)$. To prove that $y(\hat J_{\mu}\star\varpi)\to y\varpi$ in $L^1(\Omega)$, it suffices to show that
$\|y(\hat J_{\mu}\star\varpi)-\hat J_{\mu}\star(y\varpi)\|_{L^1(\Omega)}\to0$ by (4). In fact,
\begin{align*}
&\|y(\hat J_{\mu}\star\varpi)-\hat J_{\mu}\star(y\varpi)\|_{L^1(\Omega)}\leq\iint_{\mathbb{R}^2}\hat J_\mu(\tilde x,\tilde y)|\tilde y|\iint_{\Omega} |\varpi(x-\tilde x,y-\tilde y)|dxdyd\tilde xd\tilde y\\
=&\|\varpi\|_{L^1(\Omega)}\iint_{x^2+y^2\leq1}\hat J(x,y)\mu|y|dxdy\to0.
\end{align*}

Now, we prove (7). Direct computation gives
\begin{align}\nonumber
&\|\psi_\ep\varpi\|_{L^1(\Omega)}=\|(G*\omega_\ep) \varpi\|_{L^1(\Omega)}+C\| \varpi\|_{L^1(\Omega)}\\\nonumber
\leq&\|G_1\|_{L^2(\Omega)}\|\omega_\ep\|_{L^2(\Omega)}\|\varpi\|_{L^1(\Omega)}
+C\|\omega_\ep\|_{L^1(\Omega)}\|y\varpi\|_{L^1(\Omega)}\\\label{psi-tilde-omegaL1}
&+C\|y\omega_\ep\|_{L^1(\Omega)}\|\varpi\|_{L^1(\Omega)}+C\| \varpi\|_{L^1(\Omega)}<\infty.
\end{align}
By (4) and (6), $\hat J_{\mu}\star\varpi,y(\hat J_{\mu}\star\varpi)\in L^1(\Omega)$, and thus, $\psi_\ep(\hat J_{\mu}\star\varpi)\in L^1(\Omega)$. It follows again from (4) and (6) that $\hat J_{\mu}\star\varpi\to\varpi$ and $y(\hat J_{\mu}\star\varpi)\to y\varpi$ in $L^1(\Omega)$. Then
\begin{align*}
&\|\psi_\ep(\hat J_{\mu}\star\varpi-\varpi)\|_{L^1(\Omega)}\\
\leq&\|G_1\|_{L^2(\Omega)}\|\omega_\ep\|_{L^2(\Omega)}\|\hat J_{\mu}\star\varpi-\varpi\|_{L^1(\Omega)}+C\|\omega_\ep\|_{L^1(\Omega)}
\|y (\hat J_{\mu}\star\varpi-\varpi)\|_{L^1(\Omega)}\\
&+C\|y\omega_\ep\|_{L^1(\Omega)}\|\hat J_{\mu}\star\varpi-\varpi\|_{L^1(\Omega)}+C\|\hat J_{\mu}\star\varpi-\varpi\|_{L^1(\Omega)}\to 0.
\end{align*}

Finally, we prove (8).
If $-\varpi\geq1$, then $0\leq-\varpi\ln(-\varpi)\leq \varpi^2$ since $0\leq\ln (s)\leq s$ for $s\geq1$. If $0<-\varpi<1$, then $0\leq \int_0^1{(1-r)(\varpi-\omega_\ep)^2\over -2\varpi^r}dr={1\over2}\varpi-{1\over2}\varpi\ln(-\varpi)-{1\over2}\omega_\ep-\psi_\ep\varpi$, and thus, $0<\varpi\ln(-\varpi)\leq \varpi-\omega_\ep-2\psi_\ep\varpi$, where $\varpi^r=r\varpi+(1-r)\omega_\ep$.  This implies
\begin{align}\label{omega-ln-omega}|\varpi\ln(-\varpi)|\leq \varpi^2+|\varpi|+|\omega_\ep|+2|\psi_\ep\varpi|
\end{align}
for all $(x,y)\in\Omega$. By \eqref{psi-tilde-omegaL1}, we have $\psi_\ep\varpi\in L^1(\Omega)$. This, along with $\varpi\in L^1\cap L^2(\Omega)$, yields $\varpi\ln(-\varpi)\in L^1(\Omega)$.
Since $\varpi\in Y_{non}$, by  (1)-(4) and (6) we have $\hat J_{\mu}\star\varpi\in Y_{non}$. Thus,
$-(\hat J_{\mu}\star\varpi)\ln(-(\hat J_{\mu}\star\varpi))\in L^1(\Omega)$. Similar to \eqref{omega-ln-omega}, we have $|(\hat J_{\mu}\star\varpi)\ln(-(\hat J_{\mu}\star\varpi))|\leq (\hat J_{\mu}\star\varpi)^2+|(\hat J_{\mu}\star\varpi)|+|\omega_\ep|+2|\psi_\ep(\hat J_{\mu}\star\varpi)|$ for all $(x,y)\in\Omega$.
Let $B_R^c=\Omega\setminus B_R$. Then
\begin{align}\nonumber
&\iint_{B_R^c}|(-\hat J_{\mu}\star\varpi)\ln(-(\hat J_{\mu}\star\varpi))-(-\varpi)\ln(-\varpi)|dxdy\\\nonumber
\leq&
\iint_{B_R^c}\bigg((\hat J_{\mu}\star\varpi)^2+|\hat J_{\mu}\star\varpi|+|\omega_\ep|+2|\psi_\ep(\hat J_{\mu}\star\varpi)|\\\label{BRc1}
&+\varpi^2+|\varpi|+|\omega_\ep|+2|\psi_\ep\varpi|\bigg)dxdy
\end{align}
for $R>1$. By \eqref{hat-J-kappa-star-omega-estimate}, we have
\begin{align}\nonumber
\iint_{B_R^c}(\hat J_{\mu}\star\varpi)^2dxdy\leq &\iint_{\tilde x^2+\tilde y^2\leq \mu^2}\hat J_{\mu}(\tilde x,\tilde y)\iint_{B_R^c}|\varpi(x-\tilde x,y-\tilde y)|^2dxdyd\tilde x d\tilde y\\\nonumber
=&\iint_{\tilde x^2+\tilde y^2\leq \mu^2}\hat J_{\mu}(\tilde x,\tilde y)\iint_{B_R^c-(\tilde x,\tilde y)}|\varpi(\hat x,\hat y)|^2d\hat xd\hat yd\tilde x d\tilde y\\\label{BRc2}
\leq& \iint_{B_{R-1}^c}|\varpi(\hat x,\hat y)|^2d\hat xd\hat y=\|\varpi\|_{L^2(B_{R-1}^c)}^2,
\end{align}
 for $\mu\in(0,1)$ and $R>1$,
where $ B_R^c-(\tilde x,\tilde y)=\{(\hat x,\hat y)|\hat x=x-\tilde x,\hat y=y-\tilde y, (x,y)\in B_R^c\}$ and in the last inequality, we used $B_R^c-(\tilde x,\tilde y)\subset B_{R-1}^c$ since $\tilde y\in[-\mu,\mu]\subset (-1,1)$.
Similarly, we have
\begin{align}\label{BRc3}
\iint_{B_R^c}|\hat J_{\mu}\star\varpi|dxdy\leq \|\varpi\|_{L^1(B_{R-1}^c)}
\end{align}  for $\mu\in(0,1)$ and $R>1$.
Noting that
\begin{align*}
\|y(\hat J_{\mu}\star\varpi)\|_{L^1(B_R^c)}\leq &\iint_{\tilde x^2+\tilde y^2\leq \mu^2}\hat J_{\mu}(\tilde x,\tilde y)\iint_{B_R^c}\left(|y-\tilde y|+|\tilde y|\right)|\varpi(x-\tilde x,y-\tilde y)|dxdyd\tilde xd\tilde y\\
= &\iint_{\tilde x^2+\tilde y^2\leq \mu^2}\hat J_{\mu}(\tilde x,\tilde y)\iint_{B_R^c-(\tilde x,\tilde y)}\left(|\hat y\varpi(\hat x,\hat y)|+|\tilde y\varpi(\hat x,\hat y)|\right)d\hat xd\hat yd\tilde xd\tilde y\\
\leq& \|y\varpi\|_{L^1(B_{R-1}^c)}+C_0\|\varpi\|_{L^1(B_{R-1}^c)},
\end{align*}
we have
\begin{align}\nonumber
&\iint_{B_R^c}|\psi_\ep(\hat J_{\mu}\star\varpi)|dxdy\leq\iint_{B_R^c}\left(|((G_1+G_2)*\omega_\ep)(\hat J_{\mu}\star\varpi)|+C|\hat J_{\mu}\star\varpi|\right)dxdy\\\nonumber
\leq&\|G_1\|_{L^2(\Omega)} \|\omega_\ep\|_{L^2(\Omega)} \iint_{B_R^c}|\hat J_{\mu}\star\varpi|dxdy\\\nonumber
&+C\iint_{B_R^c}\left(\iint_\Omega|y-\tilde y||\varpi(\tilde x,\tilde y)|d\tilde x d\tilde y\right)|(\hat J_{\mu}\star\varpi)(x,y)|dxdy+C\|\varpi\|_{L^1(B_{R-1}^c)}\\\nonumber
\leq &\|G_1\|_{L^2(\Omega)} \|\omega_\ep\|_{L^2(\Omega)}\|\varpi\|_{L^1(B_{R-1}^c)}+C(\|\varpi\|_{L^1(\Omega)}\|y(\hat J_{\mu}\star\varpi)\|_{L^1(B_R^c)}\\\nonumber
&+\|y\varpi\|_{L^1(\Omega)}\|\hat J_{\mu}\star\varpi\|_{L^1(B_R^c)})+C\|\varpi\|_{L^1(B_{R-1}^c)}\\\nonumber
\leq &\|G_1\|_{L^2(\Omega)} \|\omega_\ep\|_{L^2(\Omega)}\|\varpi\|_{L^1(B_{R-1}^c)}+C\|\varpi\|_{L^1(\Omega)}(\|y\varpi\|_{L^1(B_{R-1}^c)}+C_0\|\varpi\|_{L^1(B_{R-1}^c)})\\\label{BRc4}
&+C\|y\varpi\|_{L^1(\Omega)}\|\varpi\|_{L^1(B_{R-1}^c)}+C\|\varpi\|_{L^1(B_{R-1}^c)}
\end{align}
for $\mu\in(0,1)$ and $R>1$. Combining \eqref{BRc1}-\eqref{BRc4}, we have
\begin{align}\nonumber
&\iint_{B_R^c}|(-\hat J_{\mu}\star\varpi)\ln(-(\hat J_{\mu}\star\varpi))-(-\varpi)\ln(-\varpi)|dxdy\\\nonumber
\leq&\|\varpi\|_{L^2(B_{R-1}^c)}^2+\|\varpi\|_{L^1(B_{R-1}^c)}+2\|\omega_\ep\|_{L^1(B_R^c)}+2\|G_1\|_{L^2(\Omega)} \|\omega_\ep\|_{L^2(\Omega)}\|\varpi\|_{L^1(B_{R-1}^c)}\\\nonumber
&+2C\|\varpi\|_{L^1(\Omega)}(\|y\varpi\|_{L^1(B_{R-1}^c)}+C_0\|\varpi\|_{L^1(B_{R-1}^c)})
+2C\|y\varpi\|_{L^1(\Omega)}\|\varpi\|_{L^1(B_{R-1}^c)}\\\label{BRc-sum}
&+2C\|\varpi\|_{L^1(B_{R-1}^c)}+\|\varpi\|_{L^2(B_{R}^c)}^2+\|\varpi\|_{L^1(B_{R}^c)}+2\|\psi_\ep\varpi\|_{L^1(B_{R}^c)}
\end{align}
for $\mu\in(0,1)$ and $R>1$. Thus, for any $\varepsilon>0$, we can choose  $R_0>1$ (independent of $\mu$) such that
\begin{align}\label{R-varepsilon}
\iint_{B_{R_0}^c}|(-\hat J_{\mu}\star\varpi)\ln(-(\hat J_{\mu}\star\varpi))-(-\varpi)\ln(-\varpi)|dxdy<{\varepsilon\over 4}.
\end{align}
Let $\nu_0>0$ small enough such that $(8+2\|G_1\|_{L^2(\Omega)} \|\omega_\ep\|_{L^2(\Omega)}+2C\|\varpi\|_{L^1(\Omega)}(1+C_0)
+2C\|y\varpi\|_{L^1(\Omega)}+C)\nu_0<\varepsilon/4$. Then there exists $\delta_0>0$ (depending on $\varepsilon$) such that for any subset $E\subset \Omega$ satisfying $|E|\leq\delta_0$, we have
\begin{align}\label{absolute continuity of integral}
\max\{\|\varpi\|_{L^2(E)}^2,\|\varpi\|_{L^1(E)},\|\omega_\ep\|_{L^1(E)},\|y\varpi\|_{L^1(E)},\|\psi_\ep\varpi\|_{L^1(E)}\}\leq \nu_0.
\end{align}
By \eqref{absolute continuity of integral} and the fact that $|E-(\tilde x,\tilde y)|=|E|$ for any $(\tilde x,\tilde y)\in \mathbb{R}^2$, a similar argument to \eqref{BRc1}-\eqref{BRc-sum} implies that
\begin{align*}
\iint_{E}(\hat J_{\mu}\star\varpi)^2dxdy\leq& \nu_0,\quad \iint_{E}|\hat J_{\mu}\star\varpi|dxdy\leq \nu_0,\\
\iint_{E}|\psi_\ep(\hat J_{\mu}\star\varpi)|dxdy
\leq& \|G_1\|_{L^2(\Omega)} \|\omega_\ep\|_{L^2(\Omega)}\nu_0+C\|\varpi\|_{L^1(\Omega)}(\nu_0
+C_0\nu_0)\\
&+C\|y\varpi\|_{L^1(\Omega)}\nu_0+C\nu_0,
\end{align*} and
\begin{align}\nonumber
&\iint_{E}|(-\hat J_{\mu}\star\varpi)\ln(-(\hat J_{\mu}\star\varpi))-(-\varpi)\ln(-\varpi)|dxdy\\\nonumber
\leq&\nu_0+\nu_0+2\nu_0+2\|G_1\|_{L^2(\Omega)} \|\omega_\ep\|_{L^2(\Omega)}\nu_0\\\label{E-small}
&+2C\|\varpi\|_{L^1(\Omega)}(\nu_0+C_0\nu_0)
+2C\|y\varpi\|_{L^1(\Omega)}\nu_0
+C\nu_0+\nu_0+\nu_0+2\nu_0
\leq {\varepsilon\over 4}
\end{align}
for   $E\subset \Omega$ satisfying $|E|\leq\delta_0$.  By Lusin's Theorem, there exists a closed subset $F\subset B_{R_0}$ such that $|B_{R_0}\setminus F|<\delta_0$ and $\varpi$ is continuous on $F$. Thus, $0< \min_{(x,y)\in F}|\varpi(x,y)|\leq \max_{(x,y)\in F}|\varpi(x,y)|<\infty$. Let  $a_F\triangleq \max_{(x,y)\in F}|\varpi(x,y)|+1$.
Since $s\ln(s)$ is uniformly continuous on $[0,a_F]$, there exists $\delta_1\in(0,\min\{\min_{(x,y)\in F}|\varpi(x,y)|,1\})$ (depending on $\varepsilon, R_0, F$) such that
\begin{align}\label{uniform continuity}
|s_2\ln(s_2)-s_1\ln(s_1)|<{\varepsilon\over 16\pi R_0} \text{ for }s_1,s_2\in[0,a_F] \text{ and }|s_2-s_1|\leq\delta_1.
\end{align}
We divide $F$ into two parts
\begin{align*}
B_{1,\delta_1}^{\mu}=\{(x,y)\in  F|\;|(\hat J_\mu\star\varpi)(x,y)-\varpi(x,y)|\leq \delta_1\},\\
B_{2,\delta_1}^{\mu}=\{(x,y)\in  F|\;|(\hat J_\mu\star\varpi)(x,y)-\varpi(x,y)|> \delta_1\}.
\end{align*}
Since $(\hat J_\mu\star\varpi)\to\varpi$ in $L^1(\Omega)$, we have
\begin{align*}
|B_{2,\delta_1}^{\mu}|\delta_1\leq\|(\hat J_\mu\star\varpi)-\varpi\|_{L^1(B_{2,\delta_1}^{\mu})}\leq \|(\hat J_\mu\star\varpi)-\varpi\|_{L^1(\Omega)}\leq \delta_0\delta_1\Longrightarrow|B_{2,\delta_1}^{\mu}|\leq \delta_0
\end{align*}
for $\mu>0$ small enough. By \eqref{uniform continuity}, we have
\begin{align}\label{B1delta1kappa-small}
&\iint_{B_{1,\delta_1}^{\mu}}|(-\hat J_{\mu}\star\varpi)\ln(-(\hat J_{\mu}\star\varpi))-(-\varpi)\ln(-\varpi)|dxdy
\leq {\varepsilon\over 16\pi R_0}|B_{1,\delta_1}^{\mu}|\leq{\varepsilon\over 4}.
\end{align}
Since $|B_{R_0}\setminus F|<\delta_0$ and
$|B_{2,\delta_1}^{\mu}|\leq \delta_0$,  we infer from \eqref{E-small} that
\begin{align}\label{BR0setminusF-small}
&\iint_{B_{R_0}\setminus F}|(-\hat J_{\mu}\star\varpi)\ln(-(\hat J_{\mu}\star\varpi))-(-\varpi)\ln(-\varpi)|dxdy
\leq {\varepsilon\over 4},\\\label{B2delta1kappa-small}
&\iint_{B_{2,\delta_1}^{\mu}}|(-\hat J_{\mu}\star\varpi)\ln(-(\hat J_{\mu}\star\varpi))-(-\varpi)\ln(-\varpi)|dxdy
\leq {\varepsilon\over 4}
\end{align}
for $\mu>0$ small enough. The conclusion \eqref{xlnx-convergence} then follows from
\eqref{R-varepsilon} and \eqref{B1delta1kappa-small}-\eqref{B2delta1kappa-small}.
\end{proof}

\subsection{Global existence  of the approximate solutions}
Now, we prove the global existence  of the approximate solutions.
\begin{lemma}\label{lem-construction of an approximate solution sequence}
Let $\tilde\omega_0\in  Y_{non}$ and $\tilde\omega_0^{\mu}$ be defined in  \eqref{tilde-omega0-kappa-def} for $\mu\in(0,1)$.
For the initial data $\vec{v}_0^{\mu}=K\ast\tilde\omega_0^{\mu}$, there exists a global smooth solution $\vec{v}^{\mu}(t)=\vec{v}_{0,0}^{\mu}+\vec{v}_{\mu}(t)$ to the 2D Euler equation such that
\[
\vec v_{\mu}(t)\in H^q(\Omega),\qquad \vec v_{\mu}\in C^0([0,T];H^q(\Omega))
\]
for every $q\ge 3$ and $T>0$, where $\vec v_{0,0}^{\mu}=K\ast \tilde\omega_{0,0}^{\mu}$. Moreover, $\lim_{y\to\pm\infty}\vec v^{\mu}(t,x,y)=(\pm1,0)$
for all $t\ge 0$ and $x\in\mathbb T_{2\pi}$, and
\[
\lim_{y\to\pm\infty}\vec v^{\mu}(t,x,y)=(\pm1,0)
\]
for all $t\ge 0$ and $x\in\mathbb T_{2\pi}$. In addition, the family $\{\vec v^{\mu}\}$ provides an approximate solution sequence which has  $L^1$ and $L^2$ vorticity control, and
\[
\tilde\omega_0^{\mu}\to \tilde\omega_0
\qquad\text{in }L^1(\Omega)\cap L^2(\Omega).
\]
\end{lemma}

\begin{proof}
We decompose $\vec{v}_{0}^{\mu}$ into the shear-energy parts: $\vec{v}_{0}^{\mu}=K\ast\tilde\omega_0^{\mu}=K\ast\tilde\omega_{0,0}^{\mu}+K\ast\tilde\omega_{0,\neq0}^{\mu}\triangleq\vec{v}_{0,0}^{\mu}+\vec{v}_{0,\neq0}^{\mu}$. Then
by Lemma \ref{tilde-omega0-kappa-properties} (5), we have $\vec{v}_{0,\neq0}^{\mu}=K\ast(\hat J_{\mu}\star\tilde\omega_{0,\neq0})=\hat J_{\mu}\star\vec{v}_{0,\neq0}\in H^q(\Omega)$ for all $q\geq3$ since $\vec{v}_{0,\neq0}\in L^2(\Omega)$.
Now we denote $\vec{v}_{\mu}$ to be the solution of the evolution equation
\begin{align}\label{euler-nonshear}
 \partial_t\vec{u}+(\vec{u}\cdot\nabla)\vec{u}+(\vec{v}_{0,0}^{\mu}\cdot\nabla)\vec{u}+(\vec{u}\cdot\nabla)\vec{v}_{0,0}^{\mu}=-\nabla p
 \end{align}
with the initial data
$\vec{v}_{\mu}(0)=\vec{v}_{0,\neq0}^{\mu}$. As in Subsection 3.2.4 of \cite{Majda-Bertozzi02}, the solution $\vec{v}_{\mu}$ to equation \eqref{euler-nonshear} exists locally in time in $H^q(\Omega)$ for $q\geq3$ and can be continued in time provided that $\|\vec{v}_{\mu}(t)\|_{H^q(\Omega)}$ remains bounded.
We use the shear-energy decomposition to derive  the BKM-type estimate  \eqref{BKM} in the cylinder version,
which proves the global existence  of the  solution $\vec{v}_{\mu}$ to the 2D Euler equation in $H^q(\Omega)$ for $q\geq3$.
The BKM criterion was originally  obtained for the 3D Euler equation on  $\mathbb{R}^3$  in
\cite{BKM84} and extended to the $\mathbb{R}^2$ version using a radial-energy decomposition for the velocity field with infinite energy (see \cite{Majda-Bertozzi02} for example). We follow the argument of \cite{Majda-Bertozzi02} and \cite{Kato1986nonlinear}.
Note that  ${\mathrm{div}}(\vec{v}_{\mu}(t))={\mathrm{div}}(\vec{v}^{\mu}(t))-{\mathrm{div}}(\vec{v}_{0,0}^{\mu})=0$ for $t\geq0$ since
${v}_{0,0,2}^{\mu}=- G\ast\partial_x \tilde\omega_{0,0}^{\mu}=0$, where ${v}_{0,0,2}^{\mu}$ is the second entry of $\vec{v}_{0,0}^{\mu}$. Then a basic energy estimate gives
\begin{align*}
{1\over 2} {d\over dt}\|\vec{v}_{\mu}(t)\|_{L^2(\Omega)}^2+\iint_{\Omega}(\vec{v}_{\mu}(t)\cdot\nabla)\vec{v}_{0,0}^{\mu}\cdot\vec{v}_{\mu}(t)dxdy=0.
 \end{align*}
Indeed, we
can first prove it for the regularized solution and then take the limit by a similar approach in Theorem 3.6 of \cite{Majda-Bertozzi02}.
Then
\begin{align}\label{basic energy estimate-eq}
{d\over dt}\|\vec{v}_{\mu}(t)\|_{L^2(\Omega)}\leq \|\vec{v}_{\mu}(t)\|_{L^2(\Omega)}\|\nabla\vec{v}_{0,0}^{\mu}\|_{L^\infty(\Omega)}
 \end{align}
and Gr\"{o}nwall's inequality implies
\begin{align}\label{basic energy estimate-gronwall}
\|\vec{v}_{\mu}(t)\|_{L^2(\Omega)}\leq \|\vec{v}_{\mu}(0)\|_{L^2(\Omega)}e^{\int_0^t\|\nabla \vec{v}_{0,0}^{\mu}\|_{L^\infty(\Omega)}ds},
 \end{align}
 where $\nabla \vec{v}_{0,0}^{\mu}$ is in the form of  $2\times2$ matrix.

 We prove that $\vec{v}_{0,0}^{\mu}\in W^{j,\infty}(\Omega)$ for $j\geq0$. Since $\tilde \omega_0\in L^2(\Omega)$, we have $\tilde \omega_{0,0}\in L^2(\Omega)$ and $\|D^j\tilde \omega_{0,0}^{\mu}\|_{L^\infty(\Omega)}\leq C(\mu,j)\|\tilde \omega_{0,0}\|_{L^2(\Omega)}$ by Lemma \ref{tilde-omega0-kappa-properties} (5). Noting that $$\|D^j\tilde \omega_{0,0}^{\mu}\|_{L^1(\Omega)}=\iint_\Omega|(D^j\hat J_\mu)\star\tilde \omega_{0,0}|dxdy\leq \|D^j\hat J_\mu\|_{L^1(\mathbb{R}^2)}\|\tilde \omega_{0,0}\|_{L^1(\Omega)}\leq C(\mu,j)\|\tilde \omega_{0,0}\|_{L^1(\Omega)},$$ we have
 \begin{align*}
 \|D^j\vec{v}_{0,0}^{\mu}\|_{L^\infty(\Omega)}=&\|K\ast D^j\tilde\omega_{0,0}^{\mu}\|_{L^\infty(\Omega)}\leq \|(\rho_1K)*D^j\tilde \omega_{0,0}^{\mu}\|_{L^\infty(\Omega)}+\|((1-\rho_1)K)*D^j\tilde \omega_{0,0}^{\mu}\|_{L^\infty(\Omega)}\\
 \leq&\|(\rho_1K)\|_{L^1(\Omega)}\|D^j\tilde \omega_{0,0}^{\mu}\|_{L^\infty(\Omega)}+\|((1-\rho_1)K)\|_{L^\infty(\Omega)}\|D^j\tilde \omega_{0,0}^{\mu}\|_{L^1(\Omega)}\\
 \leq&C(\mu,j)\|\tilde \omega_{0,0}\|_{L^2(\Omega)}+C(\mu,j)\|\tilde \omega_{0,0}\|_{L^1(\Omega)}.
 \end{align*}

 Taking derivative of \eqref{euler-nonshear} and similar to \eqref{basic energy estimate-eq}-\eqref{basic energy estimate-gronwall}, we get the high-order energy estimates ($q\geq1$):
 \begin{align}\nonumber
 {d\over dt}\|\vec{v}_{\mu}(t)\|_{H^q(\Omega)}&\leq C_q \|\vec{v}_{\mu}(t)\|_{H^q(\Omega)}\left(\|\nabla \vec{v}_{\mu}(t)\|_{L^\infty(\Omega)}+\|\vec{v}_{0,0}^{\mu}\|_{W^{q+1,\infty}(\Omega)}\right),
 \end{align}
 and
 \begin{align}\label{basic energy estimate-gronwall-high-order}
\|\vec{v}_{\mu}(t)\|_{H^q(\Omega)}&\leq \|\vec{v}_{\mu}(0)\|_{H^q(\Omega)}e^{\int_0^tC_q\left(\|\nabla \vec{v}_{\mu}(t)\|_{L^\infty(\Omega)}+\|\vec{v}_{0,0}^{\mu}\|_{W^{q+1,\infty}(\Omega)}\right)ds}.
 \end{align}
By the asymptotic behavior of $|K|$ near $(x,y)=(0,0)$ in  \eqref{K-near-0} and the exponential decay rate of $|\nabla K|$ as $|y|\to \infty$, a similar argument to Lemma A3 in \cite{Kato1986nonlinear} gives
\begin{align*}
&\|\nabla\vec{v}_{\mu}(t)\|_{L^\infty(\Omega)}\leq\|\nabla\vec{v}^{\mu}(t)\|_{L^\infty(\Omega)}+\|\nabla\vec{v}_{0,0}^{\mu}\|_{L^\infty(\Omega)}\\
\leq&
 C\bigg(\|\tilde\omega_{0}^\mu\|_{L^\infty(\Omega)}+\|\tilde\omega_{0}^\mu\|_{L^2(\Omega)}
+\|\tilde\omega_{0}^\mu\|_{L^\infty(\Omega)}
\ln\left(1+{\|\vec{v}^{\mu}(t)\|_{H^3(\Omega)}\over\|\tilde\omega_{0}^\mu\|_{L^\infty(\Omega)}}\right)\\
&+\|\tilde\omega_{0,0}\|_{L^2(\Omega)}+\|\tilde\omega_{0,0}\|_{L^1(\Omega)}\bigg)\\
\leq&C\bigg(\|\tilde\omega_{0}\|_{L^2(\Omega)}+\|\tilde\omega_{0,0}\|_{L^1(\Omega)}+\|\tilde\omega_{0}\|_{L^2(\Omega)}
\ln\bigg(1+{\|\vec{v}_{0,0}^{\mu}\|_{H^3(\Omega)}\over\|\tilde\omega_{0}^\mu\|_{L^\infty(\Omega)}}+
{\|\vec{v}_{\mu}(t)\|_{H^3(\Omega)}\over\|\tilde\omega_{0}^\mu\|_{L^\infty(\Omega)}}\bigg)\bigg),
\end{align*}
where we used \eqref{B-S law}.
 Then
\begin{align}\label{estimate-v-kappa-L-infty}
\|\nabla\vec{v}_{\mu}(t)\|_{L^\infty(\Omega)}\leq C_{\|\tilde\omega_{0}^\mu\|_{L^\infty(\Omega)},\|\tilde\omega_{0,0}\|_{L^1(\Omega)},\|\tilde\omega_0\|_{L^2(\Omega)},\|\vec{v}_{0,0}^{\mu}\|_{H^3(\Omega)}}\left(1+
\ln_+(\|\vec{v}_{\mu}(t)\|_{H^3(\Omega)})\right),
\end{align}
where $\ln_+(x)=\ln(x)$ for $x>1$ and $\ln_+(x)=0$ for $0<x\leq1$.
Plugging \eqref{basic energy estimate-gronwall-high-order} for $q=3$ into \eqref{estimate-v-kappa-L-infty}, we have
\begin{align*}
\|\nabla\vec{v}_{\mu}(t)\|_{L^\infty(\Omega)}\leq C_*\left(1+\|\vec{v}_{0,0}^{\mu}\|_{W^{4,\infty}(\Omega)}t+
\int_0^t\|\nabla \vec{v}_{\mu}(t)\|_{L^\infty(\Omega)}ds\right),
\end{align*}
where $C_*=C_{\|\tilde\omega_{0}^\mu\|_{L^\infty(\Omega)},\|\tilde\omega_{0,0}\|_{L^1(\Omega)},\|\tilde\omega_0\|_{L^2(\Omega)},
\|\vec{v}_{0,0}^{\mu}\|_{H^3(\Omega)},\|\vec{v}_{\mu}(0)\|_{H^3(\Omega)}}$
depends only on the initial data.
Then Gr\"{o}nwall's inequality implies
\begin{align*}
\|\nabla\vec{v}_{\mu}(t)\|_{L^\infty(\Omega)}\leq (C_*+\tilde C_* t)e^{C_*t},
\end{align*}
where $\tilde C_*=C_*\|\vec{v}_{0,0}^{\mu}\|_{W^{4,\infty}(\Omega)}$.
Inserting this into \eqref{basic energy estimate-gronwall-high-order}  gives an a priori bound for $\|\vec{v}_{\mu}\|_{H^q(\Omega)}$:
\begin{align}\label{BKM}
\|\vec{v}_{\mu}(t)\|_{H^q(\Omega)}\leq \|\vec{v}_{\mu}(0)\|_{H^q(\Omega)}e^{\int_0^tC_q\left( (C_*+\tilde C_* t)e^{C_*t}+\|\vec{v}_{0,0}^{\mu}\|_{W^{q+1,\infty}(\Omega)}\right)ds},
 \end{align}
 which proves the global existence  of the  solution $\vec{v}^{\mu}=\vec{v}_{0,0}^{\mu}+\vec{v}_{\mu}$ to 2D Euler equation in $H^q(\Omega)$ for $q\geq3$. This verifies (iii) of Definition \ref{Approximate solution sequence for the 2D Euler equation}. (ii) is trivially  verified. Then we prove that $\{\vec{v}^\mu\}$ has  $L^1$ and $L^2$ vorticity control.
  Let $\tilde \omega^\mu=\text{curl}(\vec{v}^\mu)$.
By Lemma \ref{tilde-omega0-kappa-properties} (4),
\begin{align}\label{L12 bound-approximate solution sequence}
\iint_{\Omega}|\tilde\omega^\mu(t)|^pdxdy=&\iint_{\Omega}|\tilde\omega_0^\mu|^pdxdy\leq \|\tilde\omega_{0}\|_{L^p(\Omega)}^p
\end{align}
for $t\geq0$, and
 $\tilde \omega_0^{\mu}\to\tilde\omega_0$ in $L^p(\Omega)$ for $p=1,2$. To  verify (i), we
 note that
 \begin{align}\nonumber
 &\|\vec{v}^\mu(t)\|_{L^2(B_R)}=\|(K*\tilde \omega^\mu)(t)\|_{L^2(B_R)}\\\nonumber
 \leq& \|((\rho_1K)*\tilde \omega^\mu)(t)\|_{L^2(\Omega)}+\|(((1-\rho_1)K)*\tilde \omega^\mu)(t)\|_{L^2(B_R)}\\\nonumber
 \leq&\|\rho_1K\|_{L^1(\Omega)}\|\tilde \omega^\mu(t)\|_{L^2(\Omega)}+C(R)\|(1-\rho_1)K\|_{L^\infty(\Omega)}\|\tilde \omega^\mu(t)\|_{L^2(\Omega)}\\\label{uniform L2 bound v kappa}
 \leq&C(R)\|\tilde \omega_0\|_{L^2(\Omega)}
 \end{align}
 for any $ R>0$,
where we used $\tilde \omega^\mu(t)=\text{curl}(\vec{v}^\mu(t))$ and \eqref{L12 bound-approximate solution sequence}.

We define the stream function by $\tilde \psi^\mu(t)=G*\tilde \omega^\mu(t)$, where $\omega^\mu(t)={\mathrm{curl}}(\vec{v}^{\mu}(t))$ is the vorticity. Then the velocity can be recovered from $\tilde \psi^\mu(t)$ by the Biot-Savart law
\begin{align}\label{B-S law}
\vec{v}^{\mu}(t)=\nabla^{\bot} (G*\tilde \omega^\mu(t))=K*\tilde \omega^\mu(t)\end{align}
in our setting. In fact, let
 $\vec{\vartheta}(t)=({\vartheta}_1(t),{\vartheta}_2(t) )\triangleq K*\tilde \omega^\mu(t)-\vec{v}^{\mu}(t)$ for  $\mu\in(0,1)$ and $t\geq0$.
 Since ${\mathrm{div}}(\vec{\vartheta}(t))=0$ and ${\mathrm{curl}}(\vec{\vartheta}(t))=0$, we have $ik\widehat{\vartheta}_{1,k}(y)+\widehat{\vartheta}_{2,k}'(y)=0$, $\widehat{\vartheta}_{1,k}'(y)-ik\widehat{\vartheta}_{2,k}(y)=0$ for $k\neq0$, $\widehat{\vartheta}_{1,0}'(y)=0$ and $\widehat{\vartheta}_{2,0}'(y)=0$.
Thus, $\widehat{\vartheta}_{1,k}''(y)-k^2\widehat{\vartheta}_{1,k}(y)=0$ and $\widehat{\vartheta}_{2,k}''(y)-k^2\widehat{\vartheta}_{2,k}(y)=0$ for $k\neq0$, which implies $
\widehat{\vartheta}_{1,k}(y)=c_{1,k}e^{ky}+\tilde c_{1,k}e^{-ky}$ and $
\widehat{\vartheta}_{2,k}(y)=c_{2,k}e^{ky}+\tilde c_{2,k}e^{-ky}$ for some $c_{1,k},\tilde c_{1,k},c_{2,k},\tilde c_{2,k}\in\mathbb{C}$. Noting that
$\vec{\vartheta}= ({\vartheta}_1,{\vartheta}_2 )=K*\tilde \omega_\mu(t)-\vec{v}_{\mu}(t)$, we have ${\vartheta}_2\in L^2 (\Omega)$, where $\tilde \omega_\mu(t)=\tilde \omega^\mu(t)-\tilde \omega_{0,0}^\mu$. Thus, $\widehat{\vartheta}_{2,k}(y)=0$ for $k\in\mathbb{Z}$, which implies $\widehat{\vartheta}_{1,k}(y)=0$ for $k\neq0$  since $ik\widehat{\vartheta}_{1,k}(y)+\widehat{\vartheta}_{2,k}'(y)=0$. By the first limit in \eqref{v-mu term1 shear decomposition} and $\vec{v}_{\mu}(t)\in L^2 (\Omega)$, we have  $\widehat{\vartheta}_{1,0}(y)=0$.

Finally, we prove that
\begin{align}\label{v-mu term2}
&\lim_{y\to\pm\infty}v^{\mu,2}(t,x,y)=-\lim_{y\to\pm\infty}\partial_x\tilde\psi^\mu(t,x,y)=-\lim_{y\to\pm\infty}(\partial_x G\ast\tilde \omega^\mu)(t,x,y)=0,\\
&\lim_{y\to\pm\infty}(\partial_y G\ast\tilde \omega_\mu)(t,x,y)=0,\;\;\lim_{y\to\pm\infty}(\partial_y G\ast\tilde \omega_{0,0}^\mu)(t,x,y)=\pm1,
\label{v-mu term1 shear decomposition}
\end{align}
which implies
\begin{align}
\label{v-mu term1}
&\lim_{y\to\pm\infty}v^{\mu,1}(t,x,y)=\lim_{y\to\pm\infty}\partial_y\tilde\psi^\mu(t,x,y)=\lim_{y\to\pm\infty}(\partial_y G\ast\tilde \omega^\mu)(t,x,y)=\pm1
\end{align}
for $t\geq0$ and $x\in\mathbb{T}_{2\pi}$, where $\vec{v}^{\mu}(t)=(v^{\mu,1}(t),v^{\mu,2}(t))$.
Indeed, $\|\tilde \omega^\mu(t)\|_{L^{p'}(\Omega)}=\|\tilde \omega^\mu(0)\|_{L^{p'}(\Omega)}\leq C\|\tilde \omega^\mu(0)\|_{H^{1}(\Omega)}$, and thus, for any $\varepsilon>0$,  there exists $R_1>0$ such that
\begin{align*}
\|\tilde \omega^\mu(t)\|_{L^{p'}(B_{R_1}^c)}<{\varepsilon\over 2\left\|\partial_x G\right\|_{L^{p}(\Omega)}},
\end{align*}
where $p\in(1,2)$ and ${1\over p}+{1\over p'}=1$. Then
\begin{align}\label{B-R1-c}
&\left|\iint_{B_{R_1}^c}\partial_x G(x-\tilde x,y-\tilde y)\tilde \omega^\mu(t,\tilde x,\tilde y)d\tilde xd\tilde y\right|
\leq\left\|\partial_x G\right\|_{L^{p}(\Omega)}
\|\tilde \omega^\mu(t)\|_{L^{p'}(B_{R_1}^c)}<{\varepsilon\over 2}
\end{align}
for $(x,y)\in\Omega$. Choose $M_1>0$ such that if $|y|>M_1$, then  $\left|\partial_x G(x-\tilde x,y-\tilde y)\right|<{\varepsilon\over 2\|\tilde \omega^\mu(t)\|_{L^1(\Omega)}}$ uniformly for $(\tilde x,\tilde y)\in B_{R_1}$. Then
\begin{align*}
&\left|\iint_{B_{R_1}}\partial_x G(x-\tilde x,y-\tilde y)\tilde \omega^\mu(t,\tilde x,\tilde y)d\tilde xd\tilde y\right|\leq{\varepsilon\over 2}
\end{align*}
for $|y|>M_1$. This, along with \eqref{B-R1-c}, gives   \eqref{v-mu term2}.
 To prove $\lim_{y\to\infty}(\partial_y G\ast\tilde \omega_\mu)(t,x,y)=0$ in \eqref{v-mu term1 shear decomposition}, we denote $C_0=\max_{x\in\mathbb{T}_{2\pi},|y|>1}(|\partial_y G|+1)<\infty$. For any $\varepsilon>0$, there exists $R_2>0$ such that
\begin{align*}
\|\tilde \omega_\mu(t)\|_{L^1(\{y>R_2\})}<{\varepsilon\over 4C_0},\;\;\|\tilde \omega_\mu(t)\|_{L^{p'}(\{y>R_2\})}<{\varepsilon\over 4\|\partial_yG\|_{L^p(B_1)}},
\end{align*}
where $p\in(1,2)$.
Since
$\iint_\Omega\tilde \omega_\mu(t)dxdy=\iint_\Omega\tilde \omega^\mu(t)dxdy-\iint_\Omega\tilde \omega_{0,0}^\mu dxdy=\iint_\Omega\tilde \omega^\mu(0)dxdy-\iint_\Omega\tilde \omega_{0,0}^\mu dxdy=0$, we have
\begin{align*}
&(\partial_y G\ast\tilde \omega_\mu)(t,x,y)=((\partial_y G+{1/ (4\pi)})\ast\tilde \omega_\mu)(t,x,y)\\
=&\iint_{\{\tilde y<R_2\}}(\partial_y G(x-\tilde x,y-\tilde y)+{1/ (4\pi)})\tilde \omega_\mu(t,\tilde x,\tilde y)d\tilde x d\tilde y\\
&+\iint_{\{\tilde y>R_2\}}(\partial_y G(x-\tilde x,y-\tilde y)+{1/ (4\pi)})\tilde \omega_\mu(t,\tilde x,\tilde y)d\tilde x d\tilde y= I +II.
\end{align*}
Choose $M_2>R_2$ such that if $y>M_2$, then $|\partial_y G(x-\tilde x,y-\tilde y)+{1/(4\pi)}|<{\varepsilon\over 4\|\tilde \omega_\mu(t)\|_{L^1(\Omega)}}$ uniformly for $\tilde y<R_2$. Then
$
|I|\leq {\varepsilon\over4}
$
for $y>M_2$. For $II$, we have
\begin{align*}
|II|=&\bigg|\iint_{\{\tilde y>R_2\}\cap\{|\tilde y-y|\leq1\}}\partial_y G(x-\tilde x,y-\tilde y)\tilde \omega_\mu(t,\tilde x,\tilde y)d\tilde x d\tilde y\\
&+\iint_{\{\tilde y>R_2\}\cap\{|\tilde y-y|\leq1\}}{1/ (4\pi)}\tilde \omega_\mu(t,\tilde x,\tilde y)d\tilde x d\tilde y\\
&+\iint_{\{\tilde y>R_2\}\cap\{|\tilde y-y|>1\}}(\partial_y G(x-\tilde x,y-\tilde y)+{1/ (4\pi)})\tilde \omega_\mu(t,\tilde x,\tilde y)d\tilde x d\tilde y\bigg|\\
\leq &\|\partial_yG\|_{L^p(B_1)}\|\tilde \omega_\mu(t)\|_{L^{p'}(\{y>R_2\})}+\|\tilde \omega_\mu(t)\|_{L^1(\{y>R_2\})}+C_0\|\tilde \omega_\mu(t)\|_{L^1(\{y>R_2\})}<{3\over4}\varepsilon
\end{align*}
for $y\in\mathbb{R}$. Combining the estimates for $I$ and $II$, we have $\lim_{y\to\infty}(\partial_y G\ast\tilde \omega_\mu)(t,x,y)=0$.
Similarly, we have $\lim_{y\to-\infty}(\partial_y G\ast\tilde \omega_\mu)(t,x,y)=0$ and $\lim_{y\to\pm\infty}(\partial_y G\ast\tilde \omega_{0,0}^\mu)(t,x,y)=\pm1$.
\end{proof}
\begin{Corollary}\label{y-tilde-omega-pseudoenergy-conserved}
Let $\{\vec{v}^\mu\}$ be the approximate solution sequence constructed in Lemma \ref{lem-construction of an approximate solution sequence}.
Then

$(1)$ for any $T>0$, there exists $C(T)>0$ (independent of $\mu$) such that  $\max_{0\leq t\leq T} \|y\tilde\omega^{\mu}(t)\|_{ L^1(\Omega)}$ $\leq C(T)$, and thus, $\tilde\omega^{\mu}(t)\in Y_{non}$ for $t\geq0$;
$\iint_\Omega y\tilde\omega^{\mu}(t,x,y)dxdy$ is conserved for all $t\geq0$;

$(2)$ the pseudoenergy
$
PE(\tilde \omega^\mu(t))={1\over2}\iint_{\Omega}(G\ast\tilde \omega^\mu)(t)\tilde \omega^{\mu}(t) dxdy
$ is conserved for all $t\geq0$.
\end{Corollary}

\begin{proof}
(1) We change the variables $(x,y)$ to $(X^\mu(t),Y^\mu(t))$, where $(X^\mu(t),Y^\mu(t))$ is the solution to $\dot{X}^\mu(t)=\partial_y\tilde\psi^\mu(t,X^\mu(t), Y^\mu(t)),
\dot{Y}^\mu(t)=-\partial_x\tilde \psi^\mu(t,X^\mu(t), Y^\mu(t))$ with the initial data $(X^\mu(0),Y^\mu(0))=(x,y)$. Noting that the vorticity $\tilde\omega^\mu$ is conserved along particle trajectories and the Jacobian of the mapping $(x,y)\to(X^\mu(t),Y^\mu(t))$ is $1$, we have
\begin{align}\nonumber
{d\over dt}\iint_\Omega |y\tilde\omega^{\mu}(t,x,y)|dxdy=&\iint_\Omega \dot{Y}^\mu(t)\tilde\omega^{\mu}(t,X^\mu(t),Y^\mu(t))sign(-Y^\mu(t))dX^\mu(t)dY^\mu(t)\\\nonumber
\leq&\|\partial_x\tilde\psi^\mu(t)\|_{L^2(\Omega)}\|\tilde\omega^{\mu}(t)\|_{L^2(\Omega)}
\leq\|\partial_xG\|_{L^1(\Omega)}\|\tilde\omega^{\mu}(t)\|_{L^2(\Omega)}^2\\\nonumber
=&\|\partial_xG\|_{L^1(\Omega)}\|\tilde\omega_0^{\mu}\|_{L^2(\Omega)}^2\leq\|\partial_xG\|_{L^1(\Omega)}\|\tilde\omega_0\|_{L^2(\Omega)}^2,
\end{align}
which, along with $y\tilde \omega^{\mu}_0\to y\tilde \omega_0$, implies that $\max_{0\leq t\leq T} \|y\tilde\omega^{\mu}(t)\|_{ L^1(\Omega)}\leq C(T)$. Moreover,
\begin{align}\nonumber
{d\over dt}\iint_\Omega y\tilde\omega^{\mu}(t,x,y)dxdy=&\iint_\Omega \dot{Y}^\mu(t)\tilde\omega^{\mu}(t,X^\mu(t),Y^\mu(t))dX^\mu(t)dY^\mu(t)\\\nonumber
=&\iint_\Omega -\partial_x\tilde\psi^\mu(t,x,y)\tilde\omega^{\mu}(t,x,y)dxdy\\\nonumber
=&-{1\over2}\iint_\Omega \partial_x|\nabla\tilde\psi^\mu(t,x,y)|^2 dxdy+\int_0^{2\pi}(\partial_x\tilde\psi^\mu\partial_y\tilde\psi^\mu)(t,x,y)|_{y=-\infty}^\infty dx\\\nonumber
=&-{1\over2}\iint_\Omega \partial_x|\nabla\tilde\psi^\mu(t,x,y)|^2 dxdy=0,
\end{align}
where we used \eqref{v-mu term2} and \eqref{v-mu term1} to ensure that $\lim_{y\to\pm\infty}(\partial_x\tilde\psi^\mu\partial_y\tilde\psi^\mu)(t,x,y)=0$ for $t>0$ and $x\in\mathbb{T}_{2\pi}$.

(2) Since $\tilde \psi^\mu(t)=G*\tilde \omega^\mu(t)$, we have
\begin{align}\nonumber
{d\over dt}PE(\tilde \omega^\mu(t))=&{1\over2}\iint_{\Omega}\partial_t\tilde \psi^\mu(t,X^\mu(t),Y^\mu(t))\tilde \omega^{\mu}(t,X^\mu(t),Y^\mu(t) )dxdy\\\nonumber
&+{1\over2}\iint_{\Omega}\nabla\tilde \psi^\mu(t, X^\mu(t),Y^\mu(t) )\cdot\nabla^\bot\tilde \psi^\mu(t,X^\mu(t),Y^\mu(t))\tilde \omega^{\mu}(t,X^\mu(t),Y^\mu(t) ) dxdy\\
=&{1\over2}\iint_{\Omega}\partial_t\tilde \psi^\mu(t,x,y)\tilde \omega^{\mu}(t,x,y )dxdy.\label{PE-derivative1}
\end{align}
On  the other hand,
\begin{align}\nonumber
{d\over dt}PE(\tilde \omega^\mu(t))=&{1\over2}\iint_{\Omega}\bigg(\partial_t(G*\tilde \omega^\mu)(t,x,y)\tilde \omega^{\mu}(t,x,y )+(G*\tilde \omega^\mu)(t,x,y)\partial_t\tilde \omega^{\mu}(t,x,y )\bigg)dxdy\\\nonumber
=&{1\over2}\iint_{\Omega}\bigg(\partial_t\tilde \psi^\mu(t,x,y)\tilde \omega^{\mu}(t,x,y )+(G*\partial_t\tilde \omega^\mu)(t,x,y)\tilde \omega^{\mu}(t,x,y )\bigg)dxdy\\
=&\iint_{\Omega}\partial_t\tilde \psi^\mu(t,x,y)\tilde \omega^{\mu}(t,x,y )dxdy.\label{PE-derivative2}
\end{align}
By \eqref{PE-derivative1}-\eqref{PE-derivative2}, we have ${d\over dt}PE(\tilde \omega^\mu(t))=\iint_{\Omega}\partial_t\tilde \psi^\mu(t,x,y)\tilde \omega^{\mu}(t,x,y )dxdy=0$.
\end{proof}

\subsection{Convergence of the approximate solutions and existence of weak solutions}
First, we prove the $L_{loc}^1$ convergence of the approximate solution sequence with $L^1$ vorticity control.
\begin{lemma}\label{convergence of an approximate solution sequence}
Let $\{\vec{v}^\mu\}$ be the approximate solution sequence constructed in Lemma \ref{lem-construction of an approximate solution sequence}. Then for any $T>0$ and $R>0$, there exists $\vec{v}\in L^1(\Omega_{R,T})$ such that $\max_{0\leq t\leq T}\iint_{B_R}|\vec{v}(t)|^2dxdy$ $\leq C(R,T)$, ${\mathrm{div}}(\vec{v})=0$, and up to a subsequence,
\begin{align}\label{limit-for-approximate solution}
\vec{v}^{\mu}\to\vec{v} \;\text{ in }\; L^1(\Omega_{R,T}),\end{align}
and
\begin{align}\label{limit-for-vorticity-approximate solution}
\textup{curl}(\vec{v}^{\mu})=\tilde\omega^{\mu}\stackrel{*}{\rightharpoonup}\tilde \omega=\textup{curl}(\vec{v}) \;\text{ in }\; \mathcal{M}(\Omega_{R,T}),\end{align}
where $\Omega_{R,T}=[0,T]\times B_R$ and $\mathcal{M}(\Omega_{R,T})=\{\mu| \mu \text{ is a Randon measure on }\Omega_{R,T} \text{ with } \mu(\Omega_{R,T})<\infty\}$. Moreover,  $\vec{v}^{\mu}(t)\in L^1(B_R)$ and
\begin{align}\label{limit-for-approximate solution-t}
\vec{v}^{\mu}(t)\to\vec{v}(t) \;\text{ in }\; L^1(B_R)\end{align}
for any $t\geq0$.
\end{lemma}
\begin{proof} By the $L^1$ vorticity control of $\{\vec{v}^\mu\}$, there exists $\tilde\omega\in \mathcal{M}(\Omega_{R,T})$ such that, up to a subsequence, \eqref{limit-for-vorticity-approximate solution} holds.
Similar to (10.33) in \cite{Majda-Bertozzi02}, $\tilde\omega\in C([0,T],H_{\text{loc}}^{-s}(\Omega))$ and
\begin{align}\label{H-s-estimate}
\max_{0\leq t\leq T}\|\varphi\tilde\omega^\mu(t)-\varphi\tilde\omega(t)\|_{H^{-s}(\Omega)}\to 0,\quad \forall\; s>1
\end{align}
for any $\varphi\in C_0^\infty(\Omega)$, where $\tilde \omega^\mu=\text{curl} (\vec{v}^\mu)$. By Lemma \ref{lem-construction of an approximate solution sequence}, we have $\tilde\omega(0)=\tilde \omega_0$.

To prove \eqref{limit-for-approximate solution}, it suffices to show that $\{\vec{v}^\mu\}$ is a Cauchy sequence in $L^1(\Omega_{R,T})$.
Let  $\rho,
 \rho_s,
(1-\rho_{s})_{>0}$ and
 $(1-\rho_{s})_{<0}$
be given in
\eqref{def-cut off functions}.
Define   $\tilde \rho_s(x,y)=\rho\left({\sqrt{x^2+y^2}\over s}\right)$ for $(x,y)\in \Omega$. Let $\delta\in(0,\pi)$ be small enough and $R'>\delta$. Then we split $\vec{v}^{\mu_1}-\vec{v}^{\mu_2}$ into five terms:
\begin{align}\nonumber
&\vec{v}^{\mu_1}-\vec{v}^{\mu_2}=K\ast(\tilde\omega^{\mu_1}-\tilde\omega^{\mu_2})\\\nonumber
=&
(\tilde\rho_\delta K)\ast(\tilde\omega^{\mu_1}-\tilde\omega^{\mu_2})+((\rho_{R'}-\tilde \rho_\delta)K)\ast(\tilde\omega^{\mu_1}-\tilde\omega^{\mu_2})\\\nonumber
&+\left((1-\rho_{R'})_{>0}\left(K+\left({1\over 4\pi},0\right)\right)\right)\ast(\tilde\omega^{\mu_1}-\tilde\omega^{\mu_2})\\\nonumber
&+\left((1-\rho_{R'})_{<0}\left(K-\left({1\over 4\pi},0\right)\right)\right)\ast(\tilde\omega^{\mu_1}-\tilde\omega^{\mu_2})\\\nonumber
&+\left(-(1-\rho_{R'})_{>0}+(1-\rho_{R'})_{<0}\right)\left({1\over 4\pi},0\right)\ast(\tilde\omega^{\mu_1}-\tilde\omega^{\mu_2})\\\label{vkappa1-vkappa2}
:=&I_1(\mu_1,\mu_2)+I_2(\mu_1,\mu_2)+I_3(\mu_1,\mu_2)+I_4(\mu_1,\mu_2)+I_5(\mu_1,\mu_2).
\end{align}
By \eqref{K-near-0} and
the  $L^1$  vorticity control of $\{\vec{v}^\mu\}$  in Lemma \ref{lem-construction of an approximate solution sequence}, we have
\begin{align}\nonumber
\|I_1(\mu_1,\mu_2)\|_{L^1(\Omega_{R,T})}\leq& \|(\tilde\rho_\delta K)\|_{L^1(\Omega)}\|\tilde\omega^{\mu_1}-\tilde\omega^{\mu_2}\|_{L^1(\Omega\times[0,T])}\\\label{I1-estimate}
\leq& C(T)\iint_{\sqrt{x^2+y^2}<2\delta}|K(x,y)|dxdy=C(T)\delta.
\end{align}
By \eqref{K-near-pm-infty} and the  $L^1$  vorticity control of $\{\vec{v}^\mu\}$, we have
\begin{align}\nonumber
&\|I_3(\mu_1,\mu_2)\|_{L^1(\Omega_{R,T})}\\\nonumber
\leq&C(R,T)\left\|\left((1-\rho_{R'})_{>0}\left(K+\left({1\over 4\pi},0\right)\right)\right)\ast(\tilde\omega^{\mu_1}(t)-\tilde\omega^{\mu_2}(t))\right\|_{L^{\infty}(\Omega)}\\\nonumber
\leq& C(R,T)\left\|(1-\rho_{R'})_{>0}\left(K+\left({1\over 4\pi},0\right)\right)\right\|_{L^{\infty}(\Omega)}\|\tilde\omega^{\mu_1}(t)-\tilde\omega^{\mu_2}(t)\|_{L^1(\Omega)}\\\label{I3-estimate}
\leq &C(R,T) R'^{-1}
\end{align}
for $R'>0$ (independent of $\mu_1, \mu_2$) sufficiently large.
Similarly,
\begin{align}\label{I4-estimate}
\|I_4(\mu_1,\mu_2)\|_{L^1(\Omega_{R,T})}\leq C(R,T) R'^{-1}
\end{align}
for $R'>0$ (independent of $\mu_1, \mu_2$) sufficiently large. Now, we fix $R'$.
To estimate $I_5(\mu_1,\mu_2)$, let $\varphi_{R'}=(-(1-\rho_{R'})_{>0}+$ $(1-\rho_{R'})_{<0})\left({1\over 4\pi},0\right)$. By the  $L^1$  vorticity control of $\{\vec{v}^\mu\}$ again, we have $\tilde \omega^{\mu}(t)\rightharpoonup\tilde \omega(t)$ in $L^1(\Omega)$ for $t>0$. This, along with the fact that  $\varphi_{R'}\in L^\infty (\Omega)$,   gives
\begin{align*}
I_5(\mu_1,\mu_2)=\iint_{\Omega}\varphi_{R'}(x-\tilde x, y-\tilde y)(\tilde\omega^{\mu_1}-\tilde\omega^{\mu_2})(t,\tilde x,\tilde y) d\tilde xd\tilde y\to0\quad \text{as}\quad \mu_1,\mu_2\to0^+
\end{align*}
 for fixed $R'$ and $(x,y,t)\in \Omega_{R,T}$. Since $|I_5(\mu_1,\mu_2)|\leq \|\tilde\omega^{\mu_1}(t)\|_{L^1(\Omega)}+\|\tilde\omega^{\mu_2}(t)\|_{L^1(\Omega)}\leq C$, by the Dominated Convergence Theorem  we have
 \begin{align}\label{I5-estimate}
\|I_5(\mu_1,\mu_2)\|_{L^1(\Omega_{R,T})}\to0 \;\;\text{as}\;\; \mu_1,\mu_2\to0^+.
\end{align}

 By \eqref{H-s-estimate}, for $(x,y,t)\in \Omega_{R,T}$ we have
\begin{align}\nonumber
&|I_2(\mu_1,\mu_2)|\leq \|(\rho_{R'}-\tilde \rho_\delta)K\|_{H^s(\Omega)}\|\rho_{2(R'+R)}(\tilde\omega^{\mu_1}(t)-\tilde\omega^{\mu_2}(t))\|_{H^{-s}(\Omega)}\to0\\\label{I2-estimate}
\Rightarrow&\|I_2(\mu_1,\mu_2)\|_{L^1(\Omega_{R,T})}\to0 \;\;\text{as}\;\; \mu_1,\mu_2\to0^+,
\end{align}
where $s>1$ and we used $(\rho_{R'}-\tilde \rho_\delta)K\in C_0^{\infty}(\Omega)$. Combining \eqref{vkappa1-vkappa2}-\eqref{I2-estimate}, taking $\delta>0$ sufficiently small and $R'>0$ sufficiently large, we obtain that $\{\vec{v}^\mu\}$ is a Cauchy sequence in $L^1(\Omega_{R,T})$. For any $t\geq0$, the proof of
$\vec{v}^{\mu}(t)\in L^1(B_R)$ and \eqref{limit-for-approximate solution-t} is the same as above.
 $\max_{0\leq t\leq T}\iint_{B_R}|\vec{v}(t)|^2dxdy\leq C(R,T)$ follows from \eqref{uniform L2 bound v kappa}.
\end{proof}
Now, we prove the existence of weak solution to the 2D Euler equation with initial vorticity $\tilde \omega_0\in Y_{non}$.
\begin{Theorem}\label{existence of weak solution to the 2D Euler equation-thm}
Let $\{\vec{v}^\mu\}$ be the approximate solution sequence constructed in Lemma \ref{lem-construction of an approximate solution sequence}.
Then for any $R, T>0$, there exists $\vec{v}\in L^2(\Omega_{R,T})$ such that
\begin{align*}
\vec{v}^\mu\to \vec{v} \;\text{ in }\;L^2(\Omega_{R,T}),
\end{align*}
 and $\vec{v}$ is a weak solution to the 2D Euler equation.
   Moreover,  $\vec{v}^{\mu}(t)\in L^2(B_R)$ and
\begin{align}\label{limit-for-approximate solution-L2-t}
\vec{v}^{\mu}(t)\to\vec{v}(t) \;\text{ in }\; L^2(B_R)\end{align}
for any $t\geq0$.  Consequently, for any initial vorticity $\tilde \omega_0\in Y_{non}$, there exists  $\vec{v}\in L^2(\Omega_{R,T})$ such that
 $\textup{curl}(\vec{v}(0))=\tilde\omega_0$ and $\vec{v}$ is a weak solution to the 2D Euler equation.
\end{Theorem}
\begin{proof}
By Proposition 25 in \cite{Milisic-Razafison16} and the fact that $\tilde \omega^{\mu}(t)\in L^2(\Omega)$ for $t\geq0$, there exists $\varphi^{\mu}(t)\in W_0^{2,2}(\Omega)$ such that $\psi=\varphi^{\mu}(t)$ solves $-\Delta\psi=\tilde \omega^{\mu}(t)$, where $W_0^{2,2}(\Omega)=\{\phi|(1+|y|^2)^{-1}\phi\in L^2(\Omega), (1+|y|^2)^{-{1\over2}}\nabla\phi\in L^2(\Omega),D^2\phi\in L^2(\Omega)\}$.  Then there exists $c_1, c_2\in \mathbb{R}$ and $d_{1j}, d_{2j}\in \mathbb{C}$, $ j\neq0$, such that $
(G\ast\tilde \omega^{\mu})(t)=\varphi^{\mu}(t)+\sum_{j\neq0} e^{ijx}(d_{1j}e^{jy}+d_{2j}e^{-jy})+c_1y+c_2.
$
We claim that $d_{1j}, d_{2j}=0$ for $ j\neq0$. In fact,
\begin{align*}
&|(G\ast\tilde \omega^{\mu})(t)|=|(G_1\ast\tilde \omega^{\mu})(t)|+|(G_2\ast\tilde \omega^{\mu})(t)|\\
\leq& \|G_1\|_{L^2(\Omega)}\|\tilde  \omega^{\mu}(t)\|_{L^2(\Omega)}+C|y|\|\tilde  \omega^{\mu}(t)\|_{L^1(\Omega)}+C\|y\tilde  \omega^{\mu}(t)\|_{L^1(\Omega)}
\end{align*}
since $\tilde \omega^\mu(t)\in L^1\cap L^2(\Omega)$ and  $y\tilde \omega^\mu(t)\in L^1(\Omega)$ by Corollary \ref{y-tilde-omega-pseudoenergy-conserved} (1). Thus, $G\ast\tilde \omega^{\mu}(t)=\varphi^{\mu}(t)+c_1y+c_2.$
By the weighted Calderon-Zygmund inequality \cite{Sauer2014},  we have
\begin{align*}
\|\nabla \vec{v}^{\mu}(t)\|_{L^2(\Omega)}=\|D^2(G\ast\tilde \omega^{\mu})(t)\|_{L^2(\Omega)}=\|D^2\varphi^{\mu}(t)\|_{L^2(\Omega)}\leq C\|\tilde \omega^{\mu}(t)\|_{L^2(\Omega)}\leq C
\end{align*}
for $t\geq0$.
By \eqref{uniform L2 bound v kappa}, $\|\vec{v}^{\mu}(t)\|_{L^2(B_R)}\leq C(R)\|\tilde \omega_0\|_{L^2(\Omega)}\leq C(R)$ for $t\geq0$. By Sobolev embedding $H^1(B_R)\hookrightarrow L^q(B_R)$ for $2< q<\infty$, we have
\begin{align*}
\|\vec{v}^{\mu}(t)\|_{L^q(B_R)}\leq C\|\vec{v}^{\mu}(t)\|_{H^1(B_R)}\leq C(R)\Longrightarrow\|\vec{v}^{\mu}\|_{L^q(\Omega_{R,T})}\leq C(R,T).
\end{align*}
This, along with \eqref{limit-for-approximate solution}, implies that there exists $\lambda\in(0,1)$ such that
\begin{align*}
\|\vec{v}^{\mu}-\vec{v}\|_{L^2(\Omega_{R,T})}\leq C\|\vec{v}^{\mu}-\vec{v}\|_{L^1(\Omega_{R,T})}^{1-\lambda}
\|\vec{v}^{\mu}-\vec{v}\|_{L^q(\Omega_{R,T})}^{\lambda}\to 0
\end{align*}
as $\mu\to 0^+$.
Similarly, for any $t\geq0$, we have  by \eqref{limit-for-approximate solution-t} that there exists $\lambda'\in(0,1)$ such that
$
\|\vec{v}^{\mu}(t)-\vec{v}(t)\|_{L^2(B_R)}\leq C\|\vec{v}^{\mu}(t)-\vec{v}(t)\|_{L^1(B_R)}^{1-\lambda'}
\|\vec{v}^{\mu}(t)-\vec{v}(t)\|_{L^q(B_R)}^{\lambda'}\to 0.
$
With the $L^2$ convergence of $\{\vec{v}^{\mu}\}$, one can verify that $\vec{v}$ is a weak solution of the 2D Euler equation by a similar argument to (A)-(C) in the proof of Theorem 10.2 in \cite{Majda-Bertozzi02}.
\end{proof}
\begin{Corollary}\label{vorticity L123}
Let $\vec{v}$ be the weak solution (obtained in Theorem \ref{existence of weak solution to the 2D Euler equation-thm}) to the 2D Euler  equation with the initial data $\tilde \omega(0)=\tilde \omega_0\in Y_{non}$, and $\tilde \omega(t)=\curl(\vec{v}(t))$ for $t\geq0$.
Then up to a subsequence,
\begin{align}\label{tilde-omega-kappa-weak convergence L1L2}
\tilde \omega^{\mu}(t)\rightharpoonup\tilde \omega(t) \text{ in }L^j(\Omega),\;\;\;\;y\tilde \omega^{\mu}(t)\rightharpoonup y\tilde \omega(t) \text{ in }L^1(\Omega),
\end{align}
$
\|\tilde \omega(t)\|_{L^j(\Omega)}\leq \|\tilde\omega (0)\|_{L^j(\Omega)}
$, $\|y\tilde \omega(t)\|_{L^1(\Omega)}\leq C(t)$,
and $\tilde \omega(t)\leq 0$ almost everywhere on $\Omega$
for all $t\geq0$ and $j=1,2$.
\end{Corollary}
\begin{proof}
By Corollary \ref{y-tilde-omega-pseudoenergy-conserved} (1) and the $L^j$ vorticity control of the approximate solution sequence $\{\vec{v}^{\mu}\}$, we obtain \eqref{tilde-omega-kappa-weak convergence L1L2} for $t\geq0$. It then follows from Lemma \ref{tilde-omega0-kappa-properties} (4) that
\begin{align*}
\|\tilde \omega(t)\|_{L^j(\Omega)}\leq \liminf_{\mu\to0^+}\|\tilde\omega^{\mu} (t)\|_{L^j(\Omega)}=\liminf_{\mu\to0^+}\|\tilde\omega^{\mu} (0)\|_{L^j(\Omega)}=\|\tilde\omega (0)\|_{L^j(\Omega)}
\end{align*}
for $j=1,2$. By Corollary \ref{y-tilde-omega-pseudoenergy-conserved} (1),
$\|y\tilde \omega(t)\|_{L^1(\Omega)}\leq \liminf_{\mu\to0^+}\|y\tilde\omega^{\mu} (t)\|_{L^1(\Omega)}\leq C(t)$.
Suppose that there exist $t_0>0$ and $E_0\subset\Omega$ such that $|E_0|>0$ and   $\tilde \omega(t_0)>0$ on $E_0$. We assume that $|E_0|<\infty$ without loss of generality. Let $\varphi\equiv1 $ on $E_0$ and $\varphi\equiv0$ on $\Omega\setminus E_0$. Then $\varphi\in L^2(\Omega)$ and
\begin{align*}
0<\iint_{E_0}\tilde \omega(t_0)dxdy=&\iint_{\Omega}\tilde \omega(t_0)\varphi dxdy=\lim_{\mu\to0^+}\iint_{\Omega}\tilde \omega^\mu(t_0)\varphi dxdy\\
=&\lim_{\mu\to0^+}\iint_{E_0}\tilde \omega^\mu(t_0)dxdy\leq 0,
\end{align*}
which is a contradiction.
\end{proof}}
\end{appendix}

\section*{Acknowledgement}
This work was completed largely while ZL was at the Georgia Institute of Technology, where he was supported by the NSF under Grant DMS-2007457. ZL is currently supported by the National Natural Science Foundation of China under Grant 12494544. HZ is supported by the National Key R\&D Program of China under Grant 2021YFA1002400, the National Natural Science Foundation of China under Grants 12101306 and 12471229, and the Natural Science Foundation of Jiangsu Province under Grants BK20210169 and BK20250169. The authors also thank Siqi Ren for helpful discussions concerning Taylor's solutions.

\end{CJK*}

\begin{thebibliography}{99}
\bibitem{Adams75} R. A. Adams,  Sobolev spaces, Pure and Applied Mathematics, Vol. 65. Academic Press [Harcourt Brace Jovanovich, Publishers], New York-London, 1975. xviii+268 pp.
\bibitem{Arnold65} V. I. Arnol$'$d,  Conditions for nonlinear stability of stationary plane curvilinear flows of an ideal
fluid, Sov. Math. Dokl., 6 (1965), 773-776.
\bibitem{Arnold69}  V. I. Arnol$'$d, On an a priori estimate in the theory of hydrodynamical stability, Am. Math. Soc.
Transl.: Series 2, 79 (1969), 267-269.
\bibitem{BD01} A. Barcilon, P. G. Drazin,  Nonlinear waves of vorticity, Stud. Appl. Math., 106 (2001),  437-479.
\bibitem{BKM84}  J. T. Beale, T. Kato, A. Majda,  Remarks on the breakdown of smooth solutions for the 3-D Euler equations, Comm. Math. Phys., 94 (1984),  61-66.
\bibitem{Benjamin-Feir1967} T. B. Benjamin, J. E. Feir,  The disintegration of wave trains on deep water. Part 1.
Theory, J. Fluid Mech. 27 (1967), 417-437.
\bibitem{Berti-Maspero-Ventura2022} M. Berti,  A. Maspero, P.  Ventura,  Full description of Benjamin--Feir instability of Stokes waves in deep water, Invent. Math., 230 (2022), 651-711.
\bibitem{Bogatov-Kichenassamy22} E. M. Bogatov, S. Kichenassamy, The solution of Liouville's equation (1850, 1853) and its impact, 	arXiv: 2205.04246.

\bibitem{Bondeson1983} A. Bondeson, Linear analysis of the coalescence instability,
 Phys. Fluids,  26  (1983), 1275-1278.
 \bibitem{Bridges-Mielke1995}
T. J. Bridges, A. Mielke,  A proof of the Benjamin--Feir instability, Arch. Rational Mech. Anal., 133 (1995),  145-198.
\bibitem{Bronski-Hur-Johnson2016}
 J. C. Bronski, V. M. Hur, M. A. Johnson, Modulational instability in equations of
KdV type. New approaches to nonlinear waves, pp. 83-133. Lecture Notes in Physics,
vol. 908. Springer, Cham, 2016.
\bibitem{Byerly59} W. E. Byerly,  An Elementary Treatise on Fourier's Series, and Spherical, Cylindrical, and Ellipsoidal Harmonics, with Applications to Problems in Mathematical Physics. New York: Dover, 1959.
\bibitem{Chen-Su2020}   G. Chen, Q. Su, Nonlinear modulational instabililty of the Stokes waves in 2d full water
waves, arXiv: 2012.15071.
\bibitem{Courant-Hilbert53}   R. Courant, D. Hilbert, Methods of Mathematical Physics, Interscience Publishers, New York, 1953.
\bibitem{Constantin-Crowdy-Krishnamurthy-Wheeler2021} A. Constantin, D. G. Crowdy, V. S. Krishnamurthy, M. H. Wheeler,  Stuart-type polar vortices on a rotating sphere, Discrete Contin. Dyn. Syst., 41 (2021),  201-215.
\bibitem{Constantin-Krishnamurthy2019}
 A. Constantin, V. S. Krishnamurthy,  Stuart-type vortices on a rotating sphere, J. Fluid Mech., 865 (2019), 1072-1084.

     \bibitem{Crowdy97}  D. G. Crowdy,  General solutions to the 2D Liouville equation, Internat. J. Engrg. Sci., 35 (1997), 141-149.
 \bibitem{Crowdy04} D. G. Crowdy,  Stuart vortices on a sphere, J. Fluid Mech., 498 (2004), 381-402.



 \bibitem{Dacorogna} B. Dacorogna, Weak continuity and weak lower semicontinuity of non-linear functionals, Vol. 922, Springer, 2006.
\bibitem{Dauxois-Fauve-Tuckerman1996} T. Dauxois, S. Fauve, L. Tuckerman,  Stability of periodic arrays of vortices, Phys. Fluids, 8 (1996), 487-495.
\bibitem{DiPerna-Majda87} R. J. DiPerna,  A. J. Majda, Concentrations in regularizations for 2-D incompressible flow, Comm. Pure Appl. Math., 40 (1987),  301-345.
 \bibitem{Dominguez-Heuer-Sayas11}  V. Dominguez, N. Heuer, F.-J. Sayas,  Hilbert scales and Sobolev spaces defined by associated Legendre functions,
    J. Comput. Appl. Math.,
235 (2011),  3481-3501.
\bibitem{Dunkerton-Montgomery-Wang2009}
T. J. Dunkerton, M. T. Montgomery, Z. Wang,  Tropical cyclogenesis in a tropical wave critical layer: easterly waves, Atmos. Chem. Phys., 9  (2009), 5587-5646.
\bibitem{Fadeev et al-1965} V. M. Fadeev, I. F. Kvabtskhava, N. N. Komarov,
 Self-focusing of local plasma currents,
Nucl. Fusion, 5 (1965), 202-209.
\bibitem{Finn-Kaw1977}  J. M. Finn, P. K. Kaw, Coalescence instability of magnetic islands, Phys. Fluids,  20 (1977), 72-78.
\bibitem{Fleischer1998} J. Fleischer, Nonlinear evolution of self-gravitating plasmas,  Phys. Scr.,  T74 (1998), 86-88.

\bibitem{holm1986nonlinear}
D. D. Holm, J. E. Marsden, T. Ratiu, Nonlinear stability of the
Kelvin--Stuart cat's eyes flow, In: AMS. 1986.
\bibitem{holm1985nonlinear} D. D. Holm, J. E. Marsden, T. Ratiu, A. Weinstein,  Nonlinear stability of fluid and plasma equilibria, Phys. Rep., 123 (1985),  116 pp.
\bibitem{Jin-Liao-Lin2019} J. Jin, S. Liao, Z. Lin,  Nonlinear modulational instability of dispersive PDE models, Arch. Ration. Mech. Anal., 231 (2019),  1487-1530.
\bibitem{Kato1986nonlinear}  T. Kato, Remarks on the Euler and Navier-Stokes equations in $\mathbf{R}^2$, Nonlinear functional analysis and its applications, Part 2 (Berkeley, Calif., 1983), 1-7, Proc. Sympos. Pure Math., 45, Part 2, Amer. Math. Soc., Providence, RI, 1986.
\bibitem{kelly1967stability}
R. E. Kelly, On the stability of an inviscid shear layer which is periodic in space and
time,  J. Fluid Mech., 27.4 (1967),  657-689.
\bibitem{kelvin1880disturbing}
L. Kelvin, On a disturbing infinity in Lord Rayleigh's solution for waves in a
plane vortex stratum,  Nature, 23.1 (1880), 45-46.
\bibitem{Klaassen-Peltier1987}
G. P. Klaassen, W. R. Peltier,   Secondary instability and transition in finite amplitude
Kelvin-Helmholtz billows, In Proc. Third Intl Symp. on Stratified Flows, 3-5 Feb. 1987,
Pasadena, California, Vol. I.
\bibitem{Klaassen-Peltier1989}
 G. P. Klaassen, W. R. Peltier,  The role of transverse secondary instabilities in the evolution of free shear layers, J. Fluid Mech., 202 (1989), 367-402.
 \bibitem{Klaassen-Peltier1991}
 G. P. Klaassen, W. R. Peltier, The influence of stratification on secondary instability in
free shear layers, J. Fluid Mech., 227 (1991), 71-106.
 \bibitem{Krishnamurthy2019}
V. S. Krishnamurthy, M. H. Wheeler, D. G. Crowdy,  A. Constantin, Steady
point vortex pair in a field of Stuart-type vorticity, J. Fluid Mech., 874 (2019), R1, 11 pp.
 \bibitem{Krishnamurthy2021}
V. S. Krishnamurthy, M. H. Wheeler, D. G. Crowdy,  A. Constantin, Liouville chains: new hybrid vortex equilibria of the two-dimensional Euler equation, J. Fluid Mech., 921 (2021), Paper No. A1, 35 pp.
\bibitem{lamb1932hydrodynamics}
H. Lamb, Hydrodynamics, 6th edition, C.U.P, 1932.

\bibitem{Laugesen2011}
R. S. Laugesen, Spectural theory of partial differential equations, Lecture Notes, 2011.

\bibitem{Liouville1853} J.  Liouville, Sur l$'$\'{e}quation aux diff\'{e}rences partielles ${d^2\log\lambda\over du dv}\pm{\lambda\over 2a^2}=0$, J. Math. Pures Appl., 18 (1853), 71-72.
\bibitem{lin2003instability}
Z. Lin, Instability of some ideal plane flows, SIAM J. Math. Anal., 35 (2003),  318-356.

\bibitem{lin2004some}
Z. Lin, Some stability and instability criteria for ideal plane flows,  Comm. Math. Phys., 246 (2004),  87-112.
\bibitem{lin2021linear}
Z. Lin,  Linear instability of Vlasov-Maxwell systems revisited-a Hamiltonian
approach,   Kinet. Relat. Models, 15 (2022), 663-679.
\bibitem{lin2022instability}
Z. Lin, C. Zeng,  Instability, index theorem, and exponential trichotomy for linear Hamiltonian PDEs, Mem. Amer. Math. Soc., 275 (2022), no. 1347, v+136 pp.
\bibitem{lin2020separable}
Z. Lin, C. Zeng, Separable Hamiltonian PDEs and Turning point
principle for stability of gaseous stars,  Comm. Pure Appl. Math., 75 (2022),  2511-2572.

\bibitem{Majda-Bertozzi02}
A. J. Majda, A. L. Bertozzi,  Vorticity and incompressible flow, Cambridge Texts in Applied Mathematics, 27. Cambridge University Press, Cambridge, 2002. xii+545 pp.
\bibitem{Martin2018}
C. I. Martin,   On the vorticity of mesoscale ocean currents,
Oceanography, 31 (2018), 28-35.
\bibitem{Milisic-Razafison13}
V. Milisic, U. Razafison, Weighted Sobolev spaces for the Laplace equation in periodic infinite
strips, 2013, ffhal-00728408v2f.
\bibitem{Milisic-Razafison16}
V. Milisic, U. Razafison, Weighted $L^p$-theory for Poisson, biharmonic and Stokes problems on periodic
unbounded strips of $\mathbb{R}^n$, Ann. Univ. Ferrara Sez. VII Sci. Mat., 62 (2016),  117-142.
\bibitem{Morrey1966} C. B. Morrey, Multiple integrals in the calculus of variations,
Springer-Verlag, Berlin, 1966.
\bibitem{Nguyen-Strauss2023} H. Q. Nguyen, W. A.  Strauss,  Proof of modulational instability of Stokes waves in deep water, Comm. Pure Appl. Math., 76 (2023), 1035-1084.


\bibitem{pierrehumbert1982two}
R. T. Pierrehumbert, S. E. Widnall, The two- and three-dimensional instabilities of a spatially periodic shear layer, J. Fluid Mech., 114 (1982), 59-82.
 \bibitem{Pontin-Priest2022} D. I. Pontin, E. R. Priest,  Magnetic reconnection: MHD theory and modelling, Living Reviews in Solar Physics,  19.1 (2022), 1-202.
\bibitem{Pritchett-Wu1979} P. L. Pritchett, C. C. Wu, Coalescence of magnetic islands, Phys. Fluids,  22 (1979),  2140-2146.
\bibitem{Priest1985} E. R. Priest, The magnetohydrodynamics of current sheets, Rep. Prog. Phys., 48 (1985), 955-1090.
\bibitem{Priest-Forbes2000}  E. R. Priest, T. G. Forbes, Magnetic Reconnection: MHD Theory and Applications,  Cambridge University Press, 2000.
\bibitem{Richardson1921}
O. W. Richardson,  The Emission of Electricity from Hot Bodies, 2nd ed. London: Longmans,
Green and Co., 1921. 320 p.
\bibitem{Rossi-Doorly-Kustrin2013}
L. Rossi, D. Doorly, D. Kustrin,
Lamination, stretching, and mixing in cat's eyes flip sequences with varying periods, Phys. Fluids 25 (2013), 073604.
\bibitem{Sakajo2019}
T. Sakajo,   Exact solution to a Liouville equation with Stuart vortex distribution on the
surface of a torus, Proc. A., 475 (2019), 20180666.
\bibitem{Sauer2014} J. Sauer, Weighted resolvent estimates for the spatially periodic Stokes equations, Ann.
Univ. Ferrara, (2014), 1-22.
\bibitem{Schindler2006} K. Schindler, Physics of Space Plasma Activity, Cambridge University Press, 2006.
\bibitem{Schmid-Burgk1965}
J. Schmid-Burgk, Zweidimensionale selbstkonsistente L$\ddot{o}$sungen station$\ddot{a}$ren Wlassov-gleichung f$\ddot{u}$r Zweikomponentenplasmen, Ludwig-Maximilians-Universit$\ddot{a}$t, M$\ddot{u}$nchen, Diplomarbeit, 1965.
\bibitem{Shukla-Sen1996}
P. K. Shukla, A. Sen, Dusty vortex streets in Saturn's rings, Physica Scripta, T63 (1996), 275-276.
\bibitem{stuart1967finite}
J. T. Stuart, On finite amplitude oscillations in laminar mixing layers,  J. Fluid Mech., 29.3 (1967),  417-440.
\bibitem{Suetin2001}
 P. K. Suetin,   Ultraspherical polynomials, Encyclopedia of Mathematics, EMS Press, 2001.
 \bibitem{Tabeling-Perrin-Fauve1987}
P. Tabeling, B. Perrin, S. Fauve, Instability of a linear array of forced vortices, Europhys. Lett., 3 (1987),  459-465.
\bibitem{Tassi2022} E. Tassi, Formal stability in Hamiltonian fluid models for
plasmas, J. Phys. A: Math. Theor., 55 (2022), 413001 (82pp).
\bibitem{Taylor2018} M. Taylor, Curvature, conformal mapping, and 2D stationary fluid flows, Preprint, 2018, 1-7.

 \bibitem{Weidmann80}  J. Weidmann, Linear Operators in Hilbert Spaces, Grad. Texts in Math., Vol. 68, Springer-Verlag, Berlin, 1980.
 \bibitem{Yoon20} J. Yoon,  H. Yim,  S.-C. Kim,   Stuart vortices on a hyperbolic sphere, J. Math. Phys.,
61 (2020), 023103.

\end{thebibliography}
\end{document}